\documentclass[11pt, a4paper]{amsart}
\usepackage{amssymb,latexsym,amsmath}
\input{xypic}
\xyoption{all}

\title[\bf Eisenstein cohomology and $L$-values]{Eisenstein Cohomology for ${\rm GL}_N$ and ratios of critical values of Rankin--Selberg $L$-functions}

\author{ \bf G\"unter Harder \ \ \and \ \ A. Raghuram} 
 
\date{June 13, 2015}      
\subjclass[2010]{11F75; 11F66, 11F67, 11F70, 22E55}

\address{G\"unter Harder: Max Planck Institut f\"ur Mathematik, 7 Vivatsgasse, 53111 Bonn, Germany.}
\email{harder@mpim-bonn.mpg.de}

\address{A. Raghuram: Department of Mathematics\\ Indian Institute of Science Education and Research, Dr. Homi Bhabha Road, Pashan, Pune Maharashtra 411008, India.} 
\email{raghuram@iiserpune.ac.in}

\numberwithin{equation}{section}   
\newtheorem{lemma}[equation]{Lemma}
\newtheorem{satz}[equation]{Theorem}
\newtheorem{thm}[equation]{Theorem}
\newtheorem{prop}[equation]{Proposition}
\newtheorem{cor}[equation]{Corollary}

\newtheorem{defn}[equation]{Definition}
\newtheorem{rem}[equation]{Remark}

\usepackage{bm}
\usepackage{upgreek}
\makeatletter
\newcommand{\bfgreek}[1]{\bm{\@nameuse{up#1}}}

\let\oldtocsection=\tocsection
\let\oldtocsubsection=\tocsubsection
\let\oldtocsubsubsection=\tocsubsubsection
\renewcommand{\tocsection}[2]{\hspace{0em}\oldtocsection{#1}{#2}}
\renewcommand{\tocsubsection}[2]{\hspace{1em}\oldtocsubsection{#1}{#2}}
\renewcommand{\tocsubsubsection}[2]{\hspace{2em}\oldtocsubsubsection{#1}{#2}}

\begin{document}  
   
\maketitle

\centerline{\bf Appendix 1 by Uwe Weselmann}

\smallskip

\centerline{\bf Appendix 2 by Chandrasheel Bhagwat and A. Raghuram}

 \tableofcontents
 

 \def\R{\mathbb{R}}
\def\C{\mathbb{C}}
\def\Z{\mathbb{Z}}
\def\Q{\mathbb{Q}}
\def\A{\mathbb{A}}
\def\F{\mathbb{F}}
\newcommand\D{{\mathbb{ D}}}
 \def\L{\mathbb{L}}
\def\bG{\mathbb{G}}
\def\N{\mathbb{N}}

\newcommand\Dm{\D_\mu}
\newcommand\Dmp{\D_{\mu^\prime}}
 
\newcommand\cal{\mathcal}
\newcommand\SMK{{\cal S}^M_{K^M_f}}
\newcommand\tMZl{\tM_{\lambda,\Z}} 
\newcommand\Gm{{\mathbb G}_m}
\newcommand\cA{\cal A}
\newcommand\cC{\cal C}
\newcommand\calL{\cal L}
\newcommand\cO{\cal O}
\newcommand\cU{\cal U}
\newcommand\cK{\cal K}   
\newcommand\cW{\cal W}     
\newcommand\HH{{\cal H}}
\newcommand\cF{\mathcal{F}} 
\newcommand\G{\mathcal{G}}
\newcommand\cB{\mathcal{B}}
\newcommand\cT{\mathcal{T}}
\newcommand\cS{\mathcal{S}}

\newcommand\Gl {{ \rm  GL}}
\newcommand\GL{{ \rm  GL}}
\newcommand\Gsp{{\rm Gsp}}
\newcommand\Lie {{ \rm Lie}} 
\newcommand\Sl{{ \rm  SL}}
\newcommand\SL{{ \rm  SL}}
\newcommand\SO {{ \rm  SO}}
\newcommand{\Sp}{\text{Sp}}

\def\ringO{\mathcal{O}}
\def\idealP{\mathfrak{P}} 
 \def\g{\mathfrak{g}}
\def\k{\mathfrak{k}}
\def\z{\mathfrak{z}}
\def\s{\mathfrak{s}}
\def\b{\mathfrak{b}}
\def\t{\mathfrak{t}}
\def\q{\mathfrak{q}}
\def\l{\mathfrak{l}}
\def\gl{\mathfrak{gl}}
\def\sl{\mathfrak{sl}}
\def\u{\mathfrak{u}} 
\def\fp{\mathfrak{p}} 
\def\p{\mathfrak{p}} 
\def\r{\mathfrak{r}}
\def\fd{\mathfrak{d}}
\def\fR{\mathfrak{R}}
\def\fI{\mathfrak{I}}
\def\fJ{\mathfrak{J}}
\def\i{\mathfrak{i}}
\def\perm{\mathfrak{S}}
\newcommand\fg{\mathfrak g}
\newcommand\fk{\mathfrak k}
\newcommand\fgK{(\mathfrak{g},K_\infty^0)}
\newcommand\gK{ \mathfrak{g},K_\infty^0 }

\newcommand\ul{\underline} 
\newcommand\Spec{\hbox{\rm Spec}} 
\newcommand\SGK{\mathcal{S}^G_{K_f}}
\newcommand\SMP{\mathcal{S}^{M_P}}
\newcommand\SMQ{\mathcal{S}^{M_Q}}
\newcommand\SMp{\mathcal{S}^{M }_{K_f^M}}
\newcommand\SMq{\mathcal{S}^{M^\prime}_{K_f^{M^\prime}}}
\newcommand\uSMP{\ul{\mathcal{S}}^{M_P}}
\newcommand\SGnK{\mathcal{S}^{G_n}_{K_f}}
\newcommand\SG{\mathcal{S}^G}
\newcommand\SGKp{\mathcal{S}^G_{K^\prime_f}}
\newcommand\piKK{{ \pi_{K_f^\prime,K_f}}}
\newcommand\piKKpkt{\pi^{\pkt}_{K_f^\prime,K_f}}
\newcommand\BSC{ \bar{\mathcal{S}}^G_{K_f}}
\newcommand\PBSC{\partial\SGK}
\newcommand\pBSC{\partial\SG}
\newcommand\PPBSC{\partial_P\SGK}
\newcommand\PQBSC{\partial_Q\SGK}
\newcommand\ppBSC{\partial_P\mathcal{S}^G}
\newcommand\pqBSC{\partial_Q\mathcal{S}^G}
\newcommand\prBSC{\partial_R\mathcal{S}^G}
\newcommand \bs{\backslash} 
 \newcommand \tr{\hbox{\rm tr}}
\newcommand \Tr{\hbox{\rm Tr}}
\newcommand\HK{\mathcal{H}^G_{K_f}}
\newcommand\HKv{\mathcal{H}^G_{K_v}}
\newcommand\HGS{\mathcal{H}^{G,\place}}
\newcommand\HKp{\mathcal{H}^G_{K_p}}
\newcommand\HKpo{\mathcal{H}^G_{K_p^0}}
\newcommand\ch{{\bf ch}}

\newcommand\M{\mathcal{M}}
\newcommand\Ml{\M_\lambda}
\newcommand\tMl{\tilde{\Ml}}
\newcommand\tM{\widetilde{\mathcal{M}}}
\newcommand\tMZ{\tM_\Z}
\newcommand\tsigma{\tilde{\sigma}}
\newcommand \pkt{\bullet}
\newcommand\tH{\widetilde{\mathcal{H}}}
\newcommand\Mot{{\bf M}} 
\newcommand\eff{{\rm eff}}
\newcommand\Aql{A_{\q}(\lambda)}
\newcommand\wl{w\cdot\lambda}
\newcommand\wlp{w^\prime\cdot\lambda} 

\def\w{{\bf w}} 
\def\d{{\sf d}}
\def\e{{\bf e}} 
\def\x{{\tt x}}
\def\y{{\tt y}}
\def\v{{\sf v}} 
\def\ff{{\bf f}}
\def\bk{{\bf k}}
 
\def\Ext{{\rm Ext}}
\def\Hom{{\rm Hom}}
\def\Ind{{\rm Ind}}
\def\aInd{{}^{\rm a}{\rm Ind}}
\def\aIndPG{\aInd_{\pi_0(P(\R)) \times P(\A_f)}^{\pi_0(G(\R)) \times G(\A_f)}}
\def\aIndQG{\aInd_{\pi_0(Q(\R)) \times Q(\A_f)}^{\pi_0(G(\R)) \times G(\A_f)}}
\def\Gal{{\rm Gal}}
\def\End{{\rm End}} 
\newcommand\Coh{{\rm Coh}}  
\newcommand\Eis{{\rm Eis}}
\newcommand\Res{\text{Res}}
\newcommand\place{\mathsf{S}}
\newcommand\emb{\mathcal{I}} 
\newcommand\LB{\mathcal{L}}  
\def\Hod{{\mathcal{H}od}}
  
 \newcommand\ip{\pi_f\circ \iota} 
\newcommand\Wp{W_{\pi_\infty\times \ip}}
\newcommand\Wpc{W^{\text{cusp}}_{\pi_\infty\times \pi_f\circ\iota}}
\newcommand\Lcusp{  L^2_{\text{cusp}}(G(\Q)\bs G(\A_f)/K_f)}
\newcommand\MiC{\tM_{\iota\circ\lambda,\C} }
\newcommand\miC{\M_{\iota\circ\lambda,\C} }
\newcommand\tpl{{}^\iota}
\newcommand\Id{\rm Id}
\newcommand\Lr{L^{\text{rat}}}
\newcommand \iso{ \buildrel \sim \over\longrightarrow} 
\newcommand \into{\hookrightarrow}
\newcommand\ppfeil[1]{\buildrel #1\over \longrightarrow}
    
\def\bfpi{\mathbf{\Pi}}
\def\bfdelta{\mathbf{\Delta}}

 \newcommand{\set}[1]{\left\{#1\right\}}

\section{\bf Introduction}

The main aim of this article is to study rank-one Eisenstein cohomology for the group $\GL_N/F,$ where $F$ is a totally real field extension of  $\Q.$ This is then used to prove rationality results for ratios of successive critical values for Rankin--Selberg $L$-functions for $\GL_n \times \GL_{n'}$ over $F$  with the parity condition that $nn'$ is even. The key idea is to interpret Langlands's constant term theorem in terms of Eisenstein cohomology.  
In \cite{harder-raghuram} we had announced the main results for $\GL_n \times \GL_{n'}$ over $\Q$ and 
in the case when $n$ is even and $n'$ is odd. In the mean-time we realized that our methods and results can be extended to the case when both $n$ and $n'$ are even and also to the situation when the base field is a totally real field. 
We hope to consider other variations of the above context in future articles. 

\medskip

For the moment, let $G = {\rm Res}_{F/\Q}(\GL_n/F)$ where $F$ is a totally real field extension of $\Q.$ 
We fix a maximal torus $T$ inside a Borel subgroup $B$ of $G.$ 
For an open-compact subgroup $K_f \subset G(\A_f),$ where 
$\A_f$ is the ring of finite ad\`eles of $\Q$, let $\SGK$ be the ad\`elic locally symmetric space; 
see \S\,\ref{sec:loc-sym-space}. 
Let $E$ be a Galois extension of $\Q$ which contains a copy of $F.$ For a dominant integral weight $\lambda \in X^*(T)$ we let $\M_{\lambda, E}$ be the algebraic irreducible representation of $G \times E$ and let $\tM_{\lambda,E}$ be the associated sheaf on $\SGK;$ see \S\,\ref{sec:sheaf-coh}.  
The fundamental object of interest is the cohomology group $H^\bullet(\SGK, \tM_{\lambda, E}).$ The main tool that we use is the Borel--Serre compactification $\BSC$ of $\SGK,$ and if  $\partial\SGK$ is the boundary of $\BSC$, then we 
have the following long exact-sequence (\S\,\ref{sec:long-e-seq}): 
$$
\cdots  \longrightarrow H^i_c(\SGK, \tM_{\lambda, E}) 
\stackrel{\mathfrak{i}^*}{\longrightarrow}   H^i(\BSC, \tM_{\lambda,E}) 
\stackrel{\mathfrak{r}^*}{\longrightarrow } H^i(\partial \SGK, \tM_{\lambda,E}) 
\longrightarrow H^{i+1}_c(\SGK, \tM_{\lambda,E}) \longrightarrow \cdots
$$
The image of cohomology with compact supports inside the full cohomology is called {\it inner} or {\it interior} cohomology and is denoted 
$H^{\bullet}_{\, !} := {\rm Image}(\mathfrak{i}^*) = {\rm Im}(H^{\bullet}_c \to H^{\bullet}).$  All these cohomology groups are 
finite-dimensional vector spaces over $E.$ They are 
Hecke-modules, i.e., there is an action of $\pi_0(G(\R)) \times \HK.$ The inner cohomology is a semi-simple module under the Hecke algebra and if $E/\Q$ is large enough then we get an isotypical decomposition:
$$
H^{\bullet}_{\, !}(\SGK, \tM_{\lambda, E}) \ = \ 
\bigoplus_{\pi_f \in \Coh_!(G, K_f, \lambda)} H^{\bullet}_{\, !}(\SGK, \tM_{\lambda, E})(\pi_f), 
$$
where $\Coh_!(G, K_f, \lambda)$ is the finite set of isomorphism types of any absolutely irreducible $\HK$-module 
which occurs with (a finite) non-zero multiplicity in this decomposition. We may also pass to the limit over all open-compact $K_f$ and get an action of $\pi_0(G(\R)) \times G(\A_f)$ on $H^\bullet(\SG, \tM_{\lambda,E})$, 
and to retrieve the cohomology group for a particular level-structure $K_f$, we can take invariants: 
$H^\bullet(\SG, \tM_{\lambda,E})^{K_f} = H^\bullet(\SGK, \tM_{\lambda,E});$ the definitions of the Hecke algebra and such Hecke-actions are reviewed in \S\,\ref{sec:hecke}. 

\medskip

We may pass to a transcendental level by taking an embedding $\iota : E \to \C$, and then use the theory of automorphic forms on $\GL_n$ to study $H^\bullet(\SGK, \tM_{{}^\iota\!\lambda, \C})$. 
It is well-known that 
inner cohomology over $\C$ contains cuspidal cohomology, and is contained in square-integrable cohomology, i.e., 
we have a chain of inclusions: 
$$
 H_{\rm cusp}^\bullet(\SGK, \tM_{{}^\iota\!\lambda, \C}) \subset 
 H_{!}^\bullet(\SGK, \tM_{{}^\iota\!\lambda, \C})   \subset  
 H_{(2)}^\bullet(\SGK, \tM_{{}^\iota\!\lambda, \C})  \subset  H^\bullet(\SGK, \tM_{\lambda,E}) \otimes_{E,\iota}\C ; 
$$
see \S\,\ref{sec:filt}. 
The square-integrable cohomology, via the work of Borel and Garland, is captured by the discrete spectrum of $\GL_n$; 
see \S\,\ref{sec:borel-garland}. Furthermore, cuspidal cohomology is understood using results about the possible infinite components of cohomological cuspidal representations; we use here the fact that cuspidal representations are globally generic (i.e., have a Whittaker model) and hence locally generic; the local components at infinity are reviewed in 
\S\,\ref{sec:repn-infinity}.   

\medskip

Our first theorem in this paper is an arithmetic characterization of cuspidal cohomology. 
We identify a subspace $H^\bullet_{!!}(\SGK, \tM_{\lambda, E})$ 
of inner cohomology, that we call strongly-inner, which by definition is spanned by all those Hecke-modules inside inner cohomology whose isotypic component in global 
cohomology is captured already by the isotypic component in inner cohomology; 
see \S\,\ref{sec:coh-!!}. 
Strongly-inner cohomology splits off in global cohomology via a Manin--Drinfeld principle, and we get a canonical decomposition 
$$
H^\bullet(\SGK,\tM_{\lambda, E}) \ = \ 
H^\bullet_{!!}(\SGK, \tM_{\lambda, E}) \oplus H^\bullet_{\rm Eis}(\SGK, \tM_{\lambda, E}), 
$$
which gives an arithmetic definition of Eisenstein cohomology inside global cohomology; 
see (\ref{eqn:manin-drinfeld-easy}). If we go to a transcendental level via $\iota: E \to \C$ then we have:
$$
 H^\bullet_{!!}(\SGK, \tM_{{}^\iota\!\lambda, \C}) = H^\bullet_{\rm cusp}(\SGK, \tM_{{}^\iota\!\lambda, \C}).
$$
See Thm.\,\ref{thm:cuspidal-cohomology} where we summarize all the known characterizations of cuspidal cohomology. 

\medskip

The proofs of the assertions about strongly-inner cohomology involve understanding the cohomology of the boundary $H^\bullet(\partial\SGK,  \tM_{\lambda, E})$, and especially to be able to pick out the Hecke modules which do not appear in boundary cohomology; this is the subject matter of all of \S\,\ref{sec:bdry-cohomology}. The boundary 
$\partial\SGK$ is stratified as $\cup_P \partial_P\SGK$, with a stratum corresponding to every $G(\Q)$-conjugacy class of parabolic subgroups. There is a spectral sequence built from the cohomology of $\partial_P\SGK$ which may be described either at an arithmetic level or at a transcendental level; see \S\,\ref{sec:spectral-sequence}. Furthermore, for a single stratum, $H^\bullet(\partial_P\SGK, \tM_{\lambda, E})$ may be described in terms of the algebraic induction of the cohomology of the Levi quotient $M_P$ of $P$ with coefficient systems depending on $\lambda$ and the set $W^P$ of Kostant representatives for $P$ in the Weyl group of $G$; see Prop.\,\ref{prop:bdry-coh-2}. The proofs about strongly-inner cohomology also make use of multiplicity one and strong multiplicity one results of 
Jacquet--Shalika~\cite{jacquet-shalika-I} \cite{jacquet-shalika-II} and 
M\oe glin--Waldspurger~\cite{moeglin-waldspurger}.

\medskip

As a first application of strongly-inner cohomology, we describe how to attach certain periods which later play an important role in the results on special values of $L$-functions.  
For this paragraph we take $n$ to be \underline{even} and let $\pi_f \in \Coh_{!!}(G, \lambda)$, i.e., $\pi_f$ is an irreducible Hecke module contributing to strongly-inner cohomology. Taking $E$ to be large enough, we know that 
for every character $\varepsilon$ of $\pi_0(G(\R))$, and for the degree $\bullet$ being an extremal degree 
$b_n^F$ or $\tilde t_n^F$ (see Prop.\,\ref{prop:cohomology-degree}), the module 
$\pi_f \times \varepsilon$ appears in $H^\bullet_{!!}(\SGK, \tM_{\lambda, E})$ with multiplicity one. We fix an 
arithmetic identification $T^\varepsilon_{\rm arith}(\lambda, \pi_f)$ between the occurrences of 
$\pi_f \otimes \varepsilon$ and $\pi_f \otimes -\varepsilon$ in degree $b_n^F$; see (\ref{eqn:t-arith}). 
Then, we go to a transcendental level via $\iota : E \to \C$, and fix an identification 
$T^\varepsilon_{\rm trans}({}^\iota\!\lambda, {}^\iota\!\pi_f)$ between the occurrences of 
${}^\iota\!\pi_f \otimes \varepsilon$ and ${}^\iota\!\pi_f \otimes -\varepsilon$ in cuspidal cohomology in degree $b_n^F$ 
by a map described purely in terms of the relative Lie algebra cohomology groups at infinity; see \S\,\ref{sec:T-trans}.
We define a nonzero complex number, which we call a {\it relative period}, as the discrepancy between these two identifications: 
$$
\Omega^\varepsilon({}^\iota\!\lambda, {}^\iota\!\pi_f) \, T^\varepsilon_{\rm trans}({}^\iota\!\lambda, {}^\iota\!\pi_f) \ = \ 
T^\varepsilon_{\rm arith}(\lambda,\pi_f) \otimes_{E,\iota} {\bf 1}.
$$
Varying $\iota$, the family of relative periods attached to $\pi_f$, gives a well-defined element of $(E \otimes \C)^\times/E^\times.$

\medskip

As a second application of strongly-inner cohomology, we go back to boundary cohomology and prove a strong form of the Manin--Drinfeld principle, by showing that a certain $E$-subspace of 
$H^\bullet(\partial\SGK, \tM_{\lambda,E})$ splits off as an isotypical Hecke-module; see 
Thm.\,\ref{thm:manin-drinfeld-strong}. To explain this result, take $N = n+n'$ and now let 
$G = {\rm Res}_{F/\Q}(\GL_N/F)$, and $P$ be the maximal parabolic subgroup with Levi quotient 
$M_P = G_n \times G_{n'} := {\rm Res}_{F/\Q}(\GL_n/F) \times {\rm Res}_{F/\Q}(\GL_{n'}/F).$ Let $Q$ be the associate parabolic subgroup of $P$ with Levi quotient 
$M_Q = G_{n'} \times G_{n}.$
Take pure weights $\mu$ and $\mu'$ for $G_n$ and $G_{n'}$; see \S\,\ref{sec:pure}. 
Let $\sigma_f \in {\rm Coh}_{!!}(G_n, \mu)$ and $\sigma'_f \in {\rm Coh}_{!!}(G_{n'}, \mu')$. We make a crucial combinatorial assumption on the weights $\mu$ and $\mu'$, that there is a Kostant representative $w \in W^P$ such that 
its length $l(w)$ is ${\rm dim}(U_P)/2,$ and $w^{-1} \cdot (\mu \otimes \mu')$ is dominant as a weight for $G$.  An obvious necessary condition for the existence of such a $w$ is that ${\rm dim}(U_P) = nn'$ is even. Without loss of generality we will take $n$ even and $n'$ of any parity. 
This condition on $\mu$ and $\mu'$ has other equivalent formulations which are captured by our {\it combinatorial lemma}; see \S\,\ref{sec:com-lemma}. This lemma is proved by Uwe Weselmann in Appendix 1. 
A consequence is that the representation algebraically (un-normalized) parabolically induced from 
$\sigma_f  \otimes \sigma'_f$ appears in boundary cohomology: 
$$
\aInd_{P(\A_f)}^{G(\A_f)}(\sigma_f  \otimes \sigma'_f) \ \hookrightarrow 
H^{b_N^F}(\partial_P\SG, \tM_{\lambda, E}), 
$$
where $b_N^F$ is a special degree corresponding to the bottom-most degrees; see \S\,\ref{sec:simple-notation}.
Pick a level structure $K_f \subset G(\A_f)$ so that the induced representation has $K_f$-fixed vectors to get a 
Hecke-stable ${\sf k}$-dimensional subspace (for some ${\sf k}$) in $H^{b_N^F}(\partial_P\SG, \tM_{\lambda, E})^{K_f}.$  
Denote this ${\sf k}$-dimensional space as 
$I^\place_b(\sigma_f,\sigma^\prime_{f},\varepsilon')_{P, w}^{K_f}$, where $w$ is coming from the combinatorial lemma, and 
$\varepsilon'$ is a signature discussed in (\ref{eqn:epsilon'-n'-cases}). The element $w \in W^P$ has an associate 
$w' \in W^Q$ which is also balanced in the sense that $l(w') = {\rm dim}(U_Q)/2.$ Another consequence is that 
$
\aInd_{Q(\A_f)}^{G(\A_f)}(\sigma'_f(n)  \otimes \sigma_f(-n')) \ \hookrightarrow 
H^{b_N^F}(\partial_Q\SG, \tM_{\lambda, E}),$ 
and taking $K_f$-invariants we get a ${\sf k}$-dimensional Hecke-stable subspace, denoted  
$I^\place_b(\sigma_f,\sigma^\prime_{f},\varepsilon')_{Q, w'}^{K_f},$  
in $H^{b_N^F}(\partial_Q\SG, \tM_{\lambda, E})^{K_f}.$ 
Thm.\,\ref{thm:manin-drinfeld-strong} states that 
the $2{\sf k}$-dimensional space 
$$
I^\place_b(\sigma_f,\sigma^\prime_{f},\varepsilon)_{P, w}^{K_f} \  \oplus \ 
I^\place_b(\sigma_f,\sigma^\prime_{f},\varepsilon)_{Q, w'}^{K_f} 
$$
splits off as an isotypical Hecke-summand inside $H^{b_N^F}(\partial\SG, \tM_{\lambda,E})^{K_f}.$ Furthermore, the duals of these modules splits off as an isotypical summand inside 
$H^{\tilde t_N^F-1}(\partial\SG, \tM_{\lambda^{\sf v},E})^{K_f},$ where 
the degree $\tilde t_N^F-1$ is coming from the top-degrees; see \S\,\ref{sec:simple-notation}. For this introduction, we let 
$$
\fR \ : \ H^{b_N^F}(\partial\SG, \tM_{\lambda,E})^{K_f} \ \to \ 
I^\place_b(\sigma_f,\sigma^\prime_{f},\varepsilon)_{P, w}^{K_f}
 \oplus I^\place_b(\sigma_f,\sigma^\prime_{f},\varepsilon)_{Q, w'}^{K_f} 
 $$ 
 denote this Hecke-projection in bottom-degree.

\medskip

We now come to our main result on rank-one Eisenstein cohomology; see Thm.\,\ref{thm:rank-one-eis}. It states that the image of full cohomology under the composition of the maps $\fR \circ \r^*$ as in: 
$$
H^{b_N^F}( \SGK,\tM_{\lambda,E})  \ \stackrel{\r^*}{\longrightarrow} \ 
H^{b_N^F}(\partial \SGK,\tM_{\lambda,E}) \ \stackrel{\fR}{\longrightarrow} \ 
I^\place_b(\sigma_f,\sigma_f', \varepsilon')_{P,w}^{K_f} \ \oplus \ I^\place_b(\sigma_f,\sigma_f', \varepsilon')_{Q, w'}^{K_f}
$$
is a ${\sf k}$-dimensional $E$-subspace of the $2{\sf k}$-dimensional target-space.  There are two-aspects to the proof: 
(i) One aspect is purely cohomological which says that the Eisenstein part of boundary cohomology is an isotropic 
subspace under Poincar\'e duality (Prop.\,\ref{prop:isotropic-subspace}) bounding the dimension of the image by 
${\sf k}$;  
see \S\,\ref{sec:app-poincare}. (ii) The other aspect of the proof is transcendental and appeals to the theory of 
automorphic forms; it also gives information about the internal structure of this image. 
Base-change to $\C$ via an embedding $\iota : E \to \C,$ and then use Langlands's constant term theorem, recalled in Thm.\,\ref{thm:langlands}, 
which says that the constant term relative to $Q$ of an Eisenstein series built from a section $f$ of an induced
representation from $P$ is the same as the standard intertwining operator $T_{\rm st}$ applied to $f$. For the interpretation of Langlands's theorem in 
cohomology, the reader should compare the diagrams (\ref{eqn:diagram-without-coh}) and (\ref{eqn:diagram-with-coh}). 
This interpretation implies in particular that the required image contains all classes of the form $(\xi, T_{\rm st}^*\xi)$; see 
(\ref{eqn:image-description-PQ}), i.e., the image is at least a ${\sf k}$-dimensional subspace in a 
$2{\sf k}$-dimensional vector space. It is helpful to think about a situation when ${\sf k}=1$, then the image is a line in a two-dimensional space, and we will see later,  that the `slope' of this line carries arithmetic information about ratios of $L$-values. Putting both the aspects together, we conclude that 
$$
{\rm Image}(\fR \circ \r^*) \ = \ 
\left\{ \left(\xi, T^P_{\rm Eis}(\xi)\right) \ : \ \xi \in I^\place_b(\sigma_f,\sigma_f', \varepsilon')_{P,w}^{K_f} \right\}, 
$$
for an operator $T^P_{\rm Eis} : I^\place_b(\sigma_f,\sigma_f', \varepsilon')_{P,w}^{K_f} \to I^\place_b(\sigma_f,\sigma_f', \varepsilon')_{Q, w'}^{K_f}$ defined over $E$. 
It is of course possible that the 
Eisenstein series constructed from a section $f$ picks up a pole at the point of evaluation which happens to be 
$s = -N/2$, but in this case we prove in Prop.\,\ref{prop:at-least-one-holomorphic}, that the relevant Eisenstein series 
built from a section of the appropriate induced representation from $Q$ is holomorphic at its point of evaluation which is 
$s = N/2.$ The proof of this proposition uses the combinatorial lemma which says that existence of a balanced 
$w \in W^P$ as above is equivalent to the fact that the points $s = -N/2$ and $s = 1-N/2$ being critical for the Rankin--Selberg $L$-function $L(s, {}^\iota\!\sigma \times {}^\iota\!\sigma'^{\sf v}),$ where ${}^\iota\!\sigma'^{\sf v}$ denotes the contragredient representation of ${}^\iota\!\sigma';$ the proof also uses Shahidi's results on local factors.  In this case, the image would consist of classes of the form $(T^Q_{\rm Eis}(\psi), \psi)$ with $\psi \in 
I^\place_b(\sigma_f,\sigma_f', \varepsilon')_{Q, w'}^{K_f}.$

\medskip

Our main result on special values of $L$-functions, see Thm.\,\ref{thm:main-theorem-l-values}, follows from considering the `slope' mentioned above, i.e., we analyze classes of the form $(\xi, T^P_{\rm Eis}(\xi))$, or of the form 
$(T^Q_{\rm Eis}(\psi), \psi).$ Passing to a transcendental level, we rewrite such a class in terms of the standard intertwining operator which is given by an integral. Normalizing the local standard intertwining operator using appropriate $L$-values we get an operator denoted $T_{\rm loc}.$ At finite places one checks that it is nonzero and  holomorphic using results of M\oe glin--Waldspurger~\cite{moeglin-waldspurger}, and further more that it is rational. 
At an archimedean place, using Speh's results (see, for example, \cite[Theorem 10b]{moeglin-edinburgh}) on reducibility for induced representations for $\GL_N(\R)$, one sees that the induced representations at hand are irreducible; next, using Shahidi's results \cite{shahidi-duke85} on 
local factors and the fact that $-N/2$ and $1-N/2$ are critical, we deduce that the standard intertwining operator is both holomorphic and nonzero; and therefore induces an isomorphism at the level of relative Lie algebra cohomology groups which are one-dimensional, and after making certain careful choice of generators, this scalar turns out to be a power of $(2 \pi i)$ as shown in Harder~\cite{harder-arithmetic}. Then we descend down to an arithmetic level via the relative periods and the arithmetic identification mentioned earlier. This exercise gives us a rationality result for such a ratio of $L$-values divided by the relative periods. 
Before stating the result, let us mention that 
the $L$-functions at hand may be thought of from different points of view: (i) as motivic $L$-functions attached to the tensor product of the pure motives $\Mot(\sigma_f)$ and $\Mot(\sigma'_f)$ that are conjecturally attached to $\sigma_f$ and 
$\sigma'_f$; (ii) as cohomological $L$-functions attached to $\sigma_f$ and $\sigma'_f$, and this is entirely from the perspective of Hecke action on the cohomology of arithmetic groups; 
(iii) as automorphic Rankin--Selberg $L$-functions attached to a pair of cuspidal automorphic representations. The interplay between these three points of view is reviewed in \S\ref{sec:eff-motives-coh-l-fns}. For the rest of this introduction we  will talk in terms of the completed cohomological $L$-function $L^{\rm coh}(\iota, \sigma \times \sigma'^{\v}, s)$ attached to 
$\sigma_f$, $\sigma'_f$ and $\iota : E \to \C.$ Suppose $n$ is even and $n'$ is odd then 
Thm.\,\ref{thm:main-theorem-l-values} says: 
$$
\frac{1}{\Omega^{\varepsilon'}({}^\iota\!\sigma_f)} \, 
\frac{L^{\rm coh}\left(\iota, \sigma \times \sigma'^{\v}, {\sf m}_0\right)}
{L^{\rm coh}\left(\iota , \sigma \times \sigma'^{\v}, 1+{\sf m}_0\right)}
\ \in  \ 
\iota(E),
$$
and moreover this ratio of $L$-values divided by the period is well-behaved under the absolute Galois group of $\Q.$ 
On the other hand, if both $n$ and $n'$ are even, then we have 
$$
\frac{L^{\rm coh}\left(\iota, \sigma \times \sigma'^{\v}, {\sf m}_0\right)}
{L^{\rm coh}\left(\iota , \sigma \times \sigma'^{\v}, 1+{\sf m}_0\right)}
\ \in  \ 
\iota(E), 
$$
which is also well-behaved under Galois automorphisms. The point of evaluation ${\sf m}_0$ (which corresponds to 
the point 
$-N/2$ for the automorphic $L$-function) is given in 
\S\,\ref{sec:arith-comb-lem}, where we also discuss an important point concerning the combinatorial lemma: we may replace $\sigma_f$, say, by Tate twists $\sigma_f(k)$ for $k \in \Z$, and try to get other ratios of $L$-values; the lemma however imposes some restrictions on the set of such permissible $k$, and it turns out that we get a rationality result for 
exactly all the successive pairs of critical values, no more and no less. 

At some other occasion we will  take the integral structure on the cohomology into account. Then a slight refinement
of our reasoning implies that the periods $ \Omega^{\varepsilon'}({}^\iota\!\sigma_f) $ are well defined modulo
a group of units $\cO^\times_{E,S} $ where we inverted   primes out of a small well controlled set $S$ of primes.
Then we can speak of the prime factorization of the above expression and ask whether  primes
dividing the above expression are visible in the structure of the cohomology. A first instance of such an event
is discussed in  ~\cite{harder-cong}. In this note the choice of the periods was simply an educated guess.
This issue is also addressed in ~\cite{harder-p-adic}.

\medskip

The above results on critical values are compatible with Deligne's conjecture \cite{deligne} on the critical values of motivic $L$-functions. This compatibility is proved in Bhagwat--Raghuram~\cite{bhagwat-raghuram} by proving an appropriate period relation for the ratio $c^+/c^-$ of Deligne's periods for the tensor product motive
$\Mot(\sigma_f) \otimes \Mot(\sigma'_f).$ See also Bhagwat~\cite{bhagwat}. 
It is interesting to note that such period relations exist for all possible combinations of parities of $n$ and $n'$. However, the methods of this article seem to break-down when $n$ and $n'$ are both odd. In Appendix 2, Bhagwat and the second author show that the analogue of the combinatorial lemma does not hold: if we ask for a situation where we have two successive critical values then it is shown in 
Prop.\,\ref{prop:appendix-2} that there does not exist an element of the Weyl group, let alone a Kostant representative, which would be needed to arrange for an induced module to appear in boundary cohomology.

\medskip

To conclude the introduction, in summary, the principal result of this article is: 

\begin{itemize} 
\item[\Small{$\bullet$}] An algebraicity theorem for ratios of critical values. See Thm.\,\ref{thm:main-theorem-l-values}. 
\end{itemize}
Towards this result on $L$-values, the other highlights are:  
\begin{itemize} 
\item[\Small{$\bullet$}] An arithmetic characterization of cuspidal cohomology. See Thm.\,\ref{thm:cuspidal-cohomology}. 
\item[\Small{$\bullet$}] Definition of the relative periods. See Def.\,\ref{defn:relative-periods}.  
\item[\Small{$\bullet$}] A Manin-Drinfeld principle for boundary cohomology. See Thm.\,\ref{thm:manin-drinfeld-strong}. 
\item[\Small{$\bullet$}] The image of Eisenstein cohomology via the constant term theorem. See Thm.\,\ref{thm:rank-one-eis}. 
\end{itemize}

\bigskip

{\small
{\it Acknowledgements:} It is a pleasure for both the authors to acknowledge the support of several institutions and mathematical colleagues over the last few years during which time this project evolved. The project began during a  hike in the hills in the Black Forest near the Mathematisches Forchugsinstitut Oberwolfach (MFO) in 2008 when the authors met for a conference organized by Steve~Kudla and Joachim~Schwermer. It evolved further during meetings at 
the Erwin Sch\"odinger Institute, Vienna; their invitations are gratefully acknowledged.  
The authors also thank Don~Blasius, Harald~Grobner, Benedict~Gross, Michael~Harris and Freydoon~Shahidi for their interest and helpful discussions.  The second author is grateful to the Max Planck Institut f\"ur Mathematik (MPIM), an Alexander von Humboldt Fellowship, an NSF grant (DMS-0856113) and travel grants from Oklahoma State University and IISER Pune which funded his visits to MPIM on innumerable occasions. Finally, both authors are grateful to MFO for hosting them as part of their  `Research in Pairs' program  in May 2014 
when this version of the manuscript was written up.}

\section{\bf Cohomology of arithmetic subgroups of $\GL_n$}
\label{sec:coh-arithmetic}

\medskip
\subsection{The ad\`elic locally symmetric space}
\label{sec:loc-sym-space}

\medskip 
 \subsubsection{\bf The base field}
 Let $F$ be a totally real number field of degree ${\sf r} = {\sf r}_F =[F:\Q]$. By a number field we mean a finite extension of $\Q$. Let $\ringO_F$ be the ring of integers of $F$. 
Let $\place_\infty = {\rm Hom}(F,\R)$ denote the set of archimedean places. Let $\bar\Q$ be the field of algebraic numbers in $\C$. We identify the sets: 
$\Sigma_F := {\rm Hom}(F,\bar{\Q}) = {\rm Hom}(F, \C) =  {\rm Hom}(F,\R) = \place_\infty.$

\medskip
\subsubsection{\bf The groups}
Let  $G_0 = \GL_n/F$, and put 
$G  =   {\rm R }_{F/\Q}(G_0) = {\rm R }_{F/\Q}( {\GL_n/F} ).$
An $F$-group will be denoted by  $G_0,H_0,$ etc.,  
and the corresponding $\Q$-group, via Weil restriction of scalars, will be denoted  by the same letter without the subscript. For any $\Q$-algebra $A,$ we have 
$G (A) = {G}_0(A \otimes_\Q F) = \GL_n(A \otimes_\Q F).$ 
Let $ {B}_0$ stand for the standard Borel subgroup of $G_0$ consisting of all upper triangular matrices, ${T}_0$ the standard torus of all diagonal matrices, and ${U}_0$ the unipotent radical of $B_0$. The center of ${G}_0$ will be denoted by $ {Z}_0$. These groups define the corresponding $\Q$-groups $G  \supset B  = T U  \supset T \supset Z $. Let $S$ be the maximal $\Q$-split torus in $Z $; we identify $S \cong \mathbb{G}_m/\Q$, by sending $x \in \mathbb{G}_m(\ )$ to the diagonal matrix with all entries equal to $x.$ We have the norm character $N_{F/\Q} : Z \to \mathbb{G}_m$, and if we restrict to $S$ then it becomes $x \mapsto x^{\sf r}.$ The character group $X^*(\mathbb{G}_m) = \Z$, with the character 
$x \mapsto x^k$ denoted by $[k].$ 

Let $ G_0^{(1)}$ stand for the group ${\rm SL}_n/F$ and $G^{(1)} = R_{F/\Q}( G_0^{(1)} )$. 
The superscript $(1)$ will mean that we have intersected with $\SL_n$; for example, $T^{(1)} = T \cap \SL_n.$

\medskip
\subsubsection{\bf The symmetric space}
\label{sec:group-infinity}

For any topological group $\G$, let $\G^0$ stand for the connected  component of the identity, and 
$\pi_0(\G) = \G/\G^0$ stand for the group of connected components. 
Note that $G (\R) = \prod_{v \in \place_{\infty}} \GL_n(\R).$ Similarly, 
$Z (\R) \cong  \prod_{v \in \place_{\infty}} \R^\times,$ where each copy of 
$\R^\times$ consists of nonzero scalar matrices in the corresponding copy of $\GL_n(\R)$. The subgroup 
$S(\R)$ of $Z (\R)$ consists of 
$\R^\times$ diagonally embedded in $\prod_{v \in \place_{\infty}} \R^\times.$
The group $K_\infty^{(1)} :=  \prod_{v \in S_{\infty}} {\rm SO}(n)$ is a maximal compact subgroup of $G^{(1)}(\R)$ and the corresponding Cartan involution $\Theta$  on $G^{(1)}(\R)$ is given
by  $g\mapsto  {}^tg^{-1}$ on each factor. Similarly, $C_\infty :=  \prod_{v \in S_{\infty}} {\rm O}(n)$ is a maximal compact subgroup of $G(\R).$ Let $K_\infty = C_\infty S(\R) = C_\infty S(\R)^0,$
and so $ K_\infty^0  =   K_\infty^{(1)} \times S(\R)^0$. 
The torus $T^{(1)}\times_\Q\R$ is the maximal split torus
which is invariant under $\Theta.$ Let $T[2]$ denote the $2$-torsion subgroup of $T(\R)$, then 
$K_\infty =K_\infty^0 \cdot T[2]$. Inclusion induces a canonical identification 
$\pi_0(K_\infty )= \pi_0(G(\R))$ which is isomorphic to an ${\sf r}$-fold product of $\Z/2\Z$.
The (generalized)  symmetric space is defined as: 
$X := G(\R)/K^0_\infty.$
On this space we have an action of $T[2]$  which acts transitively 
on the set of connected components.

\medskip
\subsubsection{\bf The  ad\`elic locally symmetric space}  
\label{sec:adelic-space}
Let $\A$ be the ring of ad\`eles of $\Q$, which we decompose into  its
infinite and its finite part: $\A=\R\times\A_f$.
The group of ad\`eles is given by $G(\A) = G(\R) \times G(\A_f)$. Elements in the ad\`elic group are denoted by underlined letters
$\underline g, \underline h,$ etc.,  and the decomposition of an element
$\underline g$ into its finite and its infinite part will be denoted $\ul g = g_\infty \times{ \ul  g}_f.$
Let $K_f = \prod_p K_p \subset G(\A_f)$ be an open compact subgroup. 
The {\it ad\`elic symmetric space} is: 
$$ 
G(\A)/K_\infty  K_f \ = \ 
X \times (G(\A_f)/K_f) \ = \  G(\R)/K^0_\infty\times  (G(\A_f)/K_f). 
$$ 
It is a product of the symmetric space $X = G(\R)/K_\infty^0$ and an infinite discrete set $G(\A_f)/K_f$. 
On this space $G(\Q)$ acts  properly discontinuously  and we get a quotient
 \begin{equation}
 \label{eqn:map-pi}
 \pi \, : \, G(\R)/K_\infty^0 \times G(\A_f)/K_f  \ \longrightarrow \    
 G(\Q) \backslash \left( G(\R)/K_\infty^0 \times G(\A_f)/K_f \right). 
\end{equation}
The target space, called the {\it ad\`elic locally symmetric space of $G$ with level structure $K_f$}, is denoted: 
$$ 
\SGK \ := \ G(\Q)\backslash G(\A)/K_{\infty}^0K_f.
$$

\smallskip

To get an idea of what this space looks like, consider the action of $G(\Q)$ on the discrete space $G(\A_f)/K_f$.
It follows from classical finiteness results that this quotient is 
finite; let $\{\underline g_f^{(i)}\}_{i=1}^{i=m}$ be a finite set of representatives. 
The stabilizer of the coset  $\ul g_f^{(i)}K_f/K_f$ in
$G(\Q)$ is equal to 
$$
\Gamma_i \ := \ \Gamma^{\ul g_f^{(i)}} \ \ := \ \ G(\Q) \ \cap\ \ul g_f^{(i)} K_f(\ul g_f^{(i)})^{-1} 
$$ 
which is an arithmetic subgroup of $G(\Q)$ that acts properly discontinuously on $X$. 
We say $K_f$ is {\it neat} if all the subgroups $ \Gamma^{\ul g_f^{(i)}}$ are torsion free. For any choice of $K_f$ we can pass to a subgroup $K_f'$ of finite index in $K_f$ which is neat; we may take $K_f'$ to be normal in $K_f$.  
We have the following decomposition of $\SGK$: 
$$
\SGK \  \cong \ \coprod_{i=1}^m  \ \Gamma_i \backslash G(\R) / K_\infty^0, 
$$
Let $\ul g = g_\infty \times \ul g_f \in G(\A);$ there is a unique index $i$ such that $\ul g_f K_f = \gamma \ul g_f^{(i)} K_f$ for some $\gamma \in G(\Q);$ 
the map which sends 
$G(\Q) \, \ul g \, K_\infty^0 K_f$ in $\SGK$ to 
$\Gamma_i \gamma^{-1}g_\infty K_\infty^0$ sets up the required homeomorphism.

\medskip
\subsection{Highest weight modules $\M_\lambda$ and their associated sheaves $\tM_\lambda$}
\label{sec:sheaf-coh}

\medskip 
\subsubsection{\bf The character module of  $T/\Q$}
\label{sec:q-lambda}

Let $E/\Q$ be a Galois extension such that $ {\rm Hom}(F,E) \neq 0.$ We denote this set of embeddings 
by $\{\tau: F\to E\}$ on which we have a transitive action of the Galois group
$\Gal(E/\Q).$   The base change 
  $T\times_\Q E$  is a split torus; more precisely (we drop the subscript $\Q$) we have $T\times E=\prod_{ \tau: F\to E }T_0\times_{F,\tau}E=\prod_{ \tau: F\to E }T_0 .$  
  Often, the field $E$ will be taken to be large enough (so that, for example, some module can be split off over $E$) and we will analyze the behavior of objects as we change $E$ to a field $E'$ always subject to the condition of being Galois over $\Q$ and containing an isomorphic copy of $F$. We also allow $E=\C.$

\smallskip
 
 We consider the group of characters of the torus $T$ over $E$: $X^*(T \times E) = \Hom(T \times E, {\mathbb G}_m),$ 
on which there is a natural action of $\Gal(E/\Q).$ Since $T = R_{F/\Q}(T_0)$, we have: 
\begin{equation}
\label{eqn:X*T}
X^*(T \times E) = \bigoplus_{\tau:F\to E } X^*(T_0 \times_\tau E)=  \bigoplus_{\tau:F\to E } X^*(T_0  ) ,  
\end{equation}
and any element $\lambda \in X^*(T)$ is of the form
$\lambda \ = \  (\lambda^\tau)_{\tau : F \to E}, $ 
with $\lambda^\tau \in X^*(T_0 \times_\tau E)$ a character of the split torus $T_0 \times_\tau E$ (since, under our hypothesis on $E$, we have $F \otimes_\Q E = \prod_{\tau : F \to E} E$). Any $\eta \in \Gal(E/\Q)$ acts on $\lambda \in X^*(T \times E)$ by permutations:
\begin{equation}
\label{eqn:galois-lambda}
{}^\eta\lambda \ = \ (({}^\eta\lambda)^\tau)_{\tau : F \to E}  \ = \  (\lambda^{\eta^{-1}\circ\tau})_{\tau : F \to E} .
\end{equation}
It is easy to see that ${}^{\eta_1\eta_2}\lambda = 
{}^{\eta_1}({}^{\eta_2}\lambda)$ for all $\eta_1,\eta_2 \in  \Gal(E/\Q).$

\smallskip

Given $\lambda\in X^*(T \times E)$, define its rationality field $E(\lambda)$ as 
$$
E(\lambda) \ := \ E^{\{\eta \in \Gal(E/\Q) \ : \ {}^\eta\lambda = \lambda\}}, 
$$ 
i.e., it is the subfield of $E$ which is fixed by the subgroup of $\Gal(E/\Q)$ which fixes $\lambda.$  
It is easy to see that $E(\lambda)$ is the subfield of $E$ generated by the values of $\lambda$ on $T(\Q).$

\medskip
 \subsubsection{\bf The Weyl group and the pairing on $X^*(T^{(1)})$. }
The normalizer  $N(T)$  of the maximal torus is a subgroup of $G/\Q$ whose 
connected component of the identity is $T/\Q.$ 
The quotient $N(T)/T = \cW$ is a finite algebraic group scheme over $\Q.$ For the base change
$G\times E\supset N(T)\times E\supset T\times E $ we have $\cW(E) =N(T)(E)/T(E)$
and $\cW(E) =: W$ will be the absolute Weyl group which is a product 
$W=\prod_{\tau : F \to E} W_0^\tau$ with each $W_0^\tau$ isomorphic to $W_0$ the symmetric group in $n$ letters which is the Weyl group of $G_0/F$ with respect to $T_0/F.$ 
The Galois group $\Gal(E/\Q)$ acts on $W$ by permutations as in (\ref{eqn:galois-lambda}) above. 
There is a positive definite symmetric pairing 
$$
\langle\ , \ \rangle : X^*(T^{(1)} \times E) \ \times \ X^*(T^{(1)} \times E)  \ \longrightarrow \ \R,
$$
which is invariant under the action of $W.$  The direct sum decomposition (\ref{eqn:X*T}) for $T^{(1)}$
is orthogonal with respect to this pairing,  and on each summand it is unique up to a scalar.
Restriction of characters gives an inclusion $X^*(T \times E) \subset X^*(T^{(1)} \times E)\oplus X^*(Z\times E) .$ 
We extend the form $\langle \ , \ \rangle$ trivially by zero on $X^*(Z \times E);$ and the Weyl group action is also trivial on this summand.

\medskip
 \subsubsection{\bf Standard basis and fundamental basis for the group of characters}
\label{sec:standard-fundamental} 
For the moment we only consider $T_0/F.$
The character module  $X^*( {T}_0)$ is a free abelian group on $\{\e_1,\dots, \e_n\}$, where, for any $t = {\rm diag}(t_1,\dots,t_n) \in {T}_0( A)$, with $A$ an $F$-algebra, 
we have $\e_i(t) = t_i\in A^\times$. For integers $b_1,\dots,b_n$, the character 
$\lambda = (b_1,\dots,b_n) := \sum b_i \e_i$ is given by $\lambda(t) = \prod_i t_i^{b_i}$.  
We will call $\{\e_1,\dots, \e_n\}$ the {\it standard basis} for $X^*( {T}_0)$. For example, 
the determinant character is given by $\bfgreek{delta}_n := {\rm det} = (1,\dots,1) = \sum_{i=1}^n \e_i$. The simple roots for 
$\SL_n$ or $\GL_n$ are given by $\{\bfgreek{alpha}_1,\dots,\bfgreek{alpha}_{n-1}\}$  where 
$\bfgreek{alpha}_i = \e_i -  \e_{i+1}$.

\smallskip

The fundamental weights $\{\bfgreek{gamma}_1,\dots, \bfgreek{gamma}_{n-1}\}$ in 
$X_\Q^*( {T_0 })  := X^*( {T_0 }) \otimes \Q$  are defined by: 
$$ 
\frac{ 2\langle \bfgreek{gamma}_i, \bfgreek{alpha}_j  \rangle}
{\ \langle \bfgreek{alpha}_j, \bfgreek{alpha}_j \rangle} = \delta_{ij}, 
$$
and the restriction of each $\bfgreek{gamma}_i$ to $Z_0$ is zero. 
(This only makes sense if $\bfgreek{gamma}_i$ is in  $X_\Q^*( {T_0 }) := X^*(T_0) \otimes \Q.$)
In terms of the standard basis the fundamental weights  are given by: 
$$
\bfgreek{gamma}_i \ = \  
(\e_1+\cdots + \e_i) - \frac{i}{n} \bfgreek{delta}_n 
\ =  \  \left(1-\frac{i}{n},\dots,1-\frac{i}{n},-\frac{i}{n},\dots,-\frac{i}{n}\right). 
$$
The basis for $X^*_\Q(T_0)$ given by 
$\{\bfgreek{gamma}_1,\dots, \bfgreek{gamma}_i,\dots, \bfgreek{gamma}_{n-1}, \bfgreek{delta}_n \}$
will be called the {\it fundamental basis}. 
This basis respects the decomposition  $T_0 = T_0^{(1)} \cdot Z_0$ 
(up to isogeny), i.e., restriction of characters gives an inclusion  
$X^*( {T_0})\subset  X^*(T_0^{(1)}) \oplus 
X^*( {Z}_0)$ and after tensoring by $\Q$ this becomes an isomorphism: 
$X_\Q^*( {T_0}) \cong X_\Q^*( {T_0}^{(1)}) \oplus X_\Q^*( {Z}_0).$ 
 The  restriction to $T_0^{(1)}$ of  the fundamental weights $\{\bfgreek{gamma}_1, \dots, \bfgreek{gamma}_{n-1}\}$ span 
$X^*(T_0^{(1)})$, and the determinant character $\bfgreek{delta}_n$ spans $X_\Q^*( {Z}_0).$ 

For example, half the sum of positive roots for $ {G}_0$ or 
$ {G}^{(1)}_0$, is in $X_\Q^*( {T}^{(1)}_0),$ and is given by
$$
\bfgreek{rho}_n \ := \ \sum_{i=1}^{n-1} \gamma_i  \ = \ 
\left(\frac{n-1}{2}, \frac{n-3}{2},\dots,\frac{-(n-1)}{2} \right). 
$$

Given $\lambda \in X^*_\Q(T_0)$ we will write 
$$
\lambda \ = \ \sum_{i=1}^{n-1} (a_i-1)  \bfgreek{gamma}_i \ + \ d \cdot \bfgreek{delta}_n. 
$$
The $(a_i-1)$ might seem strange here, but has the virtue that it will simplify expressions later on. (See, for example,  the discussion below on motivic weight.) 
The above expression for $\lambda$ is the same as writing 
$\lambda + \bfgreek{rho}_n = \sum_{i=1}^{n-1} a_i \bfgreek{gamma}_i \ + \ d \cdot \bfgreek{delta}_n.$ 
We will denote $\lambda^{(1)}= \sum_i (a_i-1) \bfgreek{gamma}_i$ and call it the semi-simple part of $\lambda$, and similarly,  $\lambda_{\rm ab} = d \cdot \bfgreek{delta}_n$ will be called the abelian part of $\lambda$.

\smallskip

We describe the dictionary between the standard and the fundamental bases. 
Let $\lambda \in X^*_\Q(T_0)$ be written as 
$\lambda = \sum_{i=1}^{n-1} (a_i-1) \bfgreek{gamma}_i+ d \cdot \bfgreek{delta}_n$. 
Formally, as a `character' of $T_0$, it may be written as: 
$$
t = {\rm diag}(t_1,\dots,t_n) \mapsto \lambda(t) = 
t_1^{a_1+a_2+\dots+a_{n-1} -(n-1)} \cdot t_2^{a_2+\dots+a_{n-1} - (n-2)} \cdots 
t_{n-1}^{a_{n-1}-1} \cdot 
(t_1\cdots t_n)^{r_{\lambda}}, 
$$ 
where 
$$
r_{\lambda} := (nd - \sum_{i=1}^{n-1} i (a_i-1))/n.
$$ 
In the standard basis if $\lambda = (b_1,\dots,b_n)$ then 
\begin{eqnarray*}
b_1 & = & a_1 + a_2 + \dots + a_{n-1} - (n-1) + r_{\lambda}, \\
b_2 & = & a_2 + \dots + a_{n-1} - (n-2) + r_{\lambda},  \\
 & \vdots & \\ 
b_{n-1} & = & a_{n-1} - 1 + r_{\lambda}, \\
b_n & = & r_{\lambda}.
\end{eqnarray*}
Conversely, 
\begin{eqnarray*}
a_i  - 1 & = & b_i - b_{i+1}, \quad  {\rm for}\  1 \leq i \leq n-1, \ {\rm and} \\
d & = & (b_1+\dots+b_n)/n.
\end{eqnarray*}

\smallskip
 
 All the above notations are adapted for characters of $T/\Q = R_{F/\Q}(T_0).$ Hence, for  
$\lambda = (\lambda^\tau)_{\tau : F \to E}$ in $X^*_\Q(T \times E )$, we have: 
\begin{equation}
\label{eqn:lambda-def}
\lambda^\tau 
\ = \ \sum_{i=1}^{n-1} (a_i^\tau-1) \bfgreek{gamma}_i \ + \ d^\tau \cdot \bfgreek{delta}_n
\ = \ (b^\tau_1, \dots, b^\tau_n). 
\end{equation}
We define $\rho_n=(\bfgreek{rho}_n)_{\tau:F\to E}$ then we  can form $\lambda+\rho_n$ and get 
 \begin{equation}
\label{eqn:lambdaplusrho-def}
(\lambda+\rho_n)^\tau 
\ = \ \sum_{i=1}^{n-1}  a_i^\tau  \bfgreek{gamma}_i \ + \ d^\tau \cdot \bfgreek{delta}_n
 \end{equation}
We also put $\delta_n=(\bfgreek{delta}_n)_{\tau:F\to E}$; actually  we are only interested in the case that
$d^\tau=d$ for all $\tau,$ and in this case we write $\lambda=\lambda^{(1)} + d\delta_n;$ 
see Lem.\,\ref{dconst} below. 
 
\medskip
\subsubsection{\bf Integral weights}
\label{sec:integral}
Let $\lambda = \sum_{i=1}^{n-1} (a_i-1) \bfgreek{gamma}_i+ d \cdot \bfgreek{delta}_n  = (b_1,\dots,b_n)  \in 
X^*_\Q(T_0).$   
We say that $\lambda$ is an integral weight if and only if $\lambda \in X^*(T_0)$, which is the same as saying  that 
$b_i \in \Z$ for all $i.$ In terms of the fundamental basis: 
$\lambda \in X^*_\Q(T_0)$ is integral if and only if
\begin{equation}
\label{eqn:integral-lambda}
\lambda \in X^*(T_0) \ \Longleftrightarrow \ 
\left\{\begin{array}{l} 
a_i \in \Z, \quad 1 \leq i \leq n-1, \\
nd \in \Z, \\
nd \equiv \sum_{i=1}^{n-1} i (a_i-1) \pmod{n}.
\end{array}\right.
\end{equation}
This implies that $r_{\lambda} = (nd -  \sum_{i=1}^{n-1} ia_i)/n \in \Z$. 
A weight $\lambda = (\lambda^\tau)_{\tau : F \to E} \in X^*_\Q(T \times E)$ is integral if and only if each $\lambda^\tau$ is integral.

\medskip
\subsubsection{\bf Dominant integral weights}
\label{sec:dominant}
Let $\lambda = (b_1,\dots,b_n) = \sum_{i=1}^{n-1} (a_i-1) \bfgreek{gamma}_i+ d \cdot \bfgreek{delta} \in 
X^*(T_0)$ be an integral weight. 
Then $\lambda$ is dominant, for the choice of the Borel subgroup being the standard upper triangular subgroup $B_0$,  if and only if 
$b_1 \geq b_2 \geq \dots \geq b_n$, or equivalently, 
\begin{equation}
\label{eqn:dominant-lambda}
\lambda \ \mbox{is dominant} \ \Longleftrightarrow \ a_i \geq 1 \  \mbox{for $1 \leq i \leq n-1.$ \ 
(There is no condition on $d$.)}
\end{equation}
A weight $\lambda = (\lambda^\tau)_{\tau : F \to E} \in X^*_\Q(T  \times E)$ is dominant-integral if and only if each 
$\lambda^\tau$ is 
dominant-integral.  
Denote the set of dominant integral weights for $T_0$ (resp., $T$) by $X^+(T_0)$ (resp., $X^+(T \times E)$).

\medskip
\subsubsection{\bf The representation $(\rho_\lambda, \M_\lambda)$}
\label{sec:h-w-module}
Let $\lambda \in X^+(T \times E)$ be a dominant integral weight and $(\rho_\lambda, \M_\lambda)$ be the absolutely irreducible finite-dimensional representation of $G \times E(\lambda)$ of highest weight $\lambda.$  We can get hold of 
$\M_\lambda$ after going to a large enough Galois extension $E/\Q$ and descend down to $E(\lambda).$ Put 
 $$
 \M_{\lambda, E} \ = \ \bigotimes_{\tau : F \to E} \M_{\lambda^\tau}, 
 $$
 where $\M_{\lambda^\tau}/E$ is the absolutely irreducible finite-dimensional representation of 
 $G_0 \times_\tau E = \GL_n/F \times_\tau E$ with highest weight  $\lambda^\tau.$ If necessary, we will write 
 $(\rho_{\lambda^\tau}, \M_{\lambda^\tau})$ for this representation. We can produce a descent data and show that $\M_{\lambda,E}$ is defined over $E(\lambda)$, and an $E(\lambda)$-structure, which will be unique up to homotheties by $E^\times$, will be denoted $\M_\lambda.$ The group $G(\Q) = \GL_n(F)$ acts on $\M_{\lambda, E}$ diagonally, i.e., 
$\gamma \in G(\Q)$ acts on a pure tensor
$\otimes_\tau m_\tau$ via: 
\begin{equation}
\label{eqn:G(Q)-action}
\gamma \cdot (\otimes_\tau m_\tau) \ = \ \otimes_\tau (\tau(\gamma) \cdot m_\tau). 
\end{equation}

\smallskip
For an alternative constructive description of $\M_\lambda$ using an algebraic version of Borel-Weil-Bott theorem, take the flag variety of all Borel subgroups  
$B\backslash G$, which works over $\Z.$ (See, for example, Demazure-Grothendieck \cite[XXII, 5.8]{demazure-grothendieck}.) 
Let $w_0$ be the element of the Weyl group of longest length which is represented by an element of $G(\Q)$. 
The weight $w_0(\lambda)$ gives a line bundle $\LB_{w_0(\lambda)}$ on $B\backslash G \times E(\lambda).$ 
The line bundle $\LB_{w_0(\lambda)}$ has global sections if and only if $\lambda$ is dominant; the representation space of $\M_\lambda$ can be taken as  the $E(\lambda)$-vector space $H^0(B\backslash G, \LB_{w_0(\lambda)})$
of global sections.  Let $A(G)$ be the algebra of regular functions on $G/\Q.$ Then
\begin{eqnarray*}
\M_\lambda & = & H^0(B\backslash G, \LB_{w_0(\lambda)}) \\
& = & 
\{f \in A(G) \otimes E(\lambda) \ : \ f(bg) = w_0(\lambda)(b)f(g), \ \forall b \in B(E(\lambda)), \forall g \in G(E(\lambda))\}.
\end{eqnarray*}
The $\rho_\lambda$-action of $g \in G(E(\lambda))$ on any  $f \in \M_\lambda$ is by right-shifts: 
$(\rho_\lambda(g)(f))(g') = f(g'g)$ for all $g,g' \in G(E(\lambda)).$ 
The highest/lowest weight vector can be explicitly written down; see \cite{harder-p-adic}.

\medskip
\subsubsection{\bf The sheaf $\tM_\lambda$ and algebraic dominant-integral weights}
\label{sec:sheaves}
Let $\M$ be a finite dimensional $\Q$-vector space and 
$ \rho : G/\Q \to \GL (\M)$
be a rational representation of the algebraic group $G/\Q.$
This representation  $\rho$ provides a sheaf $\tM$ of $\Q$-vector spaces on $\SGK$: the sections 
over an open subset $V\subset \SGK$ are given by
\begin{equation}
\label{eqn:sections}
\tM(V) = \left\{ s: \pi^{-1}(V) \to \M \ | \ s \, \, \hbox
 {locally constant and} \,\,
s( \gamma v)  = \rho(\gamma) s(v) , \gamma\in G(\Q) \right\}, 
\end{equation}
where $\pi$ is as in (\ref{eqn:map-pi}). Take $\M = \M_{\lambda,E}$ as in \S\ref{sec:h-w-module}; 
it is explained in \cite[Chap.\,3]{harder-book} that the cohomology 
$H^\bullet(\SGK, \tM_{\lambda,E}) = 0$ if the central character of $\rho_\lambda$ 
is not the type of an algebraic Hecke character of $F.$ (See, also, \cite[1.1.3]{harder-inventiones}.) 
In our situation this looks as follows:  
 
\begin{lemma}\label{dconst}
Let $\lambda \in X^+(T \times E)$, and say $\lambda = (\lambda^\tau)_{\tau : F \to E}$ with  
$\lambda^\tau = \sum_{i=1}^{n-1} (a^\tau_i-1) \bfgreek{gamma}_i + d^\tau \cdot  \bfgreek{delta}.$ If there exist 
$\tau_1$ and $\tau_2$ in $\Hom(F,E)$ 
such that $d^{\tau_1} \neq d^{\tau_2}$ then $\tM_\lambda = 0$. 
Equivalently, if $\tM_\lambda \neq 0$ then 
$\lambda^{\tau_1}_{\rm ab} = \lambda^{\tau_2}_{\rm ab}$.
\end{lemma}

\begin{proof}
If some $d^{\tau_1} \neq d^{\tau_2}$ then consider the central character of $\rho_\lambda$ on a suitable congruence subgroup of the units in the diagonal center: $\ringO_F^\times \simeq S(\ringO_F)  \subset \GL_n(F) = 
G(\Q)$ to see that every stalk is zero, and hence the sheaf is the zero sheaf. 
\end{proof}

Define $X^+_{\rm alg}(T \times E)$ to be the subset of those dominant-integral weights which satisfy the {\it algebraicity} condition that $d^{\tau_1} =  d^{\tau_2}$ for all $\tau_1$ and $\tau_2$ in $\Hom(F,E)$, i.e., 
$$
X^+_{\rm alg}(T \times E) := \{\lambda \in X^+(T \times E) \ : \ 
\lambda^{\tau_1}_{\rm ab} = \lambda^{\tau_2}_{\rm ab}, \ \forall \tau_1, \tau_2 \in \Hom(F,E)\}.
$$ 
Henceforth, we will only consider such algebraic, dominant-integral highest weights $\lambda.$ Algebraicity means that 
the central character $\omega_{\lambda} \in X^*(T\times E)$ 
factors via the norm character $N_{F/\Q}$, i.e., 
\begin{equation}
\omega_{\lambda}(x) \ = \ N_{F/\Q}(x)^d, 
\end{equation}
and it's restriction to $S = {\mathbb G}_m$ is given by 
\begin{equation}
\label{eqn:trivial-on-S(R)^0}
\omega_{\lambda}(y) \ = \ y^{{\sf r}nd}.  
\end{equation}

If $K_f$ is neat then every stalk of $\tM_\lambda$ 
is isomorphic to $\M_\lambda$ and the sheaf $\tM_\lambda$ is a local system. Note that 
 $\tM_\lambda$ is a sheaf of $E(\lambda)$-vector spaces, and the base change  
$\M_{\lambda, E} = \M_\lambda \otimes_{E(\lambda)} E$ gives a sheaf $\tM_{\lambda,E}$ of $E$-vector spaces 
on $\SGK$.

\medskip
\subsection{Cohomology of the sheaves $\tM_\lambda$}

A fundamental problem is to understand the arithmetic information contained in the sheaf cohomology groups 
$H^\bullet(\SGK, \tM_{\lambda,E}).$

\medskip
\subsubsection{\bf Functorial properties upon changing the level structure and Hecke action}
\label{sec:hecke}

An inclusion $K_{1,f} \subset K_{2,f}$ of open-compact subgroups gives a canonical  map 
$H^\bullet (\mathcal{S}^G_{K_{2,f}}, \tM_{\lambda, E}) \to H^\bullet (\mathcal{S}^G_{K_{1,f}}, \tM_{\lambda, E}).$
Pass to the limit over all open compact subgroups $K_f$ and define: 
$$
H^\bullet (\SG, \tM_{\lambda, E}) \ := \  \varinjlim\limits_{K_f} \, H^\bullet (\SGK, \tM_{\lambda, E}). 
$$
On this limit, there is an action of $\pi_0(G(\R)) \times G(\A_f)$, which we now describe. 
An element $x_\infty \in \pi_0(G(\R)) = \pi_0(K_\infty)$ is represented by an element of $T[2]$ and so 
normalizes $K_\infty.$
Let $\ul x_f \in G(\A_f)$ and put $\ul x = x_\infty \times \ul x_f.$ 
Right multiplication by $\ul x$, i.e., $\ul g \mapsto \ul g \ul x$, induces a map 
$m_{\ul x} : \mathcal{S}^G_{K_f} \longrightarrow  \mathcal{S}^G_{ \ul x_f^{-1}K_f \ul x_f}.$
This gives a canonical map in cohomology: 
$$
m_{\ul x}^\bullet : H^\bullet(\mathcal{S}^G_{ \ul x_f^{-1}K_f \ul x_f}, m_{\ul x, *}\, \tM_{\lambda, E}) 
\longrightarrow 
H^\bullet(\mathcal{S}^G_{K_f}, \tM_{\lambda, E}).
$$
We have 
$m_{\ul x, *} \, \tM_{\lambda, E} = \tM_{\lambda, E}$ as sheaves on $\mathcal{S}^G_{ \ul x_f^{-1}K_f \ul x_f}.$ 
Hence, any $\ul x$ gives a canonical map
$$
m_{\ul x}^\bullet : H^\bullet(\mathcal{S}^G_{ \ul x_f^{-1}K_f \ul x_f},  \tM_{\lambda, E}) 
\longrightarrow 
H^\bullet(\mathcal{S}^G_{K_f}, \tM_{\lambda, E}). 
$$
Passing to the limit over all $K_f$ gives the action of $\ul x$ on $H^\bullet (\SG, \tM_{\lambda, E}).$ 
The cohomology for the spaces with level structure can then be retrieved by taking invariants: 
$$
H^\bullet(\SGK, \tM_{\lambda, E}) \ \cong \ 
H^\bullet (\SG, \tM_{\lambda, E})^{K_f}.
$$
Let $\HK = C^\infty_c(G(\A_f)/ \! \!/K_f)$ 
be the Hecke-algebra consisting of all locally constant and compactly supported bi-$K_f$-invariant $\Q$-valued functions on $G(\A_f)$; the algebra structure is given by convolution of functions where we take the Haar measure on $G(\A_f)$ to be the product of local Haar measures, and for every prime $p$, the local measure is normalized so that 
${\rm vol}(G(\Z_p)) = 1.$ We have a canonical action of 
$\pi_0(G(\R)) \times \HK$ on $H^\bullet(\SGK, \tM_{\lambda, E}).$

\medskip
\subsubsection{\bf Functorial properties upon changing the field $E$}
\label{sec:change-field-E}
Consider another field $E'$, also Galois over $\Q$  with an injection $\iota : E \to E'.$ 
Then $\iota$ induces an isomorphism
$X^*(T \times E) \longrightarrow X^*(T \times E')$ written as $\lambda \mapsto {}^\iota\!\lambda.$
 (The map $\tau \mapsto \iota \circ \tau$ is a bijection from 
$\Hom(F,E)$ onto $\Hom(F,E').$ Hence any $\tau' \in \Hom(F,E')$ is of the form $\tau' = \iota \circ \tau$ for a unique 
$\tau \in \Hom(F,E)$ and we put $\tau = \iota^{-1}\circ \tau'.$ If  $\lambda=(\lambda^\tau)_{\tau:F\to E} $ then 
$^\iota{}\lambda=(\lambda^{\iota\circ\tau})_{\iota\circ \tau :F\to E^\prime}$.) 
We get an identification $\iota^* : \M_{\lambda, E} \otimes_{E,\iota} E' \iso \M_{{}^\iota\!\lambda, E'}.$
 This yields an identification $\tM_{\lambda,E}\otimes_{E,\iota} E^\prime \iso \tM_{\lambda,E^\prime}$ of sheaves   
which induces an isomorphism: 
 $$
 \iota^\bullet :  H^\bullet(\SGK, \tM_{\lambda,E})\otimes_{E,\iota}E^\prime \ \iso \ 
 H^\bullet(\SGK, \tM_{{}^\iota\!\lambda,E'}). 
 $$
The map $\iota^\bullet$ is a $\pi_0(G(\R)) \times \HK$-equivariant, since it came from a morphism of sheaves whereas the Hecke action on these cohomology groups was intrinsic to the space $\SGK$. 

We may assume $E=E^\prime$ then $\iota$ is an element  of the Galois group. 
If $\eta_1, \eta_2 \in \Gal(E/\Q)$ then it is clear that 
$(\eta_1 \circ \eta_2)^* = \eta_1^* \circ \eta_2^*,$ and hence
$(\eta_1 \circ \eta_2)^\bullet = \eta_1^\bullet \circ \eta_2^\bullet,$ i.e., 
$$
\xymatrix{
H^\bullet (\SGK, \tM_{\lambda,E}) \ar[rr]^{\eta_2^\bullet} \ar[rrdd]_{(\eta_1 \circ \eta_2)^\bullet} 
& & H^\bullet (\SGK, \tM_{{}^{\eta_2}\!\lambda,E}) \ar[dd]^{\eta_1^\bullet}  \\
& & \\
 & & H^\bullet (\SGK, \tM_{{}^{\eta_1 \circ \eta_2}\!\lambda,E})
}$$
We say that the system of cohomology groups 
$\{H^\bullet(\SGK, \tM_{{}^\eta\lambda,E})\}_{\eta \in {\rm Gal}(E/\Q)}$ is defined over $\Q.$

 \medskip
 \subsubsection{\bf The fundamental long exact sequence}
 \label{sec:long-e-seq}

Let $\BSC$ be the Borel--Serre compactification of $\SGK$, i.e., 
$\BSC = \SGK \cup \partial\SGK$, where the boundary is stratified as 
$\partial \SGK  = \cup_P \partial_P\SGK$ with $P$ running through the $G(\Q)$-conjugacy classes of proper parabolic subgroups defined over $\Q$. (See Borel--Serre \cite{borel-serre}.) 
The sheaf $\tM_{\lambda, E}$ on $\SGK$ naturally extends, 
using the definition of the Borel--Serre compactification, to a sheaf on 
$\BSC$ which we also denote by $\tM_{\lambda, E}$. Restriction from $\BSC$ to $\SGK,$ in cohomology,  induces an isomorphism 
$
H^\bullet(\BSC, \tM_{\lambda, E}) \ \iso \  H^\bullet(\SGK, \tM_{\lambda, E}).
$
The cohomology of the boundary $H^\bullet(\partial \SGK, \tM_{\lambda, E})$ and the cohomology 
with compact supports $H^\bullet_c(\SGK, \tM_{\lambda, E})$ are naturally 
modules for $ \pi_0(G(\R)) \times \HK;$ the Hecke action on these other cohomology groups are described exactly as in 
Sect.\,\ref{sec:hecke}. 
Our basic object of interest is the following long exact sequence of $\pi_0(G(\R)) \times \HK$-modules: 
$$
\cdots  \longrightarrow H^i_c(\SGK, \tM_{\lambda, E}) 
\stackrel{\mathfrak{i}^*}{\longrightarrow}   H^i(\BSC, \tM_{\lambda,E}) 
\stackrel{\mathfrak{r}^*}{\longrightarrow } H^i(\partial \SGK, \tM_{\lambda,E}) 
\longrightarrow H^{i+1}_c(\SGK, \tM_{\lambda,E}) \longrightarrow \cdots
$$
The image of cohomology with compact supports inside the full cohomology is called {\it inner} or {\it interior} cohomology and is denoted 
$H^{\bullet}_{\, !} := {\rm Image}(\mathfrak{i}^*) = {\rm Im}(H^{\bullet}_c \to H^{\bullet}).$ 
The theory of Eisenstein cohomology is designed to describe the image of 
the restriction map $\mathfrak{r}^*$. 
Our considerations in Sect.\,\ref{sec:change-field-E} apply verbatim to the cohomology groups
$H^\bullet_?(\SGK, \tM_{\lambda,E}),$ where $? \in \{empty,,c,!,\partial\};$ 
 by $H^\bullet_\partial(\SGK, \tM_{\lambda,E})$ we mean the cohomology of the boundary 
 $H^\bullet(\PBSC, \tM_{\lambda,E}).$

\medskip
\subsubsection{\bf Inner cohomology and the inner spectrum ${\rm Coh}_!(G, \lambda)$}\label{sec:innercoh}
Inner cohomology  is a semi-simple module for the Hecke-algebra. (See  Harder \cite[Chap.\,3]{harder-book}.) 
After taking $E/\Q$ to be a sufficiently large finite Galois extension there is an isotypical decomposition:  
$$  
H^\bullet_{\,!}(\SGK, \M_{\lambda,E}) \ =\ 
\bigoplus_{\pi_f \in {\rm Coh}_!(G,K_f,\lambda)} H^\bullet_{\,!}(\SGK, \M_{\lambda,E})(\pi_f), 
$$
where $\pi_f$ is an isomorphism type of an absolutely  irreducible $\HK$-module, i.e., there is an $E$-vector space $V_{\pi_f}$ with an absolutely irreducible action $\pi_f$ of $\HK$.  
Let $\HKp = C^\infty_c(G(\Q_p)/ \! \!/K_p)$ 
be the local Hecke-algebra consisting of all locally constant and compactly supported bi-$K_p$-invariant $\Q$-valued functions on $G(\Q_p).$ The local factors  $\HKp$ are commutative outside a finite set $\place = \place_{K_f}$ of primes and the factors for two different primes commute with each other.  For $p \not\in \place$ the commutative algebra 
$\HKp$ acts
on $V_{\pi_f} $  by a homomorphism $\pi_p : \HKp \to E.$ Let $V_{\pi_p}$ be the one dimensional vector space $E$ with basis $1\in E$ with the action $\pi_p$ on it. Then  
$V_{\pi_f}  =\otimes_{p\in \place} V_{\pi_p} \otimes^\prime_{p\not\in \place} V_{\pi_p}.$
The set ${\rm Coh}_!(G, K_f, \lambda)$ of isomorphism classes which occur in  the  above decomposition is called the inner spectrum of $G$ with $\lambda$-coefficients and level structure $K_f.$  The inner spectrum of $G$ with $\lambda$-coefficients is: 
$$
{\rm Coh}_!(G, \lambda) \ = \ \bigcup_{K_f} {\rm Coh}_!(G, K_f, \lambda). 
$$

\smallskip

Going back to functorial properties upon changing the field $E$ via $\iota : E \to E'$, note that 
the map $\iota^\bullet$, since it came from a morphism of sheaves, preserves inner cohomology: 
$$
 \iota^\bullet :  H^\bullet_!(\SGK, \tM_{\lambda,E}) \longrightarrow H^\bullet_!(\SGK, \tM_{{}^\iota\!\lambda,E'}). 
$$ 
Given $\pi_f \in {\rm Coh}_!(G,K_f,\lambda)$, 
we deduce ${}^\iota\pi_f \in {\rm Coh}_!(G,K_f, {}^\iota\!\lambda)$ and 
that the $\pi_f$-isotypic component  is mapped by $\iota^\bullet$ onto the ${}^\iota\pi_f$-isotypic component. Furthermore,  given any character 
$\varepsilon$ of $\pi_0(G(\R))$ we have
$$
\iota^\bullet \left(H^\bullet_!(\SGK, \tM_{\lambda,E})(\pi_f \times \varepsilon ) \right) \ = \ 
H^\bullet_!(\SGK, \tM_{{}^\iota\!\lambda,E'})({}^\iota\pi_f \times \varepsilon), 
$$
which, via an abuse of notation, we may write as
$\iota^\bullet(\pi_f \times \varepsilon) = {}^\iota\pi_f \times \varepsilon. $

\smallskip

If $E$ is large enough so that the $\pi_f$-isotypic component can be split off from inner cohomology, then we can define of the rationality field of $\pi_f$ as:
$$
E(\pi_f) \ := \  E^{\{\eta \in \Gal(E/\Q) \ : \ {}^\eta\pi_f = \pi_f\}}. 
$$ 
One may see, using strong multiplicity one for $\GL_n/F$, that the field $E(\pi_f)$ is the subfield of $E$ generated by values of $\pi_p$ for $p \notin \place.$

\bigskip
\section{\bf Analytic tools}
\label{sec:analytic-tools}

We will now go to the transcendental level, i.e., take an embedding  $\iota : E \to  \C$ and extend the ground field to $\C.$ For all of Sect.\,\ref{sec:analytic-tools}, we will work over $\C$ and therefore we suppress  
the subscript $\C.$ Starting from Sect.\,\ref{sec:bdry-cohomology} we will return to working at an arithmetic level, i.e., over a Galois extension $E/\Q$ which contains a copy of the totally real base field $F,$ but for now, we work over $\C.$

\medskip
\subsection{Cuspidal parameters and the representation $\D_\lambda$ at infinity} 
\label{sec:repn-infinity}

\medskip
\subsubsection{\bf Lie algebras and relative Lie algebra cohomology}
\label{sec:lie-algebras}
Let $\g  $ denote the Lie algebra of $G/\Q$. Similarly,  let $\g^{(1)}$, $\b$, $\t$, $\z$ and $\s$ be the Lie-algebras of 
$G^{(1)}/\Q$, $B/\Q$, $T/\Q$, $Z/\Q$ and $S$, respectively. Let $\g_\infty$ be the Lie algebra of $G(\R)$, and 
$\k_\infty$ that of $K_\infty^0.$
 
We consider Harish-Chandra modules, also called $(\g_\infty, K_\infty^0)$-modules, $(\pi,V)$,  where $V$ is the space of 
smooth $K_\infty^0$-finite vectors of a reasonable representation $(\pi, \mathcal{V})$ of $G(\R)$. For example, 
if $(\pi,\mathcal{V})$ is essentially unitary (i.e., unitary up to a twist) and irreducible, then the space of $K_\infty^0$-finite vectors are automatically smooth. Given a $(\g_\infty, K_\infty^0)$-module $(\pi,V)$
by $H^\bullet(\g_\infty, K_\infty^0; V)$ we mean the cohomology of the complex 
${\rm Hom}_{K_\infty^0}(\Lambda^\bullet(\g_\infty/\k_\infty), V).$ (See Borel--Wallach~\cite{borel-wallach}.) 
If $(\pi,\mathcal{V})$ is a (reasonable) representation of $G(\R)$, then by $H^\bullet(\g_\infty, K_\infty^0; \mathcal{V})$, we mean $H^\bullet(\g_\infty, K_\infty^0; V)$ with $V$ as above.  
It will be important later on to stay as far as possible at a `rational' level. Note: $\g_\infty = \g \otimes_\Q \R.$ Similarly, 
let $\k^{(1)}$ be the Lie algebra of the $\Q$-group $\prod_{\tau : F \to \R} \SO(n)$, then 
$\k_\infty = (\s \oplus \k^{(1)}) \otimes_\Q \R.$ Given a $(\g_\infty, K_\infty^0)$-module $(\pi,V)$,  using connectedness of $K_\infty^0$, we have: 
$$
{\rm Hom}_{K_\infty^0}(\Lambda^\bullet(\g_\infty/\k_\infty), V) \ = \ 
{\rm Hom}_{K_\infty^0}(\Lambda^\bullet(\g /\k), V) = 
{\rm Hom}_{\k}(\Lambda^\bullet(\g /\k), V). 
$$
We will consider modules $(\pi, V)$ with a $\Q$-structure, i.e., $V = V_0 \otimes_\Q \C$, and then 
\begin{equation}
\label{eqn:g-k-rationa}
H^\bullet(\g_\infty, K_\infty^0; V) \ = \ H^\bullet(\g, K_\infty^0; V) \ = \ H^\bullet(\g, \k; V_0) \otimes \C.
\end{equation}
See \cite{harder-arithmetic}.

\medskip
\subsubsection{\bf Pure dominant integral weights}
\label{sec:pure}
Not all algebraic dominant integral weights can support inner cohomology. We have the following purity condition: 

\begin{lemma}[Purity Lemma]
Let $\lambda \in X^+_{\rm alg}(T \times \C)$ be an algebraic dominant integral-weight; and say 
$\lambda = (\lambda^\nu)_{\nu : F \to \C}$, with 
$\lambda^\nu  = \sum_{i=1}^{n-1} (a^\nu_i -1)  \bfgreek{gamma}_i + d \cdot \bfgreek{delta}.$ 
Suppose there is an irreducible essentially unitary representation $\mathcal{V}$ of $G(\R)$ such that 
$H^\bullet(\gK;  \mathcal{V} \otimes \M_\lambda) \neq 0.$
Then $\lambda$ is essentially self-dual, i.e., 
for each $\nu$, we have $a^\nu_i  = a^\nu_{n-i}.$ 
\end{lemma}

\begin{proof}
This is a consequence of Wigner's Lemma. See, for example, \cite[Chap.\,3]{harder-book}. 
\end{proof}

An algebraic dominant-integral essentially self-dual weight $\lambda$ will be called a {\it pure} weight. 
Let's summarize the restrictions on all the ingredients going into $\lambda$ if it is a pure weight:

\begin{lemma}
Let $\lambda \in X^*_{\rm alg}(T \times E)$ be an algebraic weight. Suppose that 
$\lambda = (\lambda^\tau)_{\tau : F \to E}$, with 
$\lambda^\tau \ = \ \sum_{i=1}^{n-1} (a^\tau_i -1)  \bfgreek{gamma}_i + d \cdot \bfgreek{delta}
\ = \ (b^\tau_1, \dots, b^\tau_n),$ is dominant, integral and pure. Then:

\smallskip
\begin{enumerate}

\item Integrality: 
  \begin{itemize}
  \item $a^\tau_i \in \Z$ for all $\tau$  and all $i$, $nd \in \Z$, and $r_{\lambda^\tau} \in \Z,$ i.e.,    
            $nd \equiv \sum i(a^\tau_i-1) \pmod{n};$ 
  \item $b^\tau_i \in \Z$ for all $\tau$  and all $i$. 
  \end{itemize}

\smallskip
\item Dominance: 
  \begin{itemize}
  \item $a^\tau_i \geq 1$  for all $\tau$  and all $i$;    
  \item $b^\tau_1 \geq b^\tau_2 \geq \cdots \geq b^\tau_n.$
  \end{itemize}
 
\smallskip
\item Purity: 
 \begin{itemize}
 \item $a^\tau_i =  a^\tau_{n-i}$ and $2d \in \Z.$ In particular, if $n$ is odd, then $d \in \Z;$   
 \item there exists $w \in \Z$ such that $b^\tau_i + b^\tau_{n-i+1} = w$ for all $\tau$ and all $i.$ 
  Indeed, $w = 2d.$ 
 \end{itemize}
\end{enumerate}
\end{lemma}

\begin{proof}
For integrality and dominance see Sect.\,\ref{sec:integral} and Sect.\,\ref{sec:dominant}, respectively. Under purity, we need to show that $2d \in \Z.$ Recall: by integrality $r_{\lambda^\tau} \in \Z$ and, by definition, 
$nd = \sum_{i=1}^{n-1} i (a^\tau_i -1) + n r_{\lambda^\tau}. $ The purity condition (essentially self-dual) implies 
$$
2nd = nd + nd  = \sum_{i=1}^{n-1} i (a^\tau_i -1) + \sum_{i=1}^{n-1} (n-i) (a^\tau_i -1) + 2nr_{\lambda^\tau}  
= n \sum_{i=1}^{n-1} (a^\tau_i -1) + 2n r_{\lambda^\tau}.
$$ 
Hence we have  
\begin{equation}
\label{eqn:2d}
2d = \sum_{i=1}^{n-1} (a^\tau_i-1) + 2 r_{\lambda^\tau}. 
\end{equation}
This implies that $2d \in \Z.$ Furthermore, if  $n$ is odd then $d \in \Z,$ since we already had $nd \in \Z$ by integrality. 
In terms of the standard basis, the condition $a^\tau_i = a^\tau_{n-i}$ translates to 
$b^\tau_i - b^\tau_{i+1} + 1 = b^\tau_{n-i} - b^\tau_{n-i+1} + 1,$
which is the same as 
$b^\tau_i + b^\tau_{n-i+1} = b^\tau_{i+1} + b^\tau_{n-i}.$
In other words, $b^\tau_i + b^\tau_{n-i+1}$ is independent of $i;$ say,  $b^\tau_i + b^\tau_{n-i+1} = w^\tau.$ Then
$$
2nd = nd + nd = \sum_{i=1}^n b^\tau_i + b^\tau_{n-i+1} = n w^\tau,
$$
or that $w^\tau = 2d$ is independent of $\tau;$ hence $b^\tau_i + b^\tau_{n-i+1} = 2d$ for all 
$i$ and $\tau.$
\end{proof}
 \smallskip

Denote by $X^*_0(T \times E)$ the set of pure weights. 
For $\lambda \in X^*_0(T \times E)$, with the notations as above, we shall call the integer $2d$ the {\it purity weight} of $\lambda;$ furthermore, if we write $\lambda = \lambda^{(1)} + d \bfgreek{delta}$, then 
note that the dual weight of $\lambda$ is given by $\lambda^{\sf v} = \lambda^{(1)} - d \bfgreek{delta}$, i.e., 
$\lambda^\v \ = \ \lambda - 2d\bfgreek{delta},$ which implies 
$\M_\lambda^\v \ = \ \M_\lambda \otimes ({\rm det})^{-2d}.$

\medskip
\subsubsection{\bf Motivic weight}
\label{sec:motivic-weight}
Let $\lambda \in X_0^*(T \times E)$ be as above. The {\it motivic weight} of $\lambda$ is defined to be the integer: 
$$
\w_\lambda := {\rm max} \{\w_{\lambda^\tau} \ | \ \tau : F \to E\}, \ \ {\rm where} \ \ 
\w_{\lambda^\tau} = \sum_{i=1}^{n-1} a_i^\tau. 
$$
From (\ref{eqn:2d}) we get the parity conditions  
\begin{equation}
\label{eqn:parity}
\w_{\lambda^\tau}  \ \equiv  \ 2d + n - 1 \pmod{2}, \ \forall \tau, \quad \mbox{hence} \quad
\w_{\lambda} \  \equiv \ 2d + n - 1 \pmod{2}.
\end{equation}
In particular, if $n$ is odd then $\w_\lambda \equiv 0 \pmod{2}.$ These parity conditions  play an important role in the numerology concerning cuspidal parameters, Hodge pairs, critical points, etc.

\medskip
 \subsubsection{\bf Cuspidal parameters and the representation at infinity}
Let $\lambda \in X_0^*(T \times \C)$ with $\lambda = (\lambda^\nu)_{\nu : F \to \C}.$ 
By Wigner's Lemma, the infinitesimal character of an irreducible admissible $(\g, K_\infty^0)$-module $\pi_\infty$ such that $H^\bullet(\g,  K_\infty^0; \pi_\infty \otimes \M_{\lambda}) \neq 0$ is uniquely determined by $\lambda.$ 
Hence, up to isomorphism, there are only finitely many such $\pi_\infty.$ We denote this finite set by $\Coh_\infty(G,\lambda).$
Amongst this finite set of possible $(\g, K_\infty^0)$-modules, exactly one, up to twisting by sign characters (see below), is generic, i.e., admits a Whittaker model.  Only this representation can appear as the representation at infinity of a global cohomological cuspidal representation.

\smallskip

Let  $\lambda^\nu = \sum_{i=1}^{n-1} (a^\nu_i -1) \bfgreek{gamma}_i+ d \cdot \bfgreek{delta}_n.$ 
The {\it cuspidal parameter of $\lambda$} is defined as $\ell = (\ell^\nu)_{\nu : F \to \C}$, 
where $\ell^\nu  = (\ell^\nu_1, \dots, \ell^\nu_n)$ and  
\begin{equation}
\label{eqn:cuspidal-parameter}
\begin{array}{lll}
\ell^\nu_1 & \ := \ \ \ a^\nu_1 +  a^\nu_2  +  a^\nu_3 + \dots +  a^\nu_{n-1}
& = \ \ a^\nu_1 +  a^\nu_2 +  a^\nu_3 + \dots +  a^\nu_{n-1}  \ =  \ {\bf w}_{\lambda^\nu}\\
\ell^\nu_2 & \ := \  - a^\nu_1 +  a^\nu_2 +  a^\nu_3 + \dots + a^\nu_{n-1} 
& = \ \ a^\nu_2  +  a^\nu_3  +  \dots +  a^\nu_{n-2}  \\
\ell^\nu_3 & \ := \  - a^\nu_1  -  a^\nu_2 +  a^\nu_3  + \dots +  a^\nu_{n-1}  
& = \  \ a^\nu_3  + \dots + a^\nu_{n-3}  \\
&  & \vdots \\
\ell^\nu_n  & \ := \ - a^\nu_1 -  a^\nu_2 - \dots -  a^\nu_{n-1} = -{\bf w}_{\lambda^\nu}& 
\end{array}
\end{equation}
The integers $\ell^\nu_j,$ are  also called as the cuspidal parameters of $\lambda.$ 
Observe that: 

\begin{enumerate}
\item[$\bullet$] $\ell$ depends only on the semi-simple part $\lambda^{(1)}$ and not on the abelian 
part $\lambda_{\rm ab}$ of $\lambda,$ i.e., $\ell$ depends only on the 
$a^\nu_i$'s, and not on the purity weight $2d$. 

\smallskip

\item[$\bullet$] $\ell^\nu_1 >  \ell^\nu_2  >  \dots  >  \ell^\nu_{[n/2]}  >  0$ and 
$\ell^\nu_i  =  - \ell^\nu_{n-i+1}$.

\smallskip

\item[$\bullet$] 
$
\ell^\nu = ({\bf w}_{\lambda^\nu}, \, {\bf w}_{\lambda^\nu}-2a^\nu_1, \, 
{\bf w}_{\lambda^\nu}-2a^\nu_1-2a^\nu_2, \, \dots, \, -{\bf w}_{\lambda^\nu}).
$
It readily follows that  
\begin{equation}
\label{eqn:parity-cuspidal-parameter}
\ell^\nu_i  \ \equiv \  \w_\lambda  \ \equiv \ 2d+n-1 \pmod{2}, 
\end{equation}
i.e., every cuspidal parameter has the same parity as the motivic weight. 
Furthermore, if $n$ is odd, then all the cuspidal parameters, the motivic weight and the purity weight are even; 
however,  if $n$ is even, then all the cuspidal parameters and the motivic weight have the same parity which is opposite to the parity of the purity weight.
\end{enumerate}

\smallskip

For any integer $\ell \geq 1$, let $D_\ell$ be the discrete series representation of $\GL_2(\R)$ of 
lowest non-negative $K$-type corresponding to $\ell+1$ and with 
central character ${\rm sgn}^{\ell+1}$. The Langlands parameter of $D_\ell$ is 
$\Ind_{\C^\times}^{W_\R }(\chi_\ell)$, where
$W_\R $ is the Weil group of $\R$, and $\chi_\ell$ is the character of $\C^\times$ sending 
$z$ to $(z/\bar{z})^{\ell/2}$, or $\chi_\ell(re^{i t}) = e^{i \ell t}.$ The infinitesimal character of 
$D_\ell$ is $(\ell/2, -\ell/2)$. For example, a holomorphic cuspidal modular form of  weight $k$ generates $D_{k-1}$ as the representation at infinity. (For more details see Raghuram--Tanabe \cite[\S3.1.2--\S3.1.5]{raghuram-tanabe}.)

\smallskip

Let $R_n$ stand for the standard parabolic subgroup of $\GL_n$ of type $(2,2,\dots,2)$ if $n$ is even, and 
of type $(2,2,\dots,2,1)$ if $n$ is odd. For $\lambda = (\lambda^\nu)_{\nu : F \to \C} \in X_0^*(T \times \C)$, define 
\begin{equation}
\label{eqn:rep-infinity-1}
\D_{\lambda^\nu} \ := \  
\left\{
\begin{array}{ll}
{\rm Ind}_{R_n(\R)}^{\GL_n(\R)}
\left((D_{\ell^\nu_1}\otimes|\cdot|_{\mathbb R}^{-d}) \otimes \cdots \otimes
(D_{\ell^\nu_{n/2}}\otimes|\cdot|_{\mathbb R}^{-d})\right), 
& \mbox{if $n$ is even;} \\ 
& \\
{\rm Ind}_{R_n(\R)}^{\GL_n(\R)}
\left((D_{\ell^\nu_1}\otimes|\cdot|_{\mathbb R}^{-d}) \otimes \cdots \otimes
(D_{\ell^\nu_{(n-1)/2}}\otimes|\cdot|_{\mathbb R}^{-d}) \otimes
|\cdot|_{\mathbb R}^{-d} \right), 
& \mbox{if $n$ is odd,}
\end{array}\right.
\end{equation}
where ${\rm Ind}_{R_n(\R)}^{\GL_n(\R)}$ denotes normalized parabolic induction. 
Identify the set $\place_\infty$ of infinite places with the set $\Sigma_F;$ 
say, $v \in \place_\infty$ corresponds to $\nu_v \in \Sigma_F.$ Define: 
\begin{equation}
\label{eqn:rep-infinity-2}
 \D_{\lambda} = \bigotimes_{v \in \place_\infty} \D_{\lambda^{\nu_v}}. 
\end{equation}

\smallskip

We will now discuss twisting by sign characters. Identify 
the sign-character ${\rm sgn}: \R^\times \to \{ \pm 1 \}$ with $-1$, and the trivial character of $\R^\times$ with $+1.$
Let $\varepsilon = (\varepsilon_v)_{v \in \place_\infty}$ be an ${\sf r}$-tuple of signs, i.e., $\varepsilon_v \in \{\pm 1\}.$ 
Then $\varepsilon$ gives a character of order $2$ on $\prod_{ v \in \place_\infty} \R^\times$, and via the determinant map, $\varepsilon$ gives of a character of $G(\R)/G(\R)^0 = \pi_0(G(\R)).$
\begin{enumerate}
\item If $n$ is even then 
$\D_\lambda \otimes \varepsilon \ \cong \ \D_\lambda$ for every $\varepsilon.$
However, 
\item if $n$ is odd, then for every nontrivial $\varepsilon$ we have 
$\D_\lambda \otimes \varepsilon \  \not\cong \ \D_\lambda.$ This may be seen by computing central characters. 
The central character $\omega_{\D_{\lambda^{\nu_v}} \otimes \varepsilon_v}$ of the $v$-th component of 
$\D_\lambda \otimes \varepsilon$ is: 
$$
\omega_{\D_{\lambda^{\nu_v}} \otimes \varepsilon_v} \ =  \ {\rm sgn}^{\tfrac{(n-1)}{2}} \cdot |\ |^{-nd} \cdot \varepsilon_v. 
$$
(We have used the fact that $\ell^\nu_j$ is even when $n$ is odd.)
\end{enumerate}

\smallskip
We collect some basic properties of $\D_\lambda$ in the following 
\begin{prop}
\label{prop:d-lambda}
Let $\lambda \in X^*_0(T \times \C)$  and $\D_\lambda$ be as in (\ref{eqn:rep-infinity-2}). 
Let $\varepsilon = (\varepsilon_v)_{v \in \place_\infty}$ be any $r$-tuple of signs. Then:
\begin{enumerate}
\item $\D_\lambda \otimes \varepsilon$ is an irreducible essentially tempered representation. 
\item $\D_\lambda \otimes \varepsilon$ admits a Whittaker model. 
\item $H^\bullet(\g, K_\infty^0; (\D_\lambda \otimes \varepsilon) \otimes \M_\lambda) \neq 0.$
\end{enumerate}
\end{prop}

\begin{proof}
These are well-known results for $\GL_n(\R)$; we refer the reader to the very useful survey articles of 
Knapp~\cite{knapp} and M\oe glin~\cite{moeglin-edinburgh} and to the references therein. 
(Irreduciblity follows from Speh's results. 
A representation irreducibly induced from essentially discrete series representation is essentially tempered. Admitting a Whittaker model is a hereditary property and an essentially discrete series representation of $\GL_2(\R)$ has a Whittaker model. Nonvanishing of cohomology follows from Delorme's Lemma; see below.)
\end{proof}

\medskip
 \subsubsection{\bf The cohomology degrees}
 The relative Lie algebra cohomology group in Prop.\,\ref{prop:d-lambda} 
 can be computed using Delorme's Lemma; see, for example, Borel--Wallach~\cite[Thm.\,III.3.3]{borel-wallach}. We summarize what we need about these cohomology groups in the following 
 
 \begin{prop}
 \label{prop:cohomology-degree}
Let $\lambda  \in X^*_0(T \times \C),$ $\D_\lambda$ and $\varepsilon$ be as above. 
Then 
$$
H^\bullet(\g,  K_\infty^0; (\D_\lambda \otimes \varepsilon) \otimes \M_{\lambda}) \neq 0 
\ \  \Longleftrightarrow \ \ 
b^F_n \leq \bullet \leq \tilde t^F_n, 
$$
where the bottom degree $b_n^F$ and the top degree $\tilde t_n^F$ are defined as:
$$
\begin{array}{lll}
b^\Q_n := \lfloor n^2/4 \rfloor, \ & \ b^F_n := {\sf r}_F\, b^\Q_n, &  \\ 
& \\
t^\Q_n := b^\Q_n + \lceil n/2 \rceil - 1, \ & \ t^F_n := {\sf r}_F\, t^\Q_n, \ & \ \tilde t^F_n := t^F_n + {\sf r}_F -1, 
\end{array}
$$
where, recall that ${\sf r}_F = [F : \Q]$, and for any real number $x$, $\lfloor x \rfloor $ denotes the greatest integer less than or equal to $x$, and $\lceil x \rceil = - \lfloor -x \rfloor$ is the least integer greater than or equal to $x$. 
Furthermore, 
the cohomology group $H^q(\g,  K_\infty^0; (\D_\lambda \otimes \varepsilon) \otimes \M_{\lambda}),$ as a module over
$K_{\infty}/K_\infty^0 = \pi_0(K_{\infty}) = \pi_0(G_n(\R))$, in the extreme degrees 
$q = b_n^F$ or $q = \tilde t_n^F$ is given by: 
\begin{enumerate}
\item If $n$ is even, then $\D_\lambda \otimes \varepsilon = \D_\lambda$ for all $\varepsilon$, and 
$$
H^q(\g,  K_\infty^0;  \D_\lambda \otimes \M_{\lambda}) \ = \ 
\bigoplus\limits_{\varepsilon \, \in \, \widehat{\pi_0(G_n(\R))}} \varepsilon. 
$$

\item If $n$ is odd then 
$$
H^q(\g,  K_\infty^0;  (\D_\lambda \otimes \varepsilon) \otimes \M_{\lambda}) \ = \ 
(-1)^{d + (n-1)/2} \varepsilon. 
$$
\end{enumerate}
 \end{prop}

\begin{proof}
We just make a few comments, as all of this is well-known. First of all, observe that
$\g/\k = \g^{(1)}/\k^{(1)} \oplus \z/\s$, and hence 
$\Lambda^\bullet(\g/\k) = \Lambda^\bullet(\g^{(1)}/\k^{(1)}) \otimes \Lambda^\bullet(\z/\s).$ This implies: 
$$
\Hom_{K_\infty^0} (\Lambda^\bullet(\g/\k), (\D_\lambda \otimes \varepsilon) \otimes \M_{\lambda} ) \ = \ 
\Hom_{K_\infty^{(1)}}(\Lambda^\bullet(\g^{(1)}/\k^{(1)}), (\D_\lambda \otimes \varepsilon) \otimes \M_{\lambda}) \otimes
\Hom(\Lambda^\bullet(\z/\s), \C)
$$
since the Lie algebra $\z$ acts trivially on $(\D_\lambda \otimes \varepsilon) \otimes \M_{\lambda}.$ The group $H^\bullet(\g^{(1)},  K_\infty^{(1)}; (\D_\lambda \otimes \varepsilon) \otimes \M_{\lambda})$ is calculated in 
Clozel~\cite[Lem.\,3.14]{clozel}; in particular one sees that it is nonvanishing if and only if 
$b_n^F \leq \bullet \leq t_n^F.$ For $(\g,K_\infty^0)$-cohomology we need to go up to $\tilde t_n^F = t_n^F + {\sf r}_F-1$ since the dimension of $\z/\s$ is ${\sf r}_F-1.$ 
If $n$ is odd then ${\rm O}(n)/{\rm SO}(n)$ is represented by $\{\pm 1_n\}$, where $1_n$ is the identity $n \times n$-matrix. So, in case (2), it simply boils down to 
computing the central character of $(\D_\lambda \otimes \varepsilon) \otimes \M_{\lambda}.$ 
\end{proof}

\medskip
\subsection{Square-integrable cohomology}
\label{sec:square-integrable}

\medskip
\subsubsection{\bf de~Rham complex} 
\label{sec:de-rham}
For a level structure $K_f \subset G (\A_f)$, in what follows, we consider various spaces of functions on 
$G(\Q)\backslash G(\A) / K_f$ which are naturally $G(\R)$-modules. 
Assume for the moment that $K_f$ is a neat subgroup. 
Consider a pure weight $\lambda \in X^*_0(T \times \C).$ The sheaf $\tM_{\lambda }$ on 
$\SGK,$ as constructed in Sect.\,\ref{sec:sheaves}, is a local system of $\C$-vector spaces. 
The cohomology $H^\bullet(\SGK,\tM_\lambda)$ is the cohomology  of the de~Rham complex.  We have the isomorphism between the de~Rham complex and the relative Lie algebra complex
$$  
\Omega^\bullet(\SGK, \tM_\lambda ) \ = \  \Hom_{K_\infty^0} (\Lambda^\bullet(\fg/\fk), \, 
\cC^\infty(G(\Q)\backslash G(\A)/K_f, \omega_{\lambda}^{-1}|_{S(\R)^0}) \otimes  \M_{\lambda}), 
 $$
 where $\cC^\infty(G(\Q)\backslash G(\A)/K_f, \omega_{\lambda}^{-1}|_{S(\R)^0})$ consists of all smooth functions 
 $\phi : G(\A) \to \C$ such that $\phi(\gamma \, \ul g \, \ul k_f \, a_\infty) = \omega_{\lambda}^{-1}(a_\infty) \phi(\ul g),$ 
 for all $\ul g \in G(\A)$, $\gamma \in G(\Q)$, $\ul k_f \in K_f$ and $a_\infty \in S(\R)^0.$ 
 (See, for example,  \cite[Chap.\,3]{harder-book}.)  We will abbreviate 
 $\omega_{\lambda}^{-1}|_{S(\R)^0} = \omega_\infty^{-1}.$

\medskip
\subsubsection{\bf Definition of square-integrable cohomology} 
 Inside the space of smooth functions lies the subspace 
$\cC_2^\infty(G(\Q)\backslash G(\A)/K_f, \omega_\infty^{-1})$ 
of smooth square-integrable functions, which are functions 
$\phi \in  \cC^\infty(G(\Q)\backslash G(\A)/K_f, \omega_\infty^{-1})$ satisfying 
$$
\int\limits_{S(\R)^0G(\Q)\backslash G(\A)} |(\phi)(\ul g)|^2 |{\rm det}(g)|_F^{2d}  \, dg < \infty,  
$$
where the integrand being trivial on $S(\R)^0$ follows from (\ref{eqn:trivial-on-S(R)^0}). 
Define  $H^\bullet_{(2)}(\SGK,\tM_\lambda)$ as the submodule of $H^\bullet(\SGK,\tM_\lambda)$ consisting of those 
 cohomology classes which can be represented by square-integrable forms, i.e., by closed forms in  
 $$
 \Hom_{K_\infty^0}
( \Lambda^\bullet(\fg/\fk), \, 
\cC^\infty_2(G(\Q)\backslash G(\A)/K_f, \omega_\infty^{-1}) \otimes  \M_{\lambda}). 
 $$

\medskip
\subsubsection{\bf A result of Borel and Garland}
\label{sec:borel-garland}
We know from the fundamental work of Langlands \cite{langlands-eisenstein} that we have a decomposition into essentially unitary $G(\R)$-modules: 
$$
L^2(G(\Q)\backslash G(\A)/K_f, \omega_\infty^{-1}) \ \ = \ \ 
L^2_{\rm disc}(G(\Q)\backslash G(\A)/K_f, \omega_\infty^{-1}) \ \oplus \ 
L^2_{\rm cont}(G(\Q)\backslash G(\A)/K_f, \omega_\infty^{-1}), 
$$
where the first summand in the right hand side is the closure of the direct sum of irreducible subspaces. Hence, 
$\omega \in \Hom_{K_\infty^0}
( \Lambda^\bullet(\fg/\fk), \, 
\cC^\infty_2(G(\Q)\backslash G(\A)/K_f, \omega_\infty^{-1}) \otimes  \M_{\lambda})$ may be written as  
$ \omega = \omega_{\rm disc} \oplus \omega_{\rm cont}.$ 
It follows from \cite{borel-garland} that $\omega_{\rm cont}$ goes to zero in cohomology. (See also
\cite[Chap.\,3]{harder-book}.) We say that the square-integrable cohomology {\it comes} from the discrete spectrum. 
To be more precise, we define $\Coh_\infty^{(2)}(G,\lambda)$ to be the finite set of isomorphism classes of essentially unitary $G(\R)$-modules having nontrivial cohomology with $\M_\lambda$-coefficients. For each 
$\pi_\infty \in \Coh_\infty^{(2)}(G,\lambda)$ choose a representative $\mathcal{V}_{\pi_\infty}$, and let 
$V_{\pi_\infty}$ be the resulting Harish-Chandra module of $K_\infty$-finite vectors. Then, we put:
$$
W_{\pi_\infty}^{(2)} = 
\Hom_{G(\R)}\left(\mathcal{V}_{\pi_\infty}, L^2(G(\Q)\backslash G(\A)/K_f, \omega_\infty^{-1})\right) 
= 
\Hom_{G(\R)}\left(\mathcal{V}_{\pi_\infty}, L^2_{\rm disc}(G(\Q)\backslash G(\A)/K_f, \omega_\infty^{-1})\right). 
$$
It is clear that any $\Phi \in W_{\pi_\infty}^{(2)}$ sends $V_{\pi_\infty}$ to 
$\cC^\infty_2(G(\Q)\backslash G(\A)/K_f, \omega_\infty^{-1}).$ The theorem of Borel--Garland says that the induced map
\begin{equation}
\label{eqn:L2-deco-one}
 \bigoplus_{\pi_\infty\in \Coh_\infty^{(2)}(G,\lambda) } 
 W^{(2)}_{\pi_\infty} \otimes H^\bullet(\gK; V_{\pi_\infty} \otimes\M_\lambda)
  \ \longrightarrow \ 
 H_{(2)}^\bullet(\SGK, \tM_\lambda)
\end{equation}
is surjective. We can also take the action of the Hecke algebra into account; it's action via convolution on 
$L^2_{\rm disc}(G(\Q)\backslash G(\A)/K_f, \omega_\infty^{-1})$ is self-adjoint. 
For each isomorphism class $\pi_\infty \in {\rm Coh}_\infty^{(2)}(G, \lambda)$, take $\mathcal{V}_{\pi_\infty}$ as above, 
and similarly, for each isomorphism class $\pi_f$  of absolutely irreducible $\HK$-module choose a representative
 $V_{\pi_f} $ for  $\pi_f$, and define: 
 $$
 W^{(2)}_{\pi_\infty\otimes \pi_f} \ = \ 
 \Hom_{G(\R) \times \HK} \left( \mathcal{V}_{ \pi_\infty}\otimes V_{\pi_f} , \, 
L^2_{\rm disc}(G(\Q)\backslash G(\A)/K_f, \omega_\infty^{-1}) \right). 
 $$
Define $\Coh_{(2)}(G, \lambda,K_f)$ to be the set of those $\pi_f$ for which there exists a $\pi_\infty$ such that 
$W^{(2)}_{\pi_\infty\otimes \pi_f} \neq 0.$ Then we get the refined decomposition: 
\begin{equation}
\label{eqn:L2-deco-two}
\bigoplus_{\pi_\infty \times \pi_f} 
 W^{(2)}_{\pi_\infty\otimes \pi_f} \otimes 
H^\bullet(\gK; V_{\pi_\infty} \otimes\M_\lambda)
 \otimes V_{\pi_f} 
 \ \longrightarrow \ 
\bigoplus_{\pi_f \in \Coh_{(2)}(G, \lambda,K_f)} H_{(2)}^\bullet(\SGK, \tM_\lambda)(\pi_f), 
 \end{equation}
which is a surjective map onto $ H_{(2)}^\bullet(\SGK, \tM_\lambda),$ and the summand on the right hand side indexed by $\pi_f$ is in fact the $\pi_f$-isotypic component. 
This implies that  square-integrable cohomology is a semi-simple module under the action of the Hecke algebra, hence so is inner cohomology. It is a delicate question to compute the kernel of the above map. 
 We will return to this in Sect.\,\ref{sec:coh-res}. 

\smallskip 
 
 Note that, by the multiplicity one theorem for the discrete spectrum for  $\GL_n$ (proved in a special case by Jacquet \cite{jacquet-residual}, and in general by 
 M\oe glin--Waldspurger \cite{moeglin-waldspurger}), $\dim(W^{(2)}_{\pi_\infty\otimes \pi_f})=1.$

\medskip
\subsection{Cuspidal cohomology}
\label{sec:cuspidal}

\medskip
\subsubsection{\bf The cohomological cuspidal spectrum} 
\label{sec:w-cusp-pi}

 The space of square-integrable functions contains the space of smooth cusp forms
$$
 \cC_{\rm cusp}^\infty(G(\Q)\backslash G(\A)/K_f, \omega_\infty^{-1})
 \ \subset \  
  \cC^\infty_{(2)}(G(\Q)\backslash G(\A)/K_f, \omega_\infty^{-1}). 
$$
This inclusion, induces an inclusion in cohomology, and we define cuspidal cohomology by: 
$$
 H_{\rm cusp}^\bullet(\SGK, \tM_\lambda) \ := \ 
H^\bullet \left(\gK;  \, 
\cC^\infty_{\rm cusp}(G(\Q)\backslash G(\A)/K_f, \omega_\infty^{-1})  \otimes \M_\lambda \right).
$$
For $\pi_\infty \in {\rm Coh}_\infty(G, \lambda)$ and $\pi_f \in \Coh (G, \lambda,K_f),$ define: 
 $$
 W^{\rm cusp}_{\pi_\infty\otimes \pi_f} \ := \  
 \Hom_{(\fg,K_\infty)\otimes\HK} \left( 
 V_{ \pi_\infty}\otimes V_{\pi_f} , \,  
 \cC_{\rm cusp}^\infty(G(\Q)\backslash G(\A)/K_f, \omega_\infty^{-1})\right). 
 $$ 
 Define $\Coh_{\rm cusp}(G,\lambda,K_f)$ as the set of those $\pi_f \in \Coh (G, \lambda,K_f)$ 
 for which we can find a $\pi_\infty \in {\rm Coh}_\infty(G, \lambda)$
 such that $W^{\rm cusp}_{\pi_\infty\otimes \pi_f}\not= 0.$ We have a surjective map 
 $$  
 \bigoplus_{\pi_f\in \Coh (G, \lambda,K_f)}    \bigoplus_{\pi_\infty\in \Coh_\infty(G,\lambda) } 
 W^{\rm cusp}_{\pi_\infty \otimes \pi_f} \otimes 
 H^\bullet(\g, K_\infty^0; 
 V_{\pi_\infty} \otimes \M_\lambda)\otimes V_{\pi_f} 
 \ \longrightarrow \  
 H_{\rm cusp}^\bullet(\SGK, \M_\lambda), 
 $$
 which is in fact an isomorphism by a theorem of Borel \cite{borel-duke}.

 \medskip
 \subsubsection{\bf Consequence of multiplicity one and strong multiplicity one}
\label{sec:w-2}
By multiplicity one for the cuspidal spectrum 
(\cite{jacquet-residual}, \cite{moeglin-waldspurger}) for $G = R_{F/\Q}(\GL_n/F)$ if $ W^{(\rm cusp)}_{\pi_\infty\otimes \pi_f}$ is non-zero then it is of dimension one.  Furthermore, from strong multiplicity one for cuspidal representations due to Jacquet and Shalika \cite{jacquet-shalika-II}, it follows that 
$\pi_f$ is determined by the restriction  $\pi_f^\place$
of $\pi_f$ to the central subalgebra $\HGS = \otimes_{v \notin \place} \HKp$ of $\HK.$

 \medskip
 \subsubsection{\bf The character of the component group I}
\label{sec:epsilon}
 Since a cuspidal automorphic representation has a global Whittaker model it follows that the representation at infinity 
$\pi_\infty$ is isomorphic to $\D_\lambda\otimes \varepsilon$ for some sign character $\varepsilon$, because these are
the only $\pi_\infty \in \Coh_\infty(G, \lambda)$ which have a local Whittaker model. Applying the isomorphism mentioned 
in Sect.\,\ref{sec:w-cusp-pi}, and the cohomology of $\D_\lambda\otimes \varepsilon$ as described in 
Prop.\,\ref{prop:cohomology-degree} for extreme degrees $q \in \{b_n^F, \tilde t_n^F\},$ we conclude the following decomposition of cuspidal cohomology as a $\pi_0(G(\R)) \times \HK$-module: 
\begin{equation}
\label{eqn:character-component}
 H_{\rm cusp}^q(\SGK, \tM_\lambda) \ = \ 
\left\{ \begin{array}{ll} 
 \bigoplus\limits_{\pi_f \in \Coh_{\rm cusp}(G,\lambda,K_f)} \, 
 \bigoplus\limits_{\varepsilon \, \in \, \widehat{\pi_0(G_n(\R))}} \,  \varepsilon \times \pi_f, & 
 \mbox{if $n$ is even,} \\
 & \\
 \bigoplus\limits_{\pi_f \in \Coh_{\rm cusp}(G,\lambda,K_f)}  \varepsilon(\pi_f) \times \pi_f,  & 
 \mbox{if $n$ is odd,} 
 \end{array}\right.
\end{equation}
where, the canonical character $\varepsilon(\pi_f)$ that $\pi_f$ can pair with (when $n$ is odd) is given by: 
\begin{equation}
\label{eqn:character-at-infinity}
\varepsilon(\pi_f) = (\varepsilon(\pi_f)_v)_{v \in \place_\infty}, \quad 
\varepsilon(\pi_f)_v(-1) = (-1)^d\omega_{\pi_v}(-1),
\end{equation}
where $\omega_{\pi_v}$ is the central character of $\pi_v.$ 
This last assertion when $n$ is odd may be seen as follows: given $\pi_f \in \Coh_{\rm cusp}(G,\lambda,K_f)$, by strong multiplicity one, 
there is a unique $\pi_\infty$ that it can pair with to give a cuspidal automorphic representation 
$\pi = \pi_\infty \otimes \pi_f.$ Suppose, $\pi_\infty = \otimes_{v \in \place_\infty} \pi_v.$ By Prop.\,\ref{prop:cohomology-degree}, (2), we also see that $H^q(\g,  K_v^0;  \pi_v \otimes \M_{\lambda_v})$, as a 
$K_v/K_v^0 = {\rm O}(n)/{\rm SO}(n)$-module, is the sign character whose value at $-1$ is 
 $\omega_{\pi_v}(-1) (-1)^d.$ To justify the notation $\varepsilon(\pi_f)$, 
we note that $\omega_{\pi_v}$ is completely determined by $\omega_{\pi_f}$ by automorphy of 
the central character $\omega_\pi = \omega_{\pi_\infty} \otimes \omega_{\pi_f}.$ To parse it further, 
observe that $\omega_\pi = |\ |^{-nd} \otimes \omega_\pi^0$ where $\omega_\pi^0$ is a character of finite order. 
Fix $v \in \place_\infty$, and apply weak approximation to choose an $a \in F^\times$ (which will depend on $v$) 
such that $a < 0$ as an element of $F_v$ and $a > 0$ as an element of 
$F_w$ for all $w \in \place_\infty - \{v\};$ then $\omega_{\pi_v}(-1) = \omega_{\pi_f}^0(a).$ Similarly, 
we also have (when $n$ is odd):  
\begin{equation}
\label{eqn:epsilon-infinity}
\pi_\infty = \D_\lambda\otimes \varepsilon_\infty(\pi_f); \quad 
\varepsilon_\infty(\pi_f) = (\varepsilon_\infty(\pi_f)_v)_{v \in \place_\infty}; \quad 
\varepsilon_\infty(\pi_f)_v(-1) = (-1)^{(n-1)/2}\omega_{\pi_v}(-1). 
\end{equation}
The reader is also referred to the discussion in Gan--Raghuram~\cite[\S\,3]{gan-raghuram}.

\medskip
\subsubsection{\bf A bit of linear algebra}
\label{sec:lin-alg}
Let $V$ be a finite-dimensional $E$-vector space which is also a module for a commutative $\Q$-algebra $\cA.$ Let us assume that all absolutely irreducible subquotients are one-dimensional (i.e., all eigenvalues lie in $E$). 
 Define $\Spec_\cA(V)$ to be the set of isomorphism classes of absolutely irreducible $\cA$-modules over $E$ which appear as irreducible sub-quotients of $V.$ Let $W$ be an $\cA$-stable $E$-subspace of $V$ such that 
$$
\Spec_\cA(W) \ \cap \ \Spec_\cA(V/W) \ = \ \emptyset. 
$$
Then there is an $\cA$-equivariant projection $\pi_W : V \to W$, i.e., we have a splitting $V \simeq W \oplus V/W$ of 
$\cA$-modules.

\medskip
\subsubsection{\bf A filtration} 
\label{sec:filt}
We now drop the assumption that we are working  over $\C$ and go back to our coefficient system $\M_{\lambda, E}$ 
defined over $E.$ An embedding $\iota:E \to \C$ gives a chain of subspaces 
$$
 H_{\rm cusp}^\bullet(\SGK, \tM_{{}^\iota\!\lambda, \C}) \ \subset \ 
 H_{!}^\bullet(\SGK, \tM_{{}^\iota\!\lambda, \C})   \ \subset  \ 
 H_{(2)}^\bullet(\SGK, \tM_{{}^\iota\!\lambda, \C})  \ \subset  \ 
 H^\bullet(\SGK, \tM_{\lambda,E}) \otimes_{E,\iota}\C 
$$
Since  the Hecke module $H_{!}^\bullet(\SGK, \tM_{{}^\iota\!\lambda, \C}) = 
H_!^\bullet(\SGK, \tM_{\lambda,E}) \otimes_{E,\iota}\C
$ is a submodule  of a semi-simple module
it is also semi-simple.  But then already the $E$-module  $H_{!}^\bullet(\SGK, \tM_{\lambda, E})$ is semi-simple
and if $E$ is large enough  we get an isotypical decomposition  
$$
H_{!}^\bullet(\SGK, \tM_{\lambda, E})  \otimes_{E,\iota}\C \ \ = \bigoplus_{\pi_f\in \Coh_{!}(G, \lambda ,K_f)}
 H_{!}^\bullet(\SGK, \tM_{\lambda, E}) (\pi_f) \otimes_{E,\iota}\C, 
$$
where the $\pi_f$ are absolutely irreducible.

We have to understand the discrepancies between these spaces, especially
the difference between the cuspidal and the square integrable cohomology. This issue may become
very delicate for a general reductive group, but for $\Gl_n$ the situation is relatively simple.
Take $\pi_f $ in  $\Coh_!(G, \lambda ,K_f).$ Then, for any $\iota : E \to \C$, the module
${}^\iota\pi_f = \pi_f \otimes_{E, \iota}\C \in \Coh_{(2)}(G, \lambda, K_f).$
By definition, there is a ${}^\iota\pi_\infty$
such that $ W^{(2)}_{{}^\iota\pi_\infty \otimes {}^\iota\pi_f} $ is one-dimensional. We have to find criteria to decide
whether $ W^{\rm cusp}_{{}^\iota\pi_\infty \otimes {}^\iota\pi_f} =0 \text { or }  
W^{\rm cusp}_{{}^\iota\pi_\infty \otimes {}^\iota\pi_f}=
W^{(2)}_{{}^\iota\pi_\infty \otimes {}^\iota\pi_f}.  $   We will show later that 
\begin{multline}
\label{eqn:cuspidal-claim}
\mbox{the isomorphism type ${}^\iota\pi_f \in \Coh_{(2)}(G, \lambda ,K_f)$  is cuspidal}  \iff  \\
H_{!}^\bullet(\SGK, \tM_{\lambda, E}) (\pi_f)  \otimes_{E,\iota}\C  = 
H_{(2)}^\bullet(\SGK, \tM_{\lambda,E} \otimes_{E,\iota}\C ) ({}^\iota\pi_f). 
\end{multline}

\bigskip

\section{\bf Boundary cohomology}
\label{sec:bdry-cohomology}

We will be very brief in this section; our aim is only to set up notations concerning the cohomology of the boundary of the Borel--Serre compactification of the locally symmetric space $\SGK.$ 
The reader should consult \cite[1.1]{harder-lnm} and \cite[Chap.\,3]{harder-book} for more details and proofs.

\medskip
\subsection{A spectral sequence converging to boundary cohomology}
\label{sec:spectral-sequence}

Recall from Sect.\,\ref{sec:long-e-seq} the Borel--Serre compactification of $\SGK$ and the associated long exact sequence. 
We would like to understand the cohomology of the boundary: $H^\bullet(\PBSC, \tM_{\lambda,E}).$ There is a spectral sequence built out of the cohomology of the boundary strata $ \PPBSC$ which converges to the cohomology of the boundary.

\medskip
\subsubsection{\bf The spectral sequence at an arithmetic level}
\label{sec:spectral-seq-arith}
For $1 \leq i \leq n-1$, let $P_{0,i}$ be the standard maximal parabolic subgroup of $G_0 = \GL_n/F$ obtained by deleting the simple root $\bfgreek{alpha}_i.$ (See Sect.\,\ref{sec:standard-fundamental}.) Let $P_i = R_{F/\Q}(P_{0,i}).$ Then $P_i$ is a standard maximal parabolic subgroup of $G;$ note that $P_i$ is not absolutely maximal. Consider a simplex 
$\Delta_G$ whose vertices are $\{1,2,\dots,n-1\}$ where $i$ is identified with $P_i.$ Any standard parabolic subgroup 
$P$ corresponds to the simplex $\Delta_P$ whose set of vertices is the set of all $P_i$ containing $P$. 
Define $d(P) := {\rm dim}(\Delta_P).$ For example, $d(P_i) = 0$ and $d(B) = n-2.$ 
If $P \supset Q$ then $\Delta_P \subset \Delta_Q$ and by the Borel--Serre construction we have an embedding $\PQBSC \hookrightarrow  \overline{\PPBSC}$ which gives a restriction map in cohomology: 
$$
\r_{P,Q} : 
H^\bullet(\PPBSC, \tM_{\lambda,E}) \simeq H^\bullet(\overline{\PPBSC}, \tM_{\lambda,E}) \ \longrightarrow \ 
H^\bullet(\PQBSC, \tM_{\lambda,E}). 
$$ 
From this we get a spectral sequence converging to $H^\bullet(\PBSC, \tM_{\lambda,E})$ whose $E_1^{p,q}$ term is: 
$$
E_1^{p,q} \ := \ 
\bigoplus_{d(P) = p} H^q(\PPBSC, \tM_{\lambda,E}). 
$$
The boundary map $d_1 : E_1^{p,q} \to E_1^{p+1,q}$ is given by: 
\begin{equation}
\label{eqn:d-1}
d_1(\xi_P) \ = \ \sum_{\substack{Q \subset P \\ d(Q) = p+1}}  (-1)^{\sigma(P,Q)} \r_{P,Q}(\xi_P)
\end{equation}
for any $\xi_P \in H^q(\PPBSC, \tM_{\lambda,E})$ with $d(P) = p,$ and $\sigma(P,Q)$ is the place of the vertex defining $P$ in the ordered set of vertices defining $Q.$

\medskip
\subsubsection{\bf The spectral sequence at a transcendental level}
\label{sec:spectral-seq-trans}
Take an $\iota : E \to \C$ and define
$$
\Omega_P^q(\M_{{}^\iota\!\lambda}) \ := \ 
 \Hom_{K_\infty^0} \left(\Lambda^q(\fg/\fk), \, 
\cC^\infty (P(\Q)\backslash G(\A)/K_f, \omega_{\infty}^{-1}) \otimes  
\M_{{}^\iota\!\lambda}\right). 
$$
Also, define 
$$
\Omega_P^q(\M_{{}^\iota\!\lambda})^{(0)} \ := \ 
 \Hom_{K_\infty^0} \left(\Lambda^q(\fg/\fk), \, 
\cC^\infty(P(\Q)U_P(\A)\backslash G(\A)/K_f, 
\omega_{\infty}^{-1}) \otimes  \M_{{}^\iota\!\lambda}\right). 
$$
We have the inclusion $\Omega_P^q(\M_{{}^\iota\!\lambda})^{(0)} \subset \Omega_P^q(\M_{{}^\iota\!\lambda})$ and taking the constant term $\cF^P$ along $P$ gives a projection on to that subspace; recall that the constant term map 
$$
\cF^P : \cC^\infty(P(\Q)\backslash G(\A)/K_f, \omega_{\infty}^{-1}) 
\longrightarrow 
\cC^\infty(P(\Q)U_P(\A)\backslash G(\A)/K_f, \omega_{\infty}^{-1})
$$
is defined as: 
$$
\cF^P(\phi)(\ul g) \ = \ 
\int\limits_{U_P(\Q)\backslash U_P(\A)} \phi(\ul u \ul g)\, d\ul u.
$$
Hence $\Omega_P^q(\M_{{}^\iota\!\lambda})^{(0)}$ is a direct summand of $\Omega_P^q(\M_{{}^\iota\!\lambda}).$
Consider the double-complex $\Omega^{\bullet \bullet} = \Omega^{\bullet \bullet}(\M_{{}^\iota\!\lambda})$: 
$$
\xymatrix{
 & 0 & 0 &  & 0 & \\
0 \ar[r] & \Omega^{n-2,0} \ar[u] \ar[r] & \Omega^{n-2,1}  \ar[u]\ar[r]& \cdots \ar[r] & \Omega^{n-2, \d} \ar[u]\ar[r]& 0 \\
 & \vdots \ar[u]& \vdots \ar[u]&  & \vdots \ar[u] & \\
0 \ar[r] & \Omega^{1,0} \ar[u] \ar[r] & \Omega^{1,1}  \ar[u]\ar[r]& \cdots \ar[r] & \Omega^{1, \d} \ar[u]\ar[r]& 0 \\
0 \ar[r] & \Omega^{0,0} \ar[u] \ar[r] & \Omega^{0,1}  \ar[u]\ar[r]& \cdots \ar[r] & \Omega^{0, \d} \ar[u]\ar[r]& 0 \\
 & 0 \ar[u] & 0 \ar[u] &  & 0 \ar[u] & \\
}
$$
where, 
$$
\Omega^{p,q} \ = \ 
\bigoplus_{d(P) = p} \Omega^q_P(\M_{{}^\iota\!\lambda}); 
$$
the horizontal arrow $\Omega^{p,q} \to \Omega^{p,q+1}$ is exterior differentiation; the number of columns 
$\d$ is the dimension of $\fg/\fk$; the vertical arrow $\Omega^{p,q} \to \Omega^{p+1,q}$ is defined exactly as in (\ref{eqn:d-1}), i.e., as an alternating sum of maps 
$\Omega^q_P \to \Omega^q_Q$ for any $Q \subset P$ with $d(Q) = d(P)+1,$ while  
using the canonical map 
$\cC^\infty(P(\Q)\backslash G(\A), \omega_{\infty}^{-1}) \to 
\cC^\infty(Q(\Q)\backslash G(\A), \omega_{\infty}^{-1}).$ 
  The associated simple complex made from this double complex computes the cohomology of the boundary: 
$H^\bullet(\PBSC, \tM_{{}^\iota\!\lambda}).$ 
All of this remains true even if we work with the double-sub-complex 
$\Omega^{\bullet \bullet}(\M_{{}^\iota\!\lambda})^{(0)}.$  For the vertical arrows, one uses a partial constant term map: 
$$
\cF^Q_P : \cC^\infty(P(\Q) U_P(\A)\backslash G(\A), \omega_{\infty}^{-1}) \to 
\cC^\infty(Q(\Q)U_Q(\A)\backslash G(\A), \omega_{\infty}^{-1})
$$
given by:
$$
\cF^Q_P(\phi)(\ul g) \ = \ 
\int\limits_{U_Q(\Q) U_P(\A)\backslash U_Q(\A)} \phi(\ul u \, \ul g)\, d\ul u.
$$

\medskip
\subsection{Cohomology of a single boundary stratum}
\label{sec:coh-del-P}

\medskip
\subsubsection{\bf Cohomology of $\ppBSC$ as an induced representation}
It is clear from the $E_1^{p,q}$ term of the spectral sequence in Sect.\,\ref{sec:spectral-seq-arith} that to understand the cohomology of the boundary, we need to understand the cohomology of a single stratum $\PPBSC.$ 
It is known (\cite[Chap.\,3]{harder-book}) that
$$
H^\bullet(\partial_P\SGK, \tM_{\lambda,E}) \ = \ 
H^\bullet(P(\Q)\backslash G(\A)/ K_\infty^0K_f, \tM_{\lambda,E}).
$$
The space $P(\Q)\backslash G(\A)/ K_\infty^0K_f$ fibers over locally symmetric spaces of $M_P$ as we now explain. 
Let $\Xi_{K_f}$ be a complete set of representatives for $P(\A_f)\backslash G(\A_f)/K_f.$
Let $K^P_\infty = K_\infty^0 \cap P(\R),$ and for $\xi_f \in \Xi_{K_f}$, let $K_f^P(\xi_f) = P(\A_f) \cap \xi_f K_f \xi_f^{-1}.$ 
Then
$$
P(\Q)\backslash G(\A)/ K_\infty^0K_f  \ = \ \coprod_{\xi_f \, \in \, \Xi_{K_f}} P(\Q)\backslash P(\A)/ K^P_\infty K_f^P(\xi_f). 
$$
Let $\kappa_P : P \to P/U_P = M_P$ be the canonical map, and define 
$K_\infty^{M_P} = \kappa_P(K^P_\infty)$, and for $\xi_f \in \Xi_{K_f}$, let 
$K_f^{M_P}(\xi_f) = \kappa_P(K_f^P(\xi_f)).$ 
Define 
$$
\uSMP_{K_f^{M_P}(\xi_f)} \ := \ M_P(\Q)\backslash M_P(\A) / K_\infty^{M_P}K_f^{M_P}(\xi_f).
$$ 
The underline is to emphasize that we have divided out by $K_\infty^{M_P}$ which is not connected; see 
Sect.\,\ref{sec:ind-pi-0} below. Let $K_f^{U_P}(\xi_f) = U_P(\A_f) \cap \xi_f K_f \xi_f^{-1}.$ We have the fibration: 
$$
U_P(\Q)\backslash U_P(\A)/K_f^{U_P}(\xi_f) \ \hookrightarrow \ 
P(\Q)\backslash P(\A)/ K^P_\infty K_f^P(\xi_f)  \ \twoheadrightarrow \ 
\uSMP_{K_f^{M_P}(\xi_f)}.
$$
The corresponding Leray-Serre spectral sequence degenerates at $E_2$-level. See, for example, the discussion in 
Schwermer's articles \cite[Sec.\,7]{schwermer-lnm1447} and \cite[Thm.\,2.7]{schwermer-lnm988}.  The cohomology of the total space is given in terms of the cohomology of the base with coefficients in the cohomology of the fiber. For the cohomology of the fiber, one uses a classical theorem due to van Est: if $\u_P$ be the Lie algebra of $U_P$ then 
the cohomology of the fiber is the same as the unipotent Lie algebra cohomology group 
$H^\bullet(\u_P, \M_{\lambda, E})$, which is naturally an algebraic representation of $M_P$; the associated sheaf on $\uSMP_{K_f^{M_P}(\xi_f)}$ is denoted by putting a tilde on top. Putting all this together, we get: 
\begin{equation}
\label{eqn:bdry-P-coh}
H^\bullet(\partial_P\SGK, \tM_{\lambda,E})  \ = \ 
\bigoplus_{\xi_f \, \in \, \Xi_{K_f}} 
H^\bullet \left(\uSMP_{K_f^{M_P}(\xi_f)}, \widetilde{H^\bullet(\u_P, \M_{\lambda,E})}\right). 
\end{equation}
It is convenient to pass to the limit over all open compact subgroups $K_f$ and define:
$$
H^\bullet(\ppBSC, \tM_{\lambda,E}) \ := \ \varinjlim_{K_f} \, 
H^\bullet(\partial_P\SGK, \tM_{\lambda,E}). 
$$
Let $\uSMP \ := \ M_P(\Q)\backslash M_P(\A) / K_\infty^{M_P}.$ Now we can rewrite (\ref{eqn:bdry-P-coh}) as: 
$$
H^\bullet(\ppBSC, \tM_{\lambda,E})^{K_f} \ = \ 
\bigoplus_{\xi_f \, \in \, \Xi_{K_f}} 
H^\bullet \left(\uSMP, \widetilde{H^\bullet(\u_P, \M_{\lambda,E})}\right)^{K_f^{M_P}(\xi_f)}. 
$$
It is clear (using Mackey theory) that the right hand side is the $K_f$-invariants of an algebraically induced representation. We have the important result: 

\begin{prop}
\label{prop:bdry-coh-1}
The cohomology of $\ppBSC$ is given by:
$$
H^\bullet(\ppBSC, \tM_{\lambda,E}) \ = \ 
\aInd_{\pi_0(P(\R)) \times P(\A_f)}^{\pi_0(G(\R)) \times G(\A_f)}
\left( H^\bullet(\uSMP, \widetilde{H^\bullet(\u_P, \M_{\lambda,E})}) \right).
$$
The notation $\aInd$ stands for algebraic, or un-normalized, induction. 
\end{prop}
The process of induction from $\pi_0(P(\R))$  to $\pi_0(G(\R))$ is important and needs some explanation.

\medskip
\subsubsection{\bf Induction from $\pi_0(P(\R))$ to $\pi_0(G(\R))$}
\label{sec:ind-pi-0}
The parabolic subgroup $P$ is of the form $P = R_{F/\Q}(P_0)$ for a parabolic subgroup $P_0/F$ of 
$G_0 = \GL_n/F$.  
Since $F$ is totally real, we identify the set of infinite places $\place_\infty$ with $\Sigma_F = \Hom(F,\R),$ and  
$G(\R) = \prod_{\tau \in \Sigma_F} \GL_n(\R)$ and $P(\R) = \prod_{\tau \in \Sigma_F} P_0^\tau(\R)$, where 
$P_0^\tau = P_0 \times_{F, \tau} \R.$   Furthermore, suppose, $P_0$ corresponds to the partition $n = n_1 + \dots + n_k$ with $k \geq 2$ and $n_j \geq 1,$  then $M_{P_0} = \GL_{n_1} \times \cdots \times \GL_{n_k} /F.$
We have
\begin{eqnarray*}
K^{M_P}_\infty 
& = & \kappa_P(P(\R) \cap K_\infty^0) \\
& = & \kappa_P\left(P(\R) \cap (S(\R)^0\cdot \prod_\tau \SO(n))\right) \\
& = & \kappa_P\left( S(\R)^0\cdot \prod_\tau (P_0^\tau(\R) \cap \SO(n)) \right) \\
& = & \kappa_P\left( S(\R)^0\cdot \prod_\tau {\rm S}({\rm O}(n_1) \times \cdots \times {\rm O}(n_k))\right) \\
& \simeq & S(\R)^0\cdot \prod_\tau {\rm S}({\rm O}(n_1) \times \cdots \times {\rm O}(n_k)). 
\end{eqnarray*}
Note that $\pi_0(P(\R))$ has order $2^{{\sf r}k}$ and 
$\pi_0(K^{M_P}_\infty)$ has order $2^{{\sf r}(k-1)}.$ (Recall, ${\sf r} = [F:\Q].$) 
Inclusion of   components induces a canonical surjective 
map $\pi_0(P(\R)) \to \pi_0(G(\R))$ giving a short exact sequence: 
\begin{equation}
\label{eqn:ses-infinity}
1 \longrightarrow \pi_0(K^{M_P}_\infty) \longrightarrow \pi_0(P(\R)) 
\longrightarrow  \pi_0(G(\R)) \longrightarrow 1.
\end{equation}
In the definition of $\uSMP$ we have divided out by $K^{M_P}_\infty$. For $\SMP$ we divide only by $K^{M_P,0}_\infty$, 
the connected component of the identity in $K^{M_P}_\infty.$ Hence, for any sheaf $\tM$ coming from an algebraic representation 
$\M$ of $M_P/\Q$ we have: 
$$
H^\bullet(\uSMP, \tM) \ = \ H^\bullet(\SMP, \tM)^{\pi_0(K^{M_P}_\infty)}.
$$
This means that $H^\bullet(\uSMP, \widetilde{H^\bullet(\u_P, \M_{\lambda,E})})$ is a module for 
$\pi_0(P(\R)) \times P(\A_f)$ on which $\pi_0(K^{M_P}_\infty)$ acts trivially and so is naturally a module for 
$\pi_0(G(\R)) \times P(\A_f),$ which is then induced up to get a module for $\pi_0(G(\R)) \times G(\A_f)$. 
See also Rem.\,\ref{rem:ind-infinity} below.

\medskip
\subsubsection{\bf Kostant's theorem on unipotent cohomology}
\label{sec:kostant-theorem}

The structure of the unipotent cohomology group $H^{\bullet}(\u_P, \M_{\lambda, E})$ is well-known by results of Kostant \cite{kostant}, and we briefly describe these results in our situation.

\smallskip
The calculation of the unipotent cohomology group is over the field $E$. Recall that we are dealing with the split group $G \times E = \prod_{\tau : F \to E} G_0^\tau$ where $G_0^\tau = G_0 \times_{F,\tau} E.$ 

Let $\bfdelta_{G_0}$ stand for the set of roots of $G_0$ and $\bfdelta_{G_0}^+$ the subset of positive roots 
(for choice of Borel subgroup being the upper triangular subgroup). Let $\bfpi_{G_0}$ be the set of simple roots.  
The notations $\bfdelta_{G_0^\tau}$, $\bfdelta_{G_0^\tau}^+$ and $\bfpi_{G_0^\tau}$ are clear. 
Let $P = R_{F/\Q}(P_0)$ be a parabolic subgroup of $G$ as above, and we let $P_0^\tau := P_0 \times_{\tau} E.$
The Weyl group factors as $W = \prod_{\tau : F \to E} W_0^\tau$ with each $W_0^\tau$ isomorphic to the permutation group $\perm_n$ on $n$-letters.

Let $W^P$ be the set of Kostant representatives in the Weyl group $W$ of $G$ corresponding to the parabolic subgroup $P$ which is defined as: $W^P = \{ w = (w^\tau) : w^\tau \in W^\tau_0{}^{P^\tau_0} \},$ where,  
$$
W^\tau_0{}^{P^\tau_0} \ := \{w^\tau \in W^\tau_0  : \ (w^\tau)^{-1}\alpha > 0, \ \forall \alpha \in \bfpi_{M_{P_0^\tau}} \}. 
$$
(Here $\bfpi_{M_{P_0^\tau}} \subset \bfpi_{G_0^\tau}$ denotes the set of simple roots in the Levi quotient $M_{P_0^\tau}$ of $P_0^\tau.$)

For $w \in W$, and in particular for $w \in W^P$, and for $\lambda \in X^*(T)$, by $w\cdot \lambda$ we mean the twisted action of $w$ on the weight $\lambda$, i.e., 
$w \cdot \lambda = (w^\tau \cdot \lambda^\tau)_{\tau : F \to E}$ and 
$$
w^\tau\cdot \lambda^\tau = w^\tau(\lambda^\tau + \bfgreek{rho}_n) - \bfgreek{rho}_n.
$$

Since $w \in W^P$, the weight $w\cdot \lambda$ is dominant and integral as a weight for $M_P \times E$ and we can consider the irreducible finite-dimensional representation $\M_{w\cdot \lambda}$ of $M_P \times E$. Kostant's theorem asserts that as $M_P \times E$-modules, one has a multiplicity-free decomposition: 
\begin{equation}
\label{eqn:kostant}
H^q(\u_P, \M_{\lambda, E}) \ \simeq \  \bigoplus_{\substack{w \in W^P \\ l(w) = q}} \M_{w \cdot \lambda, E}.
\end{equation}

The reader should bear in mind that the above result of Kostant can be parsed further since everything in sight factors over the set of embeddings $\tau : F \to E.$  We have:

\begin{eqnarray*}
H^q(\u_P, \M_{\lambda, E}) & = & \bigoplus_{\sum q_\tau = q} \, \bigotimes_{\tau : F \to E} 
H^{q_\tau}(\u_{P^\tau_0}, \M_{\lambda^\tau, E}) \quad \mbox{(K\"unneth theorem)}\\ 
& = & \bigoplus_{\sum q_\tau = q} \, \bigotimes_{\tau :  F \to E} \, 
\bigoplus_{\substack{w^\tau \in W_0^\tau{}^{P_0^\tau} \\ l(w^\tau) = q_\tau}}
\M_{w^\tau \cdot \lambda^\tau, E} \quad \mbox{(Kostant for each $\u_{P_0^\tau}$)}\\
& = & \bigoplus_{\sum q_\tau = q}  \, 
\bigoplus_{\substack{w^\tau \in W_0^\tau{}^{P_0^\tau} \\ l(w^\tau) = q_\tau}} \, \bigotimes_{\tau :  F \to E}
\M_{w^\tau \cdot \lambda^\tau, E} \quad \mbox{(tensoring commutes with direct sum)}\\
& = & \bigoplus_{\substack{w \in W^P \\ l(w) = q}} \M_{w \cdot \lambda, E} 
\quad \mbox{(since $w = (w^\tau)$ and $l(w) = \sum_\tau l(w^\tau)).$}
\end{eqnarray*}

Applying Kostant's theorem (\ref{eqn:kostant}) to the boundary cohomology as in Prop.\,\ref{prop:bdry-coh-1}, 
while using the description of the action of $\pi_0(P(\R))$ in Sect.\,\ref{sec:ind-pi-0}, we get 

\begin{prop}
\label{prop:bdry-coh-2}
The cohomology of $\ppBSC$ is given by:
$$
H^q(\ppBSC, \tM_{\lambda, E})  \ = \ 
\bigoplus_{w \in W^P}
{}^{\rm a}{\rm Ind}_{\pi_0(P(\R)) \times P(\A_f)}^{\pi_0(G(\R)) \times G(\A_f)}
\left(H^{q - l(w)}(\SMP, \tM_{w \cdot \lambda, E})^{\pi_0(K^{M_P}_\infty)} \right).
$$
\end{prop}

\begin{rem}\label{rem:ind-infinity}{\rm 
Suppose we have an irreducible $M_P(\A_f)$-module $\pi_f$ such that for some 
$\lambda \in X^*(T)$ and some $w \in W^P$, we have $\pi_f \in {\rm Coh}(M_P, w\cdot \lambda).$ Let $\varepsilon$ be a character of $\pi_0(M_P(\R)) = \pi_0(P(\R))$. Then the induction to $\pi_0(G(\R)) \times G(\A_f)$ of 
$ \varepsilon  \times \pi_f$ will contribute to
$H^\bullet(\ppBSC, \tM_{\lambda, E})$ if and only if $\varepsilon$ is trivial on $\pi_0(K^{M_P}_\infty)$. 
}\end{rem}

\medskip
\subsection{The residual cohomology}
\label{sec:coh-res}

In this subsection we come back to the claim made in  (\ref{eqn:cuspidal-claim}).  Towards this, we start by giving a complete description of the square integrable  classes in terms of the cuspidal classes for the reductive quotients 
of certain parabolic subgroups, i.e., we give a much more explicit version of the decomposition
in (\ref{eqn:L2-deco-two}). We work at a transcendental level: our coefficient systems are $\C$-vector spaces.
The following is explained in \cite[Chap.\,3]{harder-book} in greater detail.

\medskip
\subsubsection{\bf An injectivity result in cohomology}
Write $N=u v.$ Take a parabolic subgroup 
$P_0$ of $G_0$ containing the
standard Borel subgroup, and with reductive quotient $M_0= \Gl_u\times\dots \times \Gl_u.$   
We write $P=R_{F/\Q}(P_0), M=R_{F/\Q}(M_0).$     
To such a decomposition $N=u v$, we also have a  $\Theta$-stable parabolic subgroup 
$P^\vee$ with reductive quotient $M^\vee=\Gl_v\times \Gl_v\times\dots \times\Gl_v.$
Take a highest weight  $\lambda$ whose $\tau$-component is given by
 $$
   \lambda^\tau =   a^\tau_1\bfgreek{gamma}_v +   a^\tau_2  \bfgreek{gamma}_{2v} + \dots  + 
  a^\tau_{u-1} \bfgreek{gamma}_{(u-1)v} +d\bfgreek{delta}_N
    $$ 
where in addition we have $ a^\tau_i=  a^\tau_{b-i}.$  The $\bfgreek{gamma}_{iv} $ are the dominant characters 
on $P^\vee,$ let $\q$ be the Lie algebra of $P^\vee.$ To this datum Vogan and Zuckerman  \cite{Vo-Zu} attach an irreducible Harish-Chandra module $A_{\q}(\lambda),$
which has nontrivial cohomology with coefficients in $\M_\lambda.$ This module contains 
a lowest $K_\infty$ type $\Theta(\q,\lambda)$ with highest weight $\mu_{\q}(\lambda)$ which is given by the 
formula \cite[(5.2)]{Vo-Zu}. Furthermore, 
\begin{equation*}
\begin{split}
H^\bullet(\g, K^0_\infty; A_{\q}(\lambda)\otimes\M_\lambda) & \ = \ \Hom_{K^0_\infty}(\Lambda^\bullet(\g/\k), A_{\q}(\lambda)\otimes \M_\lambda)) \\ 
& \ = \ \Hom_{K^0_\infty}(\Lambda^\bullet(\g/\k), A_{\q}(\lambda)(\Theta(\q,\lambda))\otimes \M_\lambda);  
\end{split}
\end{equation*}
and especially the differentials in the complex on the right all vanish.

\smallskip

The absolute Weyl group is $\prod_{\tau:F\to \C} W_0.$
 We choose a specific Kostant representative $w_{u,v}\in W^{P}$ whose $\tau$-component
is the  permutation in the letters $1,2,\dots, n$ given by the following rule: write $\nu=i+(j-1)v \text{ with }  1\leq i \leq u $ then $w_{u,v}(\nu) =  j+(i-1) v .$  The element $w(\lambda+\rho_N)-\rho_N\in  X^*(T\times E)$ and we get for its $\tau$ component
   \begin{align}
   \label{WAB}
   (w_{u,v} (\lambda +\rho_N)-\rho_N)^\tau  = (a^\tau_1+u-1)\bfgreek{gamma}^{M }_1+ (a^\tau_2+u-1)\bfgreek{gamma}^{M }_2 +\dots+( a^\tau_{b-1}+u-1)\bfgreek{gamma}^{M }_{v-1}+\bfgreek{gamma}_v\cr+
    (a^\tau_1+u-1)\bfgreek{gamma}^{M}_{1+v}+(a^\tau_2+u-1)\bfgreek{gamma}^{M }_{2+v} +\dots +( a^\tau_{b-1}+u-1)\bfgreek{gamma}^{M }_{2v-1}+\bfgreek{gamma}_{2v}+\dots
\end{align}
The length of  this Kostant representative is
 $$
 l(w_{u,v}) =n(u-1)(v-1)/4.    
 $$
Let $w_P$ be the longest Kostant representative which sends all the roots in $U_P$ to negative roots.
Then we define the  (reflected) Kostant representative  $ w^\prime_{u,v}$ by $ w_{u,v}=w_P w^\prime_{u,v}.$
We get 
 \begin{align}
 \label{WMU}
 w^\prime_{u,v}(\lambda+\rho)-\rho=  \mu=
   ( a^\tau_1+u-1)\bfgreek{gamma}^{M }_1+ (a^\tau_2+u-1)\bfgreek{gamma}^{M }_2 +\dots+( a^\tau_{v-1}+u-1)\bfgreek{gamma}^{M }_{v-1}+ \cr
    +
    (a^\tau_1+u-1)\bfgreek{gamma}^{M}_{1+v}+(a^\tau_2+u-1)\bfgreek{gamma}^{M }_{2+u} +\dots +( a^\tau_{v-1}+u-1)\bfgreek{gamma}^{M }_{2v-1}+ \dots \cr
    -(v+1)(\bfgreek{gamma}_v+\bfgreek{gamma}_{2v} +\dots + \bfgreek{gamma}_{(u-1)v})\phantom{          }.
    \end{align}
    We write the maximal torus $T\subset M $ as product $T_1\times T_2\times \dots\times T_u$ where the factors
    are canonically identified to $\Gm^v.$ Let $\delta_i$ be the determinant on the $i$-th factor. Then the characters $X^*(T)$ can be written as arrays: 
\begin{align*}
\begin{matrix}
w_{u,v} (\lambda +\rho_N)-\rho_N \ =  \ (\dots, \mu^{(1)}+ d_i\delta_i, \dots ), \cr
w^\prime_{u,v} (\lambda +\rho_N)-\rho_N \  = \  (\dots, \mu^{(1)}+ d^\prime_i\delta_i, \dots ), 
\end{matrix}
\end{align*}
i.e., the semi-simple component is always the same: The $\tau$-component is    $\mu^{(1)}= ( a^\tau_1+u-1)\bfgreek{gamma}^{M }_1+ (a^\tau_2+u-1)\bfgreek{gamma}^{M }_2 +\dots+( a^\tau_{v-1}+u-1)\bfgreek{gamma}^{M }_{v-1}$ and for the  coefficients of the $\delta_i$ terms we have
\begin{align}\label{dminusd}
d_i-d_{i+1} = 1,  \ \ d_i^\prime-d_{i+1}^\prime=-1. 
\end{align}
We see that $w^\prime_{u,v} (\lambda +\rho_N)-\rho_N$ is in the negative chamber.  On the $i$-th  factor of  $\Gl_u$ we have the module  $\D_{\mu^{(1)}+d_i\delta_i} $ and we define 
$\D_\mu=\otimes \D_{\mu^{(1)}+d_i\delta_i} $ and accordingly $\D_{\mu^\prime}.$ 
     
   \begin{satz}
   \label{thm:lowest-degree} We have a nonzero intertwining operator $:T^{(\rm loc) }(\D_\mu):\aInd_{P(\R)}^{G(\R)}\D_{\mu^\prime}\to
   \aInd_{P(\R)}^{G(\R)}\D_{\mu } $ 
and get a diagram
\begin{align}
  \begin{matrix}  \aInd_{P(\R)}^{G(\R)}\D_{\mu^\prime} & \ppfeil{T^{(\rm loc) }(D_\mu) }&   \Aql \cr
&& \downarrow\cr
&&  \aInd_{P(\R)}^{G(\R)}\D_{\mu }  
\end{matrix} 
\end{align}
The horizontal arrow is surjective, and the vertical arrow is injective.
The map induced by the vertical arrow in cohomology 
$$
H^q(\fg,K_\infty; \Aql \otimes \M_\lambda)  \ \longrightarrow \  
H^q(\fg, K_\infty; \aInd_{P(\R)}^{G(\R)}\D_{\mu }  \otimes \M_\lambda)
$$
is a bijection in the lowest degree of nonzero cohomology; this lowest degree is 
$$
q \ =\ v \left[\frac{u^2}{4}\right]  +\frac{n(u-1)(v-1)}{4}.
$$ 
\end{satz}
  
  \begin{proof}   
    From the description of $\aInd_{P(\R)}^{G(\R)}\D_{\mu } $  as an induced module we can compute its lowest $K_\infty$
  type. The highest weight of this lowest $K_\infty$ type  is equal  to the highest weight 
  $\mu_\q(\lambda)$  of the lowest $K_\infty$ type 
  in  $\Aql.$
   Then it follows from Delorme's computation that the complexes 
   $$
   \Hom_{K^0_\infty}(\Lambda^\bullet(\g/\k), A_{\q}(\lambda)\otimes \M_\lambda)) \to  \Hom_{K^0_\infty}(\Lambda^\bullet(\g/\k), \Ind_{P(\R)}^{G(\R)} \D_{\mu } \otimes\M_\lambda)
   $$
   are zero in degree $\bullet <q$ and are nonzero in degree $q.$ Hence we get injectivity for cohomology in this degree.
   (For more details we refer to  \cite[Chap.\,3]{harder-book}.)
 \end{proof}

In the extremal case $u=n, v =1$ the parabolic  subgroup $P$ is all of $G$ and 
$\Aql=\D_\lambda.$ In this case, and only  this case, the representation
$\Aql$ is tempered.

\medskip
\subsubsection{\bf Cuspidal spectrum is disjoint from boundary spectrum} 
We go back to the global situation, and in the following we put $w=w_{u,v}, w^\prime= w^\prime_{u,v}.$ 
The parabolic subgroup $P$ and its reductive quotient $M$ will be as above.
We consider an isotypical $\sigma_f$ in  the cuspidal  cohomology 
$H_{\rm cusp}^\pkt(\SMK, \Dmp\otimes \M_{w\cdot \lambda}).$ The K\"unneth formula yields that 
we can write $\sigma_f = \sigma_{1,f}\times\sigma_{2,f}\times  \dots\times \sigma_{v,f}$   where 
all the $\sigma_{i,f} $ occur in the cuspidal cohomology of $\Gl_u,$ hence they may be compared.
The relation (\ref{dminusd}) needs us to require that $\sigma_{i+1,f}= \sigma_{i,f}\otimes\vert\delta\vert.$
We say that $\sigma_f$ is a segment.
We assume $v\not=1$ and hence $P\not= G$. We analyze the constant term in the Eisenstein series  
(see \cite[Chap.\,3]{harder-book}): 
\begin{align*}
\cF^P\circ\Eis(\sigma\otimes \ul s): f \mapsto f + C(\sigma,\ul s)T^{\rm loc}(\sigma,\ul s)(f).
\end{align*}
We know that  under the assumption that $\sigma_f$ is a segment (and only under this assumption) the 
factor $C(\sigma,\ul s)$ has a simple pole along the lines $s_i=0.$ The  operator $T^{\rm loc}(\sigma,\ul s)$ is a product of local operators at all places
$$
T^{\rm loc}(\sigma,\ul s)=T^{\rm loc}_\infty(\sigma_\infty,\ul s)\times \prod_{p}T_p^{\rm loc}(\sigma_p,\ul s), 
$$ 
and the local factors are holomorphic as long as $\Re(s_i)\geq 0.$ We can evaluate the local operators at $\ul s=\ul 0$  then $T^{\rm loc}_\infty(\sigma_\infty,\ul 0)= T^{(\rm loc) }(\sigma_\infty) $ and at the finite places we get
homomorphisms
\begin{align} 
T_p^{\rm loc}(\sigma_p,\ul 0) : \aInd_{P(\Q_p)}^{G(\Q_p)} V_{\sigma_p}\to J_{\sigma_p}
\end{align}
where the $J_{\sigma_p} $ are essentially unitary, for almost all places  $J_{\sigma_p}^{K_p}$ is one-dimensional
and at these places the homomorphism $T_p^{\rm loc}(\sigma_p,\ul 0)$ is an isomorphism.  
We get a diagram
\begin{align}
 \begin{matrix}
\Aql\otimes J_{\sigma_f}^{K_f}  & \into  &  L^2(G(\Q)\backslash G(\A)/K_f, \omega_{\M_\lambda}^{-1}|_{S(\R)^0})\cr
 \downarrow &&  \downarrow \cF^P\cr
(\aInd_{P(\A)}^{G(\A)} V_{\sigma})^{K_f} & \into & \cA(P(\A)\backslash G(\A)/K_f)
\end{matrix}
\end{align}
We have the  inclusion $\Hom_{K_\infty}(\Lambda^\pkt(\g/\k), (\aInd_{P(\A)}^{G(\A)} V_{\sigma})^{K_f} \otimes \M_\lambda)\into
H^\pkt( \PPBSC,\M_\lambda)$ and hence we get a nontrivial contribution
\begin{align}\label{inj-low-degree}
H^q(\g, K_\infty; \Aql\otimes \M_\lambda)\otimes J_{\sigma_f}^{K_f} \into  H^q( \PPBSC,\M_\lambda)
\end{align}
where $q$ is the number in Thm.\,\ref{thm:lowest-degree} above.   Hence we can conclude

\begin{thm}
  \label{thm:cuspidal-vs-square-int} 
  An isotypical subspace $H^\pkt_{(2)}(\SGK, \M_\lambda)(\pi_f)$ in the square integrable cohomology is cuspidal
  if and only if  it's restriction to boundary cohomology 
  \begin{align*}
H^\pkt_{(2)}(\SGK, \M_\lambda)(\pi_f)  \ \longrightarrow \  H^\pkt(\partial\SGK,\M_\lambda)
\end{align*}
is the zero map, or what amounts to the same, we have
$H^\pkt_!(\SGK, \M_\lambda)(\pi_f) = H^\pkt_{(2)}(\SGK, \M_\lambda)(\pi_f).$ 
\end{thm}

Combining the results of Borel--Garland \cite{borel-garland} 
and M\oe glin--Waldspurger  \cite{moeglin-waldspurger} we get that the homomorphism
\begin{align}
\label{eqn:square-int-cohomology}
\bigoplus_{u|n} \bigoplus_{\sigma_f : \rm segment} H^\pkt(\g,K_\infty; \Aql\otimes  \M_\lambda)\otimes J_{\sigma_f}\to
H^\pkt_{(2)}(\SGK,\M_\lambda)
\end{align}
is surjective. This gives us the decomposition into isotypical spaces of $H^\pkt_{(2)}(\SGK, \M_\lambda).$ We separate 
the cuspidal part ($v=1$) from the residual part and get
\begin{align*}
H^\pkt_{(2)}(\SGK,\M_\lambda) = \bigoplus_{\pi_f : \rm cuspidal} H^\pkt_{\rm cusp}(\SGK, \M_\lambda)(\pi_f) 
\ \oplus \
\bigoplus_{\substack{u|n \\ u < n}} \bigoplus_{\sigma_f : \rm segment}  \overline{ H^\pkt(\g,K_\infty; \Aql\otimes  \M_\lambda)}\otimes J_{\sigma_f}, 
\end{align*}
where the bar on top means we have gone to its image via the map in (\ref{eqn:square-int-cohomology}). 
It follows from the theorem of Jacquet--Shalika \cite{jacquet-shalika-II} that there are no intertwining operators between the summands.

In the extremal case $u=n, v =1$ the parabolic  subgroup $P$ is all of $G$ and 
$\Aql=\D_\lambda.$ In this case and only  this case the representation
$\Aql$ is tempered, and the lowest degree of nonvanishing cohomology is the number $b_n^F.$
An easy computation shows that in the case $v>1$ the number $q <b_n^F.$  Then (\ref{inj-low-degree})  
implies that in degree $q$ 
 $$
 H^q(\g, K_\infty; \Aql\otimes \M_\lambda)\otimes J_{\sigma_f} \to  H^q( \SGK,\M_\lambda)
$$ 
is injective. This has also been proved by Grobner ~\cite{grobner}.
The above result, which we announced earlier (\ref{eqn:cuspidal-claim}),  can be  sharpened as in the following theorem. During the induction argument we use Thm.\,\ref{thm:cuspidal-vs-square-int} for
  the reductive quotients $M$ of the parabolic subgroups.

\begin{thm}
\label{thm:cuspidal-vs-boundary}  
If $\pi_f $ is cuspidal  then $\pi_f$ does not occur as an isomorphism type in the Jordan--H\"older filtration of 
$H^\bullet(\PBSC,\tM_\lambda).$
\end{thm}

\begin{proof} Let $\pi_f$ be cuspidal and if possible let 
  $V_{\pi_f}\subset H^\pkt(\partial\SGK,M_\lambda)(\pi_f) $ be an irreducible submodule. It    follows from the spectral sequence considered in Sect.\,\ref{sec:spectral-seq-arith} and Sect.\,\ref{sec:spectral-seq-trans} that it occurs a submodule in some $H^{q - l(w)}(\SMP, \tM_{w \cdot \lambda, E})^{\pi_0(K^{M_P}_\infty)}$ and if we restrict even further we may assume that it occurs in the inner cohomology and hence in some
   $H_!^{q - l(w)}(\SMP, \tM_{w \cdot \lambda, E})^{\pi_0(K^{M_P}_\infty)}(\sigma_f).$ If the $\sigma_f$ is not cuspidal then 
  $ H^{q - l(w)}(\SMP, \tM_{w \cdot \lambda, E})^{\pi_0(K^{M_P}_\infty)}(\sigma_f) $ contains a copy of $V_{\sigma_f}$
  which survives if we restrict it to the boundary $\partial\SMP$ by 
  Thm.\,\ref{thm:cuspidal-vs-square-int} above. We can repeat our argument
  until we find it as a submodule in the cuspidal cohomology of a boundary stratum, i.e., we may assume our $\sigma_f$
  is cuspidal.  But since $\pi_f$ is cuspidal it cannot be a submodule  of the finite part
of  a representation parabolically induced from a cuspidal representation of a proper parabolic subgroup; such a representation cannot be almost everywhere equivalent to a cuspidal representation by the well-known classification theorem of Jacquet and Shalika \cite[Thm.\,4.4]{jacquet-shalika-II}. 
\end{proof}

\bigskip
\section{\bf The strongly inner spectrum and applications}
\label{sec:coh-!!-periods}

\medskip
\subsection{The strongly inner spectrum: definition and properties}
\label{sec:coh-!!}

\medskip
\subsubsection{\bf Identifying the cuspidal spectrum within the inner cohomology }
\label{sec:cuspidal-inside-inner}
We return to the arithmetic context, i.e., our coefficient systems are vector spaces over $E.$  
Consider the isotypical decomposition
  $$
  H^\pkt_!(\SGK, \M_{\lambda, E}) \ = \ \bigoplus_{\pi_f\in \Coh_!(G,\lambda,K_f)}  H^\pkt_!(\SGK, \M_{\lambda, E})(\pi_f). 
    $$ 
For an embedding $\iota: E\to \C$, Thm.\,\ref{thm:cuspidal-vs-boundary} above says that the isomorphism type
${}^\iota\pi_f$ is cuspidal if and only if $H^q_!( \SGK, \tM_{{}^\iota\!\lambda})({}^\iota\pi_f) = 
H^q( \SGK,\M_{{}^\iota\!\lambda})({}^\iota\pi_f).$
This motivates us to define the set 
 $$
 \Coh_{!!}(G,\lambda,K_f) \ = \ 
 \{\pi_f\in \Coh_{!}(G,\lambda,K_f) \ \vert \  
 H^\pkt_!(\SGK, \tM_{\lambda,E} )(  \pi_f) =  H^\pkt (\SGK, \tM_{\lambda,E} )(  \pi_f)\}
 $$
 and then define the {\it strongly inner cohomology} as
 \begin{align}
\label{eqn:strongly-inner}
H_{!!}(\SGK, \tM_{\lambda,E}) \ := \ 
\bigoplus_{\pi_f\in  \Coh_{!!}(G,\lambda,K_f)}H^\pkt (\SGK, \tM_{\lambda,E})(  \pi_f). 
\end{align}
Then we get an arithmetic characterization of the cuspidal cohomology
  \begin{align}
 H^\bullet_{!!}(\SGK, \tM_{\lambda, E}) \otimes_{E,\iota} \C \ = \ H^\bullet_{\rm cusp}(\SGK, \tM_{{}^\iota\!\lambda, \C}). 
\end{align}

\medskip
\subsubsection{\bf Cuspidality of $\pi_f$ in terms of estimates on the Satake parameters} 
Let $\pi_f \in \Coh_{!}(G,\lambda,K_f).$ Let $\place$ be a finite set of places including all the archimedean places such that for any $\p \notin \place$, we have $K_\p = \GL_n(\cO_\p).$ Consider the 
local factor $\pi_\p$ at $\p,$ which is an algebra homomorphism $ \pi_\p : \HH_{K_\p} \to E$ 
of the local Hecke algebra 
$\HH_{K_\p}= \cC^\infty_c(\GL_n(F_\p)/\!\!/\GL_n(\cO_\p)).$ 
Let $\varpi_\p$ be a uniformizer at $\p$ and $q_\p$ the cardinality of the residue field $\cO_\p/\p\cO_\p.$
The Satake isomorphism says that $\pi_\p$ 
is  obtained from a homomorphism  $\vartheta_\p :T(F_\p)   \to  \bar E^\times $, which is unique modulo
the action of the Weyl group.  For $1 \leq i \leq n$, we have the cocharacter $\eta_i :\Gm  \to T_0 $ given by 
the diagonal matrix $ \eta_i(t) = {\rm diag}(1, 1, \dots, t, \dots , 1)$, 
where the $t$ is in the $i$-th place. Then $\vartheta_\fp $ is determined  by the values
 $\vartheta_{i, \p}=\vartheta_\fp (\eta_i(\varpi_\p)). $  The {\it Satake parameter} of $\pi_\p$ is 
the diagonal matrix 
$$ 
\vartheta_{\p}(\pi_\fp)=\vartheta_{\p}= 
\begin{pmatrix} 
\vartheta_{1,\fp} & & & & \cr
& \ddots & & &  \cr
 & & \vartheta_{i, \p} & & \cr
 & && \ddots& \cr
 & &   &  &\vartheta_{n, \p}
 \end{pmatrix} 
$$ 
modulo conjugation by the Weyl group. The parameters $\vartheta_{i, \p}$, for a fixed $\p$, lie in some finite algebraic extension $\tilde  E /E$ of $E$. Let $\bar{E}$ be an algebraic closure of $E.$ 
Cuspidality of $\pi_f$ can be recognized by the size of the Satake parameters: 

\begin{prop} 
\label{prop:cuspidality-satake}
Let $\pi_f\in \Coh_{!}(G,\lambda,K_f)$ and $\iota: E\to \C.$ The following are equivalent: 
\begin{enumerate}
\item ${}^\iota\pi_f$ is cuspidal.
\item For all $\p\not\in S,$ any  $\bar \iota: \bar E \to \C$ extending $\iota : E \to \C$, and any $i$ we have 
$$  q_\p^{d-\frac{1}{2}} \ < \ \vert \bar\iota(\vartheta_{i,\fp}) \vert  \ < \ q_\p^{d+\frac{1}{2}}. $$
\end{enumerate}
\end{prop}

\begin{proof}
(1)$\implies$(2) follows from the well-known estimates, due to Jacquet and Shalika \cite[Cor.\,2.5]{jacquet-shalika-I}, for the Satake parameters of a 
{\it unitary} cuspidal automorphic representation of $\GL_n$ over a number field; note that ${}^\iota\pi_f \otimes |\ |^{d}$ would be the finite part of a unitary cuspidal representation. 

For (2)$\implies$(1), we prove the contrapositive statement as follows: Since $\pi_f \in \Coh_!$, we know that 
${}^\iota\pi_f \in \Coh_{(2)}$; hence from the discussion in Sect.\,\ref{sec:borel-garland}, 
${}^\iota\pi_f \otimes |\ |^d$ is the finite part of an automorphic representation appearing in the discrete spectrum. Suppose ${}^\iota\pi_f$ is not cuspidal then  
by M\oe glin-Waldspurger \cite{moeglin-waldspurger}, there exists $a,b$ with $n=ab$ and $b > 1$, 
a unitary cuspidal representation $\sigma$ of $R_{F/\Q}(\GL_a/F)$ such that ${}^\iota\pi_f \otimes |\ |^d$ is the finite part of the unique irreducible quotient of the representation of $G(\A)$ parabolically induced from 
$$
\sigma[(b-1)/2] \times \sigma[(b-3)/2] \times \cdots \times \sigma[-(b-1)/2].
$$
(Here, $\sigma[s] := \sigma \otimes |{\rm det}|^s.$) For $\p \notin \place$, let the Satake parameter of $\sigma_\p$ be the diagonal matrix 
$\vartheta_\p(\sigma_\p) = {\rm diag}(\vartheta_{1,\fp}, \dots, \vartheta_{a, \p}).$
Since $\sigma$ is unitary cuspidal, by Jacquet and Shalika {\it loc.\,cit.}\,, we have 
$q_\p^{-\frac{1}{2}} < \vert \vartheta_{i,\fp} \vert  <  q_\p^{\frac{1}{2}}.$
Then the Satake parameter for ${}^\iota\pi_f$ is a block-diagonal matrix of the form: 
$$
\begin{pmatrix} 
\vartheta_\p(\sigma_\p) q_\p^{-(b-1)/2+d}& &  & \cr
& \vartheta_\p(\sigma_\p) q_\p^{-(b-3)/2+d} & &  \cr
& & \ddots &   \cr
 & &  & \vartheta_\p(\sigma_\p) q_\p^{(b-1)/2+d}
 \end{pmatrix}.  
$$
If for some $j$ we have $|\vartheta_{j,\p}| < 1$ then for the parameter at the $j$-th place in first block above, we have the inequality: 
$$
\left|\vartheta_{j,\p}  \, q_\p^{-(b-1)/2+d} \right| \ < \  q_\p^{-(b-1)/2+d} \ < \ q_\p^{-1/2+d}.  
$$
On the other hand, if for every $j$ we have $|\vartheta_{j,\p}| \geq 1$, then for any of the parameters in the last block we have: 
$$
\left|\vartheta_{j,\p}  \, q_\p^{(b-1)/2+d} \right| \ \geq  \  q_\p^{(b-1)/2+d} \ \geq \ q_\p^{1/2+d}.
$$
Hence the estimates in (2) are always violated. 
\end{proof}

\medskip
\subsubsection{\bf Cuspidal cohomology for $\Gl_n$}
Consider the cohomology group $H^\bullet_?(\SGK, \tM_\lambda),$
where $? = \{c, !, \mbox{empty}, \partial\}.$ 
Define $\Coh(H^\pkt_?(\SGK, \tM_\lambda))$ to be the set of isomorphism classes of absolutely irreducible 
$\HK$-modules which occur in a Jordan-H\"older series; this definition makes sense even if the module is not semi-simple. 
For an isomorphism class $\pi_f$ let  $\pi_f^{(S)}$ be the restriction
to the central subalgebra $\HH_{K_f}^{(S)}.$ The various criteria for $\pi_f$ to be cuspidal are collected together in the following 

\begin{thm} 
\label{thm:cuspidal-cohomology}
Let $\lambda \in X_0^*(T)$ be a pure weight and $\pi_f\in \Coh(H^\pkt_!(\SGK, \M_{\lambda,E})) = 
\Coh_!(G,\lambda,K_f).$ 
The following are equivalent: 
\begin{enumerate}
\smallskip  
\item  There exists  $\iota : E \to \C$  such that ${}^\iota\pi_f$ is cuspidal. 

\smallskip  
  
\item For all $\iota : E \to \C$,   ${}^\iota\pi_f$ is cuspidal. 

\smallskip  
  
  \item The isomorphism type $\pi_f^{(S)}$ does not occur in
  $\Coh(H^\pkt(\partial\SGK, \tM_{\lambda, E})).$
\smallskip  
  
 \item We have $H^\bullet(\SGK,\tM_{\lambda, E})(\pi_f)\neq (0)$ if and only if $b_n^F \leq  \bullet \leq  t_n^F.$ 

\smallskip  
  
  \item  There exists an $\iota : E \to \C$  such that  
  $(\D_{{}^\iota\!\lambda} \otimes \varepsilon_\infty({}^\iota\!\pi_f)) \otimes {}^\iota\pi_f$ is automorphic; 
  see (\ref{eqn:epsilon-infinity}). 
  
\smallskip    
  
 \item For $\iota : E \to \C$ and any  $\bar \iota: \bar E \to \C$ extending $\iota$, the  Satake parameters satisfy the estimates 
 $$  
 q_\p^{d - \frac{1}{2}} \ < \ \vert  \bar\iota(\vartheta_{i,\fp}) \vert  \ < \  q_\p^{d+\frac{1}{2}}.
 $$
 \end{enumerate} 
\end{thm}

The equivalence of $(1)$ and $(2)$ is a theorem of Clozel \cite{clozel}.

 \medskip
\subsubsection{\bf A Manin--Drinfeld principle for global cohomology}    
We consider the global cohomology 
$H^\bullet(\SGK,\tM_{\lambda, E})$ as a module under the commutative Hecke $\HH_{K_f}^{(S)}.$ Then 
the linear algebra recalled in Sect.\,\ref{sec:lin-alg} gives that strongly inner cohomology splits off within global cohomology, 
i.e., we have a canonical decomposition
\begin{align}
\label{eqn:manin-drinfeld-easy}
H^\bullet(\SGK,\tM_{\lambda, E}) \ = \ 
H_{!!}^\bullet(\SGK, \tM_{\lambda, E}) \oplus H^\pkt_{\rm Eis}(\SGK, \tM_{\lambda, E}), 
\end{align} 
which defines the Eisenstein cohomology denoted $ H^\pkt_{\rm Eis}(\SGK, \tM_{\lambda, E})$. Such a decomposition, and a finer one at that, is also given by Franke--Schwermer~\cite{franke-schwermer}; but their decomposition is only given at a transcendental level.

\medskip
\subsection{Definition of the relative periods}
\label{sec:rel-periods}
In this Sect.\,\ref{sec:rel-periods} we will only consider the group $G = R_{F/\Q}(\GL_n/F)$ when $n$ is an 
 even positive integer.   
The purpose of this section then is to define and analyze the relative periods $\Omega^\varepsilon({}^\iota\!\pi_f) \in \C^\times$, where $\lambda = (\lambda^\tau)_{\tau : F \to E}$ is a pure  weight, 
$\pi_f \in {\rm Coh}_{!!}(G,K_f,\lambda)$ for some level structure $K_f,$ $\iota : E \to \C $ and 
$\varepsilon = (\varepsilon_v)_{v \in \place_\infty}$ is a character of $\pi_0(G(\R))$.

\medskip
\subsubsection{\bf The arithmetic identification}
Consider inner cohomology in bottom degree $b_n^F$, and take the field $E$ large enough so that $\pi_f$ appears paired with every possible sign character $\varepsilon$; this is assured by (\ref{eqn:character-component}). 
Fix an $E$-linear isomorphism of $\HK$-modules
$$
T^\varepsilon_{\rm arith}(\lambda, \pi_f) \ : \ 
H^{b_n^F}_!(\SGK, \tM_{\lambda, E})(\pi_f \times \varepsilon) \ \longrightarrow \ 
H^{b_n^F}_!(\SGK, \tM_{\lambda, E})(\pi_f \times -\varepsilon), 
$$
which we call an arithmetic identification between these modules; they are both abstractly isomorphic to $\pi_f.$ 
The choice of $T^\varepsilon_{\rm arith}(\lambda, \pi_f)$ is well-defined up to $E(\pi_f)^\times$-multiplies. 
We may and will ask that these isomorphisms be compatibly chosen, i.e., for any $\iota : E \to E'$ we ask for the commutativity of
\begin{equation}
\label{eqn:t-arith}
\xymatrix{
H^{b_n^F}_!(\SGK, \tM_{\lambda, E})(\pi_f \times \varepsilon) \ar[rrr]^{T^\varepsilon_{\rm arith}(\lambda,\, \pi_f)} 
\ar[d]^{\iota^\bullet}
& & & H^{b_n^F}_!(\SGK, \tM_{\lambda, E})(\pi_f \times -\varepsilon) \ar[d]^{\iota^\bullet} \\ 
H^{b_n^F}_!(\SGK, \tM_{{}^\iota\!\lambda,\, E'})({}^\iota\pi_f \times \varepsilon) 
\ar[rrr]^{T^{\varepsilon}_{\rm arith}({}^\iota\!\lambda,\, {}^\iota\pi_f)}  
& & & H^{b_n^F}_!(\SGK, \tM_{{}^\iota\!\lambda,\, E')})({}^\iota\pi_f \times - \varepsilon) 
}
\end{equation}

\medskip
\subsubsection{\bf The transcendental identification}
\label{sec:T-trans}
Given the isomorphism type $\pi_f \in {\rm Coh}_{!!}(G,K_f,\lambda)$, it follows from 
from Thm.\,\ref{thm:cuspidal-cohomology} that for any $\iota : E \to \C$
there is a unique (up to isomorphism) irreducible admissible $(\g_\infty, K_\infty)$-module $\D_{{}^\iota\!\lambda}$ so that  $W^{\rm cusp}_{\D_{{}^\iota\!\lambda} \otimes {}^\iota\pi_f} \neq 0;$
see Sect.\,\ref{sec:w-cusp-pi}. 
Appealing to the multiplicity one result for cusp forms for $\GL_n/F$, fix a basis: \begin{equation}
\label{eqn:embed-into-cusp-forms}
W^{\rm cusp}_{\D_{{}^\iota\!\lambda} \otimes {}^\iota\pi_f} \ = \ \C \cdot \Phi(\lambda, \pi_f, \iota).
\end{equation}
 
\smallskip 
 
Define the map $\Phi^\varepsilon(\lambda, \pi_f, \iota)$ as the compositum of the three maps:

\begin{equation}
\begin{array}{ccl} 
V_{\pi_f} \otimes_{E,\iota}\C & \xrightarrow{\ \ \ \ \ \ \ \ \ \ } & 
\left(H^{b_n^F}\left(\g_\infty, K_\infty^0; \D_{{}^\iota\!\lambda} \otimes \M_{{}^\iota\!\lambda, \C} \right)(\varepsilon) \right)
\otimes V_{\pi_f} \otimes_{E,\iota}\C \\
& & \\
& \xrightarrow{\ \ \ \ \ \ \ \ \ \ } &  
H^{b_n^F}\left(\g_\infty, K_\infty^0; \, \left(\D_{{}^\iota\!\lambda} \otimes (V_{\pi_f} \otimes_{E,\iota} \C) \right) 
\otimes \M_{{}^\iota\!\lambda, \C} \right)(\varepsilon) \\
& & \\
& \xrightarrow{\Phi(\lambda, \pi_f, \iota)^\bullet} &
H^{b_n^F}\left(\g_\infty, K_\infty^0; \, 
\cC^\infty_{\rm cusp}\left(G(\Q)\backslash G(\A)/K_f,\, \omega_{\infty}^{-1}\right)
\otimes \M_{{}^\iota\!\lambda, \C}\right)(\varepsilon), 
 \end{array}
\end{equation}
where, 

\begin{enumerate}

\item[(i)] the first map is the isomorphism given by tensoring by 
$H^{b_n^F}\left(\g_\infty, K_\infty^0; \D_{{}^\iota\!\lambda} \otimes \M_{{}^\iota\!\lambda, \C} \right)(\varepsilon)$ which is a one-dimensional space; this map therefore depends on a choice of a basis for this one-dimensional space: 
$H^{b_n^F}\left(\g_\infty, K_\infty^0; \D_{{}^\iota\!\lambda} \otimes \M_{{}^\iota\!\lambda, \C} \right)(\varepsilon) = 
\C w^\varepsilon(\lambda);$ see Sect.\,\ref{para:interlude-basis} below for a special choice of $w^\varepsilon(\lambda)$;  

\item[(ii)] the second map is tautological; and 

\item[(iii)] the third is an injection, since there are only finitely-many irreducible 
$(\g_\infty, K_\infty) \times \HK$-summands in 
$\cC^\infty_{\rm cusp}\left(G(\Q)\backslash G(\A)/K_f,\, \omega_{\infty}^{-1}\right)$ 
and relative Lie algebra  cohomology commutes with direct sums. 
\end{enumerate}
Hence, $\Phi^\varepsilon(\lambda, \pi_f, \iota)$ gives an isomorphism from 
$V_{\pi_f} \otimes_{E,\iota}\C$ onto the image of $\Phi(\lambda, \pi_f, \iota)^\bullet$ and we denote this as: 
\begin{equation}
\label{eqn:Phi}
\Phi^\varepsilon(\lambda, \pi_f, \iota) \ : \ 
V_{\pi_f} \otimes_{E,\iota}\C  
\ \xrightarrow{ \ \ \ \approx \ \ \ } \ 
H^{b_n^F}_{\rm cusp}(\SGK, \tM_{{}^\iota\!\lambda, \C})({}^\iota\pi_f \times \varepsilon).
\end{equation}

\medskip

Now define 
\begin{equation}
T^\varepsilon_{\rm trans}({}^\iota\!\lambda, {}^\iota\!\pi_f) 
\ := \ 
\Phi^{-\varepsilon}(\lambda, \pi_f, \iota) \circ \Phi^\varepsilon(\lambda, \pi_f, \iota)^{-1}
\end{equation}
which is an isomorphism defined at a transcendental level, i.e., after going to $\C$ via $\iota$: 
$$
T^\varepsilon_{\rm trans}({}^\iota\!\lambda, {}^\iota\!\pi_f) \ : \ 
H^{b_n^F}_{\rm cusp}(\SGK, \M_{{}^\iota\!\lambda, \C})({}^\iota\pi_f \times \varepsilon) \ \longrightarrow \ 
H^{b_n^F}_{\rm cusp}(\SGK, \M_{{}^\iota\!\lambda, \C})({}^\iota\pi_f \times -\varepsilon).
$$

\medskip
\paragraph{\bf Interlude on $T^\varepsilon_{\rm trans}({}^\iota\!\lambda, {}^\iota\!\pi_f)$ being independent of choice of basis 
$w^\varepsilon(\lambda)$}
\label{para:interlude-basis}
We will now discuss the choice of basis $w^\varepsilon({}^\iota\!\lambda).$ Suppose, as before, 
${}^\iota\!\lambda = ({}^\iota\!\lambda^\tau)_{\tau : F \to E} = (\lambda^\nu)_{\nu : F \to \C}$; here $\nu = \iota \circ \tau.$
Since $n$ is even, the restriction to $\GL_n(\R)^0$ of $\D_{\lambda^\nu}$ breaks up as 
$\D_{\lambda^\nu} = \D_{\lambda^\nu}^+ \oplus \D_{\lambda^\nu}^-, $
and the nontrivial element of $\pi_0(\GL_n(\R))$ switches these two summands. Implicit in the proof of 
Prop.\,\ref{prop:cohomology-degree}, is the fact that  
$H^{b_n}(\gl_n,  \SO(n)\R^\times_+; \, \D_{\lambda^\nu}^+ \otimes \M_{\lambda^\nu})$ is one-dimensional, and we fix a basis for this space, say $w^+(\lambda^\nu).$ Apply K\"unneth theorem, to get that 
$w^{+\!+}({}^\iota\lambda) = \otimes_\nu \, w^+(\lambda^\nu)$
generates the one-dimensional space 
$H^{b_n^F}(\g, K_\infty^0; \,  \D_{{}^\iota\lambda}^{+\!+} \otimes \M_{{}^\iota\lambda}),$ 
where $\D_{{}^\iota\lambda}^{+\!+} = \otimes_\nu \D_{\lambda^\nu}^+.$
For any character $\varepsilon$ of $\pi_0(G_n(\R))$, our chosen basis element for  
$H^{b_n^F}\left(\g, K_\infty^0; \, \D_{{}^\iota\!\lambda} \otimes \M_{{}^\iota\!\lambda, \C} \right)(\varepsilon)$ is given by: 
\begin{equation}
\label{eqn:basis-w-epsilon-lambda}
w^\varepsilon({}^\iota\lambda) \ = \   
\sum_{a \in \pi_0(G_n(\R))} \varepsilon(a) a \cdot w^{+\!+}({}^\iota\lambda). 
\end{equation}
Consider the linear map of one-dimensional spaces 
\begin{equation}
\label{eqn:switch-infinity}
T^\varepsilon({}^\iota\lambda) : \C w^\varepsilon({}^\iota\lambda) \ \longrightarrow \ 
\C w^{-\varepsilon}({}^\iota\lambda) 
\end{equation}
defined by 
\begin{equation}
\label{eqn:correct-power-i}
T^\varepsilon({}^\iota\lambda)(w^\varepsilon({}^\iota\lambda)) \ = \ i^{{\sf r}n/2}w^{-\varepsilon}({}^\iota\lambda).
\end{equation}
It is clear that $T^\varepsilon({}^\iota\lambda)$ is independent of the choice of basis  
$w^{+\!+}({}^\iota\lambda);$ because, if we change $w^{+\!+}({}^\iota\lambda)$
by a nonzero scalar, then from (\ref{eqn:basis-w-epsilon-lambda}), 
the same scalar appears in both $w^{\pm \varepsilon}({}^\iota\lambda)$.  The scaling factor of $i^{{\sf r}n/2}$ in the right hand side of (\ref{eqn:correct-power-i}) makes $T^\varepsilon({}^\iota\lambda)$ defined over $\Q$; see \cite{harder-arithmetic}. Since the transcendental identification may be parsed as: 
 $$
 \xymatrix{
 &  V_{\pi_f} \otimes_{E,\iota}\C w^\varepsilon({}^\iota\lambda) \ar[dd]^{1 \otimes T^\varepsilon({}^\iota\lambda)} \ar[rr]
& & H^{b_n^F}_{\rm cusp}(\SGK, \M_{{}^\iota\!\lambda, \C})({}^\iota\pi_f \times \varepsilon) 
  \ar[dd]^{T^\varepsilon_{\rm trans}({}^\iota\!\lambda, {}^\iota\!\pi_f)} \\ 
  V_{\pi_f} \otimes_{E,\iota}\C \ar[ru] \ar[rd] & & & \\
 &  V_{\pi_f} \otimes_{E,\iota}\C w^{-\varepsilon}({}^\iota\lambda) \ar[rr] 
& & H^{b_n^F}_{\rm cusp}(\SGK, \M_{{}^\iota\!\lambda, \C})({}^\iota\pi_f \times -\varepsilon), 
 }
 $$
 we deduce that $T^\varepsilon_{\rm trans}({}^\iota\!\lambda, {}^\iota\!\pi_f)$ is independent of any choice of basis elements.

\medskip
\subsubsection{\bf The relative periods}
 Consider the diagram of irreducible $\HK \times \pi_0(G(\R))$-modules$/\C$: 
 $$
 \xymatrix{
 &  H^{b_n^F}_{\rm cusp}(\SGK, \M_{{}^\iota\!\lambda, \C})({}^\iota\pi_f \times \varepsilon) 
 \ar[dd]^{T^\varepsilon_{\rm trans}({}^\iota\!\lambda, {}^\iota\!\pi_f)} \ar[rr]^{\approx}
 & & H^{b_n^F}_!(\SGK, \tM_{\lambda, E})(\pi_f \times \varepsilon) \otimes_{E,\iota} \C 
 \ar[dd]^{T^\varepsilon_{\rm arith}(\lambda,\pi_f) \otimes_{E,\iota} {\bf 1}} \\
V_{\pi_f} \otimes_{E,\iota}\C \ar[ru] \ar[rd] & & & \\
 &  H^{b_n^F}_{\rm cusp}(\SGK, \M_{{}^\iota\!\lambda, \C})({}^\iota\pi_f \times -\varepsilon)  \ar[rr]^{\approx}
 & & H^{b_n^F}_!(\SGK, \tM_{\lambda, E})(\pi_f \times -\varepsilon) \otimes_{E,\iota} \C  
} $$
 
\smallskip 
 
\begin{defn}
\label{defn:relative-periods}
There exists $\Omega^\varepsilon({}^\iota\!\lambda, {}^\iota\!\pi_f) \in \C^\times$ such that 
$$
\Omega^\varepsilon({}^\iota\!\lambda, {}^\iota\!\pi_f) \, T^\varepsilon_{\rm trans}({}^\iota\!\lambda, {}^\iota\!\pi_f) \ = \ 
T^\varepsilon_{\rm arith}(\lambda,\pi_f) \otimes_{E,\iota} {\bf 1}. 
$$
If we change the choice of $T^\varepsilon_{\rm arith}(\lambda,\pi_f)$ to 
$\alpha \, T^\varepsilon_{\rm arith}(\lambda,\pi_f)$ for an $\alpha \in E(\pi_f)^\times$ then 
$\Omega^\varepsilon({}^\iota\!\lambda, {}^\iota\!\pi_f)$ changes to $\iota(\alpha) \, \Omega^\varepsilon({}^\iota\!\lambda, {}^\iota\!\pi_f),$ i.e.,  
we have an array of nonzero complex numbers 
$$
\{\dots, \Omega^\varepsilon({}^\iota\!\lambda, {}^\iota\!\pi_f), \dots \}_{\iota : E \to \C}
$$
well-defined up to multiplication by $E(\pi_f)^\times.$ 

In other words, we have defined a period 
$\Omega^\varepsilon(\lambda, \pi_f) \in (E \otimes \C)^\times/E(\pi_f)^\times.$ Sometimes, we suppress the $\lambda$ in the notation, and will just write $\Omega^\varepsilon(\pi_f)$.
\end{defn}

The reader is referred to Raghuram~\cite{raghuram-comparison} to see how the relative periods 
$\Omega^\varepsilon({}^\iota\!\pi_f)$ are related to other periods attached to ${}^\iota\!\pi_f$ which have played an 
important role in recent results on the special values of $L$-functions.

\medskip
\subsubsection{\bf Period relations under Tate twists}
\label{sec:T-Tate}
The periods $\Omega^\varepsilon({}^\iota\!\pi_f)$ have a simple behavior under Tate twists. 
For the moment, let $n$ be any positive integer. 
Let $\lambda = (\lambda^\tau)_{\tau : F \to E}$ be a pure weight, where, as before, 
$\lambda^\tau  = 
\sum_{i=1}^{n-1} (a^\tau_i-1) \bfgreek{gamma}_i + d \cdot \bfgreek{delta}_n
= (b^\tau_1,\dots, b^\tau_n).$ For $m \in \Z$, define a pure weight 
$\lambda- m\bfgreek{delta}_n := (\lambda^\tau - m\bfgreek{delta}_n)_{\tau : F \to E},$ where 
$$
\lambda^\tau - m\bfgreek{delta}_n   \ = \ 
\sum_{i=1}^{n-1} (a^\tau_i-1) \bfgreek{gamma}_i + (d-m) \cdot \bfgreek{delta}_n
\ = \ (b^\tau_1-m,\dots, b^\tau_n-m).
$$

Also, let $m {\delta}_n$ stand for the weight $(m \bfgreek{delta}_n)_{\tau : F \to E}.$  We have a class $e_{\delta_n}\in H^0(\SGK,\Q[\delta_n])$ (see \cite[Chap.\,3]{harder-book}) and the cup product by this class yields an isomorphism
\begin{align*}
T_{\rm Tate}^\bullet(m):H^\pkt_?(\SGK, \M_{\lambda,E}) \ \ppfeil{\cup e_{\delta_n^m}} \  
H^\pkt_?(\SGK, \M_{\lambda-m\delta_n,E}) 
\end{align*}
called  the $m$-th Tate twist. 
 If $\pi_f \in \Coh_?(G, \lambda, K_f)$ with $? \in \{!, !!\}$ then it is clear that 
\begin{equation}
\label{eqn:tate-twist}
T_{\rm Tate}^\bullet(m)(\pi_f \times \varepsilon) \ = \ \pi_f(m) \times (-1)^m \varepsilon, 
\end{equation}
where, $\pi_f(m) := \pi_f \otimes |\ |^m$ stands for 
$\ul g_f \mapsto  |{\rm det}(\ul g_f)|^m \pi_f(\ul g_f)$, and $\varepsilon$ is a character of $\pi_0(G(\R)).$

\smallskip

For later use, we introduce the notation $T^{\varepsilon}_{\rm Tate}(\lambda, \pi_f, m)$ for the $m$-th Tate twist on the module $\pi_f \times \varepsilon$ appearing in  ${\rm Coh}_{!!}(G,K_f,\lambda).$

\smallskip
\begin{prop}\label{prop:period-tate-twists}
If $\pi_f \in {\rm Coh}_{!!}(G,K_f,\lambda)$ then  
$\pi_f(m) \in {\rm Coh}_{!!}(G,K_f,\lambda - m\bfgreek{delta}_n)$ for $m \in \Z.$ Assume now that $n$ is even, then, for any $\varepsilon$ we have
$$
\Omega^\varepsilon({}^\iota\!\pi_f(m)) \ = \ 
\Omega^{(-1)^m\varepsilon}({}^\iota\!\pi_f). 
$$
\end{prop}

\begin{proof}
Follows from Thm.\,\ref{thm:cuspidal-cohomology}, Def.\,\ref{defn:relative-periods} and  (\ref{eqn:tate-twist}). 
\end{proof}

\medskip
\subsection{The strongly inner cohomology of the boundary }
\label{sec:manin-drinfeld}

\medskip 
\subsubsection{\bf The restriction to the maximal strata}
We have the restriction from the cohomology of the boundary to the cohomology of the maximal boundary strata
\begin{align}
\label{restriction}
H^\pkt(\PBSC, \tM_{\lambda,E}) \ \longrightarrow \  \bigoplus_{P: \rm maximal} H^\pkt(\PPBSC, \tM_{\lambda,E})
\end{align}
and for the cohomology of each boundary strata $\PPBSC$, after passing to the limit over all $K_f$, from 
Prop.\,\ref{prop:bdry-coh-2} we have:  
$$
H^\pkt(\ppBSC, \tM_{\lambda,E}) \ = \ 
\bigoplus_{w\in W^P} \aIndPG \, H^{\pkt-l(w)}(\SMP, \tM_{w \cdot \lambda,E}). 
$$ 
Our previous results applied to the quotients $M_P$  yield a  decomposition
 \begin{align*}
H^{\pkt-l(w)}(\SMP, \tM_{w \cdot \lambda,E}) \ \ = \ \ 
H_{!!}^{\pkt-l(w)}(\SMP, \tM_{w \cdot \lambda,E}) \ \oplus \ H_{\rm Eis}^{\pkt-l(w)}(\SMP, \tM_{w \cdot \lambda,E})
\end{align*}
and therefore we get homomorphisms between Hecke modules
\begin{multline}
r_{\rm max}:  H^\pkt(\pBSC, \tM_{\lambda,E})  \ \longrightarrow \  \bigoplus_{P}\bigoplus_{w\in W^P} 
\aIndPG \left(H_{!!}^{\pkt-l(w)}(\SMP, \tM_{w \cdot \lambda,E} )\right) \\ 
\oplus \aIndPG \left(H_{\rm Eis}^{\pkt-l(w)}(\SMP, \tM_{w \cdot \lambda,E} )\right).
\end{multline}
On the right hand side we project to the first set of summands to get
\begin{align*}
r_{\rm max,!}: H^\pkt(\pBSC, \tM_{\lambda,E})  \ \longrightarrow \   
\bigoplus_{P}\bigoplus_{w\in W^P} 
\aIndPG H_{!!}^{\pkt-l(w)}(\SMP, \tM_{w \cdot \lambda,E}). 
\end{align*}
This homomorphism is now surjective because the right hand side is in the kernel of all differentials in the 
spectral sequence. 
Take $K_f$-invariants and if $\place$ is the finite set of places outside of which $K_f$ is the full maximal compact, then we restrict the action of the Hecke algebra $\HK$ 
to the commutative subalgebra $\HGS = \otimes_{p \notin \place} \HKp.$ 
Then, as in the proof of Thm.\,\ref{thm:cuspidal-vs-boundary}, once again applying Jacquet and Shalika \cite[Thm.\,4.4]{jacquet-shalika-II}, we have: 
\begin{multline}
\Spec_{\HGS} \left(
\bigoplus_{P}\bigoplus_{w\in W^P} \left(\aIndPG  \, H_{!!}^{\pkt-l(w)}(\SMP, \tM_{w\cdot\lambda, E})\right)^{K_f}
\right)  \\ 
\ \cap \ \Spec_{\HGS} \left(\ker(r_{\rm max,!})^{K_f} \right)
 \ =\ \emptyset. 
\end{multline}

\medskip
\subsubsection{\bf An interlude on Kostant's representatives}

Let the notations be as in Sect.\,\ref{sec:kostant-theorem}, but take $P_0$ to be a maximal parabolic subgroup of $G_0.$ 
Let $\bfpi_{M_{P_0}} = \bfpi_{G_0} - \{\alpha_{P_0}\}$. 
 Let $w_{P_0}$ be the unique element of $W_0 = W_{G_0}$ such that $w_{P_0}(\bfpi_{M_{P_0}}) \subset \bfpi_{G_0}$ and 
$w_{P_0}(\alpha_{P_0}) < 0,$ it is the longest Kostant representative for $W^{P_0}.$ Let $Q_0$ be the parabolic 
subgroup associate to $P_0$; we have
\begin{enumerate}
\item[(i)] $w_{P_0}(\bfpi_{M_{P_0}}) = \bfpi_{M_{Q_0}}$;  
\item[(ii)] $w_{P_0}(\bfdelta_{U_{P_0}}) = -\bfdelta_{U_{Q_0}}$. \\
(Here $\bfdelta_{U_{P_0}}$ is the set of those positive roots whose root spaces are in $U_{P_0}$; similarly,  $\bfdelta_{U_{Q_0}}.$)
\end{enumerate}
For (ii), observe that if $\alpha \in \bfdelta_{U_{P_0}}$ then in the expression for $\alpha$ in terms of simple roots, the root $\alpha_P$ has to appear with a positive integral coefficient, and since $w_{P_0}(\alpha_P) < 0$, there is a negative coefficient in the expression for $w_{P_0}(\alpha)$; 
but all coefficients have the same sign, hence $w_{P_0}(\alpha) < 0$.

\begin{lemma}
\label{lem:kostant-P-Q}
With notations as above, we have: 
\begin{enumerate}
\item The map $w \mapsto w^\prime:= w_{P}w$ gives a bijection $W^P \to W^Q$. If $w= (w^\tau)_{\tau : F \to E}$, then by definition, 
$w_{P}w = (w_{P_0}w^\tau)_{\tau:F \to E}.$ 
\item This bijection has the property that $l(w^\tau) + l(w^\tau) = \dim{(U_{P_0^\tau})}$. Hence $l(w) + l(w^\prime) = \dim(U_P).$
\end{enumerate}
\end{lemma}

\begin{proof}  There are several ways to prove this lemma. We briefly sketch a proof below and for another proof the reader is referred to \cite[Chap.\,3]{harder-book}. 
The proof is the same for every component indexed by $\tau.$ We will suppress $\tau$ from the notation and just work over the group $G_0/F$ and its subgroups. 
Let $w_0 \in W_0$. (In our simplified notation, this $w_0$ is any of the $w^\tau$ for the original $w$ in $(1)$ above.) 
To prove the first statement of the lemma:  
\begin{eqnarray*}
w_{P_0} w_0 \in W^{Q_0} 
& \iff & w_0^{-1} w_{P_0}^{-1}(\beta) > 0,  \ \forall \beta \in \bfpi_{M_{Q_0}}\ \ \mbox{(now put $\beta = w_{P_0}\alpha$ which is possible by (i))} \\
& \iff & w_0^{-1}(\alpha) > 0,  \ \forall \alpha \in \bfpi_{M_P} \\
& \iff & w_0 \in W^{P_0}.
\end{eqnarray*}

To prove $l(w_0) + l(w'_0) =   \dim{(U_{P_0})}$, partition $\bfdelta_{U_{P_0}}$ as 
$\bfdelta_{U_{P_0}, w_0^{-1}}^+  \cup \bfdelta_{U_{P_0}, w_0^{-1}}^-$, the union being disjoint, where
$\bfdelta_{U_{P_0}, w_0^{-1}}^+ := \{\alpha \in \bfdelta_{U_{P_0}} : w_0^{-1}(\alpha) > 0\}$ and 
$\bfdelta_{U_{P_0}, w_0^{-1}}^- := \{\alpha \in \bfdelta_{U_{P_0}} : w_0^{-1}(\alpha) < 0\}.$
Since $w_0 \in W^{P_0}$, it follows that 
$\{\alpha \in \bfdelta_{G_0}^+ : w_0^{-1}\alpha < 0\}  = 
\{\alpha \in \bfdelta_{U_{P_0}} : w_0^{-1}\alpha < 0\} = \bfdelta_{U_{P_0}, w_0^{-1}}^-.$ 
Hence, $l(w_0) = l(w_0^{-1}) = |\bfdelta_{U_P, w_0^{-1}}^-|$. Next, one observes that the map 
$\alpha \mapsto -  w_{P_0}\alpha$ gives a bijection 
$\bfdelta_{U_{P_0}, w_0^{-1}}^+ \to \bfdelta_{U_{Q_0}, w'^{-1}_0}^-$
since $w'^{-1}_0(-w_{P_0}(\alpha)) = -w'^{-1}_0 w_{P_0}(\alpha) = -w_0^{-1}\alpha < 0$. We have 
\begin{eqnarray*}
\dim{(U_{P_0})} = |\bfdelta_{U_{P_0}}| 
& = &  |\bfdelta_{U_{P_0}, w_0^{-1}}^-| + |\bfdelta_{U_{P_0}, w_0^{-1}}^+| \\
& = & |\bfdelta_{U_{P_0}, w_0^{-1}}^-| + |\bfdelta_{U_{Q_0}, w'^{-1}_0}^-| \\
& = & l(w_0^{-1}) + l(w'^{-1}_0) = l(w_0) + l(w'_0).
\end{eqnarray*}
\end{proof}

Like the bijection between $W^P$ and $W^Q$ that was described in Lem.\,\ref{lem:kostant-P-Q} above, similarly, we have the following self-bijection of $W^P.$ 

\begin{lemma}
\label{lem:kostant-P-P}
Let the notations be as in Lem.\,\ref{lem:kostant-P-Q}. 
Let $w_G$ be the element of longest length in the Weyl group $W_G$ of $G$, and similarly, 
let $w_{M_P}$ be the element of longest length in the Weyl group $W_{M_P}$ of the Levi quotient $M_P$ of $P$. 
Then: 
\begin{enumerate}
\item The map $w \mapsto w^{\sf v}:= w_{M_P}\, w \, w_G$ gives a bijection $W^P \to W^P$. 
\item This bijection has the property that $l(w) + l(w^{\sf v}) = \dim(U_P).$
\end{enumerate}
\end{lemma}

\begin{proof}
Similar to the proof of Lem.\,\ref{lem:kostant-P-Q}. 
\end{proof}

If $\Ml$ is a highest weight module  and $\M_{\lambda^{\sf v}}$ is its dual, the we have a  nondegenerate pairing
\begin{align}
H^q(\u_P,\M_{\lambda,E}) \ \times \ H^{d_U-q}(\u_P,\M_{\lambda^{\sf v},E}) \  \ \longrightarrow \ \ E.
\end{align}

If we decompose the two cohomology modules according to (\ref{eqn:kostant}) then the pairing 
becomes a direct sum over $w\in W^P$ of nondegenerate pairings
\begin{align}
\M_{\wl, E} \ \times \ \M_{w^{\sf v}\cdot \lambda^{\sf v}, E} \ \ \longrightarrow \ \ E. 
\end{align}

\medskip 
\subsubsection{\bf The strongly inner cohomology of the boundary}
  
 We define 
  \begin{align*}
H_{!!}^\pkt(\partial\SG, \tM_{\lambda,E})  \ = \ 
\bigoplus_{P: \rm maximal}  \bigoplus_{w\in W^P} 
\aIndPG 
\left( H_{!!}^{\pkt-l(w)}(\SMP, \tM_{w \cdot \lambda,E}) \right)
\end{align*}
where each term decomposes further into isotypic components: 
\begin{align*}
H_{!!}^{\pkt-l(w)}(\SMP, \tM_{w \cdot \lambda,E}) \ = \ 
\bigoplus_{\tsigma_f\in \Coh_{!!} (M_P, w\cdot\lambda)}
H_{!!}^{\pkt-l(w)}(\SMP, \tM_{w \cdot \lambda,E})(\tsigma_f). 
\end{align*}

\smallskip

Given two maximal parabolic subgroups  $P$ and $Q,$ with reductive quotients $M_P$ and $M_Q$, respectively, 
we have to understand under what conditions we have nontrivial intertwining operators between the induced 
modules 

\begin{multline}
\label{Twwprime}
\aIndPG H_{!!}^{\pkt-l(w)}(\SMP, \tM_{w \cdot \lambda,E} )(\tsigma_f) 
\ \text{ and } \\
\aIndQG H_{!!}^{\pkt-l(w_1)}(\SMQ, \tM_{w_1 \cdot \lambda,E})(\tsigma_{f,1}), 
\end{multline}
where $w\in W^P, w_1\in W^Q,$ $\sigma_f\in \Coh_{!!} (M_P, w\cdot\lambda)$ and 
$\tsigma_{f,1}\in \Coh_{!!} (M_Q, w_1\cdot\lambda).$
It is again the theorem of Jacquet--Shalika which tells us that this will almost never be  the case unless 
we are in a special situation which is summarized in the 

\begin{prop}
We have a nontrivial intertwining  between two such modules as in (\ref{Twwprime}) if and only if 
\begin{enumerate}
\item $P=Q$ or $P$ and $Q$ are associate, and furthermore, if this condition is fulfilled, then 
\item under the bijection $w\mapsto w^\prime$ from $W^P\to W^Q$ (Lem.\,\ref{lem:kostant-P-Q}) which 
gives a bijection $\tsigma_f \to \tsigma_f^\prime$ between $ \Coh_{!!} (M_P, w\cdot\lambda)$
and $\Coh_{!!} (M_Q, w^\prime\cdot\lambda)$, we have $w_1= w^\prime$ and $(\tsigma_{f,1})=\tsigma_f^\prime$.
\end{enumerate}
\end{prop}

Therefore we see that the strongly inner cohomology of the boundary has an isotypical decomposition
\begin{multline}
\label{decobdry} 
H_{!!}^\pkt(\partial\SG, \tM_{\lambda,E}) \ = \ 
\bigoplus_{\{P,Q\}} \bigoplus_{\{w,w^\prime\}} \bigoplus_{\{\tsigma_f, \tsigma_f^\prime\}}
\aIndPG ( H_{!!}^{\pkt-l(w)}(\SMP, \tM_{w \cdot \lambda,E})(\tsigma_f) \\ 
\oplus    \aIndQG 
H_{!!}^{\pkt-l(w^\prime)}(\SMQ, \tM_{w' \cdot \lambda,E})(\tsigma^\prime_{f }). 
\end{multline}

\smallskip

Observe that in the above decomposition the cohomology in the two summands is possibly living in different degrees
because of the shifts by $l(w)$ and $l(w').$ This motivates the 

\begin{defn}
An $w = (w^\tau)_{\tau : F \to E} \in W$ is called a \underline{balanced} Kostant representative for $P$ if 
$w \in W^P$ and
$$
l(w^\tau) \ = \ \dim{(U_{P_0})}/2, \quad \forall \tau : F \to E.
$$
\end{defn}

It follows from Lem.\,\ref{lem:kostant-P-Q} 
that if $w$ is a balanced Kostant representative for $P$ then $w^\prime$ is also balanced as a Kostant representative for $Q.$ An obvious necessary condition for the existence of balanced elements in $W^P$ is that $\dim{(U_{P_0})} = nn'$ is even. 
In this paper we are only interested in the contribution coming from
balanced Kostant representatives.

\medskip
\subsubsection{\bf Interlude on induced representations}
\label{sec:interlude-induced}
Let $\tsigma_f \in \Coh_{!!}(M_P, w\cdot \lambda)$ be as above. Write $M_P = G_n \times G_{n'}$, 
$\tsigma_f = \sigma_f \otimes \sigma'_f$ with $\sigma_f \in \Coh_{!!}(G_n, \mu)$ and 
$\sigma'_f \in \Coh_{!!}(G_{n'}, \mu')$ with pure weights $\mu$ and $\mu'$ such that 
$w \cdot \lambda = \mu \otimes \mu'.$ We will henceforth assume that $\mu$ and $\mu'$ are such that there is a balanced $w \in W^P$ for which $w^{-1}(\mu \otimes \mu')$ is a dominant weight. This condition on the weights $\mu$ and $\mu'$ has many interesting and crucial consequences which are captured by the {\it combinatorial lemma}; see 
Sect.\,\ref{sec:com-lemma}.

Consider now the associate parabolic subgroup $Q$ with reductive quotient $M_Q = G_{n'} \times G_n$. 
The Kostant representative $w' \in W^Q$ given 
as in Lem.\,\ref{lem:kostant-P-Q} is also balanced, and furthermore it is easy to see that 
$$
w'\cdot\lambda \ = \ w' \cdot (w^{-1} \cdot (\mu \otimes \mu')) \ = \ 
(\mu'-n\bfgreek{delta}_{n'}) \otimes (\mu + n'\bfgreek{delta}_n).
$$
If $\sigma_f  \in \Coh_{!!}(G_n, \mu)$ then $\sigma_f(-n')  \in \Coh_{!!}(G_n, \mu +n'\bfgreek{delta}_n)$, and similarly, 
$\sigma'_f  \in \Coh_{!!}(G_{n'}, \mu')$ implies $\sigma'_f (n) \in \Coh_{!!}(G_{n'}, \mu'-n\bfgreek{delta}_{n'}).$ 
Hence the module $\tsigma_f'$ is nothing but $\sigma_f'(n) \times \sigma_f(-n').$ 
Consider the corresponding algebraically induced representations appearing as in (\ref{decobdry})  
\begin{equation}
\label{eqn:aind-P-ind}
{}^{\rm a}{\rm Ind}_{P(\A_f)}^{G(\A_f)}\left( \sigma_f \otimes \sigma_f' \right) \quad {\rm and} \quad 
{}^{\rm a}{\rm Ind}_{Q(\A_f)}^{G(\A_f)}\left( \sigma_f'(n) \otimes \sigma_f(-n')\right). 
\end{equation}
By Jacquet and Shalika \cite[(4.3)]{jacquet-shalika-II}, they are weakly-equivalent. 
Take any open-compact subgroup $K_f$, and an associated finite set of finite places $\place$ such that $K_f$ is unramified outside $\place$, 
 and consider the $K_f$-invariants as an 
$\HGS$-module. There is only one isomorphism type of simple $\HGS$-module in the $K_f$-invariants of the 
induced representations in (\ref{eqn:aind-P-ind}); 
denote this particular isomorphism type as $I^\place(\sigma_f,\sigma_f').$

\smallskip

Taking contragredients we have:
\begin{equation}
\label{eqn:ind-dual}
{}^{\rm a}{\rm Ind}_{P(\A_f)}^{G(\A_f)}\left( \sigma_f \otimes \sigma_f' \right)^\v  
\ = \ 
{}^{\rm a}{\rm Ind}_{P(\A_f)}^{G(\A_f)}\left( \sigma_f^\v(n') \otimes \sigma_f'^\v(-n) \right), 
\end{equation}
and $\sigma_f^\v(n')  \in \Coh_{!!}(G_n, \mu^\v- n' \bfgreek{delta}_n)$, and similarly, 
$\sigma_f'^\v(-n)  \in \Coh_{!!}(G_{n'}, \mu'^\v+n \bfgreek{delta}_{n'}).$
As a notational artifice, given weights $\mu \in X^*(T_n)$ and $\mu' \in X^*(T_{n'})$, we will denote the weight 
$\mu \otimes \mu' \in X^*(T_N)$ also as:  
$$
\mu \otimes \mu' \ = \ \left[\begin{array}{cc} \mu & \\ & \mu' \end{array}\right].
$$
Note that the inducing data $\sigma_f^\v(n') \otimes \sigma_f'^\v(-n)$ for the contragredient representation considered above has cohomology with respect to the weight
$$
\left[\begin{array}{cc} \mu^\v - n' & \\ & \mu'^\v+n \end{array}\right]. 
$$

\begin{lemma}
\label{lem:w-sf-v}
Let $\mu \in X^*_0(T_n)$ be a pure weight, and similarly, $\mu' \in  X^*_0(T_n).$ Assume that
there exists a balanced Kostant representative $w \in W^P$ such that 
$\lambda := w^{-1}\cdot(\mu \otimes \mu')$ is dominant. Let $w^{\sf v} \in W^P$ be the Kostant representative associated to $w$ given by Lem.\,\ref{lem:kostant-P-P}. Then $w^{\sf v}$ is also balanced and furthermore 
$$
w^{\sf v} \cdot \lambda^\v = 
\left[\begin{array}{cc} \mu^\v - n' & \\ & \mu'^\v +n \end{array}\right].
$$
\end{lemma}

\begin{proof}
That $w^{\sf v}$ is balanced follows from Lem.\,\ref{lem:kostant-P-P}. 
The rest of the proof may be parsed over the embeddings $\tau: F \to E$ and for each $\tau$ it is the same calculation, and so we just suppress $\tau$ from notation. Next, as a notational artifice, let  $w_N$ be the element of longest length in the Weyl group of $G_0 := \GL_N/F$. As a permutation matrix, we have $w_N(i,i) = \delta_{i,N-j+1}.$ For $\tau : F \to E$, let $w_N^\tau$ be the same element but now thought of as the element of longest length in $G_0 \times_{F,\tau} E.$ Then, $w_G = (w_N^\tau)_{\tau: F \to E}.$ Similarly, representing $M_{P_0}$ as block diagonal matrices,  
permits us to think of $w_{M_P}$ as an array indexed by $\tau$ with each entry being the block diagonal matrix: 
$$
w_{M_{P_0}} \ = \ \left(\begin{array}{cc} w_n& \\ & w_{n'} \end{array}\right), 
$$
where $w_n(i,j) = \delta_{i,n-j+1}$, etc. Note that $-w_N \lambda = \lambda^\v$, $-w_n \mu = \mu^\v$ and 
$-w_{n'}\mu' = \mu'^\v.$ We have:
\begin{eqnarray*}
w^{\sf v} \cdot \lambda^\v & = & 
\left( \left(\begin{array}{cc} w_n& \\ & w_{n'} \end{array}\right) w w_N\right) \cdot (-w_N \lambda) \\
& = & \left( \left(\begin{array}{cc} w_n& \\ & w_{n'} \end{array}\right) w \right) \cdot (-\lambda -2\bfgreek{rho}_N) \\
& = & \left(\begin{array}{cc} w_n& \\ & w_{n'} \end{array}\right) 
\cdot \left(\left[\begin{array}{cc} -\mu & \\ & -\mu' \end{array}\right] -2\bfgreek{rho}_N \right) \\ 
& = & \left[\begin{array}{cc} \mu^\v & \\ & \mu'^\v \end{array}\right] - 
\left(\begin{array}{cc} w_n& \\ & w_{n'} \end{array}\right)\bfgreek{rho}_N - \bfgreek{rho}_N \\ 
& = & \left[\begin{array}{cc} \mu^\v - n' & \\ & \mu'^\v +n \end{array}\right]. 
\end{eqnarray*}
\end{proof}

Again, there is only one isomorphism type of simple $\HGS$-module in the $K_f$-invariants of the 
induced representation in (\ref{eqn:ind-dual}) and let $I^\place(\sigma_f,\sigma'_f)^\v$
stand for corresponding isomorphism type of $\HGS$-module.

\medskip
\subsubsection{\bf The character of the component group II}
The assumption of wanting a balanced Kostant representative necessitates $nn'$ to be even. Without loss of generality, we take $n$ to be even and let $n'$ be any positive integer.  Given $\sigma_f \in \Coh_{!!}(G_n,\mu)$ and 
$\sigma'_f \in \Coh_{!!}(G_{n'},\mu')$ we begin by taking a character $\varepsilon'$ of $\pi_0(G_{n'}(\R))$ as: 
\begin{equation}
\label{eqn:epsilon'-n'-cases}
\varepsilon' \ = \ \left\{ 
\begin{array}{ll}
\varepsilon(\sigma'_f ) &  \mbox{if $n'$ is odd, and} \\
\mbox{any character of $\pi_0(G_{n'}(\R))$} &  \mbox{if $n'$ is even,}
\end{array} \right.
\end{equation}
where $\varepsilon(\sigma'_f )$ is as in (\ref{eqn:character-at-infinity}). Note that 
$\varepsilon(\sigma'_f (n)) = \varepsilon(\sigma'_f )$ since $n$ is an even integer. 
Now take any 
character $\varepsilon$ of $\pi_0(G_n(\R)).$
From the short exact sequence of the group of connected components in (\ref{eqn:ses-infinity}) it follows that: 
$$
\mbox{The character 
$\varepsilon \times \varepsilon'$ of $\pi_0(P(\R))$ is trivial on $\pi_0(K^{M_P}_\infty)$} 
\ \iff \ 
\varepsilon = \varepsilon'.
$$
Hence
$
(\varepsilon' \times \sigma_f) \times (\varepsilon' \times \sigma_f') 
\ \hookrightarrow \ 
H^\bullet_{!!}(\cS^{M_P}, \tM_{\mu \otimes \mu', E})^{\pi_0(K^{M_P}_\infty)}.
$
The character $\varepsilon' \times \varepsilon'$ of $\pi_0(P(\R))$ maps to $\varepsilon'$ on $\pi_0(G(\R)).$  
From  Sect.\,\ref{sec:ind-pi-0} it follows that: 
$$
{}^{\rm a}{\rm Ind}_{\pi_0(P(\R)) \times P(\A_f)}^{\pi_0(G(\R)) \times G(\A_f)}
\left((\varepsilon' \otimes \sigma_f) \otimes (\varepsilon' \otimes \sigma_f') \right) \ = \ 
\varepsilon' \otimes 
{}^{\rm a}{\rm Ind}_{P(\A_f)}^{G(\A_f)}
\left( \sigma_f \otimes \sigma_f' \right). 
$$

\medskip
\subsubsection{\bf A summary of our notation with some simplifications}
\label{sec:simple-notation}
The ambient group will be  $G = R_{F/\Q}(\GL_N/F).$ Let $n$ and $n'$ be positive integers so that 
$N = n+n'$. Let $G_0 = \GL_N/F$ and 
let $P_0$ be the standard maximal parabolic subgroup of $G_0$ of type $(n,n')$, $U_{P_0}$ its unipotent radical and $M_{P_0} = P_0/U_{P_0} = \GL_n \times \GL_{n'}/F$ is the Levi quotient. Let $Q_0$ be the standard maximal parabolic subgroup of $G_0$ of type $(n',n),$ and $U_{Q_0}$ and $M_{Q_0}$ are similarly defined. 
Let $G$, $P$, $U_P,$ $M_P,$ $Q,$ $U_Q,$ and $M_Q$ be the restriction of scalars from $F$ to $\Q$ of $G_0$, $P_0$, $U_{P_0},$ $M_{P_0},$ $Q_0,$ $U_{Q_0},$ and $M_{Q_0},$ respectively. If $n = n'$ then $Q_0 = P_0$, $Q = P$, etc.; we will call this the self-associate case. 

Let $\mu \in X^*_0(T_n)$ be a pure weight, and similarly, $\mu' \in  X^*_0(T_{n'}).$ Assume there exists a balanced Kostant representative $w \in W^P$ such that 
$\lambda := w^{-1}\cdot(\mu \otimes \mu')$ is dominant.  So $nn'$ is even, and 
without loss of generality, we take $n$ to be even, and $n'$ may be even or odd. (See Appendix 2 for the case $nn'$ is odd.) 
Let $w^{\sf v} \in W^P$ correspond to $w$ as in Lem.\,\ref{lem:kostant-P-P}. Let $w'$ (resp., $w^{\sf v}{}'$) be in $W^Q$ corresponding to $w$ (resp., $w^{\sf v}$) via 
Lem.\,\ref{lem:kostant-P-Q}. All the elements $w, w', w^{\sf v}$ and $w^{\sf v}{}'$ are balanced, i.e., their lengths are 
${\rm dim}(U_P)/2 = {\rm dim}(U_Q)/2.$ In the self-associate case ($n = n'$, $P = Q$), we note that $w' \neq w$ and 
$w^{\sf v} \neq w^{\sf v}{}'.$

Let $\sigma_f  \in \Coh_{!!}(G_n, \mu)$ and  $\sigma'_f  \in \Coh_{!!}(G_{n'}, \mu').$  Take a large enough finite set of finite places 
$\place$ containing all the places where either $\sigma_f$ or $\sigma'_f$ is ramified. Take $K_f = \prod_\p K_\p$ an open-compact subgroup of $G(\A_f)$ such that $K_\p = \GL_N(\cO_\p)$ for $\p \notin \place.$ By $\HGS$ we mean the product of local spherical Hecke-algebras outside of $\place.$ 

Take $\varepsilon'$ as in (\ref{eqn:epsilon'-n'-cases}), and for brevity, we also denote 
$\tilde\varepsilon'$ for the character $\varepsilon' \otimes \varepsilon'$ of $\pi_0(M_P(\R))$ or of $\pi_0(M_Q(\R))$.

\smallskip

In the bottom degree, we observe $b_n^F + b_{n'}^F + \tfrac12\dim(U_P) = b_N^F,$ and 
write 
$$
I^\place_b(\sigma_f,\sigma^\prime_{f},\varepsilon)_{P, w} \ := \  
 {}^{\rm a}{\rm Ind}_{\pi_0(P(\R)) \times P(\A_f)}^{\pi_0(G(\R)) \times G(\A_f)}
\left(H^{b_n^F+b_{n'}^F}_{!!}(\SMP, \tM_{w \cdot \lambda, E})^{\pi_0(K^{M_P}_\infty)}
(\tilde\varepsilon' \otimes  (\sigma_f \otimes \sigma'_f)) \right)^{K_f}, 
$$
and similarly  
\begin{equation*}
\begin{split}
I^\place_b(\sigma_f,\sigma^\prime_{f},\varepsilon)_{Q, w^\prime} & :=  \\ 
& {}^{\rm a}{\rm Ind}_{\pi_0(Q(\R)) \times Q(\A_f)}^{\pi_0(G(\R)) \times G(\A_f)}
\left(H^{b_n^F+b_{n'}^F}_{!!}(\SMQ, \tM_{w' \cdot \lambda, E})^{\pi_0(K^{M_Q}_\infty)}
(\tilde\varepsilon' \otimes ( \sigma'_f(n) \otimes \sigma_f(-n')) \right)^{K_f} . 
\end{split}
\end{equation*}

\smallskip

In top-degree we work with the contragredient modules. The module 
$\sigma_f^\v(n') \otimes \sigma_f'^\v(-n)$ is strongly inner for $M_P$ with respect to the weight 
$w^{\sf v} \cdot \lambda^{\sf v}.$ Likewise, $\sigma_f^{\sf v}{}' \otimes \sigma_f^{\sf v}$ is strongly inner 
for $M_Q$ for the weight $w^{\sf v}{}' \cdot \lambda^{\sf v}.$ 
We also have $\tilde t_n^F + \tilde t_{n'}^F + \tfrac12\dim(U_P) = \tilde t_N^F -1.$ 
In this case we write: 
\begin{equation*}
\begin{split}
I^\place_t(\sigma_f,\sigma^\prime_{f},\varepsilon')^{\sf v}_{P, w^{\sf v}}  & := \\  
& {}^{\rm a}{\rm Ind}_{\pi_0(P(\R)) \times P(\A_f)}^{\pi_0(G(\R)) \times G(\A_f)}
\left(H^{\tilde t_n^F+ \tilde t_{n'}^F}_{!!}(\SMP, \tM_{w^{\sf v} \cdot \lambda^{\sf v}, E})^{\pi_0(K^{M_P}_\infty)}
(\tilde\varepsilon' \otimes (\sigma_f^\v(n') \otimes \sigma_f'^\v(-n))) \right)^{K_f} \!\!, 
\end{split}
\end{equation*}

\noindent
and similarly  
\begin{equation*}
\begin{split}
I^\place_t(\sigma_f,\sigma^\prime_{f},\varepsilon')^{\sf v}_{Q, w^{\sf v}{}'} & := \\ 
& {}^{\rm a}{\rm Ind}_{\pi_0(Q(\R)) \times Q(\A_f)}^{\pi_0(G(\R)) \times G(\A_f)}
\left(H^{\tilde t_n^F+ \tilde t_{n'}^F}_{!!}(\SMQ, \tM_{w^{\sf v}{}' \cdot \lambda^{\sf v}, E})^{\pi_0(K^{M_Q}_\infty)}
(\tilde\varepsilon' \otimes (\sigma_f'^\v \otimes \sigma_f^\v))\right)^{K_f} . 
\end{split}
\end{equation*}

\medskip
\subsubsection{\bf A strong form of Manin--Drinfeld principle}

Our main theorem on the structure of boundary cohomology as a Hecke module 
is the following theorem which is a culmination of the entire discussion in this section.

\begin{thm}
\label{thm:manin-drinfeld-strong}
Let the notations be as in Sect.\,\ref{sec:simple-notation}. Then:
 
\begin{enumerate}

\item The $\pi_0(G(\R)) \times \HGS$-modules $I^\place_b(\sigma_f,\sigma^\prime_{f},\varepsilon')_{P, w}$ and  
$I^\place_b(\sigma_f,\sigma^\prime_{f},\varepsilon')_{Q, w^\prime},$ and similarly, the modules 
 $I^\place_t(\sigma_f,\sigma^\prime_{f},\varepsilon')^{\sf v}_{P, w^{\sf v}}$ and 
$I^\place_t(\sigma_f,\sigma^\prime_{f},\varepsilon')^{\sf v}_{Q, w^{\sf v}{}'}$  
are finite-dimensional $E$-vector spaces, all of which have the same dimension; denote this common dimension as 
${\sf k}.$ 

\smallskip

\item The sum 
$$
I^\place_b(\sigma_f,\sigma^\prime_{f},\varepsilon' )_{P, w} \ \oplus \ 
I^\place_b(\sigma_f,\sigma^\prime_{f},\varepsilon' )_{Q, w^\prime}
$$
is a $2{\sf k}$-dimensional $E$-vector space that is isotypic in $H^{b_N^F}(\pBSC, \tM_{\lambda, E})^{K_f}$. 
(When $P = Q$, note that $w' \neq w.$)
Furthermore, 
there is a $\pi_0(G(\R)) \times \HGS$-equivariant projection:  
$$
\fR_{\sigma_f, \sigma_f', \varepsilon}^b : H^{b_N^F}(\pBSC, \tM_{\lambda, E})^{K_f} \ \longrightarrow \ 
I^\place_b(\sigma_f,\sigma^\prime_{f},\varepsilon')_{P, w}\oplus 
I^\place_b(\sigma_f,\sigma^\prime_{f},\varepsilon')_{Q, w^\prime}.
$$

\smallskip

\item The sum 
$$
I^\place_t(\sigma_f,\sigma^\prime_{f},\varepsilon')^{\sf v}_{P, w^{\sf v}} \ \oplus \ 
I^\place_t(\sigma_f,\sigma^\prime_{f},\varepsilon')^{\sf v}_{Q, w^{\sf v}{}'}
$$
is a $2{\sf k}$-dimensional $E$-vector space that is isotypic in 
$H^{\tilde t_N^F-1}(\pBSC, \tM_{\lambda^{\sf v}, E})^{K_f}$. 
(When $P = Q$, note that $w^{\sf v}{}' \neq w^{\sf v}.$) 
Furthermore, 
there is a $\pi_0(G(\R)) \times \HGS$-equivariant projection:  
$$
\fR_{\sigma_f, \sigma_f', \varepsilon}^t : H^{\tilde t_N^F-1}(\pBSC, \tM_{\lambda^{\sf v}, E})^{K_f} \ \longrightarrow \ 
I^\place_t(\sigma_f,\sigma^\prime_{f},\varepsilon')^{\sf v}_{P, w^{\sf v}} \oplus
I^\place_t(\sigma_f,\sigma^\prime_{f},\varepsilon')^{\sf v}_{Q, w^{\sf v}{}'}.
$$
\end{enumerate}
\end{thm}

\bigskip
\section{\bf Eisenstein cohomology}
\label{sec:eisenstein-coh}

Recall  from Sect.\,\ref{sec:long-e-seq} the long exact sequence of $\pi_0(G(\R)) \times \HK$-modules: 
$$
\cdots  \longrightarrow H^\bullet_c(\SGK, \tM_{\lambda, E}) 
\stackrel{\mathfrak{i}^*}{\longrightarrow}   H^\bullet(\BSC, \tM_{\lambda,E}) 
\stackrel{\mathfrak{r}^*}{\longrightarrow } H^\bullet(\PBSC, \tM_{\lambda,E}) 
\stackrel{\fd^*}{\longrightarrow} H^{\bullet+1}_c(\SGK, \tM_{\lambda,E}) \longrightarrow \cdots
$$
Eisenstein cohomology is defined as: 
\begin{equation}
\label{eqn:eis-coh-with-K_f}
H^\bullet_{\rm Eis}(\PBSC, \tM_{\lambda,E}) \ := \ 
{\rm Image}\left( H^\bullet(\BSC, \tM_{\lambda,E}) 
\stackrel{\mathfrak{r}^*}{\longrightarrow } H^\bullet(\PBSC, \tM_{\lambda,E}) \right). 
\end{equation}
We may pass to the limit over all $K_f$ and also look at $\pi_0(G(\R)) \times G(\A_f)$-modules
\begin{equation}
\label{eqn:eis-coh}
H^\bullet_{\rm Eis}(\pBSC, \tM_{\lambda,E}) \ := \ 
{\rm Image}\left( H^\bullet(\bar{\mathcal{S}}^G, \tM_{\lambda,E}) 
\stackrel{\mathfrak{r}^*}{\longrightarrow } H^\bullet(\pBSC, \tM_{\lambda,E}) \right). 
\end{equation}

The reader should note the subtle difference between $H^\bullet_{\rm Eis}(\pBSC, \tM_{\lambda,E})$ as above, and 
a complement to strongly inner cohomology in global cohomology denoted $H^\bullet_{\rm Eis}(\SG, \tM_{\lambda,E})$. Since strongly inner cohomology does not intertwine with boundary cohomology, we get a surjective map: 
$$
\xymatrix{
H^\bullet_{\rm Eis}(\SG, \tM_{\lambda,E})  \ar@{>>}[r] & 
H^\bullet_{\rm Eis}(\pBSC, \tM_{\lambda,E}).
}$$

\medskip
\subsection{Poincar\'e duality and maximal isotropic subspace of boundary cohomology}

\medskip
\subsubsection{\bf Poincar\'e duality}

Let the notations be as in Sect.\,\ref{sec:simple-notation}. Let $\d := \d_N^F = \dim(\SGK).$
We have the following Poincar\'e duality pairing for sheaf cohomology on $\SGK$:  
\begin{equation}
\label{eqn:poincare-1}
H^\bullet(\SGK, \tM_{\lambda, E}) \times H^{\d-\bullet}_c(\SGK, \tM_{\lambda^\v, E})\  \longrightarrow \ E.  
\end{equation}
The boundary $\PBSC$ is a compact manifold with corners with $ \dim(\PBSC) = \d-1.$ We have: 
\begin{equation}
\label{eqn:poincare-2}
H^\bullet(\PBSC, \tM_{\lambda, E}) \times H^{\d-1-\bullet}(\PBSC, \tM_{\lambda^\v, E}) \ \longrightarrow \ E
\end{equation}
We will denote either of the two dualities simply by $(\ , \ )$. (The reader is referred to \cite[Chap.\,3]{harder-book} 
for a discussion of Poincar\'e duality for manifolds with corners.)

\medskip
\subsubsection{\bf Compatibility of duality isomorphisms with connecting homomorphism}

Consider the following diagram: 
\begin{equation}
\label{eqn:poincare-compatible}
\xymatrix{
H^\bullet(\SGK, \tM_{\lambda, E}) \ar[d]^{\r^*} & \times & H^{\d-\bullet}_c(\SGK, \tM_{\lambda^\v, E}) & 
\longrightarrow & E \\
H^\bullet(\PBSC, \tM_{\lambda, E}) & \times &  H^{\d-1-\bullet}(\PBSC, \tM_{\lambda^\v, E}) \ar[u]^{\fd^*} & \longrightarrow & E
}
\end{equation}
where the horizontal arrows are the Poincar\'e duality pairings (\ref{eqn:poincare-1}) and (\ref{eqn:poincare-2}), 
and the vertical arrows $\r^*$ and $\fd^*$ are as in the long exact sequence. 
For any class $\xi \in H^\bullet(\SGK, \tM_{\lambda, E})$ and 
any $\varsigma \in  H^{\d-1-\bullet}(\PBSC, \tM_{\lambda^\v, E})$ we have (see \cite[Chap.\,3]{harder-book}): 
\begin{equation}
\label{eqn:duality-compatible}
(\r^*(\xi), \, \varsigma) \ = \ (\xi, \, \fd^*(\varsigma)).
\end{equation}

\medskip
\subsubsection{\bf Maximal isotropic subspaces}

The following proposition says that Eisenstein cohomology is a maximal isotropic subspace of boundary cohomology under Poincar\'e duality.

\begin{prop}
\label{prop:isotropic-subspace}
Under the duality pairing (\ref{eqn:poincare-2}) for boundary cohomology, we have: 
$$
H^\bullet_{\rm Eis}(\PBSC, \tM_{\lambda, E}) \ = \ 
H^{\d-1-\bullet}_{\rm Eis}(\PBSC, \tM_{\lambda^\v, E})^\perp.
$$
\end{prop}

\begin{proof}
This is an exercise in using (\ref{eqn:duality-compatible}). 
Let $\r^*(\xi) \in H^\bullet_{\rm Eis}(\PBSC, \tM_{\lambda, E})$. Then 
$$
(\r^*(\xi), \r^*(\varsigma')) = (\xi, \fd^*\r^*(\varsigma')) = (\xi,0) = 0, 
\quad \forall r^*(\varsigma') \in H^{\d-1-\bullet}_{\rm Eis}(\PBSC, \tM_{\lambda^\v, E}). 
$$
Hence the left hand side is contained in the right hand side. 
For the reverse inclusion, suppose $\xi' \in H^\bullet(\PBSC, \tM_{\lambda, E})$ is orthogonal to 
$H^{\d-1-\bullet}_{\rm Eis}(\PBSC, \tM_{\lambda^\v, E})$, then 
$$
0 = (\xi' , \r^*(\varsigma')) = (\fd^*(\xi'), \varsigma'), 
\quad \forall  \varsigma' \in H^{\d-1-\bullet}(\SGK, \tM_{\lambda^\v, E}).  
$$
Nondegeneracy of the duality pairing (\ref{eqn:poincare-1}) in degree $\d-1-\bullet$ 
implies $\xi' \in {\rm Ker}(\fd^*) = {\rm Im}(\r^*).$ Hence, 
$\xi' \in H^\bullet_{\rm Eis}(\PBSC, \tM_{\lambda, E}).$ 
\end{proof}

\medskip
\subsection{The main result on rank-one Eisenstein cohomology}
Notations are as in  Thm.\,\ref{thm:manin-drinfeld-strong}. 
Consider the following maps  
starting from global cohomology $H^{b_N^F}( \SG,\tM_{\lambda,E})^{K_f}$ 
and ending with an isotypic component in boundary cohomology:

\begin{equation}
\label{eqn:eis-coh-image}
\xymatrix{
H^{b_N^F}( \SG,\tM_{\lambda,E})^{K_f}  \ar[d]^{\r^*} \\ 
H^{b_N^F}(\pBSC,\tM_{\lambda,E})^{K_f}  \ar[d]^{\fR_{\sigma_f, \sigma'_f, \varepsilon'}^b}  \\ 
I^\place_b(\sigma_f,\sigma_f', \varepsilon')_{P,w} \ \oplus \ I^\place_b(\sigma_f,\sigma_f', \varepsilon')_{Q, w'}
}\end{equation}

Recall, from Thm.\,\ref{thm:manin-drinfeld-strong},  that the target space 
$I^\place_b(\sigma_f,\sigma_f', \varepsilon')_{P,w} \oplus I^\place_b(\sigma_f,\sigma_f', \varepsilon')_{Q,w'}$  
is a $E$-vector space of dimension $2{\sf k}$.  In the self-associate case just change the $Q$ to $P$. 
Our main result on Eisenstein cohomology (see Thm.\,\ref{thm:rank-one-eis}
below) says that the image of 
$H^{b_N^F}_{\rm Eis}(\pBSC,\tM_{\lambda,E})^{K_f} = {\rm Im}(\r^*)$ under $\fR_{\sigma_f, \sigma'_f, \varepsilon'}^b$ is 
a middle-dimensional (i.e., ${\sf k}$-dimensional) subspace of this $2{\sf k}$-dimensional space. 
(It helps to have a mental picture of when ${\sf k} = 1$, i.e., of a line in an ambient two-dimensional space; we will see later that the ``slope" of this line contains arithmetic information about $L$-values.) 
The proof of this main result also needs the analogue of (\ref{eqn:eis-coh-image}) in top-degree. We now state and prove this main result on Eisenstein cohomology.

\medskip
\subsubsection{\bf The image of Eisenstein cohomology under $\fR_{\sigma_f, \sigma'_f, \varepsilon'}^\bullet$}

\begin{thm}
\label{thm:rank-one-eis}
Let the notations be as in \ref{sec:simple-notation}. 
Furthermore, for brevity, let 
\begin{eqnarray*}
\fI^b(\sigma_f, \sigma_f', \varepsilon') \ & :=  &\ 
\fR_{\sigma_f, \sigma'_f, \varepsilon'}^b(H^{b_N^F}_{\rm Eis}(\pBSC,\tM_{\lambda,E})^{K_f}), \\
\fI^t(\sigma_f, \sigma_f', \varepsilon')^\v \ & := & \ 
\fR_{\sigma_f, \sigma'_f, \varepsilon'}^t(H^{\tilde t_N^F-1}_{\rm Eis}(\pBSC,\tM_{\lambda^\v,E})^{K_f}).
\end{eqnarray*}

\begin{enumerate}
\item In the non-self-associate cases ($n \neq n'$) we have: 

\smallskip

\begin{enumerate}
\item $\fI^b(\sigma_f, \sigma_f', \varepsilon')$ is a ${\sf k}$-dimensional $E$-subspace of 
$I^\place_b(\sigma_f,\sigma_f', \varepsilon')_{P,w} \oplus I^\place_b(\sigma_f,\sigma_f', \varepsilon')_{Q, w'}.$ 

\smallskip

\item  $\fI^t(\sigma_f, \sigma_f', \varepsilon')^\v$ is a ${\sf k}$-dimensional $E$-subspace of 
$I^\place_t(\sigma_f,\sigma_f', \varepsilon')^\v_{P, w^{\sf v}} \oplus 
I^\place_t(\sigma_f,\sigma_f', \varepsilon')^\v_{Q, w^{\sf v}{}'}.$
\end{enumerate}

\medskip

\item In the self-associate case ($n=n'$) the same assertions hold by putting $Q = P.$ 
\end{enumerate}
\end{thm}

\smallskip
\subsubsection{\bf Proof of Thm.\,\ref{thm:rank-one-eis}} 
The proof of (2) is almost identical to the proof of (1), and so we give the details only for (1). 
The proof of (1) involves two steps: 
\begin{enumerate}
\item[(i)] The first step is to show that both $\fI^b(\sigma_f, \sigma_f', \varepsilon')$ and 
$\fI^t(\sigma_f, \sigma_f', \varepsilon')^\v$ are at least ${\sf k}$-dimensional; this is achieved by going to a transcendental level and appealing to Langlands's constant term theorem and producing enough cohomology classes in the image. 

\item[(ii)] The second step, after invoking properties of the Poincar\'e duality pairing, is to show that both 
$\fI^b(\sigma_f, \sigma_f', \varepsilon')$ and $\fI^t(\sigma_f, \sigma_f', \varepsilon')^\v$ have dimension exactly 
${\sf k}.$ 
\end{enumerate}

We take up these two arguments in the paragraphs \ref{sec:langlands-constant-term} and \ref{sec:app-poincare} below.

\medskip
\paragraph{\bf The cohomological meaning of the constant term theorem of Langlands}
\label{sec:langlands-constant-term}
Take an embedding $\iota : E \to \C$ and pass to a transcendental level via $\iota.$ To show that 
$\fI^b(\sigma_f, \sigma_f', \varepsilon')$ or $\fI^t(\sigma_f, \sigma_f', \varepsilon')^\v$ 
has dimension at least ${\sf k}$ as an $E$-vector space, it suffices to show that their base-change to 
$\C$ via $\iota$ has dimension at least ${\sf k}$ as a $\C$-vector space, i.e., we would like to show: 
\begin{multline}
\label{eqn:nonzero-via-Langlands}
\dim_\C\left(
\fR_{{}^\iota\!\sigma_f, {}^\iota\!\sigma'_f, \varepsilon'}^b
(H^{b_N^F}_{\rm Eis}(\pBSC,\tM_{{}^\iota\!\lambda})^{K_f}) 
\right) \geq {\sf k}
\quad {\rm and} \quad \\
\dim_\C\left(
\fR_{{}^\iota\!\sigma_f, {}^\iota\!\sigma'_f, \varepsilon'}^t(H^{\tilde{t}_N^F}_{\rm Eis}(\pBSC,\tM_{{}^\iota\!\lambda})^{K_f})
\right) \geq {\sf k}. 
\end{multline}
In Sect.\,\ref{sec:aut-l-fns} we briefly introduce the $L$-functions at hand, recall the celebrated theorem of Lanlgands on the constant term of an Eisenstein series, and then we will come back to the proof of this part in 
Sect.\,\ref{sec:conclude-proof-nonzero-image}.

\medskip
\paragraph{\bf Application of Poincar\'e duality}
\label{sec:app-poincare}
The proofs of (1)(a) and (1)(b), assuming that we have proved (\ref{eqn:nonzero-via-Langlands}), is an exercise involving properties of Poincar\'e duality; especially that it is non-degenerate, Hecke-equivariant,  and that the Eisenstein part is maximal isotropic. 
The bare-bones linear algebra looks like: suppose we have decompositions 
$V = V_P \oplus V_Q$ and $W = W_P \oplus W_Q,$
where $V_P,V_Q,W_P,$ and $W_Q$ are all ${\sf k}$-dimensional vector spaces over $E$; suppose also that 
we have a non-degenerate pairing $( \ ,\ ) : V \times W \to E$ such that $(V_P, W_Q) = (V_Q , W_P) = 0$ and the pairing is non-degenerate on $V_P \times W_P$ and $V_Q \times W_Q;$ furthermore, suppose we are given subspaces 
$\fI \subset V$ and 
$\fJ \subset W$ such that $\dim_E(\fI) \geq {\sf k},$ $\dim_E(\fJ) \geq {\sf k},$ and $(\fI, \fJ) = 0.$ Then it is easy to see 
that $\dim_E(\fI) = {\sf k} = \dim_E(\fJ).$

This concludes the proof of Thm.\,\ref{thm:rank-one-eis} under the assumption 
that we have proved (\ref{eqn:nonzero-via-Langlands}).

\medskip
\subsection{A theorem of Langlands on the constant term of an Eisenstein series}
\label{sec:aut-l-fns}

We will need some details from the Langlands--Shahidi method in our context. 
The reader is referred to Kim \cite{kim-fields} and Shahidi \cite{shahidi-pcmi} \cite{shahidi-book} for details, 
proofs and further references.

\medskip
\subsubsection{\bf $\bfgreek{alpha}_P$, $\bfgreek{gamma}_P$, $\bfgreek{rho}_P$, $\bfgreek{delta}_P$} 
Let the notations be as in Sect.\,\ref{sec:simple-notation}. In particular, 
$P_0 = M_{P_0}U_{P_0}$ is the standard $(n,n')$ parabolic subgroup of $\GL_N/F,$ and $P = R_{F/\Q}(P_0)$, etc.  
We write 
$$
M_{P_0} = \left\{m = {\rm diag}(h,h') =   
\left(\begin{array}{cc} h & 0 \\ 0 & h' \end{array}\right) \ :\ 
h \in \GL_n, \ h' \in \GL_{n'} \right\},
$$
and let 
$$
A_{P_0} := Z_{M_{P_0}} = 
\left\{a =  
\left(\begin{array}{cc} t 1_n & 0 \\ 0 & t' 1_{n'} \end{array}\right) \ :\ 
t,t' \in \GL_1 \right\}.
$$

Fix an identification $X^*(A_{P_0}) = \Z^2$, by letting $(k,k') \in \Z^2$ correspond to the character 
that sends $a$ to $t^kt'{}^{k'}$. We have 
$$
X^*(A_P \times E) = \bigoplus_{\tau: F \to E} X^*(A_{P_0} \times_\tau E) = \bigoplus_{\tau :F \to E} \Z^2.
$$ 
Similarly, fix $X^*(M_{P_0}) = \Z^2$, by letting $(k,k') \in \Z^2$ correspond to the character 
that sends ${\rm diag}(h,h')$ to ${\rm det}(h)^{k}{\rm det}(h')^{k'}$. 
Restriction from $M_{P_0}$ to $A_{P_0}$ gives an inclusion $X^*(M_{P_0}) \hookrightarrow X^*(A_{P_0})$ which is given by $(k,k') \mapsto (nk,n'k')$. Clearly, $X^*(M_{P_0}) \otimes \Q = X(A_{P_0}) \otimes \Q$, which fixes an identification $X^*(M_{P_0}) \otimes \Q = \Q^2$ via $X(A_{P_0}) \otimes \Q = \Q^2.$ 
Similarly, $X^*(M_{P_0}) \otimes \R = X^*(A_{P_0}) \otimes \R$ and fix 
$X^*(M_{P_0}) \otimes \R = \R^2$. This fixes $X^*(M_P) \otimes \R = \oplus_{\tau:F \to E} \R^2.$

\smallskip

Let $\bfgreek{rho}_{P_0}$ be half the sum of positive roots whose root 
spaces appear in $U_{P_0}$. The restriction of $\bfgreek{rho} _{P_0}$ to $A_{P_0}$ is in 
$X^*(A_{P_0}) \otimes \Q \hookrightarrow \mathfrak{a}_{P_0}^* := X(A_{P_0}) \otimes \R = X^*(M_{P_0}) \otimes \R$ and under the above identification of the latter with $\R^2$, one has $\bfgreek{rho}_{P_0} = (n'/2, -n/2)$, or using the notations of \ref{sec:standard-fundamental}, we have
$$
\bfgreek{rho}_{P_0} = \frac12\sum_{\stackrel{1\leq i \leq n}{n+1 \leq j \leq N}} \e_i - \e_j = 
\frac{n'}{2}(\e_1 + \cdots + \e_n) - \frac{n}{2}(\e_{n+1} + \cdots + \e_{n+n'}). 
$$
And $\bfgreek{rho}_P = (\bfgreek{rho}_{P_0^\tau})_{\tau : F \to E},$ with each $\bfgreek{rho}_{P_0^\tau}$ given as above. 
 
\smallskip

Let $\bfgreek{alpha}_{P_0} =  \bfgreek{alpha}_n = \e_n - \e_{n+1}$ be the unique simple root of $G_0$ that is not amongst the roots of $M_{P_0}.$ Consider the corresponding fundamental weight 
$\bfgreek{gamma}_{P_0} = \bfgreek{gamma}_n =  \langle \bfgreek{rho}_{P_0}, \bfgreek{alpha}_{P_0} \rangle^{-1} 
\bfgreek{rho}_{P_0}.$ 
Identify $X^*(T_{N,0}) \otimes \R$ with $\R^N$ using the $\e_i$'s, and let $(\ ,\ )$ be the usual euclidean 
inner product on $\R^N$. It is easy to see that 
$\langle \bfgreek{rho}_{P_0}, \bfgreek{alpha}_{P_0} \rangle = 
\frac{2 (\bfgreek{rho}_{P_0}, \bfgreek{alpha}_{P_0})}{(\bfgreek{alpha}_{P_0},\bfgreek{alpha}_{P_0})} = N/2$, and hence 
$$
\bfgreek{gamma}_{P_0} =  \frac{2}{N}\bfgreek{rho}_{P_0}.
$$
We also have $\bfgreek{alpha}_P = (\bfgreek{alpha}_{P_0^\tau})_{\tau: F \to E}$ and 
$\bfgreek{gamma}_P = (\bfgreek{gamma}_{P_0^\tau})_{\tau: F \to E}$, with 
$N\, \bfgreek{gamma}_{P_0^\tau}=  2\, \bfgreek{rho}_{P_0^\tau}.$ 

\smallskip

Let $\bfgreek{delta}_{P_0}$ be the modular character of $M_{P_0}(k)$, where $k$ is any local field (such as $k=\R$ or 
$k = F_v$ any $p$-adic completion of $F$), which is defined as 
$\bfgreek{delta}_{P_0}(m) = |{\rm det}({\rm Ad}_{\mathfrak{u}_P}(m))|$ for $m \in M_{P_0}(k)$. If 
$m = {\rm diag}(h, h')$ with $h \in \GL_n(k)$ and $h' \in \GL_{n'}(k)$ then 
$$
\bfgreek{delta}_{P_0}({\rm diag}(h,h')) = |{\rm det}(h)|^{n'}|{\rm det}(h')|^{-n}.
$$
Note that that $|2\bfgreek{rho}_{P_0}(\cdot)| = \bfgreek{delta}_{P_0}(\cdot)$. Also, 
$\bfgreek{delta}_P$ the modular character of $M_P(k)$, for $k = \R$ or $k=\Q_p$,  is defined via the various completions of $F$ over that place; for example, $\bfgreek{delta}_P$ on $M_P(\Q_p)$ is $\prod_{\p | p} \bfgreek{delta}_{P_0, \p},$ 
where by $\bfgreek{delta}_{P_0, \p}$ we mean the character as above on $M_{P_0}(F_\p).$

\medskip
\subsubsection{\bf Induced representations}
Let $\sigma$ (resp., $\sigma'$) be a cuspidal automorphic representation of $G_n(\A)$ (resp., $G_{n'}(\A)$). The relation with our previous `arithmetic' notation is that given $\sigma_f \in \Coh_{!!}(G_n,\mu)$ and given $\iota : E \to \C$, think of 
${}^\iota\sigma_f$ to be the finite part of a cuspidal automorphic representation ${}^\iota\sigma$, etc. The $\iota$ is fixed, and we suppress it until otherwise mentioned. Consider the induced representation $I_P^G(s, \sigma \otimes \sigma')$ consisting of all smooth functions 
$f : G(\A) \to V_{\sigma} \otimes V_{\sigma'}$ such that 
\begin{equation}
\label{eqn:f-in-induced-repn}
f(mug) \ = \ 
\bfgreek{delta}_P(m)^{\frac{1}{2}} \bfgreek{delta}_P(m)^{\frac{s}{N}} \, (\sigma \otimes \sigma')(m) \, f(g)
\end{equation}
for all $m \in M_P(\A)$, $u \in U_P(\A)$, and $g \in G(\A)$; where $V_\sigma$ (resp., $V_{\sigma'}$) is the subspace inside the space of cusp forms on $G_n(\A)$ (resp., $G_{n'}(\A)$) realizing the representation $\sigma$ (resp., 
$\sigma'$). In other words,
$$
I_P^G(s, \sigma \otimes \sigma') = 
{\rm Ind}_{P(\A)}^{G(\A)}((\sigma  \otimes |\ |^{\frac{n'}{N}s}) \otimes (\sigma' \otimes |\ |^{\frac{-n}{N}s})), 
$$
where ${\rm Ind}_P^G$ denotes the normalized parabolic induction. In terms of 
algebraic or un-normalized induction, we have
\begin{equation}
\label{eqn:a-ind-s=-N/2}
I_P^G(s,\sigma \otimes \sigma') = {}^{\rm a}{\rm Ind}_{P(\A)}^{G(\A)}
((\sigma \otimes |\ |^{\frac{n'}{N}s + \frac{n'}{2}}) \otimes 
(\sigma'  \otimes |\ |^{\frac{-n}{N}s - \frac{n}{2}})). 
\end{equation}

\medskip
\subsubsection{\bf Standard intertwining operators}
\label{sec:T_st}

There is an element $w_{P_0} \in W_G$, the Weyl group of $G$, which is uniquely determined by the property
$w_{P_0}(\bfpi_G - \{\bfgreek{alpha}_P\}) \subset \bfpi_G$ and $w_{P_0}(\alpha) < 0.$ This element looks like 
$w_{P_0} = (w_{P_0}^\tau)_{\tau:F \to E}$, where for each $\tau$, as a permutation matrix in $\GL_N$, we have 
$$
w_{P_0}^\tau = \left[\begin{array}{ll} & 1_n \\ 1_{n'} & \end{array}\right].
$$

The parabolic subgroup $Q$, which is associate to $P$, corresponds to 
$w_{P_0}(\bfpi_G - \{\bfgreek{alpha}_P\}).$ Since 
$w_{P_0}^\tau{}^{-1} {\rm diag}(h,h') w_{P_0}^\tau = {\rm diag}(h',h)$ for all $ {\rm diag}(h,h') \in M_{P_0^\tau}$, we get  
$w_{P_0}(\sigma \otimes \sigma') =  \sigma' \otimes \sigma$ as a representation of $M_{Q}(\A)$. 
The global  standard intertwining operator:
$$
T_{\rm st}^{PQ}(s, \sigma \otimes \sigma') :  I_P^G(s, \sigma \otimes \sigma') \ \longrightarrow \  
I_Q^G(-s, \sigma' \otimes \sigma)
$$
is given by the integral
$$
(T_{\rm st}^{PQ}(s, \sigma \otimes \sigma')f)(g) = \int_{U_{Q}(\A)}f(w_{P_0}^{-1}ug)\, du. 
$$
Often, we will abbreviate $T_{\rm st}^{PQ}(s, \sigma \otimes \sigma')$ as $T_{\rm st}(s, \sigma \otimes \sigma').$
The global standard intertwining operator factorizes as a product of local standard intertwining operators:
$T_{\rm st}(s, \sigma \otimes \sigma') = 
\otimes_v T_{\rm st}(s, \sigma_v \otimes \sigma'_v)$
where the local operator 
$$
T_{\rm st}(s, \sigma_v \otimes \sigma'_v) : 
I_P^G(s, \sigma_v \otimes \sigma'_v) \ \longrightarrow \  
I_Q^G(-s,  \sigma_v' \otimes \sigma_v)
$$
is given by a similar local integral.

\medskip
\subsubsection{\bf Ratios of $L$-functions}
\label{sec:ratio-l-fns}
A fundamental observation of Langlands \cite{langlands-euler} says that if $v$ is a finite place of $F$ where both 
$\sigma$ and $\sigma'$ are unramified, and suppose $f_v^0$ (resp., $\tilde f_v^0$) is 
the normalized spherical vector of $I_P^G(s, \sigma_v \otimes \sigma'_v)$ (resp., 
$I_Q^G(-s, \sigma'_v \otimes \sigma_v)$) then 
$$
T_{\rm st}(s, \sigma_v \otimes \sigma'_v)f_v^0 \ = \ 
\frac{L(s, \sigma_v \times \sigma_v'^{\sf v})}{L(s+1, \sigma_v \times \sigma_v'^{\sf v})} \, \tilde f_v^0. 
$$
If ${\rm diag}(\vartheta_{1,v},\dots,\vartheta_{n,v})$ (resp., ${\rm diag}(\vartheta'_{1,v},\dots,\vartheta'_{n',v})$) is the 
Satake parameter of $\sigma_v$ (resp., $\sigma'_v$) then the local `Rankin--Selberg' $L$-function above is given by:
$$
L(s, \sigma_v \times \sigma_v'^{\sf v}) \ = \ \prod_{\substack{1\leq i \leq n \\ 1 \leq j \leq n'}}
(1 - \vartheta_{i,v}\vartheta_{j,v}'^{-1} q_v^{-s})^{-1}.
$$

Now, if $f \in I_P^G(s, \sigma \otimes \sigma')$ is a pure-tensor $f = \otimes_v f_v$, with $f_v = f_v^0$ outside a finite set $\place$ of places including the archimedean ones, then 
\begin{equation}
\label{eqn:T-stand-ratio-L-fns}
T_{\rm st}(s, \sigma \otimes \sigma') f \ = \ 
\frac{L^\place(s, \sigma \times \sigma'^{\sf v})}{L^\place(s+1, \sigma \times \sigma'^{\sf v})} \, 
\otimes_{v \notin \place} \tilde f_v^0 \otimes_{v \in \place} 
T_{\rm st}(s, \sigma_v \otimes \sigma'_v)f_v,
\end{equation}
where $L^\place(s, \sigma \times \sigma'^{\sf v}) = \prod_{v \notin \place} L(s, \sigma_v \times \sigma_v'^{\sf v})$ is the partial $L$-function. 

\smallskip

Starting from \ref{sec:eff-motives-coh-l-fns}, we will work with the cohomological $L$-function which is better adapted for studying arithmetic properties of $L$-values. To contrast then the above $L$-function will be referred to as `automorphic' $L$-function.

\medskip
\subsubsection{\bf Eisenstein series}
\label{sec:eisenstein-series}

Let $f \in I_P^G(s, \sigma \times \sigma')$ be as in (\ref{eqn:f-in-induced-repn}). For $\ul g \in G(\A)$, the value 
$f(\ul g)$ is a cusp form on $M_P(\A).$ By the defining equivariance property of $f$, the complex number 
$f(\ul g)(\ul m)$ determines and is determined by $f(\ul m \ul g)(\ul 1)$ for any $\ul m \in M_P(\A).$ 
Henceforth, we identify 
$f \in I_P^G(s, \sigma \times \sigma')$ with the complex valued function $\ul g \mapsto f(\ul g)(\ul 1),$ i.e., we have embedded: 
$$
I_P^G(s, \sigma \times \sigma') \ \hookrightarrow \ 
\cC^\infty\left(U_P(\A)M_P(\Q)\backslash G(\A), \omega_\infty^{-1}\right) 
\ \subset \ \cC^\infty\left(P(\Q)\backslash G(\A), \omega_\infty^{-1}\right), 
$$
where $\omega^{-1}_\infty$ is a simplified notation for the central character of $\sigma \otimes \sigma'$ restricted to 
$S(\R)^\circ.$ If $\sigma_f \in \Coh(G_n, \mu)$, $\sigma_f' \in  \Coh(G_{n',} \mu')$ and $\iota : E \to \C$ then 
$\omega_\infty$ is the product of the central characters $\omega_{\M_{{}^\iota\! \mu}} \omega_{\M_{{}^\iota\! \mu'}}$ restricted to $S(\R).$ 

Given $f \in I_P^G(s, \sigma \times \sigma')$, thought of as a complex-valued function on $P(\Q)\backslash G(\A)$, define the corresponding Eisenstein series $\Eis_P(s, f) \in \cC^\infty\left(G(\Q)\backslash G(\A), \omega_\infty^{-1}\right)$ by averaging over $P(\Q)\backslash G(\Q)$: 
\begin{equation}
\label{eqn:Eis-P}
\Eis_P(s, f)(\ul g) \ := \ \sum_{\gamma \in P(\Q)\backslash G(\Q)} f(\gamma \ul g), 
\end{equation}
as a formal sum, and we will discuss below holomorphy of this Eisenstein series at a particular point of evaluation.

\medskip
\subsubsection{\bf Constant term of an Eisenstein series}
\label{sec:constant-term-eisenstein-series}
For the parabolic subgroup $Q$, recall the constant term map from Sect.\,\ref{sec:spectral-seq-trans} denoted
$\cF^Q : \cC^\infty(G(\Q)\backslash G(\A), \omega_\infty^{-1}) 
\to \cC^\infty(M_Q(\Q)U_Q(\A)\backslash G(\A), \omega_\infty^{-1}),$ and given by: 
$$
\cF^Q(\phi)(\ul g) \ = \ 
\int\limits_{U_Q(\Q)\backslash U_Q(\A)} \phi(\ul u \, \ul g)\, d\ul u.
$$
Consider the following diagram of maps: 
$$
\xymatrix{
I_P^G(s, \sigma \times \sigma') \ar@{^{(}->}[rr] \ar[dd]_{T_{\rm st}(s, \sigma \times \sigma')}& & 
\cC^\infty\left(P(\Q)\backslash G(\A), \omega_\infty^{-1}\right) \ar[d]^{\Eis_P} \\
& & \cC^\infty\left(G(\Q)\backslash G(\A), \omega_\infty^{-1}\right) \ar[d]^{\cF^Q} \\
I_Q^G(-s, \sigma' \times \sigma) \ar@{^{(}->}[rr] & & 
\cC^\infty\left(Q(\Q)\backslash G(\A), \omega_\infty^{-1}\right)
}$$

\begin{thm}[Langlands]
\label{thm:langlands}
Let $f  \in I_P^G(s, \sigma \times \sigma').$ 
\begin{enumerate}
\item In the non-self-associate cases ($n \neq n'$), we have: 
\smallskip
   \begin{enumerate}
   \item $\cF^P \circ \Eis_P(s, f) \ = \ f.$ 
   \smallskip

   \item $\cF^Q \circ \Eis_P(s, f) \ = \ T_{\rm st}(s, \sigma \times \sigma')(f).$ 
   \end{enumerate}
   \smallskip

\item In the self-associate cases ($n = n'$ and $P = Q$), we have: 
$$\cF^P \circ \Eis_P(s, f) \ = \ f + T_{\rm st}(s, \sigma \times \sigma')(f).$$ 

\end{enumerate}

\end{thm}

Using (\ref{eqn:T-stand-ratio-L-fns}), for $f = \otimes f_v \in I_P^G(s, \sigma \times \sigma')$, we get
in the non-self-associate cases ($n \neq n'$): 
\begin{equation}
\label{eqn:langlands-ratio-L-fns}
\cF^Q (\Eis_P(s, f)) \ = \ 
\frac{L^\place(s, \sigma \times \sigma'^{\sf v})}{L^\place(s+1, \sigma \times \sigma'^{\sf v})} \, 
\otimes_{v \notin \place} \tilde f_v^0 \otimes_{v \in \place} 
T_{\rm st}^{PQ}(s, \sigma_v \otimes \sigma'_v)f_v. 
\end{equation}
For future reference, let's record the same statement if we started from a representation induced from $Q$. 
Let $\tilde f = \otimes \tilde f_v \in I_Q^G(w, \sigma' \times \sigma)$, we get
\begin{equation}
\label{eqn:langlands-ratio-L-fns-Q-to-P}
\cF^P (\Eis_Q(w, \tilde f)) \ = \ 
\frac{L^\place(w, \sigma' \times \sigma^{\sf v})}{L^\place(w+1, \sigma' \times \sigma^{\sf v})} \, 
\otimes_{v \notin \place}  f_v^0 \otimes_{v \in \place} 
T_{\rm st}^{QP}(w, \sigma'_v \otimes \sigma_v) \tilde f_v. 
\end{equation}
We leave it to the reader to formulate a similar statement in the self-associate case.

\medskip
\subsubsection{\bf Evaluating at $s = -N/2$ and $w = N/2$}
\label{sec:-n/2}
Recall the induced representations which appear in boundary cohomology as in Sect.\,\ref{sec:simple-notation}, 
and recall also how the complex variable $s$ appears in an induced representation as 
in (\ref{eqn:a-ind-s=-N/2}). Putting the two together, we see: 
$$
{}^{\rm a}{\rm Ind}_{P(\A)}^{G(\A)}
\left( \sigma \otimes \sigma' \right) \ = \ 
I_P^G(s, \sigma \otimes \sigma')|_{s = -N/2}. 
$$
Similarly, we have
$$
{}^{\rm a}{\rm Ind}_{Q(\A)}^{G(\A)}
\left( \sigma'(n) \otimes \sigma(-n') \right) \ = \ 
I_Q^G(w, \sigma' \otimes \sigma)|_{w = N/2}. 
$$
It might be helpful to bear in mind the following picture: 
$$
\xymatrix{
 & & & & \\
 & & & & \\
{}^{\rm a}{\rm Ind}_{P(\A)}^{G(\A)}
\left( \sigma \otimes \sigma' \right) \ar@/^4pc/[rrrr]^{T^{PQ}_{\rm st}(s, \sigma \otimes \sigma')|_{s = -N/2}} & & & &
{}^{\rm a}{\rm Ind}_{Q(\A)}^{G(\A)}
\left( \sigma'(n) \otimes \sigma(-n') \right) \ar@/^4pc/[llll]^{T^{QP}_{\rm st}(w, \sigma' \otimes \sigma)|_{w = N/2}}\\
 & & & & \\
 & & & & 
}$$

\medskip
\subsubsection{\bf Holomorphy of one of the Eisenstein series at point of evaluation} 
\label{sec:holomorphy-eisenstein}

The aim of this section is to show that it is not possible for $\Eis_P(s,f)$ to have a pole at $s = -N/2$  \underline{and} 
$\Eis_Q(w, \tilde f)$ to have a pole at $w = N/2.$ But before stating the result 
(see Prop.\,\ref{prop:at-least-one-holomorphic} below) we need to briefly review some well-known analytic properties of Rankin--Selberg $L$-functions.

\medskip
\paragraph{\bf An interlude on analytic properties of Rankin--Selberg $L$-functions}
\label{sec:analytic-rankin-selberg} 
(We refer the reader to Shahidi's book \cite[10.1]{shahidi-book} for further details and references.) 
Let $\sigma$ (resp., $\sigma'$) be a cuspidal automorphic representation of 
$G_n(\A)$ (resp., $G_{n'}(\A)$). For any place $v$ the local $L$-function 
$L(s, \sigma_v \times \sigma_v'^{\sf v})$ is defined as a local Artin L-function via the local Langlands correspondence. For a finite unramified place $v$ we gave this definition in  Sect.\,\ref{sec:ratio-l-fns} using the Satake parameters. There are other ways of defining the local factor, for example, via the approach of zeta integrals and integral representations; it is a fact that we end up with the same local factor; see {\it loc.\,cit.} The global $L$-function is defined for $\Re(s) \gg 0$ as an Euler product: 
$L(s, \sigma \times \sigma'^{\sf v}) = \prod_v L(s, \sigma_v \times \sigma'^{\sf v}_v). $
For any finite set of places $\place$, the partial $L$-function is defined for $\Re(s) \gg 0$ as
$L^\place(s, \sigma \times \sigma'^{\sf v}) = \prod_{v \notin \place} L(s, \sigma_v \times \sigma'^{\sf v}_v).$ These Euler products, defined 
{\it a priori} only in a half-plane, admit a meromorphic continuation to all of $\C.$ 
We need the following important results (see \cite[Thm.\,10.1.1]{shahidi-book}): 

\begin{thm}
\label{thm:rankin-selberg-analytic}
Let $\sigma$ (resp., $\sigma'$) be a cuspidal automorphic representation of 
$G_n(\A)$ (resp., $G_{n'}(\A)$). 
\begin{enumerate}
\item Suppose $n \neq n'$ then $L(s, \sigma \times \sigma'^{\sf v})$ extends to an entire function of $s$. 

\smallskip
\item Suppose $n = n'$ then $L(s, \sigma \times \sigma'^{\sf v})$ extends to an entire function of $s$, unless
$\sigma \simeq \sigma'$, and in which case $L(s, \sigma \times \sigma^{\sf v})$
extends to a meromorphic function of $s$ with only one pole, which is located at $s=1$ and is 
a simple pole. 

\smallskip
\item Functional equation: $L(s, \sigma \times \sigma'^{\sf v}) = \varepsilon(s, \sigma \times \sigma'^{\sf v})
L(1-s, \sigma^{\sf v} \times \sigma').$ (The epsilon factor on the right hand side is an exponential function.)

\smallskip
\item Suppose that $\sigma$ and $\sigma'$ are  \underline{unitary}  then for any finite set of 
places $\place$ we have: 
$$
L^\place(s, \sigma \times \sigma'^{\sf v}) \neq 0, \quad \Re(s) \geq 1.
$$

\end{enumerate}
\end{thm}

Let us also recall the following definition adapted from Deligne \cite{deligne}. Given $\sigma$, we can write the representation at infinity as 
$\sigma_\infty = \otimes_{v \in \place_\infty} \sigma_v;$
the tensor product is over all the archimedean places of $F.$ The $L$-factor at infinity is 
$L(s, \sigma_\infty \times \sigma'_\infty) = \prod_{v \in \place_\infty}  L(s, \sigma_v \times \sigma'_v)$ is 
a product of $\Gamma$-factors times exponential functions.

\begin{defn}
\label{def:critical-aut-l-fn}
Let $\sigma$ (resp., $\sigma'$) be a cuspidal automorphic representation of 
$G_n(\A)$ (resp., $G_{n'}(\A)$). Assume they are of cohomological type, i.e., $\sigma_f \in \Coh(G_n, \mu)$  
for some $\mu \in X^*_0(T_n)$ (resp., $\sigma'_f \in \Coh(G_{n'},\mu')$ for some $\mu' \in X^*_0(T_{n'})$). 
\begin{enumerate}
\item If $n \equiv n' \pmod{2}$ then take $m \in \Z.$
\item If $n \not\equiv n' \pmod{2}$ then take $m \in \tfrac12 + \Z.$
\end{enumerate}
We say that such an $m$ is critical for the Rankin--Selberg $L$-function $L(s, \sigma \times \sigma'^{\sf v})$ if the local factors at infinity on both sides of the functional equation are regular (holomorphic) at $s=m,$ 
i.e., $L(s, \sigma_\infty \times \sigma'^{\sf v}_\infty)$ and 
$L(1-s, \sigma^{\sf v}_\infty \times \sigma'_\infty)$ are finite at $s=m.$
\end{defn}

At this moment the reader can ignore the possibly half-integral nature of the critical point $m$ which is caused by the so-called motivic normalization; see Sect.\,\ref{sec:shift-s-variable}. This amongst several related arithmetic issues will be discussed in Sect.\,\ref{sec:eff-motives-coh-l-fns} and Sect.\,\ref{sec:critical-points-comb-lemma} 
on cohomological $L$-functions and effective motives.

\medskip
\paragraph{\bf Statement of holomorphy concerning certain Eisenstein series}

\begin{prop}
\label{prop:at-least-one-holomorphic}
Notations are as in Sect.\,\ref{sec:simple-notation}. Let $\sigma \in \Coh(G_n, \mu)$ 
(resp., $\sigma' \in \Coh(G_{n'}, \mu')$) be a cohomological cuspidal automorphic representation of 
$G_n(\A)$ (resp., $G_{n'}(\A)$). Suppose $nn'$ is even. 
 Assume that the points $-N/2$ and $1-N/2$ are critical for the Rankin--Selberg $L$-function 
 $L(s, \sigma \times \sigma'^{\sf v})$. 
Let $f = \otimes f_v \in I_P^G(s, \sigma \times \sigma')$ and 
$\tilde f = \otimes \tilde f_v \in I_Q^G(w, \sigma' \times \sigma).$
Then: 
\begin{enumerate}
\item If $\tfrac{N}{2} + d - d' \leq 0$ then $\Eis_P(s,f)$ is holomorphic at $s = -N/2$, 
\smallskip
\item If $\tfrac{N}{2} + d - d' \geq 0$  then $\Eis_Q(w, \tilde f)$ is holomorphic at $w = N/2.$
\end{enumerate}
\end{prop}

\begin{proof}[Proof of Prop.\,\ref{prop:at-least-one-holomorphic}]
The proof needs a few lemmas: 
\ref{lem:eis-p-pole}, \ref{lem:eis-p-pole} and \ref{lem:local-rankin-selberg}. We first state these lemmas and then prove them.

\medskip
\paragraph{\bf Some lemmas}

\begin{lemma}
\label{lem:eis-p-pole}
If $\Eis_P(s,f)$ has a pole at $s = -N/2$ then $L^\place(1 - \tfrac{N}{2}, \sigma \times \sigma'^{\sf v}) = 0,$ where 
$\place = \place_\infty \cup \place_f$ is a finite set of places which is 
the union of the set $\place_\infty$ of all the archimedean places and the set $\place_f$ of all the finite places where either $\sigma$ or $\sigma'$ is ramified.
\end{lemma}

\begin{lemma}
\label{lem:eis-q-pole}
If $\Eis_Q(w, \tilde f)$ has a pole at $w = N/2$ then $L^\place(1 + \tfrac{N}{2}, \sigma' \times \sigma^{\sf v}) = 0,$ 
where $\place$ is as in the previous lemma. 
\end{lemma}

\begin{lemma}
\label{lem:local-rankin-selberg}
Let $v$ be a finite place of $F$, and let $\pi$ and $\pi'$ be irreducible admissible unitary generic representations of 
$\GL_n(F_v)$ and $\GL_{n'}(F_v)$, respectively. Suppose $L(s, \pi \times \pi')$ has a pole at $s = t_0 \in \R$, then 
$t_0 < 1.$ 
\end{lemma}

\medskip
\paragraph{\bf Proofs of Lemmas \ref{lem:eis-p-pole} and \ref{lem:eis-q-pole}}
\label{sec:proofs-using-local-const}
Recall (\ref{eqn:langlands-ratio-L-fns})
$$
\cF^Q (\Eis_P(s, f)) \ = \ 
\frac{L^\place(s, \sigma \times \sigma'^{\sf v})}{L^\place(s+1, \sigma \times \sigma'^{\sf v})} \, 
\otimes_{v \notin \place} \tilde f_v^0 \otimes_{v \in \place} 
T_{\rm st}(s, \sigma_v \otimes \sigma'_v)f_v. 
$$
The constant term map $\cF^Q$ is an integration over a compact space and will not affect the analytic discussion below. 
If $\Eis_P(s, f)$ has a pole at $s = -N/2$, then looking at the right hand side, either
\begin{enumerate}
\item $L^\place(s, \sigma \times \sigma'^{\sf v})/L^\place(s+1, \sigma \times \sigma'^{\sf v})$ has a pole at $-N/2$, or 
\item for some $v \in \place$, the local standard intertwining operator $T_{\rm st}(s, \sigma_v \otimes \sigma'_v)$ has a pole at $-N/2.$ 
\end{enumerate}

\smallskip

In case (1), since the numerator $L^\place(s, \sigma \times \sigma'^{\sf v})$ is entire, the only way the ratio can have a pole is if the denominator vanishes at that point, i.e., $L^\place(1-N/2, \sigma \times \sigma'^{\sf v}) = 0.$ 

\smallskip

In case (2), if $T_{\rm st}(s, \sigma_v \otimes \sigma'_v)$ has a pole at $s=-N/2$ of order $k_v \geq 0$, and if $k = \sum_{v \in \place} k_v$, then we claim that the numerator $L^\place(s, \sigma \times \sigma'^{\sf v})$ will have a zero 
of order at least $k$, i.e., all the poles of the local standard intertwining operators at bad places will cancel against zeros of the partial $L$-function appearing in the numerator; hence the only way $\Eis_P(s, f)$ can have a pole at $s=-N/2$ is 
when $L^\place(1-N/2, \sigma \times \sigma'^{\sf v}) = 0,$ which will prove the lemma. 

For the proof of the claim made above: suppose  $T_{\rm st}(s, \sigma_v \otimes \sigma'_v)$ has a pole
at $s=-N/2$ of order $k_v \geq 0.$ We appeal to Shahidi's results on local constants, which we briefly review here. 
Fix a nontrivial additive character $\psi_v$ of the base field and let $\lambda_{\psi_v}(s, \sigma_v \times \sigma_v')$ be the standard Whittaker functional (defined as a certain integral) on $I(s, \sigma_v \otimes \sigma'_v)$. We are using the fact here that cuspidal representations of $\GL_n$ are globally generic, and hence locally generic everywhere. We know that 
$\lambda_{\psi_v}(s, \sigma_v \times \sigma_v')$ extends to an entire function which is nonzero for all $v$; see 
Shahidi~\cite[Prop.\,3.1]{shahidi-duke80} for archimedean $v$ and \cite[Prop.\,3.1]{shahidi-ajm81} for finite $v$. By multiplicity one for Whittaker models there is a complex number $C_{\psi_v}(s, \sigma_v \times \sigma'_v)$, called a `local constant',  such that 
\begin{equation}
\label{eqn:local-constant}
C_{\psi_v}(s, \sigma_v \times \sigma'_v) \left(\lambda_{\psi_v}(-s, \sigma'_v \times \sigma_v) \circ 
T_{\rm st}(s, \sigma_v \otimes \sigma'_v) \right) \ = \ \lambda_{\psi_v}(s, \sigma_v \times \sigma_v').
\end{equation}
Furthermore, the local constant is related to local $L$- and $\varepsilon$-factors as: 
\begin{equation}
\label{eqn:local-constant-gamma}
C_{\psi_v}(s, \sigma_v \times \sigma'_v) \ = \ \varepsilon(s, \psi_v, \sigma_v \times \sigma'^{\sf v}_v) \,
\frac{L(1-s, \sigma_v^{\sf v} \times \sigma'_v)}{L(s, \sigma_v \times \sigma'^{\sf v}_v)}.
\end{equation}
See Shahidi~\cite[(5.1.4)]{shahidi-book}. Back to our proof of the above claim: suppose 
$T_{\rm st}(s, \sigma_v \otimes \sigma'_v)$ has a pole
at $s=-N/2$ of order $k_v \geq 0$, since the right hand side of (\ref{eqn:local-constant}) is holomorphic, the local constant $C_{\psi_v}(s, \sigma_v \times \sigma'_v)$ has a zero at $s=-N/2$ of order $k_v.$ The epsilon factor is an exponential function (hence entire and nonvanishing), and local $L$-factors are nowhere vanishing, hence the denominator of the right hand side of (\ref{eqn:local-constant-gamma}), i.e., $L(s, \sigma_v \times \sigma'^{\sf v}_v)$ has a pole 
at $s=-N/2$ of order $k_v$. We conclude that 
$L_\place(s, \sigma \times \sigma'^{\sf v}) = \prod_{v \in \place}L(s, \sigma_v \times \sigma_v'^{\sf v})$
has a pole at $s=-N/2$ of order $k = \sum_{v \in \place}k_v.$ But, 
$L(s, \sigma \times \sigma'^{\sf v}) = 
L_\place(s, \sigma \times \sigma'^{\sf v})L^\place(s, \sigma \times \sigma'^{\sf v})$
is entire (Thm.\,\ref{thm:rankin-selberg-analytic}), hence $L^\place(s, \sigma \times \sigma'^{\sf v})$ must have a zero 
at $s=-N/2$ of order at least $k$ to cancel against the pole at $s=-N/2$ of order $k$ of $L_\place(s, \sigma \times \sigma'^{\sf v})$--as claimed above.  
Proof of Lem.\,\ref{lem:eis-q-pole} is identical to the above proof of Lem.\,\ref{lem:eis-p-pole}.

\medskip
 \paragraph{\bf Proof of Lemma \ref{lem:local-rankin-selberg}}
 This is well-known and follows from the theory of zeta integrals. For example, in the case $n \neq n'$, see 
 Prop.\,6.2(i) and Prop.\,6.3  in Cogdell \cite{cogdell}. More generally, it is embedded in the work of 
 Jacquet, Piatetskii-Shapiro and Shalika \cite{jacquet-ps-shalika-AJM}.

\medskip
\paragraph{\bf Conclusion of proof of Prop.\,\ref{prop:at-least-one-holomorphic}}
To prove (1), 
suppose $-N/2+d-d' \leq 0,$ and if possible suppose also that ${\rm Eis}_P(s,f)$ has a pole at $s = -N/2.$ 
From Lem.\,\ref{lem:eis-p-pole} we know that 
$$
L^\place(1-\tfrac{N}{2}, \sigma \times \sigma'^{\sf v}) = 
L^\place(1-(\tfrac{N}{2}+d -d'),  {}^\circ\!\sigma \times {}^\circ\!\sigma'^{\sf v}) = 0, 
$$
where $\sigma = |\ |^{-d} \otimes {}^\circ\!\sigma$ with ${}^\circ\!\sigma$ being a unitary cuspidal representation; and 
similarly, $\sigma' = |\ |^{-d'} \otimes {}^\circ\!\sigma'.$
Since $N/2+d-d' \leq 0,$ from Lem.\,\ref{lem:local-rankin-selberg}, for any place $v$ we know that  
$L(1-(\tfrac{N}{2}+d -d'),  {}^\circ\!\sigma_v \times {}^\circ\!\sigma_v'^{\sf v})$ is finite. Hence, for the completed $L$-function we get $L(1-(\tfrac{N}{2}+d -d'),  {}^\circ\!\sigma \times {}^\circ\!\sigma'^{\sf v}) = 0$; 
but this is not possible by 
(4) of Thm.\,\ref{thm:rankin-selberg-analytic} since $1-(N/2+d-d') \geq 1;$ whence we conclude that 
${\rm Eis}_P(s,f)$ is finite at $s = -N/2.$ The proof of (2) is identical, and is left to the reader.
\end{proof}

\medskip
\subsubsection{\bf Conclusion of proof of Thm.\,\ref{thm:rank-one-eis}}
\label{sec:conclude-proof-nonzero-image}
We now prove the claim made in Sect.\,\ref{sec:langlands-constant-term}. The notations are now as in the 
statement of Thm.\,\ref{thm:rank-one-eis}, and we would like to show  
$$
\dim_E(\fI^b(\sigma_f, \sigma_f', \varepsilon')) \geq {\sf k} \quad {\rm and} \quad 
\dim_E(\fI^t(\sigma_f, \sigma_f', \varepsilon')^\v) \geq {\sf k}.
$$
The argument for top-degree is the same as that for the bottom degree, so we give the details of the proof only in bottom-degree.  
Take an embedding $\iota : E \to \C$ and pass to a transcendental level via $\iota.$ To show 
$\dim_E(\fI^b(\sigma_f, \sigma_f', \varepsilon')) \geq {\sf k},$ it suffices to show that its base-change to 
$\C$ via $\iota$ has dimension at least ${\sf k}$, i.e., it suffices to prove 
$$
\dim_\C(\fI^b({}^\iota\!\sigma_f, {}^\iota\!\sigma_f',  \varepsilon')) \geq {\sf k}, 
$$
where, for brevity, we have 
$$
\fI^b({}^\iota\!\sigma_f, {}^\iota\!\sigma_f',  \varepsilon') \ := \  
\fR_{{}^\iota\!\sigma_f, {}^\iota\!\sigma'_f, \varepsilon'}^b
(H^{b_N^F}_{\rm Eis}(\pBSC,\tM_{{}^\iota\!\lambda})^{K_f}).
$$

Given $\iota$, we know that ${}^\iota\!\sigma_f$ (resp., ${}^\iota\!\sigma_f'$) 
is the finite part of a cuspidal automorphic representation ${}^\iota\!\sigma$ (resp., ${}^\iota\!\sigma'$) of 
$G_n(\A)$ (resp., $G_{n'}(\A)$). 
Suppose $T^{PQ}_{\rm st}(s, {}^\iota\!\sigma \otimes {}^\iota\!\sigma')$ is holomorphic at $s = -N/2$, which is the same as saying that $\Eis_P(s,f)$ is holomorphic at $s = -N/2$ for an $f \in I_P^G(s, {}^\iota\!\sigma \otimes {}^\iota\!\sigma')$. 
Then, consider the following diagram: 

{\Small
\begin{equation}
\label{eqn:diagram-without-coh}
\xymatrix{
 & \cC^\infty\left(G(\Q)\backslash G(\A), \omega_\infty^{-1}\right) \ar[dd]^{\cF^P \oplus \cF^Q}  & \\
 & & \\
 \cC^\infty\left(P(\Q)\backslash G(\A), \omega_\infty^{-1}\right) \ar[uur]^{\Eis_P} \ar@{^{(}->}[r] \ &   
 \cC^\infty\left(P(\Q)\backslash G(\A), \omega_\infty^{-1}\right) \oplus 
 \cC^\infty\left(Q(\Q)\backslash G(\A), \omega_\infty^{-1}\right)
 & \ \cC^\infty\left(Q(\Q)\backslash G(\A), \omega_\infty^{-1}\right) \ar@{_{(}->}[l] \\
 I_P^G(-N/2, {}^\iota\!\sigma \otimes {}^\iota\!\sigma') \ar@{^{(}->}[u] 
 \ar[rr]^{T_{\rm st}(-N/2, {}^\iota\!\sigma \otimes {}^\iota\!\sigma')}  
 &  & I_Q^G(N/2, {}^\iota\!\sigma'  \otimes {}^\iota\!\sigma) \ar@{^{(}->}[u]
}
\end{equation}}

\noindent where, for brevity, we denote $\omega_\infty^{-1}$ for the inverse of the central character of $\M_{{}^\iota\!\lambda}$ restricted to 
$S(\R)^0.$ Tensor the above diagram by $\tM_{{}^\iota\!\lambda}$ and take cohomology, i.e., apply the functor 
$H^{b_N^F}(\g, K_\infty^0; -).$ Then take the $\varepsilon'$-isotypic component for the action of $\pi_0(K_\infty)$ 
and also take $K_f$-invariants to get: 

{\Small
\begin{equation}
\label{eqn:diagram-with-coh}
\xymatrix{
 & H^{b_N^F}( \SG,\tM_{{}^\iota\!\lambda})^{K_f} \ar[dd]^{(\cF^P \oplus \cF^Q)^*}  & \\
 & & \\
 H^{b_N^F}_{!!}(\ppBSC, \tM_{{}^\iota\!\lambda})^{K_f} \ar[uur]^{\Eis_P^*} \ar@{^{(}->}[r] \ &   
 H^{b_N^F}_{!!}(\ppBSC, \tM_{{}^\iota\!\lambda})^{K_f} \oplus 
 H^{b_N^F}_{!!}(\pqBSC, \tM_{{}^\iota\!\lambda})^{K_f}
 & \  H^{b_N^F}_{!!}(\pqBSC, \tM_{{}^\iota\!\lambda})^{K_f} \ar@{_{(}->}[l] \\
I^\place_b({}^\iota\!\sigma_f, {}^\iota\!\sigma_f', \varepsilon')_{P, {}^\iota\!w} \ar@{^{(}->}[u] 
 \ar[rr]^{T_{\rm st}(-N/2, {}^\iota\!\sigma \otimes {}^\iota\!\sigma', \varepsilon')^*}  
 &  & 
 I^\place_b({}^\iota\!\sigma_f, {}^\iota\!\sigma_f', \varepsilon')_{Q, {}^\iota\!w'} \ar@{^{(}->}[u]
}
\end{equation}}

\noindent 
Also the map $(\cF^P \oplus \cF^Q)^*$ is the same as 
the restriction $\r^*$ to boundary cohomology followed by a projection afforded by the Manin--Drinfeld principle. 
It is clear now that the 
required image $\fI^b({}^\iota\!\sigma_f, {}^\iota\!\sigma_f', \varepsilon')$ contains the image of 
$(\cF^P \oplus \cF^Q)^* \circ {\rm Eis}_P^*$, i.e., after applying Thm\,\ref{thm:langlands} we see that
\begin{equation}
\label{eqn:image-description-PQ}
\fI^b({}^\iota\!\sigma_f, {}^\iota\!\sigma_f', \varepsilon') \ \supset \ 
\left\{ \left(\xi \, , \, T_{\rm st}(-\tfrac{N}{2}, {}^\iota\!\sigma \otimes {}^\iota\!\sigma')^*\xi \right) \ : \ 
\xi \in I^\place_b({}^\iota\!\sigma_f, {}^\iota\!\sigma_f', \varepsilon')_{P, {}^\iota\!w}
 \right\}.
\end{equation}
Hence $\dim_\C(\fI^b({}^\iota\!\sigma_f, {}^\iota\!\sigma_f', \varepsilon')) \geq 
\dim_\C(I^\place_b({}^\iota\!\sigma_f, {}^\iota\!\sigma_f', \varepsilon')_{P, {}^\iota\!w}) = 
{\sf k}.$ 

\smallskip

Suppose $\Eis_P(s,f)$ has a pole at $s = -N/2$ then we start from the parabolic subgroup $Q$ while bearing in mind that necessarily $\Eis_Q(w,\tilde f)$ and $T_{\rm st}^{QP}(w, \sigma' \times \sigma)$ are holomorphic at $w=N/2$, by 
Prop.\,\ref{prop:at-least-one-holomorphic}.  In this case, we get 
\begin{equation}
\label{eqn:image-description-QP}
\fI^b({}^\iota\!\sigma_f, {}^\iota\!\sigma_f', \varepsilon') \ \supset \ 
\left\{ \left(T_{\rm st}(N/2, {}^\iota\!\sigma' \otimes {}^\iota\!\sigma)^*\tilde\xi  \, ,\, \tilde\xi \right) \ : \ 
\tilde\xi \in  I^\place_b({}^\iota\!\sigma_f, {}^\iota\!\sigma_f', \varepsilon')_{Q, {}^\iota\!w'}
 \right\}.
\end{equation}
Hence, $\dim_\C(\fI^b({}^\iota\!\sigma_f, {}^\iota\!\sigma_f', \varepsilon')) \geq 
\dim_\C(I^\place_b({}^\iota\!\sigma_f, {}^\iota\!\sigma_f', \varepsilon')_{Q, {}^\iota\!w'}) = 
{\sf k}.$ 
This concludes the proof of Thm.\,\ref{thm:rank-one-eis}.

\bigskip
\section{\bf $L$-functions}
\label{sec:l-function}

\medskip
\subsection{Cohomological $L$-functions and effective motives}
\label{sec:eff-motives-coh-l-fns}

There is a well-known conjectural dictionary between cohomological cuspidal automorphic representations of 
$\GL_n$ and pure rank $n$ motives. We briefly review this dictionary while recasting it in the the context of strongly inner Hecke summands on the one hand and pure effective motives on the other. In the discussion below,  
we consider the motives only via their Betti, de~Rham and $\ell$-adic realizations as in Deligne \cite{deligne}.

Let $\lambda \in X^*_0(T_n \times E)$ be a pure weight and $\sigma_f \in \Coh_{!!}(G_n, \lambda)$ be a strongly inner absolutely irreducible Hecke module. Recall that we have a Galois extension $E/\Q$ which contains a copy of $F$; furthermore, $\lambda = (\lambda^\tau)_{\tau : F \to E}$ and $E/\Q$ is taken large enough so that $\sigma_f$ is defined over $E$.
Then it is conjectured that we can attach a motive  $\Mot(\sigma_f)$ which is  defined over $F$ has coefficients in 
$E(\sigma_f)$. 
There is also an effective motive $\Mot_{\rm eff}(\sigma_f)$ which is a   Tate twist of $\Mot(\sigma_f)$. 
This Tate twist is chosen so that the eigenvalues of the Frobenius $\Phi_p^{-1}$ in any $\l-$adic realization
 $\Mot_{\rm eff}(\sigma_f)_{\acute et,\l}$ are algebraic integers. 
The motivic 
$L$-function attached to $\Mot_{\rm eff}(\sigma_f)$ is the same as a the  cohomological $L$-function attached to 
$\sigma_f.$

\medskip
\subsubsection{\bf Langlands parameters and Hodge pairs for effective motives}
Fix $\iota : E \to \C$, then ${}^\iota\!\sigma := {}^\iota\!\sigma_f \otimes {}^\iota\!\sigma_\infty$ 
is a cohomological cuspidal automorphic representation. Identify the sets $\Hom(F,\C) = \Hom(F,\R) = \place_\infty$; say, $v \in \place_\infty$ corresponds to $\nu \in \Hom(F,\C)$; we also write $v \mapsto \nu_v$ or 
$\nu \mapsto v_\nu.$ 
As in \ref{sec:change-field-E}, 
the map $\tau \mapsto \iota \circ \tau$ identifies $\Hom(F,E)$ with $\Hom(F,\C)$, and the weight 
${}^\iota\!\lambda \in X^*_0(T_n \times \C)$ is written as ${}^\iota\!\lambda = ({}^\iota\!\lambda^\nu)_{\nu:F \to \C}$ where 
${}^\iota\!\lambda^\nu = \lambda^{\iota^{-1}\circ \nu}.$  
The representation ${}^\iota\!\sigma_\infty$ factors as: 
$\D_{{}^\iota\!\lambda} = \otimes_{v \in \place_\infty}  \D_{{}^\iota\!\lambda^{\nu_v}}$ up to a signature character 
$\epsilon_\infty(\sigma_f)$ which is relevant when $n$ is odd. 
For each $\nu$, let $\varrho({}^\iota\!\lambda^\nu) := \varrho(\D_{{}^\iota\!\lambda^\nu})$ be the Langlands parameter of $\D_{{}^\iota\!\lambda^\nu}$, which is an  
$n$-dimensional semi-simple representation of the Weil group $W_\R$ of $\R$, as given by the local Langlands correspondence for $\GL_n(F_{v_\nu}) = \GL_n(\R);$ see Knapp \cite{knapp}. The restriction to 
$\C^\times \simeq W_\C \subset W_\R$ up to a half-integral twist (see (\ref{eqn:langlands-betti}) below) 
is given by: 
$$
|\ |^{\frac{(1-n)}{2}}\varrho({}^\iota\!\lambda^\nu)|_{\C^\times} \ = \ 
\bigoplus_{j=1}^n \, (z \mapsto z^{p_j^\tau}\bar z^{q_j^\tau}), 
\quad p_j^\tau = p_j({}^\iota\!\lambda^\nu), q_j^\tau = q_j({}^\iota\!\lambda^\nu).
$$
(Here $\tau = \iota^{-1} \circ \nu.$) When we restrict from $W_\R$ to $W_\C$ the signature $\epsilon_\infty(\sigma_f)$ disappears. One can read off the $p_j^\tau$ and $q_j^\tau$ 
from the cuspidal parameters; we have: 
\begin{equation}
\label{eqn:langlands-parameters}
p_j^\tau = \frac{\ell_j^\tau-2d+1-n}{2}, \quad q_j = \frac{-\ell_j^\tau-2d+1-n}{2}
\end{equation}
The parity condition in (\ref{eqn:parity-cuspidal-parameter}) gives that $p_j^\tau, q_j^\tau \in  \Z.$ 

\smallskip

On the other hand, for $\iota : E \to \C$, we have a rank $n$ motive ${}^\iota\Mot = \Mot({}^\iota\!\sigma)$ over $F$ with coefficients in $\iota(E) \subset \C.$ For $\nu : F \to \C$, the Betti realization 
$H_B(\nu, {}^\iota\Mot)$ is an $n$-dimensional $\iota(E)$-vector space together with a Hodge decomposition:
$$
H_B(\nu, {}^\iota\Mot) \otimes_{\iota(E)} \C \ = \ \bigoplus_{(p,q) \in \Hod(\nu, {}^\iota\Mot)} H^{p,q}(\nu, {}^\iota\Mot), 
$$
which is a representation of $W_\C = \C^\times$, where $z \in \C^\times$ acts on 
$H^{p,q}(\nu, {}^\iota\Mot)$ via $z^{-p}\bar z^{-q}.$ The set of Hodge pairs for ${}^\iota\Mot$ is 
$$
\Hod({}^\iota\Mot) \ = \ \bigcup_{\nu : F \to \C} \Hod(\nu, {}^\iota\Mot).
$$ 
The motive $\Mot(\sigma_f)$ is pure, i.e., there exists an integer $\tilde \w$ such that for any $\iota$  
if $(p,q) \in \Hod({}^\iota\Mot)$ then $p+q = \tilde \w.$ (If ${}^\iota\Mot$ is a piece of $H^\bullet(X, \iota(E))$ for a variety $X/F$,  then $H_B(\nu, {}^\iota\Mot)$ is the corresponding piece in $H^\bullet_{\rm Betti}(X \times_\nu \C, \iota(E)).$)

\smallskip

Under the conjectural correspondence $\sigma_f \leftrightarrow \Mot(\sigma_f)$, the exponents in the Langlands parameter, up to a specific half-integral twist, and the Hodge numbers determine each other: 
\begin{equation}
\label{eqn:langlands-betti}
|\ |^{\frac{(1-n)}{2}}\varrho({}^\iota\!\lambda^\nu)|_{\C^\times} \ \simeq \ 
H_B(\nu, {}^\iota\Mot) \otimes_{\iota(E)} \C,  
\end{equation}
where the isomorphism is as representations of $W_\C.$ Now (\ref{eqn:langlands-parameters}) and (\ref{eqn:langlands-betti}) give the Hodge pairs: 
{\small
\begin{equation}
\label{eqn:hodge-numbers-M}
\begin{split}
\Hod(\nu, {}^\iota\Mot) =  
\left\{\left(\frac{\ell_1^\tau+2d+n-1}{2}, \frac{-\ell_1^\tau+2d+n-1}{2} \right) \right., &  
\left(\frac{\ell_2^\tau+2d+n-1}{2}, \frac{-\ell_2^\tau+2d+n-1}{2} \right), \dots \\ 
& \dots, \left. \left(\frac{\ell_n^\tau+2d+n-1}{2}, \frac{-\ell_n^\tau+2d+n-1}{2} \right) \right\}. 
\end{split}
\end{equation}}

\noindent The purity weight $\tilde \w$ of $\Mot(\sigma_f)$ is given by 
$$
\tilde \w  \ = \ 2d + n-1.
$$

\smallskip

We will consider a suitable Tate twist of $\Mot(\sigma_f)$ to make it effective. 
(Suppose a motive $\Mot$ defined over $F$ with coefficients in $E$ is effective then for some 
$\iota : E \to \C$ and $\nu : F \to \C$ the Hodge pairs in $\Hod(\nu, {}^\iota\Mot)$ are of the form 
$\{(\w, 0), (\w-a_1, a_1), (\w-a_1-a_2, a_1+a_2), \dots, (0,\w)\},$ 
where $\w = a_1+\dots + a_{n-1}$ with $a_j \geq 1.$)  For $m \in \Z$, let $\Mot(m) := \Mot \otimes \Q(m)$ where 
$\Q(m)$ is a Tate motive; see, for example, \cite[\S\,2.1]{raghuram-riemann}. 

If $(p,q)$ is a Hodge pair for $\Mot$ then $(p-m,q-m)$ is a Hodge pair for $\Mot(m).$ Recall the definition of the motivic weight from \ref{sec:motivic-weight}: 
$\w_\lambda = {\rm max} \{\w_{\lambda^\tau} \ | \ \tau : F \to E\}$, where
$\w_{\lambda^\tau} = a_1^\tau+\dots +a_{n-1}^\tau = \ell_1^\tau.$ Define: 
\begin{equation}
\label{eqn:mot-eff-mot}
\Mot_\eff(\sigma_f) \ = \ \Mot(\sigma_f)\left(\frac{2d+n-1-\w_\lambda}{2}\right). 
\end{equation}
Suppose $\tau_0$ is such that $\w_\lambda = \w_{\lambda^{\tau_0}}.$ Take any $\iota$ and $\nu$ as above. 
The Hodge pairs in $\Hod(\nu, {}^\iota\Mot_{\rm eff}(\sigma_f))$ are of the form: 
$$
\left\{
\left( \frac{\ell_1^\tau+\ell_1^{\tau_0}}{2}, \frac{-\ell_1^\tau+\ell_1^{\tau_0}}{2} \right), 
\left( \frac{\ell_2^\tau+\ell_1^{\tau_0}}{2}, \frac{-\ell_2^\tau+\ell_1^{\tau_0}}{2} \right), \dots, 
\left( \frac{\ell_n^\tau+\ell_1^{\tau_0}}{2}, \frac{-\ell_n^\tau+\ell_1^{\tau_0}}{2} \right)
\right\}, 
$$ 
where $\tau = \iota^{-1}\circ \nu.$ In particular, for a pair $\iota_0$ and $\nu_0$ such that 
$\iota_0^{-1}\circ \nu_0 = \tau_0$ we get: 
$$
\Hod(\nu_0, {}^{\iota_0}\Mot_{\rm eff}(\sigma_f)) \ = \ 
\{(\w_\lambda, 0), (\w_\lambda-a_1^{\tau_0}, a_1^{\tau_0}), 
(\w-a_1^{\tau_0}-a_2^{\tau_0}, a_1^{\tau_0}+a_2^{\tau_0}), \dots, (0,\w_\lambda)\}. 
$$

\smallskip

Let's record a consequence of (\ref{eqn:langlands-betti}) for the behavior of the correspondences 
$\sigma_f \leftrightarrow \Mot(\sigma_f)$ and $\sigma_f \leftrightarrow \Mot_{\rm eff}(\sigma_f)$ under taking duals. 
Given $\sigma_f \in \Coh_{!!}(G_n, \lambda)$, the contragredient module $\sigma_f^{\sf v}$ is in 
$\Coh_{!!}(G_n, \lambda^{\sf v})$, where $\lambda^{\sf v} = -w_{G_n}(\lambda)$ is the dual weight. Further, the local Langlands correspondence commutes with taking contragredients. On the motivic side, given a motive $\Mot$ over $F$ with coefficients in $E$, the dual motive $\Mot^{\sf v}$ is also defined over $F$ with coefficients in $E$ whose realizations are the respective duals  of the realizations of $\Mot.$ In particular, if $(p,q)$ is a Hodge pair for $\Mot$ then $(-p,-q)$ is a Hodge pair for $\Mot^{\sf v}$. Putting these observations together with (\ref{eqn:langlands-betti}), we leave it to the reader to check that: 
\begin{equation}
\label{eqn:rep-mot-duals}
\Mot(\sigma_f^{\sf v}) \ = \ \Mot(\sigma_f)^{\sf v}(1-n).  
\end{equation}
For effective motives, from (\ref{eqn:mot-eff-mot}) and (\ref{eqn:rep-mot-duals}), we get: 
\begin{equation}
\label{eqn:rep-eff-mot-duals}
\Mot_{\rm eff}(\sigma_f^{\sf v}) \ = \ \Mot_{\rm eff}(\sigma_f)^{\sf v}(-\w_\lambda). 
\end{equation}
(The reader who wishes to check this will need to use that $d(\lambda^{\sf v}) = -d(\lambda)$ and 
$\w_{\lambda^{\sf v}} = \w_\lambda.$)

\medskip
\subsubsection{\bf Motivic $L$-functions} 

Let $\Mot$ be a motive defined over $F$ with coefficients in $E$. Its 
$\ell$-adic realization $H_\ell(\Mot)$ admits an action of the absolute Galois group of $F$, and for each $\iota : E \to \C$, 
we can take the associated Artin $L$-function. The coefficients appearing in the local Euler factors at all finite places 
initially take values in $E_{\mathfrak{l}}$ for $\mathfrak{l} | \ell$; the usual expectation of $\ell$-independence make them take values in $E$, and after applying $\iota$ we get a $\C$-valued local $L$-function. Taking the Euler product over all finite places, we get the finite part of the $L$-function which we denote as 
$L_f(\iota, \Mot, s).$ (See Deligne~\cite[2.2]{deligne} for a precise definition.) 

At the archimedean places there is a recipe given by Serre to define the local factors, which we will discuss in detail 
in \ref{sec:critical-points-comb-lemma}, and this gives the $L$-factor at infinity $L_\infty(\iota, \Mot, s)$ and 
$\L_\infty(\Mot, s).$ The completed $L$-function defined as: 
$$
L(\iota, \Mot, s) := L_\infty(\iota, \Mot, s)L_f(\iota, \Mot, s), 
$$
conjecturally extends to a meromorphic function to all of $\C$ and satisfies a functional equation of the form
$$
L(\iota, \Mot, s) \ = \ \varepsilon(\iota, \Mot, s)L(\iota, \Mot^{\sf v}, 1-s), 
$$
where the $\varepsilon$-factor on the right hand side is a nonzero constant times an exponential function. Taking the product over all $\iota$ we define $\L(\Mot, s)$ and it has a similar functional equation. All these comments will be applied to $L$-functions for $\Mot(\sigma_f)$ and $\Mot_{\rm eff}(\sigma_f)$.

\smallskip

On the other hand, for $\iota : E \to \C$ we have the standard automorphic $L$-function of the cuspidal automorphic representation ${}^\iota\!\sigma.$ This may be defined in terms of integral representations as, for example, in 
Jacquet~\cite{jacquet-corvallis}, or may be defined as Langlands $L$-functions as in 
\ref{sec:analytic-rankin-selberg}. It is known that one gets the same $L$-function irrespective of which definition is used. 
The product of local $L$-functions over all finite places is denoted  
$L^{\rm aut}_f(s, {}^\iota\!\sigma) = L^{\rm aut}(s, {}^\iota\!\sigma_f)$, and the product 
over all archimedean places giving the $L$-factor at infinity is denoted
$L^{\rm aut}_\infty(s, {}^\iota\!\sigma) = L^{\rm aut}(s, {}^\iota\!\sigma_\infty).$ The completed $L$-function is defined as: 
$$
L^{\rm aut}(s, {}^\iota\!\sigma) \ = \ L^{\rm aut}_\infty(s, {}^\iota\!\sigma)L^{\rm aut}_f(s, {}^\iota\!\sigma). 
$$
For automorphic $L$-functions it is known that $L^{\rm aut}(s, {}^\iota\!\sigma)$ admits an analytic continuation to an entire function (unless $n=1$ and $\sigma_f$ is the trivial representation whence the $L$-function is the Dedekind zeta function of $F$ which has a pole at $s=1$), and satisfies a functional equation: 
$$
L^{\rm aut}(s, {}^\iota\!\sigma) \ = \ \varepsilon(s, {}^\iota\!\sigma) L^{\rm aut}(1-s, {}^\iota\!\sigma^{\sf v}).
$$ 
As with motivic $L$-functions, take the product over all $\iota$ to define: 
$$
\L^{\rm aut}_f(s, \sigma) \ = \ \prod_{\iota: E \to \C} L^{\rm aut}_f(s, {}^\iota\!\sigma).
$$
Similarly, $\L^{\rm aut}_\infty(s, \sigma)$ and 
$\L^{\rm aut}(s, \sigma) := \L^{\rm aut}_\infty(s, \sigma) \L^{\rm aut}_f(s, \sigma).$

\smallskip

The key point in the expected comparison between the motivic $L$-functions and automorphic $L$-function is that
there is an analogue at finite places of (\ref{eqn:langlands-betti}), i.e., for any rational prime $p$ and a prime ideal $\p$ of $F$ above $p$, the Langlands parameter $\varrho({}^\iota\!\sigma_\p)$ (which is an $n$-dimensional semi-simple representation of the Weil group $W_{F_\p}$ of $F_\p$) is related to the local Galois representation at $\p$ on the $\ell$-adic realization $H_\ell({}^\iota\Mot(\sigma_f))$ via: 
\begin{equation}
\label{eqn:langlands-ell}
|\ |^{\frac{(1-n)}{2}}\varrho({}^\iota\!\sigma_\p) \ \simeq \ 
H_\ell({}^\iota\Mot(\sigma_f)). 
\end{equation}
Taking the corresponding Artin $L$-factors on both sides and then taking a product over all $\p$  gives: 
\begin{equation}
\label{eqn:motivic-auto-l-fn}
L\left(\iota, \Mot(\sigma_f), s\right) \ = \ L^{\rm aut}\left(s + \tfrac{1-n}{2}, {}^\iota\sigma_f\right).
\end{equation}
Using (\ref{eqn:mot-eff-mot}) we get: 
\begin{equation}
\label{eqn:eff-motivic-auto-l-fn}
L\left(\iota, \Mot_{\rm eff}(\sigma_f), s\right) \ = \ L^{\rm aut}\left(s + \tfrac{2d-\w}{2}, {}^\iota\!\sigma_f\right).
\end{equation}
We have similar relations for $\L$-functions by taking a product over all $\iota.$ The shift $(2d-\w)/2$ in the $s$-variable can be seen by working intrinsically in the context of the cohomology of arithmetic groups by considering a cohomological $L$-function attached to $\sigma_f$ which we now discuss.

\medskip
\subsubsection{\bf Cohomological $L$-functions}
We briefly review certain characteristic properties of the cohomological $L$-functions while referring the reader to Harder~\cite[Chap.\,3]{harder-book} for the precise definition and further details. 

For simplicity we deal with the special case of 
$F = \Q.$ Given $\sigma_f \in \Coh_{!!}(\GL_n/\Q, \lambda)$, let $p$ be a rational prime where $\sigma_f$ is unramified. 
The weight $\lambda$ gives an integral sheaf $\M_{\lambda, \Z}$ and we consider the integral cohomology by which we mean the image of the cohomology of $\M_{\lambda, \Z}$ in the cohomology of $\M_{\lambda}.$
Let $\chi : \bG_m \to T$ be a dominant co-character also thought of as a dominant character of the dual torus $T^\vee.$ 
The characteristic function of $\GL_n(\Z_p) \chi(p) \GL_n(\Z_p)$, denoted ${\bf ch}(\chi),$ is an element of the local Hecke algebra which does not in general preserve integral cohomology, however, 
$$
p^{\langle \chi, \lambda^{(1)} \rangle - \langle \chi, \lambda_{\rm ab} \rangle} {\bf ch}(\chi)
$$ 
does indeed preserve the integral cohomology; the exponent of $p$ appearing above is optimal for this purpose. 
Let $r_\chi$ be the algebraic irreducible representation of the dual group 
$G^\vee = \GL_n(\C)$ with highest weight $\chi$. Let $\vartheta_p \in T^\vee \subset G^\vee$ be the Satake parameter of $\sigma_p.$ There is a well-known expression for $r_\chi(\vartheta_p)$ in terms of dominant characters $\chi' < \chi$ of $T^\vee$ 
from which it follows that the polynomial appearing in the denominator of an Euler factor at $p$ can be normalized so that we see only integral coefficients: 
\begin{equation}
\det({\rm Id} - p^{c(\chi,\lambda)} r_\chi(\vartheta_p)X) \in \ringO_E[X], \quad {\rm where} \quad
c(\chi, \lambda) \ := \ \langle \chi, \lambda^{(1)} \rangle - \langle \chi, \lambda_{\rm ab} \rangle + 
\langle \chi, \bfgreek{rho}_n \rangle. 
\end{equation}
One defines the local Euler factor at $p$ for the cohomological $L$-function to be the formal expression: 
\begin{equation}
L^{\rm coh}_p(\sigma_f, r_\chi, s) \ = \ 
\frac{1}{\det({\rm Id} - p^{c(\chi,\lambda)} r_\chi(\vartheta_p)p^{-s})}.
\end{equation}
Any $\iota : E \to \C$ induces an embedding $\iota : \ringO_E[X] \hookrightarrow \C[X]$ giving an honest-to-goodness complex valued function: 
\begin{equation}
L^{\rm coh}_p(\iota, \sigma_f, r_\chi, s) \ = \ 
\frac{1}{\iota(\det({\rm Id} - p^{c(\chi,\lambda)} r_\chi(\vartheta_p)X))|_{X = p^{-s}}}.
\end{equation}

We now explicate the relation between the cohomological $L$-function and the automorphic $L$-function. 
We are interested in only the standard $L$-function of $\GL_n$ which corresponds to the standard representation of 
the dual group $\GL_n(\C)$ whose highest weight is $\e_1$ (see  \ref{sec:standard-fundamental}). 
It is easy to see that
\begin{equation}
\label{eqn:c-e1-lambda}
c(\e_1, \lambda)  
\ = \ \langle \e_1, \lambda^{(1)} \rangle - \langle \e_1, \lambda_{\rm ab} \rangle + \langle \e_1, \bfgreek{rho}_n \rangle  
\ = \  \frac{\w+1-n}{2} - d + \frac{n-1}{2}  
\ = \ \frac{\w - 2d}{2}.
\end{equation}
For the standard $L$-function of $\GL_n$, in the notations for the data going in to defining the $L$-function, 
we drop the representation of the dual group $r_\chi = r_{\e_1}$. From (\ref{eqn:c-e1-lambda}) we get that 
the entries in the diagonal matrix 
$$
p^{\tfrac{\w - 2d}{2}} \vartheta(\sigma_p) 
$$
are in $\ringO_E.$ We deduce: 
$$
p^{\tfrac{n-1}{2}}  \vartheta(\sigma_p) \ = \  p^{\tfrac{n-1-\w+2d}{2}} p^{\tfrac{\w - 2d}{2}}\vartheta(\sigma_p) \ 
\in \ p^{\Z} \ringO_E \ \subset \ E. 
$$ 
Hence it makes sense to apply the embedding $\iota : E \to \C$ to $p^{(n-1)/2}  \vartheta(\sigma_p).$ 
Next, it is well-known that taking the Satake parameter of an unramified (generic) representation $\sigma_p$ is rational up to a half-integral twist stemming from the fact that such normalized parabolic induction is not a rational operation due to the presence of $\bfgreek{rho}_n$ (see, for example, Gross~\cite{gross}); we have: 
$$
\iota \left(p^{\tfrac{n-1}{2}}  \vartheta(\sigma_p)\right) \ = \  p^{\tfrac{n-1}{2}} \vartheta({}^\iota\!\sigma_p). 
$$
It follows that 
\begin{equation*}
\begin{split}
\iota\left(p^{\tfrac{\w-2d}{2}}  \vartheta(\sigma_p) \right) 
& \ = \ \iota\left(p^{\tfrac{\w-2d+1-n}{2}} p^{\tfrac{n-1}{2}}   \vartheta(\sigma_p) \right) 
\ = \ p^{\tfrac{\w-2d+1-n}{2}}  \iota\left(p^{\tfrac{n-1}{2}}   \vartheta(\sigma_p) \right) \\
& \ = \ 
p^{\tfrac{\w-2d+1-n}{2}}  p^{\tfrac{n-1}{2}}   \vartheta({}^\iota\!\sigma_p)  
\ = \ p^{\tfrac{\w-2d}{2}} \vartheta({}^\iota\!\sigma_p). 
\end{split}
\end{equation*}
We then get the following relation between 
the standard cohomological $L$-function and the standard automorphic $L$-function: 
\begin{equation}
\label{eqn:coh-aut-l-fn-p}
\begin{split}
L^{\rm coh}_p(\iota, \sigma_f, s) 
& = \ \det\left({\rm Id} - \iota\left(p^{\tfrac{\w-2d}{2}}  \vartheta(\sigma_p)\right) p^{-s}\right)^{-1} \\
& =  \ \det\left({\rm Id} - p^{\tfrac{\w-2d}{2}}  \vartheta({}^\iota\sigma_p) p^{-s}\right)^{-1} \\
& =  \ \det\left({\rm Id} - \vartheta({}^\iota\sigma_p) p^{-(s + \tfrac{2d-\w}{2})} \right)^{-1} \\ 
& = \ L^{\rm aut}_p \left(s + \tfrac{2d - \w}{2}, {}^\iota\!\sigma_f \right). 
\end{split}
\end{equation}
When $\sigma_f$ is ramified at $p$, we {\it define} the cohomological $L$-function by (\ref{eqn:coh-aut-l-fn-p}), i.e., 
$$
L^{\rm coh}_p(\iota, \sigma_f, s) \ := \ L^{\rm aut}_p \left(s + \tfrac{2d - \w}{2}, {}^\iota\!\sigma_f \right),
\quad \mbox{when $\sigma_p$ is ramified}.
$$
Similarly, at any archimedean place $v \in \place_\infty$ we define
$$
L^{\rm coh}_v(\iota, \sigma, s) \ := \ L^{\rm aut}_v \left(s + \tfrac{2d - \w}{2}, {}^\iota\!\sigma_v \right). 
$$

All the local factors are pieced together to deduce the following relation between $L$-functions: 
\begin{equation}
\label{eqn:coh-aut-l-fns-finite}
L^{\rm coh}_f(\iota, \sigma_f, s)  \ = \  L^{\rm aut} \left(s + \tfrac{2d - \w}{2}, {}^\iota\!\sigma_f \right). 
\end{equation}
Similarly, we get the following relation between the completed $L$-functions: 
\begin{equation}
\label{eqn:coh-aut-l-fns-completed}
L^{\rm coh}(\iota, \sigma, s)  \ = \  L^{\rm aut} \left(s + \tfrac{2d - \w}{2}, {}^\iota\!\sigma \right). 
\end{equation}

From (\ref{eqn:eff-motivic-auto-l-fn}) and (\ref{eqn:coh-aut-l-fns-finite}) we get that the cohomological $L$-function attached to $\sigma_f$ and the motivic $L$-function attached to $\Mot_{\rm eff}(\sigma_f)$ coincide: 
\begin{equation}
\label{eqn:coh-eff-mot-l-fns-finite}
L^{\rm coh}_f(\iota, \sigma_f, s)  \ = \ 
L\left(\iota, \Mot_{\rm eff}(\sigma_f), s\right). 
\end{equation}
Similarly, for the completed $L$-functions.

\smallskip

A fundamental property about the cohomological $L$-function is that it is invariant under Tate twists, i.e., 
suppose $\sigma_f \in \Coh_{!!}(G_n, \lambda)$ and $m \in \Z$, then 
$\sigma_f(m) \in \Coh_{!!}(G_n, \lambda - m \delta_n) .$ and 
\begin{equation}
\label{eqn:coh-l-fn-tate-twist}
L^{\rm coh}_f(\iota, \sigma_f(m), s)  \ = \  L^{\rm coh}_f(\iota, \sigma_f, s). 
\end{equation}
This may be seen using (\ref{eqn:coh-aut-l-fns-finite}) after noting that 
$d(\lambda - m \delta_n) = d(\lambda) - m$ and 
$\w_{\lambda - m {\delta}_n} = \w_{\lambda}.$ 

\smallskip

Finally, let's comment that the above discussion of cohomological $L$-functions was when the base field was $\Q$, however, everything goes through {\it mutatis mutandis} for a totally real field $F.$  In particular, 
(\ref{eqn:coh-aut-l-fns-finite}), (\ref{eqn:coh-aut-l-fns-completed}) and (\ref{eqn:coh-eff-mot-l-fns-finite}) are true for a totally real $F.$

\medskip
\subsection{Critical points and the combinatorial lemma}
\label{sec:critical-points-comb-lemma}

In this subsection we write down the set of critical points of Rankin--Selberg $L$-functions for $G_n \times G_{n'}.$  
The discussion about critical points may be carried out by working entirely in the automorphic set-up, or granting the dictionary between motives and automorphic representations, the same discussion may be carried out 
working entirely in the motivic world. We present both the approaches. We also see that the condition on highest 
weights about the existence of a balanced Kostant representative, which guaranteed that a certain induced 
representation appeared in boundary cohomology in a special degree, turns out to be exactly equivalent to the 
condition that the $L$-values which intervene in rank-one Eisenstein cohomology to be critical values; 
this we state as a {\it combinatorial lemma}.

\medskip
\subsubsection{\bf The motivic $\Gamma$-factors at infinity and the critical set} 

We briefly present the calculation of the critical set of motivic  Rankin--Selberg $L$-functions.

\medskip
\paragraph{\bf Serre's $\Gamma$-factors} 
Suppose $\Mot$ is a pure regular rank $n$ motive defined over $F$ with coefficients in $E.$ Let $\w = \w(\Mot)$ be the purity weight of $\Mot,$ i.e., for any Hodge type $(p,q)$ of $\Mot$ we have $p+q = \w.$ By regularity we mean that any nonzero Hodge number is $1.$ Fix $\iota : E \to \C.$ For any $\nu : F \to \C$ suppose the Hodge types are:
$$
\Hod(\nu, {}^\iota\Mot) \ = \ \{(p_1^\tau,q_1^\tau), \dots, (p_n^\tau,q_n^\tau) \}, \quad (\tau = \iota^{-1}\circ \nu : F \to E).
$$
Further, we suppose that there is no middle Hodge type, i.e., for every $\tau$ and $j$, $p_j^\tau \neq \w/2.$ In this situation, the $\Gamma$-factor at the infinite place $v_\nu$ corresponding to $\nu$ is given by:
$$
L_{v_\nu}(\iota, \Mot, s) \ = \ 
\prod_{p_j^\tau < \tfrac{\w}{2}} \Gamma_\C(s - p_j^\tau).
$$
See Serre~\cite[(25)]{serre}; except that for us $\Gamma_\C(s) = 2(2\pi)^{-s}\Gamma(s).$ 
Taking the product over all the archimedean places, we get: 
\begin{equation}
\label{eqn:serre-gamma}
L_\infty(\iota, \Mot, s) \ = \ 
\prod_{\nu : F \to \C} 
\prod_{p_j^\tau < \tfrac{\w}{2}} \Gamma_\C(s - p_j^\tau) \ = \ 
\prod_{\tau : F \to E} 
\prod_{p_j^\tau < \tfrac{\w}{2}} \Gamma_\C(s - p_j^\tau). 
\end{equation}

\medskip
\paragraph{\bf $\Gamma$-factors and functional equations} 
Continuing with the notations above, supposing $\Mot = h^i(X, E)(m)$ by which we mean cohomology in degree $i$ with coefficients in $E$ of a smooth projective variety $X$ over $F$ together with a Tate twist by an integer $m.$ 
Then $\Mot$ is pure of weight 
$\w = i - 2m.$ By Poincar\'e duality and the hard Lefschetz theorem, which are known to be true in all the three realizations, we get: $\Mot^{\sf v}(1) \ = \ \Mot(\w+1)$; see, for example, Nekov\'ar \cite[(1.3.1)]{nekovar}. This kind of an essential self-duality of $\Mot$ is however not true for a general pure motive: for example, it is known by Ramakrishnan and Wang \cite{ramakrishnan-wang} 
that there is a cohomological cuspidal representation of $\GL_6/\Q$ which is not essentially self-dual; hence the corresponding motive could not possibly satisfy the above relation. However, the Betti realization does satisfy an analogous relation which may be checked using the Hodge types, and in particular, we have: 
$$
L_\infty(\iota, \Mot^{\sf v}, 1-s) \ = \ 
L_\infty(\iota, \Mot^{\sf v}(1), -s) \ = \ 
L_\infty(\iota, \Mot(\w+1)-s) \ = \ 
L_\infty(\iota, \Mot, \w+1-s). 
$$
The functional equation may be written as: 
$$
L_\infty(\iota, \Mot, s)L_f(\iota, \Mot, s) \ = \ 
\varepsilon(\iota, \Mot, s)  L_\infty(\iota, \Mot, \w+1-s) L_f(\iota, \Mot^{\sf v}, 1-s). 
$$

\medskip
\paragraph{\bf The critical set for motivic $L$-functions}
\label{sec:critical-mot-l-fn}
An integer $m$ is said to be critical for $L(\iota, \Mot, s)$ if the $\Gamma$-factors on either side of the functional equation are regular, i.e., do not have poles at $s=m.$ This means that both 
$$
L_\infty(\iota, \Mot, s) \quad {\rm and} \quad L_\infty(\iota, \Mot, \w+1-s)
$$
are regular at $s = m.$ Define
\begin{equation}
\label{eqn:p0-q0}
p_{\rm max} \ := \ {\rm max}\{p_j^\tau : p_j^\tau < \w/2\}, \ \ {\rm and} \ \ 
q_{\rm min} = \w - p_{\rm max} = {\rm min}\{q_j^\tau : q_j^\tau > \w/2\}. 
\end{equation}
It is an easy exercise using (\ref{eqn:serre-gamma}) to see that 
\begin{equation}
\label{eqn:critical-set-mot-l-fn}
\begin{split}
\mbox{The critical set for $L(\iota, \Mot, s)$} & = \ \{ m \in \Z : p_{\rm max} < m \leq q_{\rm min}\} \\
& = \ \{p_{\rm max}+1, p_{\rm max}+2, \dots, q_{\rm min}-1, q_{\rm min}\}.
\end{split}
\end{equation}
The critical set is centered around $(\w+1)/2$ and has $q_{\rm min} - p_{\rm max}$ integers. The assumption that there is no middle Hodge type assures us that the critical set is nonempty. The length $q_{\rm min}-p_{\rm max}$ of the critical set may also be interpreted as the smallest possible positive difference between the holomorphic and anti-holomorphic parts of the Hodge types, i.e., 
$$
q_{\rm min} - p_{\rm max} = {\rm min}\{\ |p-q| \ : \ (p,q) \in \Hod({}^\iota\Mot)\}.
$$
It follows from (\ref{eqn:p0-q0}) and (\ref{eqn:critical-set-mot-l-fn}) that the critical set for $L(\iota, \Mot, s)$ is independent of $\iota.$

\medskip
\paragraph{\bf The critical set for $L(\iota, \Mot_{\eff}(\sigma_f) \times \Mot_{\eff}(\sigma_f'^{\sf v}), s)$}
Let $\sigma_f \in \Coh_{!!}(G_n, \mu)$ and $\sigma_f' \in \Coh_{!!}(G_{n'}, \mu')$, with  
$\mu = (\mu^\tau)_{\tau : F \to E}$ and $\mu' = (\mu'^\tau)_{\tau : F \to E}.$ Then 
$\sigma_f'^{\sf v} \in \Coh_{!!}(G_{n'}, \mu'^{\sf v}).$ Let the motivic weights be $\w = \w(\mu)$ and 
$\w' = \w(\mu') = \w(\mu'^{\sf v}).$ Similarly, let the cuspidal parameters for $\mu$ be 
$\ell = (\ell^\tau)_{\tau : F \to E}$ with $\ell^\tau = (\ell_i^\tau)_{1 \leq i \leq n}.$ The cuspidal parameters for 
$\mu'$ are the same as the cuspidal parameters for dual weight $\mu'^{\sf v}$ and these we denote by 
$\ell' = (\ell'{}^{\tau})_{\tau : F \to E}$ with $\ell'^\tau = (\ell_j'\!{}^\tau)_{1 \leq j \leq n'}.$ For $\iota : E \to \C$ and 
$\nu : F \to \C$, the Hodge types are given by: 
$$
\Hod(\nu, {}^\iota\Mot_{\rm eff}(\sigma_f)) \ = \ 
\left\{ \left(\frac{\ell_i^\tau + \w}{2}, \frac{-\ell_i^\tau + \w}{2}\right)_{1 \leq i \leq n} \right\}. 
$$
Similarly, 
$$
\Hod(\nu, {}^\iota\Mot_{\rm eff}(\sigma_f'^{\sf v})) \ = \ 
\left\{ \left(\frac{\ell_j'\!{}^\tau + \w'}{2}, \frac{-\ell_j'\!{}^\tau + \w'}{2}\right)_{1 \leq j \leq n'} \right\}. 
$$
The Hodge types of the tensor product motive are given by: 
\begin{equation}
\label{eqn:hodge-tensor-product-eff-mot}
\Hod(\nu, {}^\iota\Mot_{\rm eff}(\sigma_f) \otimes {}^\iota\Mot_{\rm eff}(\sigma_f'^{\sf v})) \ = \ 
\left\{ \left(\frac{\ell_i^\tau + \ell_j'\!{}^\tau + \w + \w'}{2}, 
\frac{-\ell_i^\tau - \ell_j'\!{}^\tau+ \w + \w'}{2}\right)_{\substack{1 \leq i \leq n \\ 1\leq j \leq n'}} \right\}. 
\end{equation}
We would like to see under what condition there is no middle Hodge type in such a tensor product motive. This begets the 
\begin{defn}
\label{def:cuspidal-width}
We say that $\sigma_f \in \Coh_{!!}(G_n, \mu)$ and $\sigma_f' \in \Coh_{!!}(G_{n'}, \mu')$ have disjoint cuspidal parameters if $\ell_i^\tau \neq \ell_j'\!{}^\tau$ for all $1 \leq i \leq n$, $1\leq j \leq n'$ and $\tau : F \to E.$ Also, we define the cuspidal width between $\sigma_f$ and $\sigma_f'$ to be 
$$
\ell(\sigma_f, \sigma_f') = \ell(\mu, \mu') = 
{\rm min}\{\ |\ell_i^\tau - \ell_j'\!{}^\tau| \ : \  1 \leq i \leq n,\ 1\leq j \leq n',\ {\rm and}\ \tau : F \to E \}.
$$
\end{defn}
Note that the cuspidal parameters are disjoint if and only if $\ell(\sigma_f, \sigma_f') > 0.$ Also, the cuspidal width depends only on the semi-simple parts of the weights, and furthermore, since the cuspidal parameters of a weight and its dual weight are equal, we deduce: 
$$
\ell(\mu, \mu') \ = \ \ell(\mu^{(1)}, \mu'^{(1)}) \ = \ \ell(\mu, \mu'^{\sf v}), \ \mbox{etc.} 
$$

\begin{prop}
\label{prop:disjoint-param-middle-hodge}
The tensor product motive ${}^\iota\Mot_{\rm eff}(\sigma_f) \otimes {}^\iota\Mot_{\rm eff}(\sigma_f'^{\sf v})$ has no middle Hodge type if and only if $\sigma_f$ and $\sigma_f'$ have disjoint cuspidal parameters. Furthermore, cuspidal parameters being disjoint implies that  $nn'$ is even. 
\end{prop}

\begin{proof}
The weight of the tensor product motive is $\w + \w'$; hence there is a middle Hodge type if and only if 
$\ell_i^\tau + \ell_j'\!{}^\tau = 0$ or that $\ell_j'\!{}^\tau = \ell_{n-i+1}^\tau$  
for some $i,j,\tau.$ 
Suppose $n$ and $n'$ are both odd, then for each $\tau$, we have $\ell_i^\tau = 0$ for 
$i = (n+1)/2$ and $\ell_j'\!{}^\tau = 0$ for $j = (n'+1)/2$; hence $\ell(\sigma_f, \sigma_f') = 0.$ 
\end{proof}

For the next proposition, let's note the parity condition: $\ell(\mu, \mu') \equiv  \w + \w' \pmod{2}.$ 

\begin{prop}
\label{prop:critical-set-tensor-product-eff-mot}
Suppose $\sigma_f \in \Coh_{!!}(G_n, \mu)$ and $\sigma_f' \in \Coh_{!!}(G_{n'}, \mu')$ have disjoint cuspidal parameters. 
Then the critical set for $L(\iota, \Mot_{\eff}(\sigma_f) \times \Mot_{\eff}(\sigma_f'^{\sf v}), s)$ is the finite set: 
$$
\left\{ m \in \Z \ : \ 
\frac{-\ell(\mu, \mu') + \w + \w'}{2} < m \leq \frac{\ell(\mu, \mu') + \w + \w'}{2} \right\}.
$$
The critical set is centered around $(1+\w+\w')/2$ and the number of critical points is the cuspidal width 
$\ell(\mu, \mu').$  
\end{prop}

\begin{proof}
From (\ref{eqn:hodge-tensor-product-eff-mot}) it follows that 
$p_{\rm max} = (-\ell(\mu, \mu') + \w + \w')/2$ and $q_{\rm min} = (\ell(\mu, \mu') + \w + \w')/2.$ The rest is clear from the discussion in \ref{sec:critical-mot-l-fn}.
\end{proof}

\medskip
\subsubsection{\bf The automorphic $\Gamma$-factors at infinity and the critical set} 

We briefly present the calculation of the critical set of automorphic  Rankin--Selberg $L$-functions.

\medskip
\paragraph{\bf The Rankin--Selberg $\Gamma$-factors}
Suppose $\sigma_f \in \Coh_{!!}(G_n, \mu)$ and $\sigma_f' \in \Coh_{!!}(G_{n'}, \mu').$ 
Assume that $n$ is even. For $\iota : E \to \C$, 
we know from Thm.\,\ref{thm:cuspidal-cohomology} 
that the representations at infinity of ${}^\iota\!\sigma$ (resp., ${}^\iota\!\sigma'$) are given by 
$\D_{{}^\iota\!\mu}$ (resp., $\D_{{}^\iota\!\mu'}$ up to a signature character when $n'$ is odd). For each archimedean place $v$ and corresponding 
$\nu_v : F \to \C$ we know the local archimedean Langlands parameters
$\varrho({}^\iota\!\sigma_v) = \varrho({}^\iota\!\mu^{\nu_v})$ (resp., $\varrho({}^\iota\!\sigma'_v) = 
\varrho({}^\iota\!\mu^{\nu_v})$). For notational brevity, for any integer $k$, let's denote by $I(k)$ the $2$-dimensional irreducible representation of $W_\R$ obtained by inducing the unitary character 
$z \mapsto (z/\bar z)^{k/2}$, i.e., the character $r e^{i\theta} \mapsto e^{i k\theta}$, of $\C^\times = W_\C.$ 
Using the local Langlands correspondence for $\GL_n(\R)$ and $\GL_{n'}(\R)$ we have: 

\begin{equation}
\label{eqn:langlands-parameters-infinity}
\varrho({}^\iota\!\sigma_v) \ = \ \bigoplus_{i = 1}^{n/2} I(\ell_i^\tau) \otimes |\ |^{-d}, \quad {\rm and} \quad 
\varrho({}^\iota\!\sigma_v'^{\sf v}) \ = \ \left\{
\begin{array}{ll}
\bigoplus_{j = 1}^{(n'-1)/2} I(\ell_j'\!{}^\tau) \otimes |\ |^{d'} \, 
\oplus \, {\rm sgn}^{\epsilon'_v} |\ |^{d'}, & \mbox{if $n'$ is odd}, \\
& \\
\bigoplus_{j = 1}^{n'/2} I(\ell_j'\!{}^\tau) \otimes |\ |^{d'}, & \mbox{if $n'$ is even.}
\end{array}\right.
\end{equation}
(Here $\iota \circ \tau = \nu_v.$) The following properties are readily checked: 
(i) $I(k)$ is irreducible for $k \neq 0$, and $I(0) = 1\!\!1 \oplus {\rm sgn};$  
(ii) $I(k) \simeq I(-k);$ 
(iii) $I(k) \otimes {\rm sgn} \simeq I(k);$  and 
(iv) $I(k) \otimes I(k') = I(k+k') \oplus I(|k-k'|).$

Under the assumption that $\sigma_f$ and $\sigma_f'$ have disjoint cuspidal parameters, if $n'$ is odd then 
\begin{equation}
\label{eqn:langlands-parameters-tensor-product-odd}
\varrho({}^\iota\!\sigma_v) \otimes \varrho({}^\iota\!\sigma_v'^{\sf v}) \ = \ 
\bigoplus_{i = 1}^{n/2} I(\ell_i^\tau) \otimes |\ |^{d'-d} \ 
\bigoplus_{\substack{1 \leq i \leq n/2 \\ 1 \leq j \leq (n'-1)/2}} 
\left(I(\ell_i^\tau + \ell_j'\!{}^\tau) \otimes |\ |^{d'-d} \oplus I(|\ell_i^\tau-\ell_j'\!{}^\tau|) \otimes |\ |^{d'-d} \right), 
\end{equation}
and if $n'$ is even then 
\begin{equation}
\label{eqn:langlands-parameters-tensor-product-even}
\varrho({}^\iota\!\sigma_v) \otimes \varrho({}^\iota\!\sigma_v'^{\sf v}) \ = \ 
\bigoplus_{\substack{1 \leq i \leq n/2 \\ 1 \leq j \leq n'/2}} 
\left(I(\ell_i^\tau + \ell_j'\!{}^\tau) \otimes |\ |^{d'-d} \oplus I(|\ell_i^\tau-\ell_j'\!{}^\tau|) \otimes |\ |^{d'-d} \right). 
\end{equation}

From the matching up of local $L$-factors on either side of the local Langlands correspondence 
(see Knapp~\cite[(3.6), (3.8)]{knapp}), if $n'$ is odd then from (\ref{eqn:langlands-parameters-tensor-product-odd}) we get 
 \begin{equation}
 \label{eqn:L-factor-infinity-automorphic-odd}
 \begin{split}
 L_v(s, \iota, \sigma \times \sigma'^{\sf v}) \ = \ 
& \prod_{i = 1}^{n/2} \Gamma_\C \left(s+d'-d+ \tfrac{\ell_i^\tau}{2}\right) \cdot \\
& \prod_{\substack{1 \leq i \leq n/2 \\ 1 \leq j \leq (n'-1)/2}} 
 \Gamma_\C\left(s+d'-d+ \tfrac{\ell_i^\tau+\ell_j'\!{}^\tau}{2}\right) 
 \Gamma_\C\left(s+d'-d+ \tfrac{|\ell_i^\tau-\ell_j'\!{}^\tau|}{2}\right), 
 \end{split}
 \end{equation}
and if $n'$ is even then from (\ref{eqn:langlands-parameters-tensor-product-even}) we get: 
\begin{equation}
 \label{eqn:L-factor-infinity-automorphic-even}
 L_v(s, \iota, \sigma \times \sigma'^{\sf v}) \ = \ 
 \prod_{\substack{1 \leq i \leq n/2 \\ 1 \leq j \leq n'/2}} 
 \Gamma_\C\left(s+d'-d+ \tfrac{\ell_i^\tau+\ell_j'\!{}^\tau}{2}\right) 
 \Gamma_\C\left(s+d'-d+ \tfrac{|\ell_i^\tau-\ell_j'\!{}^\tau|}{2}\right). 
 \end{equation}

\medskip
\paragraph{\bf The critical set for automorphic Rankin--Selberg $L$-functions}

\begin{prop}
\label{prop:critical-set-tensor-product-aut}
Suppose $\sigma_f \in \Coh_{!!}(G_n, \mu)$ and $\sigma_f' \in \Coh_{!!}(G_{n'}, \mu')$ have disjoint cuspidal parameters. 
Assume that $nn'$ is even. If $N = n+n'$ then let $\epsilon_N \in \{0,1\}$ be such that $N \equiv \epsilon_N \pmod{2}.$ 
Then the critical set for $L(s, \iota, \sigma_f \times \sigma_f'^{\sf v})$ is the finite set of half-integers:  
$$
\left\{ m \in \tfrac{\epsilon_N}{2} + \Z \ : \ 
\frac{2-\ell(\mu, \mu') + 2(d-d')}{2} \ \leq \ m \ \leq \ \frac{\ell(\mu, \mu') + 2(d-d')}{2} \right\}.
$$
The critical set is centered around $\tfrac12 + d-d'$ and has cardinality equal to the cuspidal width 
$\ell(\mu, \mu').$   
\end{prop}

\begin{proof}
The parity issue is dictated by Def.\,\ref{def:critical-aut-l-fn}. We remind the reader that 
$\ell_i \equiv 2d + n-1 \pmod{2}$, hence $\ell(\mu,\mu') + 2(d-d') \equiv \epsilon_N \pmod{2}.$ 

The proof follows from (\ref{eqn:L-factor-infinity-automorphic-odd}) and (\ref{eqn:L-factor-infinity-automorphic-even});  
we briefly sketch the details in the case $n$ is even and $n'$ is odd, leaving the other case to the reader. Then 
$\epsilon_N = 1$. Suppose $m = \tfrac12 + m_0$. Note that 
$L_\infty(s, \iota, \sigma \times \sigma'^{\sf v})$ is regular at $\tfrac12 + m_0$ if and only if for all 
$\tau : F \to E$, and for all indices $i,j$ with $1 \leq i \leq n/2$ and $1 \leq j \leq (n'+1)/2$, bearing in mind that 
$\ell'_{(n'+1)/2} = 0$, we have: 
\begin{itemize}
\item[] $\frac{1 + 2m_0 + 2(d'-d) + \ell_i^\tau+\ell_j'\!{}^\tau}{2} \geq 1 \ \iff \ 
2m_0 - 2(d-d') + \ell_i^\tau+\ell_j'\!{}^\tau \geq 1$, and 
\medskip
\item[] $\frac{1 + 2m_0 + 2(d'-d) + |\ell_i^\tau - \ell_j'\!{}^\tau|}{2} \geq 1 \ \iff \ 
2m_0 - 2(d-d') +  |\ell_i^\tau - \ell_j'\!{}^\tau| \geq 1$;
\end{itemize}
putting both together we have 
\begin{verse}
$L_\infty(s, \iota, \sigma \times \sigma'^{\sf v})$ is regular at 
$\tfrac12 + m_0 \ \ \iff \ \  2m_0 - 2(d-d') + \ell(\mu,\mu') \geq 1.$
\end{verse}
Similarly, 
\begin{verse}
$L_\infty(1-s, \iota, \sigma^{\sf v} \times \sigma')$ is regular at 
$\tfrac12 + m_0 \ \ \iff \ \  -2m_0 + 2(d-d') + \ell(\mu,\mu') \geq 1.$
\end{verse}
Putting both conditions together concludes the proof. 
\end{proof}

\medskip
\begin{cor}
\label{cor:comb-lem-easy-critical}
Suppose $\sigma_f \in \Coh_{!!}(G_n, \mu)$ and $\sigma_f' \in \Coh_{!!}(G_{n'}, \mu')$ have disjoint cuspidal parameters. 
Assume that $nn'$ is even, and let $N = n+n'$. The points 
$-\tfrac{N}{2}$ and $1  -\tfrac{N}{2}$ are critical for $L(s, \iota, \sigma_f \times \sigma_f'^{\sf v})$ if and only if 
$$
-\frac{N}{2} + 1 - \frac{\ell(\mu,\mu')}{2} \ \leq \ (d-d') \ \leq \ 
-\frac{N}{2} - 1 + \frac{\ell(\mu,\mu')}{2}. 
$$
\end{cor}

\begin{proof}
Follows easily from Prop.\,\ref{prop:critical-set-tensor-product-aut}. 
\end{proof}

\medskip
\begin{cor}
\label{cor:ratio-l-values-infty}
The ratio of local $L$-values at any archimedean place $v \in \place_\infty$  is given by: 
$$
\frac{L_v\left(-\tfrac{N}{2}, \iota, \sigma \times \sigma'^{\sf v}\right)}
{L_v\left(1-\tfrac{N}{2}, \iota, \sigma \times \sigma'^{\sf v}\right)} 
\ = \ 
\frac{(2\bfgreek{pi})^{\tfrac{nn'}{2}}}
{\prod_{i=1}^{n/2} \prod_{j=1}^{n'}\left(\tfrac{-N+2(d'-d)+|\ell_i^\tau - \ell_j^\tau|}{2}\right)}. 
$$
In particular, the ratio of $L$-values at infinity is given by:  
$$
\frac{L_\infty\left(-\tfrac{N}{2}, \iota, \sigma \times \sigma'^{\sf v}\right)}
{L_\infty\left(1-\tfrac{N}{2}, \iota, \sigma \times \sigma'^{\sf v}\right)} 
\ \approx_{\Q^\times} \ 
\bfgreek{pi}^{[F:\Q]nn'/2} \ = \ \bfgreek{pi}^{{\rm dim}(U_P)/2}, 
$$
where $\approx_{\Q^\times}$ means up to a nonzero rational number. 
\end{cor}
\begin{proof}
Follows from (\ref{eqn:L-factor-infinity-automorphic-odd}) and (\ref{eqn:L-factor-infinity-automorphic-even}), and 
the well-known relation: $\Gamma(s+1) = s\Gamma(s).$  
\end{proof}
This ratio of $L$-values will be relevant in the discussion around Thm.\,\ref{thm:c-infinity}.

\medskip
\paragraph{\bf Comments about shifts in the $s$-variable}
\label{sec:shift-s-variable}

The relation (\ref{eqn:eff-motivic-auto-l-fn}) can be rewritten as a relation between cohomological Rankin--Selberg $L$-functions or the $L$-function of the tensor product of effective motives and the automorphic Rankin--Selberg $L$-functions as:  
\begin{equation}
\label{eqn:eff-motivic-auto-l-fn-tensor-product}
L\left(\iota,  \Mot_{\rm eff}(\sigma_f) \otimes \Mot_{\rm eff}(\sigma_f'^{\sf v}),  s\right) \ = \ 
L^{\rm aut}\left(s + \tfrac{2(d-d')-\w-\w'}{2}, {}^\iota\!\sigma_f \times {}^\iota\!\sigma_f^{\sf v} \right).
\end{equation}
The reader may verify that Prop.\,\ref{prop:critical-set-tensor-product-eff-mot} and 
Prop.\,\ref{prop:critical-set-tensor-product-aut}
are compatible with the above shift in the $s$-variable.

\smallskip

Note that ${}^\iota\!\sigma \otimes |\ |^d =: {}^\circ {}^\iota\!\sigma$ is a unitary cuspidal automorphic representation. 
The critical set for the $L$-function attached to the unitary representations
$$
L(s, \iota, {}^\circ\!\sigma \times  {}^\circ\!\sigma'^{\sf v}) \ = \ 
L(s, {}^\circ {}^\iota\!\sigma \times {}^\circ {}^\iota\!\sigma'^{\sf v}) \ = \ 
L(s + d -d', \iota, \sigma \times  \sigma'^{\sf v})
$$ 
would be centered around $\tfrac12$.

\medskip
\subsubsection{\bf A combinatorial lemma}
\label{sec:com-lemma}

Suppose $\sigma_f \in \Coh_{!!}(G_n, \mu)$ and $\sigma_f' \in \Coh_{!!}(G_{n'}, \mu')$. Assume that $nn'$ is even. Let  
$N = n+n'.$ Recall, from Sect.\,\ref{sec:simple-notation} that the 
assumption on the existence of a balanced Kostant representative 
$w \in W^P$ such that $\lambda:= w^{-1} \cdot (\mu \otimes \mu')$ is dominant implies that the induced representation 
${}^{\rm a}{\rm Ind}_{P(\A_f)}^{G(\A_f)}( \sigma_f \otimes \sigma_f' )$ appears in 
$H^{b_N^F}_{!!}(\partial_P\cS^{G}, \tM_{\lambda, E})$, and for the parabolic subgroup $Q$ that is associate to $P$, the representation 
${}^{\rm a}{\rm Ind}_{Q(\A_f)}^{G(\A_f)}( \sigma_f'(n) \otimes \sigma_f(-n'))$ appears in 
$H^{b_N^F}_{!!}(\partial_Q\cS^G, \tM_{\lambda, E}).$ 
Furthermore, the same assumption also implies that the duals of these induced representations appear in boundary cohomology in degree $\tilde t_N^F-1.$ 
The lemma below characterizes the existence of such a balanced Kostant representative as a purely combinatorial condition involving the cuspidal parameters and the abelian parts of $\mu$ and $\mu'.$ Amazingly, the same combinatorial condition also characterizes the fact that the $L$-values which intervene in (the proof of) the main theorem on Eisenstein cohomology Thm.\,\ref{thm:rank-one-eis} are critical values. 
We further elaborate on the arithmetic content of this lemma in Sect.\,\ref{sec:arith-comb-lem} below.

\begin{lemma}
\label{lem:comb-lemma}
Suppose $\sigma_f \in \Coh_{!!}(G_n, \mu)$ and $\sigma_f' \in \Coh_{!!}(G_{n'}, \mu')$ have disjoint cuspidal parameters. 
(From Prop.\,\ref{prop:disjoint-param-middle-hodge}, necessarily $nn'$ is even.) 
Let $\mu \otimes \mu' \in X^*(T_N).$ Then the following are equivalent: 
\begin{enumerate}
\item There is a balanced Kostant representative $w \in W^P$ such that $w^{-1}\cdot(\mu \otimes \mu')$ is a dominant integral weight. 

\medskip

\item $-\frac{N}{2} + 1 - \frac{\ell(\mu,\mu')}{2} \ \leq \ (d-d') \ \leq \ 
-\frac{N}{2} - 1 + \frac{\ell(\mu,\mu')}{2}.$

\medskip

\item The points $-\tfrac{N}{2}$ and $1-\tfrac{N}{2}$ are critical for the automorphic Rankin--Selberg $L$-function
$L(s, \iota, \sigma_f \times \sigma_f'^{\v}).$

\end{enumerate}
\end{lemma}

\begin{proof}
The  equivalence (1) $\iff$ (2) was conjectured in our announcement \cite{harder-raghuram}, and was later proved by 
Uwe Weselman; his proof is given in Appendix 1. 
The equivalence (2) $\iff$ (3) is exactly the statement of Cor.\,\ref{cor:comb-lem-easy-critical} above. 
\end{proof}

\medskip
\paragraph{\bf Sufficiently large cuspidal width} 
Note that a necessary condition for the inequalities in (2) to hold is that the cuspidal width is at least $2$, i.e., 
$\ell(\mu,\mu') \geq 2.$ For $L$-functions for $G_2 \times G_1$, i.e., for $L$-functions for Hilbert modular forms twisted by algebraic Hecke characters of $F$, in the notations of \ref{sec:example-hilbert}, the condition 
$\ell(\mu,\mu') \geq 2$ translates to $k^0 = {\rm min}(k_\tau) \geq 3.$

\medskip
\paragraph{\bf Comparison of notations with \cite{harder-raghuram}}
The combinatorial lemma, especially (1) $\iff$ (2), was conjectured by us \cite[Conj.\,5.1]{harder-raghuram} in our announcement, and was proved by Uwe Weselmann; see the appendix for his proof. In this short paragraph  we just point to some notational differences: The quantity $p(\mu)$ of \cite{harder-raghuram} is the largest critical point of the cohomological $L$-function which is the same the motivic $L$-function for the effective motives. From 
Prop.\,\ref{prop:critical-set-tensor-product-eff-mot} we can see that $p(\mu)$ is equal to our  
$(\ell(\mu, \mu') + \w + \w')/2.$  It should now be clear that the above formulation of the combinatorial lemma is the same as in \cite{harder-raghuram}.

\medskip
\paragraph{\bf An arithmetic consequence of the combinatorial lemma} 
\label{sec:arith-comb-lem}
This article is about an application of Eisenstein cohomology to prove rationality 
results for ratios of successive $L$-values. 
The combinatorial lemma says that we can prove such rationality results for a ratio of $L$-values when these are critical values, and furthermore, the lemma also says that every ratio of critical values is captured by these techniques, and if either of the $L$-values is not critical then our methods do not apply. This may be justified as follows: Applying the lemma, and as should be clear from the proof of Thm.\,\ref{thm:rank-one-eis}, we are able to prove a rationality result for the ratio
$$
\frac{L^{\rm aut}\left(-\tfrac{N}{2}, \iota, \sigma_f \times \sigma_f'^{\v}\right)}
{L^{\rm aut}\left(1-\tfrac{N}{2}, \iota, \sigma_f \times \sigma_f'^{\v}\right)}
\ = \ 
\frac{L^{\rm coh}\left(\iota, \sigma_f \times \sigma_f'^{\v}, -\tfrac{N}{2} + \tfrac{(\w+\w')}{2}-(d-d')\right)}
{L^{\rm coh}\left(\iota , \sigma_f \times \sigma_f'^{\v}, 1-\tfrac{N}{2} + \tfrac{(\w+\w')}{2}-(d-d')\right)}. 
$$
For brevity, let 
$$
{\sf m}_0 := -\tfrac{N}{2} + \tfrac{(\w+\w')}{2}-(d-d').
$$ Now, given $\sigma_f \in \Coh_{!!}(G_n, \mu)$ and $\sigma_f' \in \Coh_{!!}(G_{n'}, \mu')$, let's take all possible Tate twists.  
Note that this means, we have fixed the semi-simple parts 
$\mu^{(1)}$ and $\mu'^{(1)}$ and we are letting $d$ and $d'$, i.e., the abelian parts $\mu_{\rm ab} = d \bfgreek{delta}_n$ 
$\mu'_{\rm ab} = d' \bfgreek{delta}_{n'},$ to vary. The inequalities in (2) of 
Lem.\,\ref{lem:comb-lemma} bounds $(d-d')$ since the cuspidal width 
$\ell(\mu,\mu') = \ell(\mu^{(1)}, \mu'^{(1)})$ depends only on the semi-simple parts of the weights. As $(d-d')$ varies between its lower bound $-\frac{N}{2} + 1 - \frac{\ell(\mu,\mu')}{2}$ to its upper-bound 
$-\frac{N}{2} - 1 + \frac{\ell(\mu,\mu')}{2}$, one may check using Prop.\,\ref{prop:critical-set-tensor-product-eff-mot}, 
that $m_0$ varies from one less than the largest critical point to the lowest critical point, i.e., 
the ratio 
$$
\frac{L^{\rm coh}\left(\iota, \sigma_f \times \sigma_f'^{\v}, {\sf m}_0\right)}
{L^{\rm coh}\left(\iota , \sigma_f \times \sigma_f'^{\v}, 1+{\sf m}_0\right)}
$$ 
runs exactly over all possible successive critical values; no more and no less!

\medskip
\subsection{The main result on special values of $L$-functions}
\label{sec:special-values}

\medskip
\subsubsection{\bf The main result on $L$-values}
\label{sec:subsec-special-values}

\begin{thm}
\label{thm:main-theorem-l-values}
Let $\mu \in X^*_0(T_n)$ 
(resp., $\mu' \in X^*_0(T_{n'})$) be a pure weight. Suppose $\sigma_f \in \Coh_{!!}(G_n, \mu)$ and 
$\sigma_f' \in \Coh_{!!}(G_{n'}, \mu')$ have disjoint cuspidal parameters; in particular $nn'$ is even.  Assume that there is a balanced Kostant representative $w \in W^P$ such that $\lambda := w^{-1}\cdot(\mu \otimes \mu')$ is a dominant integral weight. Let  ${\sf m}_0 = -\tfrac{N}{2} + \tfrac{(\w+\w')}{2}-(d-d').$ Let the rest of the notations be as in 
Sect.\,\ref{sec:simple-notation}.

\begin{enumerate}

\item Suppose that $n$ is even and $n'$ is odd. For $\iota : E \to \C$ we have
$$
\frac{1}{\Omega^{\varepsilon'}({}^\iota\!\sigma_f)} \, 
\frac{L^{\rm coh}\left(\iota, \sigma \times \sigma'^{\v}, {\sf m}_0\right)}
{L^{\rm coh}\left(\iota , \sigma \times \sigma'^{\v}, 1+{\sf m}_0\right)}
\ \in  \ 
\iota(E). 
$$
Moreover, for any $\tau \in {\rm Gal}(\bar\Q/\Q)$  
we have: 
$$
\tau \left(
\frac{1}{\Omega^{\varepsilon'}({}^\iota\!\sigma_f)} \, 
\frac{L^{\rm coh}\left(\iota, \sigma \times \sigma'^{\v}, {\sf m}_0\right)}
{L^{\rm coh}\left(\iota , \sigma \times \sigma'^{\v}, 1+{\sf m}_0\right)} \right)
\ = \ 
\frac{1}{\Omega^{{}^{\tau\!\varepsilon'}}( {}^{\tau \circ \iota}\sigma_f)} \, 
\frac{L^{\rm coh}\left(\tau \circ \iota, \sigma \times \sigma'^{\v}, {\sf m}_0\right)}
{L^{\rm coh}\left(\tau \circ \iota , \sigma \times \sigma'^{\v}, 1+{\sf m}_0\right)}. 
$$

\smallskip

\item Suppose that both $n$ and $n'$ are even. For $\iota : E \to \C$ we have
$$
\frac{L^{\rm coh}\left(\iota, \sigma \times \sigma'^{\v}, {\sf m}_0\right)}
{L^{\rm coh}\left(\iota , \sigma \times \sigma'^{\v}, 1+{\sf m}_0\right)}
\ \in  \ 
\iota(E). 
$$
Moreover, for any $\tau \in {\rm Gal}(\bar\Q/\Q)$  
we have: 
$$
\tau \left(
\frac{L^{\rm coh}\left(\iota, \sigma \times \sigma'^{\v}, {\sf m}_0\right)}
{L^{\rm coh}\left(\iota , \sigma \times \sigma'^{\v}, 1+{\sf m}_0\right)} \right)
\ = \ 
\frac{L^{\rm coh}\left(\tau \circ \iota, \sigma \times \sigma'^{\v}, {\sf m}_0\right)}
{L^{\rm coh}\left(\tau \circ \iota , \sigma \times \sigma'^{\v}, 1+{\sf m}_0\right)}. 
$$
\end{enumerate}
\end{thm}

\medskip
\subsubsection{\bf Proof of Thm.\,\ref{thm:main-theorem-l-values} in the case $n$ is even and $n'$ is odd}
The notations are as in Sect.\,\ref{sec:simple-notation}.
If $-N/2+d-d' \leq 0$ then we go from $P$ to $Q$ as suggested by 
Prop.\,\ref{prop:at-least-one-holomorphic}, and we will present the details of the proof below. 
The case when $-N/2+d-d' \geq 0$ is absolutely similar and we will leave the details to the reader. 
Recall from Thm.\,\ref{thm:rank-one-eis} that $\fI^b(\sigma_f, \sigma_f', \varepsilon')$, which is the image of global cohomology in the $2{\sf k}$-dimensional $E$-vector space 
$I^\place_b(\sigma_f,\sigma_f', \varepsilon')_{P,w} \oplus I^\place_b(\sigma_f,\sigma_f', \varepsilon')_{Q,w'}$, 
is a ${\sf k}$-dimensional $E$-subspace; and furthermore, from the proof of this theorem, it is clear that we 
get an intertwining operator: 
\begin{equation}
\label{eqn:T_eis}
T_{\rm Eis}(\sigma, \sigma', \varepsilon') \ : \ 
I^\place_b(\sigma_f,\sigma_f', \varepsilon')_{P,w} \ \longrightarrow \  
I^\place_b(\sigma_f,\sigma_f', \varepsilon')_{Q, w'}, 
\end{equation}
defined over $E$, such that 
$$
\fI^b(\sigma_f, \sigma_f', \varepsilon') \ = \ 
\left\{ \left(\xi, \, T_{\rm Eis}(\sigma, \sigma', \varepsilon')(\xi)\right) \ | \ 
\xi \in I^\place_b(\sigma_f,\sigma_f', \varepsilon')_{P,w} \right\}.
$$
 The idea of the proof is to take $T_{\rm Eis}(\sigma, \sigma', \varepsilon')$ to a transcendental level, use the constant term theorem which gives $L$-values, and after introducing the relative periods, to descend down to an arithmetic level, giving a rationality result for the said $L$-values divided by the relative periods.

\medskip
\paragraph{\bf The Eisenstein operator at a transcendental level}
Consider the Eisenstein operator in (\ref{eqn:T_eis}), and pass to $\C$ via an embedding $\iota : E \to \C.$ 
Recall the part of the proof of the main theorem on Eisenstein cohomology, where we used Langlands's constant term theorem after going to a transcendental level; we get 
\begin{equation}
\label{eqn:T-eis-pass-to-C}
T_{\rm Eis}(\sigma, \sigma', \varepsilon') \otimes_{E, \iota} {\bf 1}_\C \ = \ 
T_{\rm st}(-\tfrac{N}{2}, {}^\iota\!\sigma \otimes {}^\iota\!\sigma')^\bullet(\varepsilon'), 
\end{equation}
where the right hand side is the map induced at the level of relative Lie algebra cohomology 
(Sect.\,\ref{sec:conclude-proof-nonzero-image}) by the global standard intertwining operator 
(Sect.\,\ref{sec:T_st}), i.e., $T_{\rm st}(-\tfrac{N}{2}, {}^\iota\!\sigma \otimes {}^\iota\!\sigma')^\bullet(\varepsilon')$ is the intertwining operator 
\begin{multline}
H^{b_N^F}(\g, K_\infty^0, \aInd_{P(\A)}^{G(\A)}({}^\iota\!\sigma \otimes {}^\iota\!\sigma')^{K_f} 
\otimes \M_{{}^\iota\!\lambda, \C})(\varepsilon') 
\ \longrightarrow \\ 
H^{b_N^F}(\g, K_\infty^0, \aInd_{Q(\A)}^{G(\A)}({}^\iota\!\sigma'(n) \otimes {}^\iota\!\sigma(-n'))^{K_f} 
\otimes \M_{{}^\iota\!\lambda, \C})(\varepsilon'). 
\end{multline}
Since the global standard intertwining operator is a product of local operators we get
\begin{equation}
\label{eqn:T-st-global-to-local}
T_{\rm st}(-\tfrac{N}{2}, {}^\iota\!\sigma \times {}^\iota\!\sigma')^\bullet( \varepsilon') \ = \ 
T_{\rm st}(-\tfrac{N}{2}, {}^\iota\!\sigma_\infty \times {}^\iota\!\sigma'_\infty)^\bullet(\varepsilon') \otimes 
T_{\rm st}(-\tfrac{N}{2}, {}^\iota\!\sigma_f \times {}^\iota\!\sigma_f'). 
\end{equation}

\medskip
\paragraph{\bf The local operator at infinity}
We begin by showing  that 
$T_{\rm st}(-\tfrac{N}{2}, {}^\iota\!\sigma_\infty \times {}^\iota\!\sigma'_\infty)^\bullet(\varepsilon')$ is a nonzero isomorphism between one-dimensional vector spaces. 

\begin{prop}
The induced representations
$$
\aInd_{P(\R)}^{G(\R)}({}^\iota\!\sigma_\infty \otimes {}^\iota\!\sigma'_\infty) \quad {\rm and} \quad 
\aInd_{Q(\R)}^{G(\R)}({}^\iota\!\sigma'_\infty(n) \otimes {}^\iota\!\sigma_\infty(-n'))
$$
are irreducible representations of $G(\R).$ The operator 
$T_{\rm st}(-\tfrac{N}{2}, {}^\iota\!\sigma \times {}^\iota\!\sigma')_\infty$
is an isomorphism between these two induced representations, and hence induces an isomorphism 
$T_{\rm st}(-\tfrac{N}{2}, {}^\iota\!\sigma \times {}^\iota\!\sigma')_\infty^\bullet(\varepsilon')$ between the one-dimensional $\C$-vector spaces: 
\begin{multline*}
H^{b_N^F}(\g, K_\infty^0, \aInd_{P(\R)}^{G(\R)}({}^\iota\!\sigma_\infty \otimes {}^\iota\!\sigma'_\infty)
\otimes \M_{{}^\iota\!\lambda, \C})(\varepsilon') 
\ \longrightarrow \\ 
H^{b_N^F}(\g, K_\infty^0, \aInd_{Q(\R)}^{G(\R)}({}^\iota\!\sigma'_\infty(n) \otimes {}^\iota\!\sigma_\infty(-n'))
\otimes \M_{{}^\iota\!\lambda, \C})(\varepsilon'). 
\end{multline*}
\end{prop}

\begin{proof}
Irreducibility of the induced representations follows by applying results of Birgit Speh; see 
\cite[Thm.\,10b]{moeglin-edinburgh}. We briefly sketch the details in one case. (Let's introduce some convenient and well-known notation. For any local field $\F$, suppose $\pi_1$ and $\pi_2$ are representations of $\GL_{n_1}(\F)$ and 
$\GL_{n_2}(\F)$, then by $\pi_1 \times \pi_2$ we denote the representation of $\GL_{n_1+ n_2}$ obtained by normalized  parabolical induction from the representation $\pi_1 \otimes \pi_2$ of the Levi subgroup 
$\GL_{n_1}(\F) \times \GL_{n_2}(\F).$) Translating to normalized induction, and using transitivity of normalized parabolic induction, we may write $\aInd_{P(\R)}^{G(\R)}({}^\iota\!\sigma_\infty \otimes {}^\iota\!\sigma'_\infty)$ as: 
$$
\D_{{}^\iota\!\ell_1}(-d-\tfrac{n'}{2}) \times \cdots \times \D_{{}^\iota\!\ell_{n/2}}(-d-\tfrac{n'}{2}) \times 
\D_{{}^\iota\!\ell'_1}(-d'-\tfrac{n}{2}) \times \cdots \times \D_{{}^\iota\!\ell_{(n'-1)/2}}(-d'-\tfrac{n}{2}) \times 
{}^\iota\!\varepsilon'(-d'-\tfrac{n}{2}).
$$
To check irreducibility, by (ii) and (iii) of \cite[Thm.\,10b]{moeglin-edinburgh}, we need to check, 
for $1 \leq i \leq n/2$ and $1 \leq j \leq (n'+1)/2$, that  
$$
- \frac{|{}^\iota\!\ell_i - {}^\iota\!\ell'_j|}{2} + |d-d'+\tfrac{N}{2}| \ \notin \ \{1, 2, 3, \dots \}. 
$$
We leave it to the reader to see that this follows from the inequalities in (ii) of Lem.\ref{lem:comb-lemma}. It is also easy to see along the same lines that the irreducibility holds if both $n$ and $n'$ are even. 

Now, we have the standard intertwining operator $T_{\rm st}(-\tfrac{N}{2}, {}^\iota\!\sigma \times {}^\iota\!\sigma')_\infty$ between two irreducible representations. To show it is an isomorphism, it suffices to show that it is nonzero. But, suppose 
$T_{\rm st}(-\tfrac{N}{2}, {}^\iota\!\sigma \times {}^\iota\!\sigma')_\infty = 0.$ Then using Shahidi's results 
\cite{shahidi-duke80} on local factors (see (\ref{eqn:local-constant}) and (\ref{eqn:local-constant-gamma})), we deduce that 
$C_{\psi_\infty}(s, \sigma_\infty \times \sigma'_\infty)$ has a pole at $s = -N/2.$ Since local $L$-factors are nonvanishing everywhere, we further deduce that $L(1+N/2, \sigma_\infty^{\sf v} \times \sigma'_\infty)$ is a pole; but this is not possible since $-N/2$ is a critical point. Hence, $T_{\rm st}(-\tfrac{N}{2}, {}^\iota\!\sigma \times {}^\iota\!\sigma')_\infty \neq 0$; whence an isomorphism. The rest is clear since any functor takes an isomorphism to an isomorphism. 
\end{proof}

The reader should appreciate the point that the conditions of the combinatorial lemma are crucial in the proof of the above proposition.

\medskip
\subparagraph{\bf Choice of bases at infinity}
Delorme's Lemma (see \cite[Thm.\,III.3.3]{borel-wallach}) describes the relative Lie algebra cohomology of a parabolically induced representation in terms of the cohomology of the inducing data. This allows us to  
make the following identifications: 
\begin{multline}
\label{eqn:delorme-1}
H^{b_N^F}(\g_N, K_{N,\infty}^0, \aInd_{P(\R)}^{G(\R)}({}^\iota\!\sigma_\infty \otimes {}^\iota\!\sigma'_\infty)
\otimes \M_{{}^\iota\!\lambda, \C})(\varepsilon')  \ \cong \\
H^{b_n^F}(\g_n, K_{n,\infty}^0, {}^\iota\!\sigma_\infty \otimes \M_{{}^\iota\!\mu, \C})(\varepsilon') \otimes
H^{b_{n'}^F}(\g_{n'}, K_{n',\infty}^0, {}^\iota\!\sigma'_\infty \otimes \M_{{}^\iota\!\mu', \C})(\varepsilon'), 
\end{multline}
and 
\begin{multline}
\label{eqn:delorme-2}
H^{b_N^F}(\g_N, K_{N,\infty}^0, \aInd_{Q(\R)}^{G(\R)}({}^\iota\!\sigma'_\infty(n) \otimes {}^\iota\!\sigma_\infty(-n'))
\otimes \M_{{}^\iota\!\lambda, \C})(\varepsilon')  \ \cong \\
H^{b_n^F}(\g_n, K_{n,\infty}^0, 
{}^\iota\!\sigma_\infty(-n') \otimes \M_{{}^\iota\!\mu+n'\bfgreek{delta}_n, \C})(\varepsilon') \otimes
H^{b_{n'}^F}(\g_{n'}, K_{n',\infty}^0, 
{}^\iota\!\sigma'_\infty(n) \otimes \M_{{}^\iota\!\mu'-n\bfgreek{delta}_{n'}, \C})(\varepsilon'). 
\end{multline}
In the above identifications, the balanced Kostant representative $w \in W^P$ and its associate $w' \in W^Q$ have played an important role, since ${}^\iota\!w \cdot {}^\iota\!\lambda = {}^\iota\!\mu \otimes {}^\iota\!\mu'$ 
and ${}^\iota\!w' \cdot {}^\iota\!\lambda = ({}^\iota\!\mu' -n\bfgreek{delta}_{n'})  \otimes ({}^\iota\!\mu+n'\bfgreek{delta}_{n}).$

Recall that we fixed a basis element $w^{\varepsilon'}({}^\iota\mu)$ of the one-dimensional space
$H^{b_n^F}(\g_n, K_{n,\infty}^0, {}^\iota\!\sigma_\infty \otimes \M_{{}^\iota\!\mu, \C})(\varepsilon')$ in 
(\ref{eqn:basis-w-epsilon-lambda}), and recall also the map $T^{\varepsilon'}({}^\iota\mu)$ in 
(\ref{eqn:switch-infinity}) which sends $w^{\varepsilon'}({}^\iota\mu)$ to $w^{-\varepsilon'}({}^\iota\mu).$
We have the map 
\begin{equation}
\label{eqn:T-trans-comes-from-here}
T_{\rm Tate}(-n', {}^\iota\!\sigma_\infty, -\varepsilon') \circ T^{\varepsilon'}({}^\iota\mu)
\end{equation}
which takes 
$w^{\varepsilon'}({}^\iota\mu)$ to $w^{\varepsilon'}({}^\iota\mu+n'\bfgreek{delta}_n);$ we have used $n'$ is odd. 
Similarly, we have 
$$
T_{\rm Tate}(n, {}^\iota\!\sigma'_\infty, \varepsilon') 
$$
which identifies $H^{b_{n'}^F}(\g_{n'}, K_{n',\infty}^0, {}^\iota\!\sigma'_\infty \otimes \M_{{}^\iota\!\mu', \C})(\varepsilon')$  
with $H^{b_{n'}^F}(\g_{n'}, K_{n',\infty}^0, 
{}^\iota\!\sigma'_\infty(n) \otimes \M_{{}^\iota\!\mu'-n\bfgreek{delta}_{n'}, \C})(\varepsilon')$, since $n$ is even. 
Hence, we get a map 
$$
T_{\rm Tate}(-n', {}^\iota\!\sigma_\infty, -\varepsilon') \circ T^{\varepsilon'}({}^\iota\mu) \otimes 
T_{\rm Tate}(n, {}^\iota\!\sigma'_\infty, \varepsilon')
$$
identifying the right hand side of (\ref{eqn:delorme-1}) with the right hand side of (\ref{eqn:delorme-2}); this defines the corresponding isomorphism of the left hand sides which we denote 
$T_{\rm loc}({}^\iota\!\sigma_\infty, {}^\iota\!\sigma'_\infty)(\varepsilon').$

\medskip
\subparagraph{\bf The quantity $c({}^\iota\!\sigma_\infty, {}^\iota\!\sigma'_\infty)$}
Observe that both the operators 
$T_{\rm st}(-\tfrac{N}{2}, {}^\iota\!\sigma_\infty \otimes {}^\iota\!\sigma'_\infty)^\bullet(\varepsilon')$ 
and 
$T_{\rm loc}({}^\iota\!\sigma_\infty, {}^\iota\!\sigma'_\infty)(\varepsilon')$
are isomorphisms between the same one-dimensional $\C$-vector spaces. This defines a nonzero complex number
$c({}^\iota\!\sigma_\infty, {}^\iota\!\sigma'_\infty)$ as: 
$$
T_{\rm st}(-\tfrac{N}{2}, {}^\iota\!\sigma_\infty \otimes {}^\iota\!\sigma'_\infty)^\bullet(\varepsilon') \ = \ 
c({}^\iota\!\sigma_\infty, {}^\iota\!\sigma'_\infty) \, T_{\rm loc}({}^\iota\!\sigma_\infty, {}^\iota\!\sigma'_\infty)(\varepsilon'). 
$$
The following theorem is due to Harder \cite{harder-arithmetic}: 
\begin{thm}
\label{thm:c-infinity}
$$ 
c_\infty({}^\iota\!\sigma_\infty, {}^\iota\!\sigma'_\infty) \ \approx_{\Q^\times} \,( i\bfgreek{pi})^{[F:\Q] nn'/2},
$$
where $\approx_{\Q^\times}$ means equality up to a nonzero rational number. 
\end{thm}
From Cor.\,\ref{cor:ratio-l-values-infty}, this may be stated in terms of cohomological $L$-factors at infinity as 
$$
c_\infty({}^\iota\!\sigma_\infty, {}^\iota\!\sigma'_\infty) \ \approx_{\Q^\times} \ 
i^{{\sf r}nn'/2} \cdot 
\frac{L^{\rm coh}_\infty\left(\iota, \sigma \times \sigma'^{\v}, {\sf m}_0\right)}
{L^{\rm coh}_\infty\left(\iota , \sigma \times \sigma'^{\v}, 1+{\sf m}_0\right)}.
$$
(Recall, ${\sf r} = {\sf r}_F = [F:\Q].$) In particular, we have: 
\begin{equation}
\label{eqn:c-infinity-L-values}
T_{\rm st}(-\tfrac{N}{2}, {}^\iota\!\sigma_\infty \otimes {}^\iota\!\sigma'_\infty)^\bullet(\varepsilon') 
 \ \approx_{\Q^\times} \ 
i^{{\sf r}nn'/2} \cdot
\frac{L^{\rm coh}_\infty\left(\iota, \sigma \times \sigma'^{\v}, {\sf m}_0\right)}
{L^{\rm coh}_\infty\left(\iota , \sigma \times \sigma'^{\v}, 1+{\sf m}_0\right)} \,
T_{\rm loc}({}^\iota\!\sigma_\infty, {}^\iota\!\sigma'_\infty)(\varepsilon').
\end{equation}

\medskip
\paragraph{\bf The local operator at finite places} 
We define an operator $T_{\rm loc}(-\tfrac{N}{2}, {}^\iota\!\sigma \times {}^\iota\!\sigma')_v$ at any finite place as: 
\begin{equation}
\label{eqn:T-loc-finite}
T_{\rm loc}(-\tfrac{N}{2}, {}^\iota\!\sigma \times {}^\iota\!\sigma')_v \ := \  
\left(\frac{L(-\tfrac{N}{2}, {}^\iota\!\sigma_v \times {}^\iota\!\sigma_v'^{\sf v})}
{L(1-\tfrac{N}{2}, {}^\iota\!\sigma_v \times {}^\iota\!\sigma_v'^{\sf v})}\right)^{-1} \, 
T_{\rm st}(-\tfrac{N}{2}, {}^\iota\!\sigma \otimes {}^\iota\!\sigma')_v. 
\end{equation}
By \cite[Prop.\,I.10]{moeglin-waldspurger}, we know that $T_{\rm loc}(-\tfrac{N}{2}, {}^\iota\!\sigma \times {}^\iota\!\sigma')_v$ is finite and nonzero. (In \cite{moeglin-waldspurger}, the normalization factor on the right hand side also involves the local epsilon-factor, but this doesn't matter to us, since the epsilon-factor is an exponential function which is holomorphic and nonvanishing everywhere.)  
Langlands's calculation (see Sect.\,\ref{sec:ratio-l-fns}) about the standard intertwining operator at any unramified place says that $T_{\rm loc}(-\tfrac{N}{2}, {}^\iota\!\sigma \times {}^\iota\!\sigma')_v$ takes the 
normalized spherical vector $f_v^0$ to the normalized spherical vector $\tilde f_v^0$. 

At a ramified place we have basically the same. We have the $\HKv \times E$  module 
$I_P^G(-N/2,  \sigma_v \otimes \sigma'_v)$, and we have the base extension
  $$
  I_P^G(s, \sigma_v \otimes \sigma'_v) \ = \ 
  I_P^G(  \sigma_v \otimes \sigma'_v)\otimes_{E,\iota}\C\otimes \vert\gamma_P\vert_v^s.
  $$
It will be shown in
Harder~\cite{harder-badprimes} that  the evaluation of 
$ T_{\rm loc}(s, {}^\iota\!\sigma \times {}^\iota\!\sigma')_v $ at $s=-\frac{N}{2}$
is the base extension of an operator over $E$: 
$$
T_{\rm loc}(-\tfrac{N}{2}, {}^\iota\!\sigma \times {}^\iota\!\sigma')_v : I_P^G(-N/2, \sigma_v \otimes \sigma'_v)^{K_v}
\ \longrightarrow \ 
I_Q^G( N/2, \sigma^\prime_v \otimes \sigma_v)^{K_v}.
  $$ 
  This operator is nonzero provided $K_v$ is sufficiently small. (In general, it may not be sufficient to simply have  
  $I_P^G( -N/2, \sigma_v \otimes \sigma'_v)^{K_v}\not= 0$ to guarantee nonvanishing of the operator.) 
  The reason for this rationality result is that the integral defining the intertwining operator
  can be written as a finite sum of geometric series. 
  
  \bigskip
  
    Putting (\ref{eqn:T-st-global-to-local}), (\ref{eqn:c-infinity-L-values}) and (\ref{eqn:T-loc-finite}) together,  
while keeping Thm.\,\ref{thm:langlands} and Prop.\,\ref{prop:at-least-one-holomorphic} in mind, we deduce: 
\begin{multline}
\label{eqn:T-st-to-T-loc}
T_{\rm st}(-\tfrac{N}{2}, {}^\iota\!\sigma \times {}^\iota\!\sigma')^\bullet( \varepsilon') \ \approx_{\Q^\times} \\ 
i^{{\sf r}nn'/2} \cdot
\frac{L^{\rm coh}\left(\iota, \sigma \times \sigma'^{\v}, {\sf m}_0\right)}
{L^{\rm coh}\left(\iota , \sigma \times \sigma'^{\v}, 1+{\sf m}_0\right)} \, 
\left(T_{\rm loc}({}^\iota\!\sigma_\infty, {}^\iota\!\sigma'_\infty)(\varepsilon') \otimes 
T_{\rm loc}(-\tfrac{N}{2}, {}^\iota\!\sigma \times {}^\iota\!\sigma')_f\right).
\end{multline}

\medskip
\paragraph{\bf The operator $T_{\rm loc}$} 
From the definitions of $T_{\rm loc}({}^\iota\!\sigma_\infty, {}^\iota\!\sigma'_\infty)(\varepsilon')$ and 
$T_{\rm loc}(-\tfrac{N}{2}, {}^\iota\!\sigma \times {}^\iota\!\sigma')_f$, the definition of 
$T_{\rm trans}$ in Sect.\,\ref{sec:T-trans} and $T_{\rm Tate}$ in Sect.\,\ref{sec:T-Tate}, and 
after identifying $M_P$ with $M_Q$, it is clear that: 
\begin{multline}
\label{eqn:T-loc-to-T-trans}
T_{\rm loc}({}^\iota\!\sigma_\infty, {}^\iota\!\sigma'_\infty)(\varepsilon') \otimes 
T_{\rm loc}(-\tfrac{N}{2}, {}^\iota\!\sigma \times {}^\iota\!\sigma')_f \ = \\ 
\smallskip
\frac{1}{i^{{\sf r}n/2}}
\aIndPG 
\left(
(T_{\rm Tate}^{-\varepsilon'}({}^\iota\!\mu, {}^\iota\!\sigma_f, -n') \circ 
T_{\rm trans}^{\varepsilon'}({}^\iota\!\mu, {}^\iota\!\sigma_f)) \, \otimes \, 
T_{\rm Tate}^{\varepsilon'}({}^\iota\!\mu', {}^\iota\!\sigma'_f, n)
\right). 
\end{multline}
Using the definition (Def.\,\ref{defn:relative-periods}) of the relative period 
$\Omega^\varepsilon({}^\iota\!\sigma_f),$ we can rewrite 
(\ref{eqn:T-loc-to-T-trans}) as: 
\begin{multline}
\label{eqn:T-trans-to-T-arith}
T_{\rm loc}({}^\iota\!\sigma_\infty, {}^\iota\!\sigma'_\infty)(\varepsilon') \otimes 
T_{\rm loc}(-\tfrac{N}{2}, {}^\iota\!\sigma \times {}^\iota\!\sigma')_f \ = \\ 
\smallskip
\frac{1}{i^{{\sf r}n/2} \Omega^\varepsilon({}^\iota\!\sigma_f)}
\aIndPG
\left(
(T_{\rm Tate}^{-\varepsilon'}({}^\iota\!\mu, {}^\iota\!\sigma_f, -n') \circ 
T_{\rm arith}^{\varepsilon'}({}^\iota\!\mu, {}^\iota\!\sigma_f)) \, \otimes \, 
T_{\rm Tate}^{\varepsilon'}({}^\iota\!\mu', {}^\iota\!\sigma'_f, n) 
\right). 
\end{multline}
Finally, we use the fact that both $T_{\rm Tate}$ and $T_{\rm arith}$ are defined over $E$, and furthermore, the process of algebraic induction is also rationally defined; hence we may take the $\iota$ out: 
\begin{multline}
\label{eqn:pull-out-iota} 
\aIndPG
\left((T_{\rm Tate}^{-\varepsilon'}({}^\iota\!\mu, {}^\iota\!\sigma_f, -n') \circ 
T_{\rm arith}^{\varepsilon'}({}^\iota\!\mu, {}^\iota\!\sigma_f)) \, \otimes \, 
T_{\rm Tate}^{\varepsilon'}({}^\iota\!\mu', {}^\iota\!\sigma'_f, n) \right)  \ = \\
\medskip
\aIndPG
\left((T_{\rm Tate}^{-\varepsilon'}(\mu, \sigma_f, -n') \circ 
T_{\rm arith}^{\varepsilon'}(\mu, \sigma_f)) \, \otimes \, 
T_{\rm Tate}^{\varepsilon'}(\mu', \sigma'_f, n) \right) \otimes_{E, \iota} {\bf 1}_\C.
\end{multline}
Putting (\ref{eqn:T-loc-to-T-trans}), (\ref{eqn:T-trans-to-T-arith}) and (\ref{eqn:pull-out-iota}) together, we have: 
\begin{multline}
\label{eqn:T-loc-to-T-arith} 
T_{\rm loc}({}^\iota\!\sigma_\infty, {}^\iota\!\sigma'_\infty)(\varepsilon') \otimes 
T_{\rm loc}(-\tfrac{N}{2}, {}^\iota\!\sigma \times {}^\iota\!\sigma')_f \ = \\ 
\smallskip
\frac{1}{i^{{\sf r}n/2} \Omega^\varepsilon({}^\iota\!\sigma_f)} \,
\aIndPG
\left((T_{\rm Tate}^{-\varepsilon'}(\mu, \sigma_f, -n') \circ 
T_{\rm arith}^{\varepsilon'}(\mu, \sigma_f)) \, \otimes \, 
T_{\rm Tate}^{\varepsilon'}(\mu', \sigma'_f, n) \right) \otimes_{E, \iota} {\bf 1}_\C. 
\end{multline}

\medskip
\paragraph{\bf The concluding step in the proof} 
From (\ref{eqn:T-eis-pass-to-C}), (\ref{eqn:T-st-to-T-loc}) and (\ref{eqn:T-loc-to-T-arith}) it follows that:  

{\Small
\begin{multline}
\label{eqn:T-eis-all-the-way}
T_{\rm Eis}(\sigma, \sigma', \varepsilon') \otimes_{E, \iota} {\bf 1}_\C \ \approx_{\iota(E)^\times} \\
\medskip
\frac{1}{\Omega^\varepsilon({}^\iota\!\sigma_f)}
\frac{L^{\rm coh}\left(\iota, \sigma \times \sigma'^{\v}, {\sf m}_0\right)}
{L^{\rm coh}\left(\iota , \sigma \times \sigma'^{\v}, 1+{\sf m}_0\right)} \, 
\aIndPG
\left((T_{\rm Tate}^{-\varepsilon'}(\mu, \sigma_f, -n') \circ 
T_{\rm arith}^{\varepsilon'}(\mu, \sigma_f)) \, \otimes \, 
T_{\rm Tate}^{\varepsilon'}(\mu', \sigma'_f, n) \right) \otimes_{E, \iota} {\bf 1}_\C, 
\end{multline}}
where, we have used that $i^{{\sf r}n(n'-1)/2}$ is just $\pm 1$ which gets subsumed into $\iota(E)^\times.$ 
We deduce 
\begin{equation}
\label{eqn:L-value-by-period}
\frac{1}{\Omega^\varepsilon({}^\iota\!\sigma_f)}
\frac{L^{\rm coh}\left(\iota, \sigma \times \sigma'^{\v}, {\sf m}_0\right)}
{L^{\rm coh}\left(\iota , \sigma \times \sigma'^{\v}, 1+{\sf m}_0\right)} \ \in \ \iota(E).
\end{equation}
Furthermore, Galois equivariance is also clear from (\ref{eqn:T-eis-all-the-way}) since we know the behavior of all the objects as we change $\iota.$

\medskip
\subsubsection{\bf Proof of Thm.\,\ref{thm:main-theorem-l-values} in the case when both $n$ and $n'$ are even}
The proof is almost identical, but for some minor variations which we now explain: 
since $n'$ is even, the operator in (\ref{eqn:T-trans-comes-from-here}) simplifies to 
$T_{\rm Tate}(-n', {}^\iota\!\sigma_\infty, -\varepsilon')$. Hence, the operator 
$T_{\rm loc}({}^\iota\!\sigma_\infty, {}^\iota\!\sigma'_\infty)(\varepsilon') \otimes 
T_{\rm loc}(-\tfrac{N}{2}, {}^\iota\!\sigma \times {}^\iota\!\sigma')_f$ in (\ref{eqn:T-loc-to-T-trans}) is already rationally 
defined; in particular, the step taken in (\ref{eqn:T-trans-to-T-arith}) is not needed in this situation, i.e., there is no need to introduce any period. Furthermore, the term $i^{{\sf r}nn'/2}$ as in (\ref{eqn:c-infinity-L-values}) is just $\pm 1$ and so can be absorbed into $\iota(E).$ 
The rest of the argument is the same eventually giving us the result that 
\begin{equation}
\label{eqn:L-value-by-period-even-even}
\frac{L^{\rm coh}\left(\iota, \sigma \times \sigma'^{\v}, {\sf m}_0\right)}
{L^{\rm coh}\left(\iota , \sigma \times \sigma'^{\v}, 1+{\sf m}_0\right)} \ \in \ \iota(E).
\end{equation}
It is worth commenting that if $n = n'$ is even, i.e., if we are in the self-associate situation, then in (\ref{eqn:T_eis}) change $Q$ to $P$ and everything else goes through exactly as above giving the result in 
(\ref{eqn:L-value-by-period-even-even}).

\medskip
\subsubsection{\bf Example of Hilbert modular forms}
\label{sec:example-hilbert}

Let $\bk = (k_\tau)_{\tau : F \to E}$ be a ${\sf r}_F$-tuple of integers. Let 
$k_0 = {\rm max}\{k_\nu\}.$ Let $m \in \Z.$ Consider the weight $\lambda \in X^*(T_2 \times E)$ given by 
$$
\lambda = (\lambda^\tau), \quad \lambda^\tau = (k_\tau -2)\bfgreek{rho}_2 + (-m - \tfrac{k_0}{2})\bfgreek{delta}_2
$$
Then $\lambda$ is dominant if $k_\tau \geq 2$ and is integral if $k_\tau \equiv k_0 \pmod{2},$ and furthermore, by definition, it satisfies the algebraicity condition and clearly $\lambda$ is also pure.  
Let $\sigma_f \in \Coh_{!!}(G_2, \lambda).$ Take $\iota : E \to \C.$ There is a holomorphic Hilbert modular cusp form 
${}^\iota\ff$ of weight $\bk$, such that if $\pi({}^\iota\ff)$ is the associated cuspidal automorphic representation of 
$\GL_2(\A_F)$ then 
$$
{}^\iota\sigma = \pi({}^\iota\ff) \otimes | \ |^{m + \tfrac{k_0}{2}}.
$$
(This can be seen from Raghuram-Tanabe~\cite[Thm.\,1.4]{raghuram-tanabe}.) The cuspidal parameters of 
$\sigma_f$ are $\ell = (\ell^\tau)$ with $\ell^\tau = (\ell^\tau_1, \ell^\tau_2)$ with 
$\ell^\tau_1 = k_\tau - 1$ and $\ell^\tau_2 = -\ell^\tau_1 = 1-k_\tau.$ The motivic weight of $\lambda$ is 
$\w = \w_\lambda = k_0-1.$ The abelian part of $\lambda$ is $d = d_\lambda = -m-k_0/2.$ 

Let $\Mot(\sigma_f)$ be the motive attached to $\sigma_f.$ It's Hodge types are: 
$$
\Hod(\nu, {}^\iota\Mot(\sigma_f)) \ = \ 
\left\{\left(\frac{k_\tau-k_0-2m}{2}, \frac{2-k_\tau -k_0 -2m}{2} \right), 
\left(\frac{2-k_\tau -k_0 -2m}{2}, \frac{k_\tau-k_0-2m}{2} \right) \right\}, 
$$
where $\tau = \iota^{-1}\circ\nu.$ The weight of $\Mot(\sigma_f)$ is $\tilde\w = 2d+n-1 = -2m-k_0+1.$
The effective motive $\Mot_{\rm eff}(\sigma_f)$ attached to $\sigma_f$ has Hodge types: 
$$
\Hod(\nu, {}^\iota\Mot_{\rm eff}(\sigma_f)) \ = \ 
\left\{\left(\frac{k_\tau+k_0-2}{2}, \frac{-k_\tau + k_0}{2} \right), 
\left(\frac{-k_\tau +k_0}{2}, \frac{k_\tau+k_0-2}{2} \right) \right\}, 
$$
which are independent of the abelian part of $\lambda.$ 
The weight of $\Mot_{\rm eff}(\sigma_f)$ is the motivic weight $\w = k_0-1$ of $\lambda$. To compute the critical set, let's denote $k^0 = {\rm min}(k_\tau)$ and observe that 
$$
p_{\rm max} = \frac{k_0 - k^0}{2}, \quad q_{\rm min} = \frac{k_0+k^0}{2}-1.
$$ 
The critical set for $L(\iota, \Mot_{\rm eff}(\sigma_f), s)$ is: 
$$
\left\{ m \in \Z \ : \ \frac{k_0 - k^0}{2} < m < \frac{k_0 + k^0}{2} \right\}.
$$

For $\iota : E \to \C$, we can see that the motivic $L$-function attached to $\iota$ and the 
effective motive $\Mot_{\eff}(\sigma_f)$ is the Hecke-Shimura $L$-function, $L^{\rm Hecke}(s, {}^\iota\ff)$, attached to ${}^\iota\ff$: 
\begin{equation*}
\begin{split}
L(\iota, \Mot_{\eff}(\sigma_f), s) 
& \ = L^{\rm aut}\left(s + \frac{2d-\w}{2}, {}^\iota\sigma_f\right) \ = \ 
L^{\rm aut}_f \left(s + \frac{-k_0-2m-(k_0-1)}{2}, \pi({}^\iota\ff) \otimes  | \ |^{m + \tfrac{k_0}{2}}\right) \\
& \ = \  L^{\rm aut}_f\left(s - \frac{k_0-1}{2}, \pi({}^\iota\ff)\right) \ = \ L^{\rm Hecke}(s, {}^\iota\ff), 
\end{split}
\end{equation*}
where the last equality follows from  \cite[Thm.\,1.4, (1)]{raghuram-tanabe}. The moral is that irrespective of the integer $m$, we always end up with same $L$-function for the effective motive.  

\smallskip

The reader should work out the case of $\GL_2 \times \GL_1$ (resp., $\GL_2 \times \GL_2$)  
explicating our Thm.\,\ref{thm:main-theorem-l-values} using the well-known results of Shimura 
on the special values of the standard $L$-function attached to a Hilbert modular form and a Hecke character of $F$ 
\cite[Thm.\,4.3]{shimura-duke} (resp., the special values of the Rankin--Selberg $L$-function attached to a pair of Hilbert modular forms \cite[Thm.\,4.2]{shimura-duke}). Some more explicit examples may be produced using symmetric power transfers of Hilbert modular forms and special value results in Raghuram~\cite{raghuram-imrn}, \cite{raghuram-preprint}, and the examples in Grobner--Raghuram~\cite[\S\,8]{grobner-raghuram}.

\bigskip
\bigskip

\section*{\bf Appendix 1:  Proof of the combinatorial lemma.}
\begin{center}{\it by Uwe Weselmann}\end{center}

\bigskip

The aim of this appendix is to prove the combinatorial Lem.\,\ref{lem:comb-lemma}. In the following we use the notations of the main article without further comment, but with some minor modifications
explained below. Let us suppose, that the assumptions of Lem.\,\ref{lem:comb-lemma}. are satisfied. So $n$ denotes an even integer, $n'$ an arbitrary positive integer, and $N=n+n'$. For simplicity we will start to work inside the character group $X^*_\Q(T_0\times_\tau E)$  for a fixed embedding $\tau:F\hookrightarrow E$. Since this can and will be identified with $\Q^n =X^*_\Q(T_0)$, we will drop the
$\tau$ from all notations until the final proof. In $X^*_\Q(T_0)$
the half sum of the positive weights on $\GL_n$ will be denoted by $\bfgreek{rho}_n$, the determinant by $\bfgreek{delta}_n=\det_n$
and the $i$-th standard basis vector by $e_i$.
Write $$\mu +\bfgreek{rho}_n = (\bar b_1,\ldots,\bar b_n)\in \Q^n =X^*_\Q(T_0), \quad \text{ i.e. }
\bar b_i\;=\; b_i +\frac{n+1}2-i$$
 with the $b_i$ from Sect.\,\ref{sec:standard-fundamental}. Recall that the $\bfgreek{gamma}_i$, which are extensions of the fundamental weights, may be written in the form
$$\bfgreek{gamma}_i=\sum_{j=1}^{i} e_j-\frac in\cdot \sum_{j=1}^n e_j\qquad\text{ for } i=1,\ldots,n-1.$$
Thus
$$ \mu+\bfgreek{rho}_n= \sum_{i=1}^{n-1} a_i \gamma_i+ d\cdot \bfgreek{delta}_n =
  \sum_{j=1}^{n}\left(\sum_{i=j}^{n-1} a_i\right) e_j + \left(-\sum_{i=1}^{n-1} \frac in a_i + d\right)\cdot \sum_{j=1}^n e_j.$$
If $\mu+\bfgreek{rho}_n$ is essentially self dual, we have $a_{n-i}=a_i$ and thus
$$\sum_{i=1}^{n-1} \frac in a_i=\frac 12\cdot\left(\sum_{i=1}^{n-1} \frac in a_i +\sum_{i=1}^{n-1} \frac in a_{n-i}\right)
= \frac 12\cdot\left(\sum_{i=1}^{n-1} \frac in a_i +\sum_{i=1}^{n-1} \frac {n-i}n a_{i}\right)=
\frac 12\cdot\left(\sum_{i=1}^{n-1}  a_i \right).$$
We deduce $\bar b_j=\beta_j+d$, where
$$ 
\quad \beta_j\quad=\quad \sum_{i=j}^{n-1} a_i- \frac 12\cdot\left(\sum_{i=1}^{n-1}  a_i\right)
\;=\; \frac 12\cdot\left( \sum_{i=j}^{n-1} a_j-\sum_{i=1}^{j-1} a_j\right).
$$
Thus $\beta_j\;=\; \frac{l_j}2$ in the notations of (\ref{eqn:cuspidal-parameter}). The weight being essentially self dual
now means $\beta_{n+1-j}=-\beta_j$. Observe $a_i= \beta_{i}-\beta_{i+1}$ and $\beta_1> \beta_2>\ldots > \beta_n$.

\medskip
Similarly write $\mu'+\bfgreek{rho}_{n'} =(\bar b_1',\ldots,\bar b_{n'}')\in \Q^{n'}$ with
$\bar b_i'=\beta_i'+d'$ and $\beta'_{n'+1-i}=-\beta'_i$.

\medskip
The motivic weights are ${\bf w}=2\beta_1$ and ${\bf w'}=2\beta_1'$. 
The Hodge numbers of the effective motives are
$$(\beta_1+\beta_i,\beta_1-\beta_i)\qquad i=1,\ldots,n\qquad\text{ for } {\bf M}_{\rm eff}$$
$$(\beta_1'+\beta_j',\beta_1'-\beta_j')\qquad j=1,\ldots,n'\qquad\text{ for } {\bf M}_{\rm eff}'.$$
 Then ${\bf M}_{\rm eff}\otimes {\bf M}_{\rm eff}'$ has the Hodge numbers
$$(\beta_1+\beta_1'+\beta_i+\beta_j,\; \beta_1+\beta_1'-\beta_i-\beta_j)\quad\text{ for }
(i,j)\in S ,\quad\text{ where }$$
$$S=\set{1,\ldots, n}\times \set{1,\ldots,n'}. $$
In the following  the characters of the split torus of $\GL_N$ will be written as pairs $\tilde\lambda=(\lambda,\lambda')\in \Z^n\times \Z^{n'}$,
especially we write $\tilde\mu=(\mu,\mu').$
\medskip
We introduce the half integer
$$\tilde p(\tilde\mu) \; := \;\min\set{\beta_i+\beta_j'> 0\; |\; (i,j)\in S}\; =\; 
\min\set{\nu> 0\;|\quad  h^{\beta_1+\beta_1'+\nu,\beta_1+\beta_1'-\nu}\ne 0}.$$  
With the integer $p(\tilde \mu)=\min\{p>\frac{{\bf w}+{\bf w'}}2|\; h^{p,{\bf w}+{\bf w'}-p}\ne 0\}$ from \cite[3]{harder-raghuram} we have $$\tilde p(\tilde\mu) \; =\;  p(\tilde\mu)-\frac{{\bf w}+{\bf w}'}2\; =\;  p(\tilde\mu)- \beta_1-\beta_1'.$$  
 Recall that it is an assumption of Lem.\,\ref{lem:comb-lemma}.,  that  $\sigma_f$ and $\sigma_f'$ have disjoint cuspidal parameters. This is equivalent, by Prop.\,\ref{prop:disjoint-param-middle-hodge}, to the fact that  the middle Hodge number of the tensor motive is zero, which means $\beta_i+\beta_j'\ne 0$ for all $(i,j)\in S$. 
  In view of the symmetries $\beta_{n+1-j}=-\beta_j$ and $\beta'_{n'+1-i}=-\beta'_i$ we thus get
\begin{align}\label{pmin}
  &\tilde p(\tilde\mu)\;=\; \min_{(i,j)\in S} |\beta_i-\beta_j'|\;=\; \min_{(i,j)\in S^+} |\beta_i-\beta_j'|, \quad\text{ where }\\
  \notag S^+\; =\; &\left\{ (i,j)\in S\; |\; 1\le i\le \frac {n+1}2,\; 1\le j\le \frac{n'+1}2\right\}
      =\left\{ (i,j)\in S\, |\; \beta_i\ge 0,\beta_j'\ge 0\right\}.
\end{align}
\medskip

We have
$$ \tilde\mu+\bfgreek{rho}_N\; =\; (\mu,\mu')+\bfgreek{rho}_N\; =\;  \left(\mu+\bfgreek{rho}_n+\frac {n'}2\cdot \bfgreek{delta}_n,\;
\mu'+\bfgreek{rho}_{n'}-\frac n2\cdot \bfgreek{delta}_{n'}\right),$$
and if we write $\tilde\mu+\rho_N=(\tilde b_1,\ldots,\tilde b_N)$, then we get
\begin{align*} \tilde b_i= \bar b_i+\frac{n'}2 &=\beta_i + \tilde d\qquad\;\; \text{  with } \tilde d=d+\frac{n'}2\quad
\text{ for }1\le i\le n \quad\text{ and }\\
\tilde b_{j+n}= \bar b_{j}'-\frac n2 &= \beta_{j}'+\tilde d' \qquad \text{ with }\tilde d'=d'-\frac n2 \quad\text{ for }
1\le j \le n'.
\end{align*}

Let $$\tilde a(\tilde\mu) = (d-d')+\frac{n+n'}2=\tilde d -\tilde d'.$$

\begin{lemma}\label{Permutation} If we have $|\tilde a(\tilde \mu)|=|\beta_i-\beta_j'|$ for some $(i,j)\in S^+$, then there does not exist $w\in W^P$ such that
$\tilde\mu+\bfgreek{rho}_N$ is of the form $w(\lambda+\bfgreek{rho}_N)$ with a dominant weight $\lambda$.
\end{lemma}
\begin{proof} If $|\tilde a(\tilde\mu)|=|\beta_i-\beta_j'|$ then either $\tilde d-\tilde d'=\beta_i-\beta_j'$ which is equivalent to
  $\tilde b_{n+1-i}=\tilde d-\beta_i=\tilde d'-\beta_j'=\tilde b_{N+1-j}$ or we have $\tilde d'-\tilde d=\beta_i-\beta_j'$, which is equivalent
  to $ \tilde b_i=\tilde b_{n+j}$. Thus at least two entries of $\tilde\mu+\bfgreek{rho}_N$ coincide.
  But if $\lambda$ is dominant, then  $\lambda+\bfgreek{rho}_N$ has strictly decreasing entries, and
  all entries of  $w(\lambda+\bfgreek{rho}_N)$ are different for every $w\in W^P$.
  \end{proof}

\smallskip

\begin{lemma}\label{Laenge}  Assume  $|\tilde a(\tilde\mu)|\ne |\beta_i-\beta_j'|$ for all $(i,j)\in S^+$.
Then there exists a unique $w\in W^P$, such that $\tilde\mu+\bfgreek{rho}_N=w(\lambda+\bfgreek{rho}_N)$
with a dominant weight $\lambda\in \Z^N$.
\begin{itemize}
\item[(a)] If \quad $|\tilde a(\tilde\mu)|< \tilde p(\tilde\mu)$, \quad  then  \quad  $l(w)=\frac {dim(U_P)}2$;
\item[(b)] If \quad  $\tilde a(\tilde\mu)> \tilde p(\tilde\mu)$, \quad  then  \quad  $l(w)<\frac {dim(U_P)}2$;
\item[(c)] If \quad  $\tilde a(\tilde \mu)< -\tilde p(\tilde\mu)$, \quad  then  \quad  $l(w)>\frac {dim(U_P)}2$.
\end{itemize}
\end{lemma}

\begin{proof} 
The assumption implies, by the arguments of  Lem.\,\ref{Permutation}, that the entries of $\tilde\mu+\bfgreek{rho}_N$ are pairwise different. Therefore there exists a unique permutation $w\in W$ such that $w^{-1}(\tilde\mu+\bfgreek{rho}_N)$ is regular dominant (i.e. the entries form a strictly decreasing sequence). Since the entries differ by integers, we get an equation $\tilde\mu+\bfgreek{rho}_N=w(\lambda+\bfgreek{rho}_N)$ with a dominant $\lambda$. Since we already know $\tilde b_1>\ldots>\tilde b_n$ and $\tilde b_{n+1}>\ldots>\tilde b_N$, we have $w\in W^P$.

\medskip
Since the Weyl group $W$ is the symmetric group $W=S_N$, the length  can be computed by the formula:
$$ l(w)\quad =\quad \# \set{(i,j)\in \N\times\N\; |\;  1\le i<j\le N,\quad  \tilde b_i< \tilde b_j}.$$
But since the first $n$ entries  are already in a decreasing order as are the last $n'$ entries, $l(w)$ is the cardinality of the set
$$ S_<\quad =\quad \set{ (i,j)\in S\; |\;  \tilde b_i< \tilde b_{n+j}},\qquad \text{ where } S=  \set{1,\ldots,n}\times \set{1,\ldots,n'}.$$
Observe  $\#(S)= n\cdot n'=dim(U_P)$. To calculate $l(w)$ we will determine the intersections $Q_{i,j}\cap S_<\;\; $,  where
$$ Q_{i,j}=\set{(i,j),\; (n+1-i,j),\; (i,n'+1-j),\; (n+1-i,n'+1-j)}$$
for $(i,j)\in S^+$. Observe that $S=\bigcup_{(i,j)\in S^+} Q_{i,j}$ is a disjoint partition in subsets of order $4$
(for $j<\frac{n'+1}2$ ) respectively of order $2$ for $j=\frac{n'+1}2$. Observe that $i\ne \frac{n+1}2$ for all $i$, since $n$ is assumed to be even.

\medskip
(a) If $|\tilde a(\tilde\mu)|< \tilde p(\tilde\mu)$ we have $|\tilde d-\tilde d'| < |\beta_i-\beta_j'|$ for all $(i,j)\in S^+$.
   In the case $\beta_i>\beta_j'\ge 0$ this means
   $$ \beta_j'-\beta_i< \tilde d-\tilde d'< \beta_i-\beta_j'\qquad \text{ and consequently }$$
   $$ \tilde d-\beta_i=\tilde b_{n+1-i}\; <\; \tilde b_{n+n'+1-j}=\tilde d'-\beta_j' \;\le\;
    \tilde b_{n+j}= \tilde d'+ \beta_j'\; <\;  \tilde b_i    =\tilde d+\beta_i .$$
  Therefore  $Q_{i,j}\cap S_<\; =\; \{(n+1-i,j),(n+1-i,n'+1-j)\}$.
  Thus $\#(Q_{i,j}\cap S_<)= \frac 12 \cdot \#(Q_{i,j}).$  (This means $2=\frac 12\cdot 4$ in the case
  $j<n'+1-j$ and $1=\frac 12\cdot 2$ in the case $\beta_j=0  \Leftrightarrow j=n'+1-j$.)
\smallskip
   In the case $\beta_j'>\beta_i> 0$  we get similarly
   $$ \beta_i-\beta_j'< \tilde d-\tilde d'< \beta_j' -\beta_i\quad \text{ and therefore }\quad
    \tilde b_{n+n'+1-j} < \tilde b_{n+1-i} < \tilde b_{i} < \tilde b_{n+j}.$$
   Thus we have $Q_{i,j}\cap S_<\; =\; \{(i,j),(n+1-i,j)\}$ in this case, and again 
   $\#(Q_{i,j}\cap S_<)=\frac 12\cdot \#(Q_{i,j})$.
   Summing up over all indices we get
   $$l(w)=\#(S_<)= \sum_{(i,j)\in S^+} \#(Q_{i,j}\cap S_<)= \frac 12 \sum_{(i,j)\in S^+} \#(Q_{i,j})
   =\frac 12 \#(S)= \frac 12 \dim(U_P). $$

\medskip
  (b)  In this case we have $\tilde d -\tilde d' > | \beta_i-\beta_j'|$ for at least one
  pair $(i,j)\in S^+$. Consequently  $\tilde d -\tilde d' > \beta_i-\beta_j'$ and $\tilde d -\tilde d' > \beta_j'-\beta_i$.  This implies
  $\tilde b_i> \tilde b_{n+j}\ge \tilde b_{n+n'+1-j}$ and $\tilde b_{n+1-i}> \tilde b_{n+n'+1-j}$. Thus
  $\#(Q_{i,j}\cap S_<) \subset\{(n+1-i,j)\}$ in case $\beta_j'>0$ and $\#(Q_{i,j}\cap S_<)=\emptyset$ in case $\beta_j'=0$.
  In both cases we get   $$\#(Q_{i,j}\cap S_<)\quad <\quad \frac 12\cdot \#(Q_{i,j}).$$
  For those pairs $(i,j)\in S^+ $ for which we do not have $\tilde d -\tilde d' >|\beta_i-\beta_j'|$, we must have
  $\tilde d -\tilde d'=|\tilde d -\tilde d'|<|\beta_i-\beta_j'|$ and get $\#(Q_{i,j}\cap S_<)=\frac 12\cdot \#(Q_{i,j})$ by the
  calculations in (a).
  Summing up we conclude  $l(w)=\#(S_<)< \frac 12 \#(S)= \frac 12 \dim(U_P)$.

  \medskip
  (c) This case is dual to (b): If we have $\tilde d-\tilde d'< -| \beta_i-\beta_j'|$ for some pair $(i,j)\in S$, then we get
   $Q_{i,j}\cap S_<\; \supset\;  \{( (i,j),(n+1-i,n'+1-j),(n+1-i,j)\}$, and thus
   $$\#(Q_{i,j}\cap S_<)\quad >\quad \frac 12\cdot \#(Q_{i,j}).$$
   Since for all other pairs we have $\#(Q_{i,j}\cap S_<)\ge \frac 12\cdot \#(Q_{i,j})$ by the same reasoning respectively by (a), we conclude
   $l(w)=\#(S_<)>\frac 12 \#(S)= \frac 12 \dim(U_P)$.
   \end{proof}
   
\medskip

\begin{prop}  The existence of  $w\in W^P$ of length $l(w)=\frac {dim(U_P)}2$, such that $\tilde\mu+\bfgreek{rho}_N=w(\lambda+\bfgreek{rho}_N)$
   with a dominant weight $\lambda\in \Z^N$, is equivalent to  
   $$-\tilde p(\tilde\mu)+1\;\le\; \tilde a(\tilde\mu)\;\le\; \tilde p(\tilde\mu)-1.$$
\end{prop}

\begin{proof} This is an immediate consequence of the two lemmas, once we have proved that $|\tilde a(\tilde\mu)|< \tilde p(\tilde\mu)$ is equivalent to the inequality in the proposition.
But this follows from the fact that the difference between the half integers $\tilde a(\tilde\mu)$ and $\tilde p(\tilde\mu)$ is an integer: Observe $2d=2\bar b_1-2\beta_1=2 b_1+n-1-{\bf w}\equiv n-1-{\bf w}$ mod $2$, which implies
 $2\tilde a(\tilde\mu)= 2d-2d'+N \equiv n-1-{\bf w}-(n'-1-{\bf w'})+n+n'\equiv {\bf w}+{\bf w'}$ mod $2$. Furthermore $2\tilde p(\tilde\mu)= 2 p(\tilde\mu)-({\bf w}+{\bf w}')\equiv {\bf w}+{\bf w}'$ mod
 $2$, since $p(\tilde\mu)$ is  an integer. Consequently $2\tilde a(\tilde\mu)\equiv 2\tilde p(\tilde\mu)$ mod $2$.
 \end{proof}

\medskip
 \begin{cor} The conditions (1) and (2) of Lem.\,\ref{lem:comb-lemma}. are equivalent.
 \end{cor}

\begin{proof} One has to apply the proposition to all embeddings $\tau:F\to E$. 
 Write $\tilde\mu =(\mu,\mu')=(\tilde\mu^\tau)_{\tau:F\to E}\in \prod_\tau X^*(T_0\times_\tau E)$. Condition
 (1) means (compare 5.3.7.) that for each embedding $\tau: F\to E$ there exists $w^\tau\in W^P$ of length $\dim(U_{P_0})/2$ such that
 $(w^{\tau})^{-1}(\tilde\mu^\tau+\bfgreek{rho}_N)-\bfgreek{rho}_N$ is dominant.
 
  After adding $\frac N2$ to the inequality (2) it reads:
 ${-\frac{\ell(\mu,\mu')}2+1\le \tilde a(\tilde\mu)\le \frac{\ell(\mu,\mu')}2-1}$.
 The relations $\beta_i^\tau=\frac{ \ell_i^\tau}2$ imply, in view of (\ref{pmin}) and Def.\,\ref{def:cuspidal-width}, that 
$\frac{\ell(\mu,\mu')}2 =\min_{\tau:F\to E}\; \tilde p(\tilde\mu^\tau),$  so that the condition (2) is equivalent to
 $$-\tilde p(\tilde\mu^\tau)+1\;\le\; \tilde a(\tilde\mu)\;\le\; \tilde p(\tilde\mu^\tau)-1\qquad\text{ for all }\tau:F\to E.$$
\end{proof}

\medskip

 \begin{cor}
 Conjecture 5.1. of \cite{harder-raghuram} is true.
 \end{cor}
 
\begin{proof} This is clear from the proposition and the preceeding corollary. Observe that we have $\tilde a(\tilde \mu)= a(\tilde \mu)+\frac N2$, where $a(\tilde\mu)=d-d'$ in the notations of \cite{harder-raghuram}.
 \end{proof}

\medskip

\begin{rem}{\rm 
It should be noted that the sets $\Phi=\set{\beta_1,\ldots,\beta_n}$ and $\Phi'=\set{\beta_1',\ldots,\beta_n'}$ are characteristic sets in the sense of \cite{weselman}:
The set $\Phi$ is either of type $B_*$ (half integers) or of type $D_*$ (integers).
If $n=2g$, then the weight satisfies ${\bf w}=2\beta_1 \equiv a_g$ mod $2$. Thus for $a_g$ even we have a characteristic set of type $D_g$, and for $a_g$ odd the characteristic set is of type $B_g$.}
\end{rem}

\medskip
If $n'$ is odd, the set $\Phi'$ is of type $C_*$ (i.e., $0\in \Phi'$). If $n'$ is even it may be of type $B_*$ or of type $D_*$. 
If $n'$ is odd and if the representations $\sigma$ and $\sigma'$ are self equivalent, then the central
Eisenstein class in Theorem 5.2 of \cite{harder-raghuram} (i.e., where the usual Rankin--Selberg $L$-function is evaluated at $s=0$ and at $s=1$) is a twisted lift of an endoscopic $L$-packet on $\Sp_{N-1}$ by the results of \cite{weselman}.

\bigskip

{U{\tiny WE} W{\tiny ESELMAN}:}  Mathematisches Institut, Im Neuenheimer Feld 288,  D-69121 Heidelberg

{\it Email address:} weselman@mathi.uni-heidelberg.de

\bigskip
\bigskip

\section*{\bf Appendix 2: The case of $\GL_n \times \GL_{n'}$ when both $n$ and $n'$ are odd}
\begin{center}{\it by Chandrasheel Bhagwat and A. Raghuram}\end{center}

\bigskip

After the announcement \cite{harder-raghuram}, the authors of this appendix proved certain period relations 
in \cite{bhagwat-raghuram} for the periods of a tensor product $\Mot \otimes \Mot'$ of two pure motives of with ranks $n$ and $n'.$ If the parities of $n$ and $n'$ are different then these period relations together with Deligne's conjecture \cite{deligne} exactly predicts the main result on $L$-values Thm.\,\ref{thm:main-theorem-l-values}, (1). It was already observed by Yoshida 
\cite{yoshida-ajm} that if $n$ and $n'$ are both even, then the two periods $c^\pm$ of Deligne for the tensor product motive are equal; this predicts Thm.\,\ref{thm:main-theorem-l-values}, (2). In the case when both $n$ and $n'$ are odd, 
in \cite{bhagwat-raghuram}, an interesting period relation for the ratio $c^+/c^-$ for $\Mot \otimes \Mot'$ is proved; this suggests that there should be a result analogous to Thm.\,\ref{thm:main-theorem-l-values}, however, it turns out that the methods of the paper break down in this situation. 
The aim of this second appendix is to discuss the combinatorial lemma (Lem.\,\ref{lem:comb-lemma}) in the situation when the Levi subgroup is $\GL_n \times \GL_{n'}$ with both $n$ and $n'$ being odd positive integers. We will see that certain simple combinatorial conditions on the highest weights $\mu$ and $\mu'$ assure us of the existence of two successive critical points, however, the same conditions will give that there is no element of the Weyl group which makes the character $\mu \otimes \mu'$ dominant. This is already seen when the base field $F = \Q$, and so we will work over $\Q$ for notational simplicity. 

 \begin{prop} 
 \label{prop:appendix-2}
Let $\sigma \in \Coh(\GL_n,\mu)$ and $\sigma' \in \Coh(\GL_{n'},\mu')$
for pure weights $\mu$ and $\mu'$ respectively. Assume that $n$ and $n'$ are odd positive integers. 
Let $\epsilon, \epsilon'  \in \{0,1\}$ be such that $\sigma_\infty = \D_\mu \otimes {\rm sgn}^\epsilon$ and 
$\sigma'_\infty = \D_{\mu'} \otimes {\rm sgn}^{\epsilon'}.$ Put $\epsilon_0 = \epsilon + \epsilon' \pmod{2}.$ 
Let $\ell_i$ and $\ell'_j$ denote the cuspidal parameters as in (\ref{eqn:cuspidal-parameter}); recall that all the 
$\ell_i$ and $\ell'_j$ are even, and $\ell_{(n+1)/2} = \ell'_{(n'+1)/2} = 0;$ and because of the last property, we modify the definition the cuspidal width as 
$$
\ell^+(\mu, \mu') \ := \ \min\{ |\ell_i - \ell'_j| \ : \ 1 \leq i \leq (n-1)/2, 1 \leq j \leq (n'-1)/2\}. 
$$
Then we have: 
\begin{enumerate}
\item If $\ell^+(\mu, \mu') = 0$ then $L(s, \sigma \times \sigma'^{\sf v})$ has no critical points. 

\medskip

\item If $\ell := \ell^+(\mu, \mu') > 0$ and $\epsilon = \epsilon'$, i.e., $\epsilon_0 = 0,$ then the critical set of integers for 
$L(s, \sigma \times \sigma^{\sf v})$ is given by: 
{\small 
$$
\left\{(d-d') - [\tfrac{\ell}{2}]_e - 1, \ \dots \ , (d-d')-3, (d-d')-1 \ ; \  (d-d')+2, (d-d')+4, \ \dots \ , (d-d')+ [\tfrac{\ell}{2}]_e \right\}, 
$$}
\noindent where $[\ell/2]_e$ is the largest even integer less than or equal to $\ell/2.$ 

\medskip

\item   If $\ell = \ell^+(\mu, \mu') > 0$ and $\epsilon \neq \epsilon'$, i.e., $\epsilon_0 = 1,$ then the critical set of integers for $L(s, \sigma \times \sigma^{\sf v})$ is  given by: 
{\small 
$$
\left\{(d-d') - [\tfrac{\ell}{2}]_o - 1, \ \dots \ , (d-d')-2, (d-d') \ ; \  (d-d')+1, (d-d')+3, \ \dots \ , (d-d')+ [\tfrac{\ell}{2}]_o \right\}, 
$$}
\noindent where $[\ell/2]_o$ is the largest odd integer less than or equal to $\ell/2.$

\medskip

\item The points $ - \tfrac{N}{2}$ and $1 - \tfrac{N}{2}$ are critical for $L(s,~\sigma \times\sigma'^{\sf v})$ 
if and only if  $\ell^+(\mu, \mu') > 0$, $\epsilon \neq \epsilon'$ and $d - d' = -\tfrac{N}{2}$.
 
 \medskip
 
 \item Suppose $d - d' = -N/2$ then there does not exist an element $w \in W$ such that $w^{-1} \cdot (\mu \otimes \mu')$ is dominant.
 
 \end{enumerate}
 \end{prop}

\begin{proof}
For the proofs of (1), (2) and (3), recall that the critical set for the Rankin-Selberg 
$L$-function $L(s, \sigma \times \sigma'^{\sf v})$ is the set of all integers $m$ such that for both $L_{\infty}(s, \sigma \times \sigma'^{\sf v})$ and $L_{\infty}(1-s, \sigma^{\sf v} \times \sigma')$ are regular at $s = m$.
The Langlands parameters for $\sigma_\infty$ and $\sigma'_\infty$ are given by
$$
\rho(\sigma) = \bigoplus \limits_{i=1}^{(n-1)/2} I(\ell_i)\otimes |~|^{-d} \ \oplus \ \text{sgn}^{\epsilon} \otimes |~|^{-d}, 
\ \ {\rm and} \ \ 
\rho(\sigma') = \bigoplus \limits_{j=1}^{(n'-1)/2} I(\ell'_j) \otimes |~|^{-d'} \ \oplus \ \text{sgn}^{\epsilon'} \otimes |~|^{-d'}. 
$$
If $\ell^+(\mu, \mu') > 0$ then the local $L$-factor at infinity for $\sigma \times \sigma'^{\sf v}$ is given by: 
{\small 
\begin{eqnarray*}
L_{\infty}(s, \sigma \times \sigma'^{\sf v}) & = &  
\Gamma_{\R} \left(s + \epsilon_{0} + d'-d\right)
\prod \limits_{i=1}^{(n-1)/2}
 \Gamma_{\C} \left(s + d' - d + \frac{\ell_i}{2}\right)
 \prod \limits_{j=1}^{(n'-1)/2} \Gamma_{\C}\left(s + d' - d + \frac{\ell'_j}{2}\right) \\
&& \cdot
\prod \limits_{\substack{1 \leq i \leq (n-1)/2 \\ 1 \leq j \leq (n'-1)/2}} 
\Gamma_{\C} \left(s + d' - d + \frac{\ell_i + \ell'_j}{2}\right)~  \Gamma_{\C} \left(s + d' - d + \frac{|\ell_i - \ell'_j|}{2}\right). \end{eqnarray*}}
Now, $L_{\infty}(s, \sigma \times \sigma'^{\vee})$ is regular at $s = m$ if and only if for all indices $1 \leq i \leq (n+1)/2$,~ $1 \leq j \leq (n'+1)/2$ (except when $i = (n+1)/2,~ j = (n'+1)/2$), we have:
\begin{equation}
\label{eqn:odd-odd-1} 
\begin{split}
m - (d-d') + |\ell_i - \ell'_j|/2 \ & \geq \ 1, \\
m - (d-d') + \epsilon_{0}  \ & \notin \ \{0, -2, -4, -6, \dots \}.
\end{split}
 \end{equation}
Similarly, the factor $L_{\infty}(1-s, ~\sigma^{\sf v} \times \sigma')$ is regular at $s = m$ if and only if for all indices $1 \leq i \leq (n+1)/2$,~ $1 \leq j \leq (n'+1)/2$ (except when $i = (n+1)/2, ~j = (n'+1)/2$), we have:
 \begin{equation} 
 \label{eqn:odd-odd-2}
 \begin{split}
 -m + (d-d') + |\ell_i - \ell'_j|/2 \ & \geq \ 0, \\
  1-m + d-d' + \epsilon_{0} \ & \notin \ \{0, -2, -4, -6, \dots \}. 
\end{split}
\end{equation}
We leave it to the reader to check that (2) and (3) follow from (\ref{eqn:odd-odd-1}) and (\ref{eqn:odd-odd-1}). For (1), 
suppose $\ell^+(\mu,\mu') = 0$ then $\ell_i = \ell'_j$ for some $i,j$; the local factor at infinity for 
$\sigma \times \sigma'^{\sf v}$ will have factors $\Gamma((s+d-d')/2)\Gamma((s+d-d'+1)/2)$ and it is easy to see then that there are no critical points. Also, (4) easily follows from (1), (2) and (3). 

\medskip
 
We now prove (5). Suppose $\mu = \left( \mu_1, \mu_2, \ldots, \mu_{n} \right)$ and $\mu' = \left( \mu'_1, \mu'_2, \ldots, \mu'_{n'} \right)$ with non-increasing sequences $\mu_i$ and $\mu'_j.$ By definition, we have 
$$
\mu \otimes \mu' = \left( \mu_1,~\mu_2,~\ldots,~\mu_{n},~ \mu'_1,~\mu'_2,~\ldots,~\mu'_{n'} \right).
$$ 
Let 
$\bfgreek{rho}_{N} = \left( \frac{N-1}{2}, ~ \frac{N-3}{2},~ \ldots, ~ \frac{1-N}{2} \right)$ be the half sum of positive roots. It is clear that 
$$
\left(\mu \otimes \mu' + \bfgreek{rho}_{N} \right)_{j} =  
\begin{cases} \mu_j + \frac{N+1}{2} - j & \quad \text{if} \quad 1 \leq j \leq n, \\
  \mu'_{j-n} + \frac{N+1}{2} - j & \quad \text{if} \quad 1+n \leq j \leq N.  
\end{cases}
$$

By definition of the twisted action of the Weyl group we have $w\cdot(\mu \otimes \mu') : = 
w(\mu \otimes \mu' + \bfgreek{rho}_{N}) - \bfgreek{rho}_{N}$. The weight $w^{-1} \cdot(\mu \otimes \mu')$ is dominant if and only if 
$$ 
(w^{-1} \cdot (\mu \otimes \mu'))_{j} \ \geq \ (w^{-1}\cdot(\mu \otimes \mu'))_{j+1} \quad \forall ~j \geq 1.
$$
Equivalently, 
 $$ 
 (w^{-1}(\mu \otimes \mu' + \bfgreek{rho}_N))_{j} \ \geq \ (w^{-1}(\mu \otimes \mu' + \bfgreek{rho}_N))_{j+1} + 1 
 \quad \forall ~j \geq 1. 
 $$
In particular, for $w^{-1} \cdot(\mu \otimes \mu')$ to be dominant, it is necessary that 
$(\mu \otimes \mu' + \bfgreek{rho}_N)_{j}$ are all distinct as $1 \leq j \leq N$. However,  observe that 
$$ 
(\mu \otimes \mu' + \bfgreek{rho}_N)_{\frac{n+1}{2}}  \ = \  d + n'/2  \ = \ d' - n/2 \ = \ 
(\mu \otimes \mu' + \bfgreek{rho}_N)_{n + \frac{n'+1}{2}}, 
$$
since $\mu_{(n+1)/1} = d$, $\mu'_{(n'+1)/1} = d'$, and by hypothesis $d - d' = -N/2.$ 
In other words, there is no element $w$ of the Weyl group $W$ so that the weight $w^{-1} \cdot (\mu \otimes \mu')$ is  dominant.
\end{proof}

If we hope to use the theory of Eisenstein cohomology to prove a result on successive $L$-values, it is clear from the main body of this article that we need the points $-N/2$ and $1-N/2$ to be critical. However, in the situation when $n$ and $n'$ are odd, the conditions (given in part (4) of the above proposition) which ensure criticality of these two points also say (in part (5)) that the relevant induced representation does not occur in the boundary cohomology of $\GL_N$ for any coefficient system $\lambda$; because, if there were such a dominant integral weight $\lambda$, then we would need $\mu \otimes \mu' = w \cdot \lambda$ for some Kostant representative $w \in W^P$, but then this would mean $w^{-1} \cdot (\mu \otimes \mu')$ is dominant contradicting (5) above.


 \end{document}